\pdfoutput=1
\documentclass[aos,preprint]{imsart}
\usepackage[square,sort,comma,numbers]{natbib}
\usepackage{xr}
\usepackage{enumitem}
\usepackage{xspace}
\usepackage{tikz}

\usepackage{textcase}
\usepackage{bm}
\usepackage{etoolbox}
\makeatletter
\patchcmd{\specialsection}{\MakeUppercase}{\MakeTextUppercase}{}{}
\patchcmd{\specialsection}{\MakeUppercase}{\MakeTextUppercase}{}{}
\makeatother

\RequirePackage[OT1]{fontenc}
\RequirePackage{amsthm,amsmath}
\RequirePackage[numbers]{natbib}

\usepackage{url}% for url's in bib
\usepackage[colorlinks=True,citecolor=blue,urlcolor=blue,pagebackref=true,backref=true]
{hyperref}
\usepackage[capitalize,noabbrev]{cleveref} % Load cleveref after hyperref 

\crefname{equation}{equation}{equations}

\renewcommand*{\backrefalt}[4]{%
    \ifcase #1 \footnotesize{(Not cited.)}%
    \or        \footnotesize{(Cited on page~#2.)}%
    \else      \footnotesize{(Cited on pages~#2.)}%
    \fi}

\usepackage{xr}

\makeatletter
\newcommand*{\addFileDependency}[1]{% argument=file name and extension
  \typeout{(#1)}
  \@addtofilelist{#1}
  \IfFileExists{#1}{}{\typeout{No file #1.}}
}
\makeatother

\usepackage{graphics}
\usepackage{graphicx}
\usepackage{psfrag}
\usepackage{enumitem}

\usepackage{color}
\usepackage{float}

\usepackage{amssymb,bbm}
\usepackage[citecolor=blue]{hyperref}
\usepackage{nicefrac}

 \usepackage{tabularx}%

\usepackage{enumitem}
\usepackage{booktabs}
\usepackage{caption,subcaption}

\usepackage{mathtools}

\usepackage{graphicx}
\graphicspath{{figs/}}
\newtheorem{theos}{Theorem}
\newtheorem{props}{Proposition}
\newtheorem{lems}{Lemma}
\newtheorem{cors}{Corollary}

\newlength{\widebarargwidth}
\newlength{\widebarargheight}
\newlength{\widebarargdepth}

\makeatletter
\long\def\@makecaption#1#2{
        \vskip 0.8ex
        \setbox\@tempboxa\hbox{\small {\bf #1:} #2}
        \parindent 1.5em  %% How can we use the global value of this???
        \dimen0=\hsize
        \advance\dimen0 by -3em
        \ifdim \wd\@tempboxa >\dimen0
                \hbox to \hsize{
                        \parindent 0em
                        \hfil
                        \parbox{\dimen0}{\def\baselinestretch{0.96}\small
                                {\bf #1.} #2
                                }
                        \hfil}
        \else \hbox to \hsize{\hfil \box\@tempboxa \hfil}
        \fi
        }
\makeatother

\long\def\comment#1{}

\usepackage{xcolor}
\definecolor{OliveGreen}{cmyk}{0.91,0,0.88,0.2} % PANTONE 349

\newenvironment{carlist}
 {\begin{list}{$\bullet$}
 {\setlength{\topsep}{0in} \setlength{\partopsep}{0in}
  \setlength{\parsep}{0in} \setlength{\itemsep}{\parskip}
  \setlength{\leftmargin}{0.15in} \setlength{\rightmargin}{0.08in}
  \setlength{\listparindent}{0in} \setlength{\labelwidth}{0.08in}
  \setlength{\labelsep}{0.1in} \setlength{\itemindent}{0in}}}
 {\end{list}}

\newcommand{\bcar}{\begin{carlist}}
\newcommand{\ecar}{\end{carlist}}

\newcommand{\err}{\ensuremath{\epsilon}}

\newcommand{\mydefn}{: \, =}

\newcommand{\noisesd}{\sigma}

\DeclareMathOperator{\tr}{trace}
\DeclareMathOperator{\Exs}{\mathbb{E}}

\DeclareMathOperator{\Cov}{\mathrm{Cov}}

\newcommand{\smallop}[1]{{\operatorname{#1}}}

\newcommand{\indistrb}{\ensuremath{ \stackrel{\smallop{d}}{\longrightarrow}}}

\newcommand{\indist}{\indistrb}

\newcommand{\inEllOne}{\ensuremath{\stackrel{L_1}{\longrightarrow}}}
\newcommand{\almostSurely}{\ensuremath{\stackrel{\smallop{a.s.}}{\longrightarrow}}}

\newcommand{\y}{\ensuremath{y}}
\newcommand{\x}{\ensuremath{\mathbf{x}}}

\newcommand{\filtration}{\ensuremath{\mathcal{F}}}
\newcommand{\real}{\ensuremath{\mathbb{R}}}
\newcommand{\X}{\ensuremath{\mathbf{X}}}
\newcommand{\W}{\ensuremath{\mathbf{W}}}
\newcommand{\w}{\ensuremath{\mathbf{w}}}

\newcommand{\error}{\ensuremath{\mathbf{\epsilon}}}

\newcommand{\Id}{\ensuremath{\mathbf{I}}}

\newcommand{\V}{\ensuremath{\mathbf{V}}}
\newcommand{\inprod}[2]{\ensuremath{\big\langle #1, #2 \big \rangle}}

\newcommand{\myinprod}[2]{\inprod{#1}{#2}}

\newcommand{\Ncal}{\mathcal{N}}

\newcommand{\BigoP}{\ensuremath{O_p }}
\newcommand{\Bigo}{\ensuremath{O }}

\newcommand{\smalloP}{\ensuremath{o_p}}

\newcommand{\thetabf}{\ensuremath{\boldsymbol{\theta}}}

\newcommand{\thetaDecorr}{\ensuremath{ \widehat{ \boldsymbol{\theta}}_{\scriptscriptstyle{\operatorname{OD}}} }}
\newcommand{\thetaLS}{\ensuremath{ \widehat{ \boldsymbol{\theta}}_{\scriptscriptstyle{\operatorname{LS}}} }}

\newcommand{\z}{\ensuremath{\mathbf{z}}}
\newcommand{\Z}{\ensuremath{\mathbf{Z}}}
\newcommand{\tuneParScaled}{\gamma}

\newcommand{\thetastar}{{\ensuremath{\boldsymbol{\theta}^*}}}
\newcommand{\thetastarScalar}{\ensuremath{\theta}^*}

\newcommand{\Prob}{\ensuremath{\mathbb{P}}}

\newcommand{\Dim}{\ensuremath{d}}

\newcommand{\XscaleMat}{\ensuremath{\mathbf{\Gamma}}}

\newcommand{\DelMat}{\ensuremath{\Delta}}

\newcommand{\SigmaLim}{\ensuremath{\mathbf{B}}}
\newcommand{\SigmaMat}{\ensuremath{\mathbf{S}}}

\newcommand{\e}{\ensuremath{\mathbf{e}}}

\newcommand{\RNum}[1]{\uppercase\expandafter{\romannumeral #1\relax}}

\newcommand{\diag}{\mathrm{diag}}
\newcommand{\SigmaMatLB}{\ensuremath{\mathbf{L}}}
\newcommand{\MaxVal}{\ensuremath{C}}

\newcommand{\ExsXsqLB}{\ensuremath{\mathbf{G}}}
\newcommand{\wiener}{\ensuremath{\mathrm{w}}}

\newcommand{\ceil}[1]{\ensuremath{\left \lceil #1 \right \rceil}}

\newcommand{\basis}{\ensuremath{\mathbf{e}}}
\newcommand{\armMeans}{\thetastar}
\newcommand{\armCount}{\ensuremath{N}}

\newcommand{\greedyEps}{\ensuremath{\varepsilon}}

\newcommand{\A}{\ensuremath{\mathbf{A}}}
\newcommand{\B}{\ensuremath{\mathbf{B}}}
\newcommand{\errorstd}{\ensuremath{\sigma}}

\def\E{{\Exs}} %Exs is a really weird choice 
\newcommand{\dir}{{\mathbf{v}}}
\def\<{{\langle}}
\def\>{{\rangle}}

\newcommand{\selection}{\ensuremath{\psi}}
\newcommand{\Selections}{\Psi}

\newcommand{\reals}{{\mathbb{R}}}
\newcommand{\normal}{{\mathcal{N}}}
\newcommand{\thetahat}{{\widehat{\boldsymbol{\theta}}}}

\newcommand{\CI}{{\mathcal{I}}}

\newcommand{\thetaLSScalar}{\theta_{\scriptscriptstyle{\operatorname{LS}}} }
\newcommand{\thetaDecorrScalar}{ \widehat{\theta}_{\scriptscriptstyle{\operatorname{OD}}} }
\newcommand{\armCountMat}{\mathbf{S}}
\def\mahal{{\mathbf{M}}}

\def\Unif{{\mathrm{Unif}}}
\def\sigmahat{{\widehat{\sigma}}}

\newcommand{\parDelta}{\Delta}

\newenvironment{talign*}
 {\csname align*\endcsname} {\endalign}

\newcommand{\qtext}[1]{\quad\text{#1}\quad}

\newcommand{\numobs}{\ensuremath{n}}

\newcommand{\NormalQuantile}[1]{z_{#1}}

\newcommand{\mompar}{\ensuremath{\Delta}}

\newcommand{\mymatrix}[1]{\ensuremath{\mathbf{#1}}}
\newcommand{\myvector}[1]{\ensuremath{\mathbf{#1}}}

\newcommand{\Bias}{\ensuremath{\myvector{b}}}
\newcommand{\ZeroMartin}{\ensuremath{\myvector{v}}}

\newcommand{\FullNseq}[5]{\ensuremath{ \{ #1_{#2} \}_{#3=#4}^{#5}}}
\newcommand{\Nseq}[2]{\FullNseq{#1}{#2}{#2}{1}{\numobs}}

\newcommand{\NseqNosub}[2]{ \{#1\}_{#2=1}^\numobs}

\newcommand{\Nsum}[1]{\ensuremath{ \sum_{#1=1}^\numobs}}

\newcommand{\RV}{Z}
\newcommand{\realseq}{a}

\newcommand{\matsqrt}[1]{\ensuremath{#1^{\frac{1}{2}}}}

\newcommand{\MyIdDim}{\ensuremath{\Id}}
\newcommand{\PopSigma}{\mathbf{\Sigma}}

\DeclareSymbolFont{symbolsC}{U}{txsyc}{m}{n}
\DeclareMathSymbol{\notnimartin}{\mathrel}{symbolsC}{61}

\newcommand{\matsnorm}[2]{\ensuremath{|\!|\!| #1 | \! | \!|_{{#2}}}}

\newcommand{\maxnorm}[1]{\|#1\|_{\scriptsize{\operatorname{max}}}}

\newcommand{\opnorm}[1]{\matsnorm{#1}{\scriptsize{\operatorname{op}}}}

  \newcommand{\frobnorm}[1]{\matsnorm{#1}{\scriptsize{\operatorname{F}}}}

\newcommand{\myparagraph}[1]{\subsection*{\textbf{#1}}}

\newcommand{\dims}{d}

\newcommand{\lebesgue}{\nu}

\newcommand{\greedy}{\varepsilon}

\newcommand{\OLS}{OLS\xspace}

\newcommand{\Loss}{\mathrm{\ell}_\mahal}

\newcommand{\Y}{\mathbf{y}}

\newcommand{\actionset}{\mathcal{A}}
\newcommand{\action}{\mathbf{a}}
\newcommand{\thetaRidge}{\ensuremath{\hat{\mathbf{\theta}}_{\text{ridge}}}}
\newcommand{\lambdaRidge}{\ensuremath{\lambda_{\text{ridge}}}}

\newcommand{\CIset}{\mathbf{A}}
\newcommand{\chisq}{\chi^2}

\newcommand{\diagCov}{\ensuremath{\mathbf{D}}}

\newcommand{\Basis}{\ensuremath{\mathbf{V}}}
\newcommand{\enorm}[1]{ \| #1 \|}
\newcommand{\thetaBasis}{\ensuremath{ \widehat{ \boldsymbol{\theta}}_{\dir, \scriptscriptstyle{\operatorname{diagOD}}} }}

\newcommand{\algoDir}{\ensuremath{u}}
\newcommand{\randDir}{\ensuremath{v}}
\newcommand{\thetaVec}{\ensuremath{\mathbf{\theta}}}
\newcommand{\thetaBasisLS}{\ensuremath{ \widehat{ \boldsymbol{\theta}}_{\dir, \scriptscriptstyle{\operatorname{LS}}} }}

\renewcommand{\real}{\bf{R}}

\newcommand{\InvSigmaV}{\mathbf{\Omega}}
\newcommand{\OpTuning}{\beta}

\newcommand{\zero}{\bm 0}
\newcommand{\ip}[1]{\langle #1 \rangle}
\newcommand{\norm}[1]{\Vert #1 \Vert}
\newcommand{\cE}{\ensuremath{{\widehat{U}}}}

\newcommand{\thetaODNew}{\widehat{\theta}_{OD}}
\newcommand{\gammahat}{\widehat{\gamma}} 
\providecommand{\convprob}{\stackrel{\smallop{p}}{\longrightarrow}}
\newcommand{\banditdir}{\e_j}
\def\CIrisk{{\sf Risk}}

\newcommand{\Event}{\ensuremath{\mathcal{E}}}
\newcommand{\uphat}{\ensuremath{\widehat{u}}}
\newcommand{\lowhat}{\ensuremath{\widehat{\ell}}}
\newcommand{\hackCI}{\ensuremath{[\lowhat_\numobs, \uphat_\numobs]}}
\newcommand{\smallv}{\ensuremath{\xi_\numobs}}
\newcommand{\smallvtil}{\ensuremath{\tilde{v}_\numobs}}

\newcommand{\ExsData}{\ensuremath{\Exs_{\thetabf}}}
\newcommand{\ProbData}{\ensuremath{\Prob_{\thetabf}}}

\newcounter{example}[section]

\theoremstyle{definition}
\newtheorem{definition}{Definition}[section]

\begin{document}
%%%%%%% TITLE PAGE %%%%%%%%%%%%%%%%%%%%%%%%%%%%%%%%%%%%%%%%%%%%%%%%%%%

\begin{frontmatter}

% "Title of the paper"
\title{Near-optimal inference in adaptive linear regression}
\runtitle{Near-optimal inference in adaptive linear regression}

\begin{aug}
%%%%%%%%%%%%%%%%%%%%%%%%%%%%%%%%%%%%%%%%%%%%%%
%%Only one address is permitted per author. %%
%%Only division, organization and e-mail is %%
%%included in the address.                  %%
%%Additional information can be included in %%
%%the Acknowledgments section if necessary. %%
%%%%%%%%%%%%%%%%%%%%%%%%%%%%%%%%%%%%%%%%%%%%%%
\author[A]{\fnms{Koulik} \snm{Khamaru}\ead[label=e1]{kk1241@stat.rutgers.edu}},
\author[B]{\fnms{Yash} \snm{Deshpande}\ead[label=e2,mark]{ydeshpande@voleon.com}},\\
\author[C]{\fnms{Tor} \snm{Lattimore}\ead[label=e3,mark]{tor.lattimore@gmail.com}},
\author[D]{\fnms{Lester} \snm{Mackey}\ead[label=e4,mark]{lmackey@microsoft.com}},
\and
\author[E]{\fnms{Martin J.} \snm{Wainwright}\ead[label=e5,mark]{wainwrigwork@gmal.com}}
%%%%%%%%%%%%%%%%%%%%%%%%%%%%%%%%%%%%%%%%%%%%%%
%% Addresses                                %%
%%%%%%%%%%%%%%%%%%%%%%%%%%%%%%%%%%%%%%%%%%%%%%
\address[A]{Department of Statistics,
Rutgers University,
\printead{e1}}

\address[B]{ 
Voleon Group, \printead{e2}}

\address[C]{DeepMind, \printead{e3}}
\address[D]{Microsoft Research, \printead{e4}}

\address[E]{Lab for Information and Decision Systems,
Statistics and Data Science Center,
Massachusetts Institute of Technology \\
Departments of Statistics and EECS,
UC Berkeley
\printead{e5}}

\end{aug}

\begin{abstract}
When data is collected in an adaptive manner, even simple methods like
ordinary least squares can exhibit non-normal asymptotic behavior.  As
an undesirable consequence, hypothesis tests and confidence intervals
based on asymptotic normality can lead to erroneous results.  We
propose a family of online debiasing estimators to correct these
distributional anomalies in least squares estimation.  Our proposed
methods take advantage of the covariance structure present in the
dataset and provide sharper estimates in directions for which more
information has accrued.  We establish an asymptotic normality
property for our proposed online debiasing estimators under mild
conditions on the data collection process and provide asymptotically
exact confidence intervals. We additionally prove a minimax lower
bound for the adaptive linear regression problem, thereby providing a
baseline by which to compare estimators. There are various conditions
under which our proposed estimators achieve the minimax lower bound.
We demonstrate the usefulness of our theory via applications to
multi-armed bandit, autoregressive time series estimation, and active
learning with exploration.
\end{abstract}

\begin{keyword}[class=MSC2020]
\kwd[Primary ]{62E20}
\kwd[; secondary ]{6208}
\end{keyword}

\begin{keyword}
\kwd{Adaptive linear regression}
\kwd{Online debiasing}
\kwd{Minimax lower bound}
\kwd{Asymptotic normality}
\kwd{Multi-armed bandits}
\kwd{Autoregressive time series}
\end{keyword}

\end{frontmatter}

\section{Introduction} 
\label{sec:introduction}
Consider a prediction problem in which we observe $\numobs$ datapoints
of the form $(\x_i, \y_i) \in \real^\Dim \times \real$ with covariate
vector $\x_i$ and response $\y_i$ linked via the linear model
\begin{align}
\label{eqn:linear-reg-model}
\y_i = \myinprod{\x_i}{\thetastar} + \error_i \quad \mbox{for $i = 1,
  \ldots, \numobs$.}
\end{align}
Here the vector $\thetastar \in \real^\Dim$ is an unknown parameter of
interest, and $\error_i$ is additive noise.  When the datapoints are
generated via some i.i.d. sampling process, this model, and in
particular the behavior of the ordinary least squares (OLS) estimate
$\thetaLS$, is very well-understood.  The focus of this paper is the
more challenging setting in which the covariate vectors $\Nseq{\x}{i}$
have been \emph{adaptively} collected, meaning that the choice of
$\x_i$ can depend on the entire set of previous observations $\{\x_j,
\y_j\}_{j =1}^{i-1}$.

More precisely, given a filtration $\Nseq{\filtration}{i}$, assume
that $\x_i$ is $\filtration_{i - 1}$-measurable and
that the additive error $\Nseq{\error}{i}$ is a martingale difference
sequence with respect to $\Nseq{\filtration}{i}$, with 
\begin{align}
\Exs[\error_i \mid \filtration_{i - 1} ] = 0, \quad \mbox{and} \quad
\Exs[\error_i ^2 \lvert \filtration_{i-1}] = \sigma^2,
\end{align}
for some non-random scalar $\sigma^2 > 0$.  We refer to the
combination of the linear observation
model~\eqref{eqn:linear-reg-model} with such (potentially) adaptive
collection procedures as the \emph{adaptive linear regression model}.
Instances of adaptive linear regression arise in a variety of
applications, including multi-armed
bandits~\citep{lattimore2020bandit}, active
learning~\citep{fontaine2019online}, times series
modeling~\citep{box2015time}, stochastic
control~\citep{aastrom2012introduction}, and adaptive stochastic
approximation schemes~\citep{lai1982least,deshpande2017accurate}.

Let us discuss some known results for the OLS estimate $\thetaLS$.
The estimate can be expanded in the form
\begin{align}
  \label{EqnLSDecompose}
\thetaLS = \SigmaMat_\numobs^{-1}\X_n^\top \y_n = \thetastar
  + \SigmaMat_\numobs^{-1} \sum_{i=11}^\numobs \x_i \error_i, \qquad
\text{where} \quad \SigmaMat_\numobs \mydefn \sum_{i=1}^\numobs \x_i \x_i^\top.
\end{align}
This decomposition reveals that the statistical properties of the OLS
estimate depend on the martingale transform $\Nsum{i} \x_i \error_i$,
along with the random matrix $\SigmaMat_\numobs$.  There is a lengthy
literature on conditions under which the OLS estimate is
consistent~\citep{lai1982least,aastrom2012introduction,box2015time,goodwin1977dynamic,lai1979adaptive,lai1979strong}.
Notably, \citet[Thm.~1]{lai1982least} show that the OLS estimate is
strongly consistent, meaning that $\thetaLS \almostSurely \thetastar$,
whenever
\begin{align}
\label{eqn:strong-consistency-cond}\lambda_{\min}(\SigmaMat_\numobs) \almostSurely \infty \quad \text{and}
\quad
\frac{\log  \lambda_{\max}(\SigmaMat_\numobs)}{\lambda_{\min}(\SigmaMat_\numobs)}
\almostSurely 0.
\end{align}
Arguably, these conditions for consistency are quite mild.  In
contrast, \citet[Theorem 3]{lai1982least} also show that asymptotic
normality of the least squares estimator in the adaptive linear
regression model holds under a stability condition that is
substantially more restrictive---namely, the existence of a sequence
$\{\SigmaLim_\numobs\}_{\numobs \geq 1}$ of \emph{non-random} strictly
positive definite matrices such that
\begin{align}
  \label{eqn:Cond-Cov-stability} \SigmaLim_\numobs^{-1} \SigmaMat_\numobs \stackrel{\smallop{p}}{\longrightarrow} \MyIdDim.
\end{align}
Moreover, \citet[Example 3]{lai1982least} demonstrate through the
example of a unit root autoregressive model that the OLS estimator
fails to be asymptotically normal in absence of the stability
property~\eqref{eqn:Cond-Cov-stability}.  In such cases, confidence
intervals and other forms of inference performed using Gaussian limit
theory are no longer valid.

%%%%%%%%%%%%%%%%%%%%%%%%%%%%%%%%%%%%%%%%%%%%%%%%%%%%%%%%%%%%%%%%%%%%%%%%%%%%%%%

\subsection*{Contributions}

In this paper, we propose and analyze a new family of estimators for
the parameter vector $\thetastar$ (or linear functionals thereof)
based on online debiasing techniques.  We show that, under mild
conditions, our proposed estimators are both asymptotically unbiased
and asymptotically normal. The underlying assumptions are less
stringent than the stability condition~\eqref{eqn:Cond-Cov-stability}
and are satisfied by a large class of models for data generation and
protocols for choosing covariate vectors.  We provide a detailed
discussion of three such example classes in \cref{sec:applications}.
By deriving minimax lower bounds on the performance of any estimator,
we show that our estimators are minimax optimal.  We also show that
the asymptotic performance of these estimators are near-optimal in an
instance dependent sense, in that they match the performance of the
best problem-specific behavior up to a logarithmic factor.

%%%%%%%%%%%%%%%%%%%%%%%%%%%%%%%%%%%%%%%%%%%%%%%%%%%%%%%%%%%%%%%%%%%%%%%%%%%%%%%

\subsection*{Related work} 

The broader literature on bandit algorithms and experimentation
focuses mostly on a single statistical objective, with standard
examples being minimizing regret or selecting an optimal arm with high
probability. In the papers~\citep{xu2013estimation,villar2015multi},
the authors empirically observed that bandit algorithms induce bias,
which can be problematic for ex-post inference.  Later
works~\citep{nie2018adaptively,shin2019sample,shin2019bias}
characterize the sign and bound the magnitude of this bias. In the
paper~\citep{hadad2019confidence}, the authors develop estimators that
use propensity scores for the multi-armed bandit setting, a special
case of the stochastic regression model~\eqref{eqn:linear-reg-model}
in which the covariate vectors $\x_i$ are restricted to standard basis
vectors. However, it is not clear how to extend this approach to
general designs. Also in the bandit setting, Zhang et
al.~\citep{zhang2020inference} develop a least squares estimator that
exploits an assumed batch structure, meaning that only a fixed, finite
number of adaptive decisions are made. This approach, however, does
not apply to more general schemes that make adaptive decisions at each
round. Recently, Zhang et al.~\citep{zhangstatistical} proposed a
weighted \mbox{M-estimator} for contextual bandit problems where the
bandit algorithm is known. It is also not clear how to generalize this
approach to a more general data collection scheme or to the case when
the data collection algorithm is not completely known.

There is also a parallel line of work that exploits concentration of
measure results (e.g., see the papers~\citep{BouLugMas13, wainwright2019high}) to
develop confidence regions that are valid uniformly in time.  This
approach has its roots in the bandits
literature~\citep{abbasi2011online,jamieson2014lil} and has been
refined in more recent
work~\citep{howard2018uniform,kaufmann2018mixture}.  An advantage of
this approach is that it yields bounds that are uniform in time. On
the flip side, it requires very strong exponential tail conditions on
the error sequence in contrast to the relatively mild moment conditions
that we impose.  Overall, we view this line of work as being
complementary to our goal of developing corrected estimators that obey
asymptotic normality.

This paper builds upon and extends past work, due to a subset of the
current authors~\citep{deshpande2017accurate,deshpande2019online},
using online debiasing techniques. In~\cref{thm:Minimax-Lowerbound},
we prove a lower bound that shows how the matrix sequence
$\SigmaMat_\numobs^{-1}$ controls the fundamental difficulty of the
problem, and this lower bound also motivates the particular form of
debiasing proposed in this paper.  The construction used in past
work~\citep{deshpande2017accurate,deshpande2019online} is based on a
non-adaptive upper bound of the form $\lambda_* \MyIdDim$, where the
scalar $\lambda_*$ is chosen to be much larger than
$\lambda_{\max}(\SigmaMat_\numobs^{-1})$ with high probability.  By
sharp contrast, our analysis instead makes use of an adaptive upper
bound that simultaneously respects the structure of
$\SigmaMat_\numobs^{-1}$ and leads to a stable martingale transform;
this particular construction and our analysis thereof allows us to
obtain sharper guarantees than past
work~\citep{deshpande2017accurate,deshpande2019online}.

\subsection*{Notation} Let us summarize some notation used throughout
the remainder of the paper.  For a positive integer $n$, we make use
of the convenient shorthand $[n] \mydefn \{1, 2, \ldots, n\}$.  We use
$\myvector{e}_j$ to denote the $j$th standard basis vector in
$\real^{\Dim}$.  For a matrix $\mymatrix{M}$, we use the notation
$\opnorm{\mymatrix{M}}$ and $\frobnorm{\mymatrix{M}}$ to denote the
operator norm (maximum singular value) and the Frobenius norm of the
matrix $\mymatrix{M}$, respectively; similarly, we use the notation
$\maxnorm{\mymatrix{M}}$ to denote the maximum entry in absolute
value.  For a square matrix $\mymatrix{S}$, the quantities
$\lambda_{\max}(\mymatrix{S})$ and $\lambda_{\min}(\mymatrix{S})$
respectively denote the maximum and minimum eigenvalue of the matrix
$\mymatrix{S}$.  The quantity $\tr(\mymatrix{S})$ denotes the sum of
diagonal entries of the square matrix $\mymatrix{S}$. For a pair of
squares matrices $(\mymatrix{A}, \mymatrix{B})$ of compatible
dimensions, we use the notation $\mymatrix{A} \succcurlyeq
\mymatrix{B}$ to indicate that the difference matrix $\mymatrix{A} -
\mymatrix{B}$ is positive semidefinite; we use the notation
$\mymatrix{A} \succ \mymatrix{B}$ when the difference matrix
$\mymatrix{A} - \mymatrix{B}$ is positive definite. The relations
$\mymatrix{A} \preccurlyeq \mymatrix{B}$ and $\mymatrix{A} \prec
\mymatrix{B}$ are defined analogously. For a symmetric positive
semidefinite matrix $\mymatrix{S}$, we use $\matsqrt{\mymatrix{S}}$ to
denote its symmetric matrix square root.

For a sequence of random variables $\{\RV_\numobs\}_{\numobs \geq 1}$
and a random variable $\RV$, we write $\RV_\numobs
\stackrel{\smallop{p}}{\longrightarrow} \RV$ to mean that the sequence
of random variables $\{\RV_\numobs\}_{\numobs \geq 1}$ converges to
$\RV$ in probability; the notation $\RV_\numobs \indistrb \RV$
indicates convergence in distribution.  For a sequence of real-valued
random variables $\{\RV_\numobs\}_{\numobs \geq 1}$ and a sequence of
non-zero real numbers $\{\realseq_\numobs\}_{\numobs \geq 1}$, we
write $\RV_\numobs = \smalloP(\realseq_\numobs)$ to mean that the
ratio $\tfrac{\RV_\numobs}{\realseq_\numobs}
\stackrel{\smallop{p}}{\longrightarrow} 0$.  We write $\RV_\numobs =
\BigoP(\realseq_\numobs)$ to mean that the ratio
$\RV_\numobs/\realseq_\numobs$ is stochastically bounded.  More
precisely, for every scalar $\epsilon >0$, there exits a positive real
number $C_{\epsilon}$ such that \mbox{$\sup \limits_{\numobs \geq 1}
  \, \Prob [\RV_\numobs / \realseq_\numobs > C_\epsilon] < \epsilon$.}

%%%%%%%%%%%%%%%%%%%%%%%%%%%%%%%%%%%%%%%%%%%%%%%%%%%%%%%%%%%%%%%%%%%%%%%%%%%%%%%%%%%%%%%%%%%%%%

\section{From ordinary least squares to online debiasing}
\label{SecMotivation}

In this section, we begin by motivating the work by discussing how
classical theory about ordinary least squares estimate can break down
when data is collected in an adaptive manner.  We then introduce an
online debiasing approach to computing alternative estimates.

\subsection{Breakdown of the ordinary least squares estimator}
\label{sec:least-square-estimators}

\begin{figure}[!ht]
  \centering
  \begin{subfigure}{0.45\linewidth}
      \includegraphics[width=6cm]{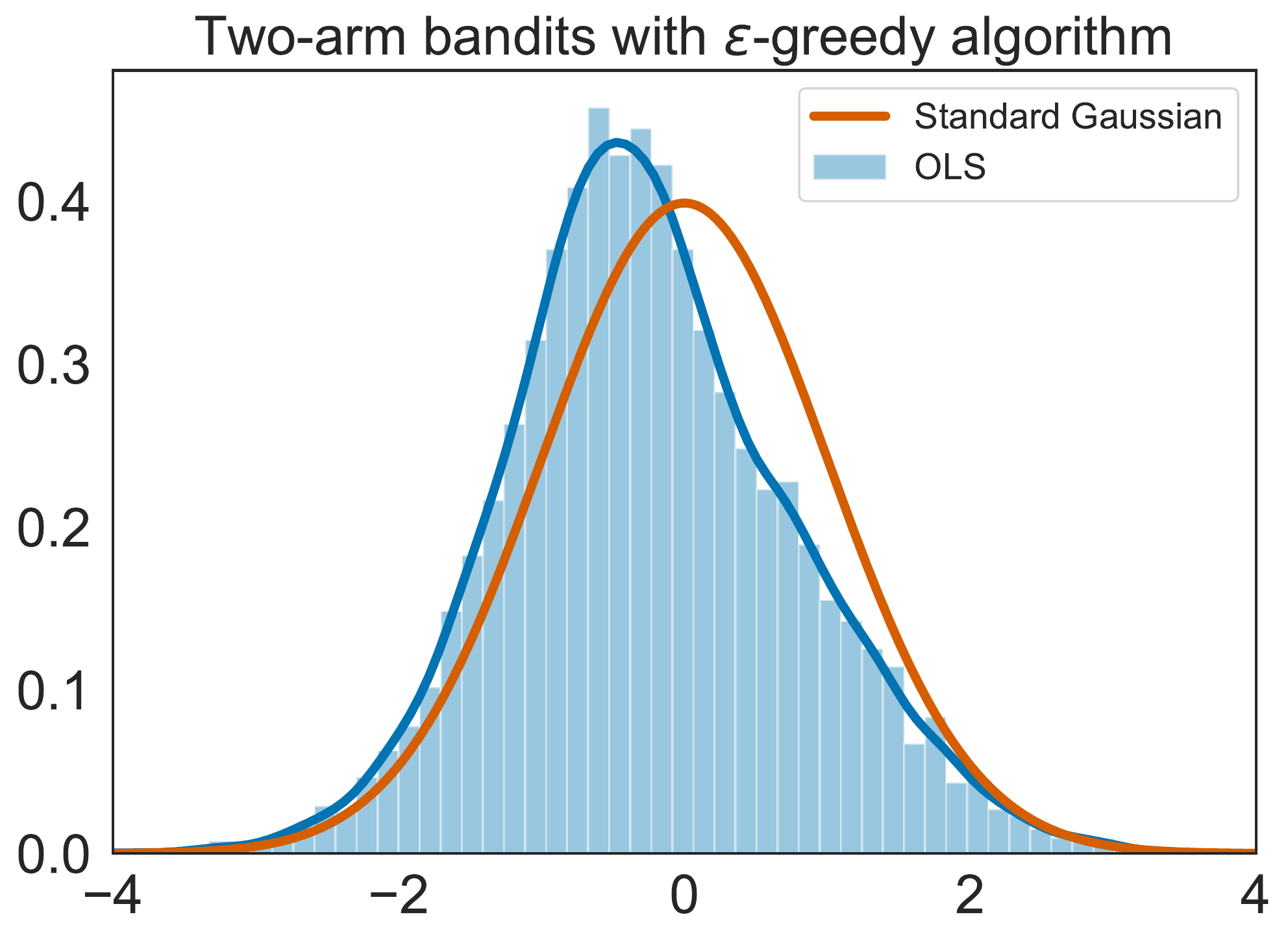}
    \caption{Distribution of $\frac{\thetahat_2 -
        \theta_2^*}{\sqrt{(\SigmaMat_n^{-1})_{22}} }$}
  \end{subfigure}
  \quad
  \begin{subfigure}{0.45\linewidth}
    \raisebox{0.14in}{
      \includegraphics[width=6cm]{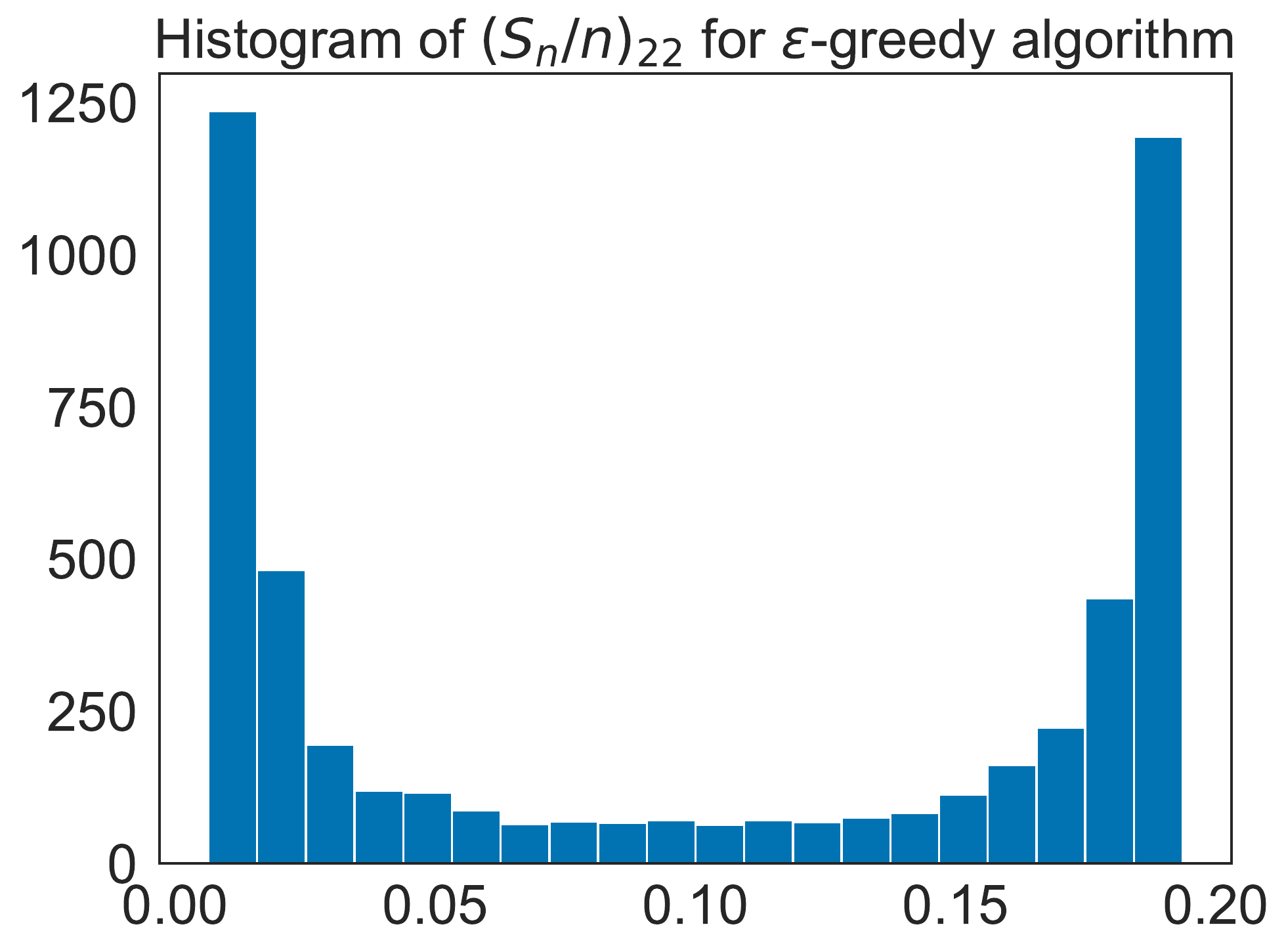}
      }
    \caption{Distribution of $\left(\SigmaMat_n/n\right)_{22}$}
  \end{subfigure}
\caption{Quantitative behavior of the first coordinate $\thetahat_1$
  of the ordinary least squares estimator $\thetaLS := (\thetahat_1,
  \thetahat_2)$ on a dataset drawn from the $\greedy$-greedy two-armed
  bandit model of \cref{sec:least-square-estimators}. The results are
  obtained with a dataset of size $n = 1000$ and $5000$ independent
  replications. (a) The distribution of the rescaled difference $\frac{\thetahat_2 -
    \theta_2^*}{\sqrt{(\SigmaMat_n^{-1})_{22}} }$ is far from standard
  Gaussian.  (b) The bimodal distribution of $\left(\SigmaMat_n/n
  \right)_{22}$ suggests that the scaled covariance matrix
  $\SigmaMat_n/n$ does not converge to a deterministic matrix.}
 \label{fig:toy-bandit}
\end{figure}

Let us begin by considering the behavior of the OLS estimator
$\thetaLS$ from equation~\eqref{EqnLSDecompose}.  When the covariates
$\{\x_i\}_{i \geq 1}$ are either fixed or independently sampled from a
fixed distribution, it has several optimality properties.
Accordingly, it is natural to ask what the performance of the \OLS
estimator is when the covariates $\{\x_i\}$ are drawn in an adaptive
manner.

In order to fix ideas, let us consider a two-armed bandit
problem~\citep{lattimore2020bandit}, a special case of the linear
regression model~\eqref{eqn:linear-reg-model} with each $\x_i$ chosen
to be either $(1,0)^\top$ or $(0, 1)^\top$ based on the prior data
$\{\x_j, \y_j \mid j \leq i - 1 \}$.  In order to generate the
covariates $\{\x_i\}_{i = 1}^{n}$, suppose that we apply the
$\greedy$-greedy selection algorithm, a popular choice for tackling
bandit problems~\citep{lattimore2020bandit}.

A simple simulation reveals some interesting phenomena.  We generated
linear regression data using $\greedy$-greedy selection algorithm with
the choices $\greedy = 0.1$, $\thetastar = (0.3, 0.3)^\top$, and noise
variables $\error_i \stackrel{i.i.d.}{\sim} \Ncal(0, 1)$.  Let
$\thetahat_2$ denote the first coordinate of the \OLS estimator fit to
the bandit data.  Figure~\ref{fig:toy-bandit}(a) demonstrates that the
distribution of $\thetahat_2$, even after proper re-centering and
scaling, does not converge to a standard normal distribution.  As an
undesirable consequence, the confidence intervals for $\thetastarScalar_2$,
usually constructed using the quantiles of a standard normal random
variable, are not valid.

Let us try to understand why the \OLS estimate 
fails to be asymptotically normal.
Figure~\ref{fig:toy-bandit}(b) plots a histogram 
of the $(2,2)$ entry of the scaled sample covariance 
matrix $\SigmaMat_n/n$.
The bimodal behavior suggests that $\SigmaMat_n/n$ fails to converge to a 
non-random matrix $\SigmaLim$ and indicates that the stability 
condition~\eqref{eqn:Cond-Cov-stability}
is not satisfied. Indeed, in a recent paper~\citep{zhang2020inference},
the authors show that when $\theta_1^* = \theta_2^*$ as in our example, the \OLS estimator, 
after proper centering and scaling, converges to
a distribution which is \emph{not a standard Gaussian} distribution.

It turns out that this distributional anomaly of the \OLS estimator is
neither specific to the two-armed bandit
problem~\citep{lattimore2020bandit} nor to the $\greedy$-greedy
algorithm used to simulate the data for
Figure~\ref{fig:toy-bandit}. The same phenomenon was documented in the
time-series and forecasting literature half a century ago, dating back
to the works of~\citet{white1958limiting,dickey1979distribution}
and~\citet{lai1982least}. More recent
work~\citep{deshpande2017accurate,zhang2020inference} has highlighted
that a similar phenomenon commonly occurs in multi-armed bandit
problems when using popular selection algorithms, including Thompson
sampling and the upper confidence bound (UCB)
algorithm~\citep{lattimore2020bandit}.

In Section~\ref{sec:W-decorr}, we rectify the distributional anomaly
of the \OLS estimator by proposing an estimator based on the online
debiasing principles of~\citet{deshpande2017accurate} and show that
our online debiasing estimator exhibits asymptotic normality even in
the absence of the stability condition~\eqref{eqn:Cond-Cov-stability}.
In Section~\ref{sec:applications}, we demonstrate the usefulness of
our theory via applications to the multi-armed bandit problems,
autoregressive time series, and active learning problems with
exploration.

\subsection{Online debiasing estimator}
\label{sec:W-decorr}

In this section, we propose and analyze an estimator based on an
online debiasing technique motivated by the
work of~\citet{deshpande2017accurate}. At a high-level, the estimator
involves a specific perturbation of the ordinary least squares
estimator $\thetaLS$.  This perturbation is constructed via a linear
combination of the prediction errors $\NseqNosub{\y_i -
  \myinprod{\x_i}{\thetaLS}}{i}$ along with a carefully chosen
sequence of weight vectors $\Nseq{\w}{i}$.  The key property ensured by the
construction is that the weight vector $\w_i$ is $\filtration_{i - 1}$
measurable for each $i \in [\numobs]$.

Concretely, for weight vectors $\Nseq{\w}{i}$, we
compute the \mbox{\emph{online debiasing estimate}}
\begin{align}
\label{eqn:ThetaDecorr-defn}
\thetaDecorr & \mydefn \thetaLS + \SigmaMat_\numobs^{-\frac{1}{2}}
\sum_{i = 1}^{\numobs} \w_i (\y_i - \myinprod{\x_i}{\thetaLS}).
\end{align}
Here the reader should recall our earlier definition
$\SigmaMat_\numobs \mydefn \sum_{i=1}^\numobs \x_i \x_i^\top$, and
throughout, we assume that the sample covariance $\SigmaMat_\numobs$
is invertible. The matrix $\SigmaMat_\numobs^{-\frac{1}{2}}$ denotes a
symmetric matrix square root of $\SigmaMat_\numobs^{-1}$.

Of course, there is an infinite family of estimators of the
form~\eqref{eqn:ThetaDecorr-defn}, and the key question is how to
define the weight vectors.  In this paper, we propose an estimator in
which the sequence $\Nseq{\w}{i}$ is obtained by solving an
optimization problem that takes three inputs:
\begin{enumerate}
\item[(i)] the original data $\NseqNosub{(\x_i, \y_i)}{i}$,
\item[(ii)] a non-random scalar $\tuneParScaled_{\numobs} \in (0, 1]$,
  and
\item[(iii)] a sequence of symmetric positive semidefinite matrices
  $\Nseq{\XscaleMat}{i}$ such that $\XscaleMat_{i} \in \filtration_{i -
  1}$ for each $i \in [\numobs] \mydefn \{1, \ldots, \numobs \}$.
\end{enumerate}

In order to simplify notation, we adopt the shorthand $\z_i \mydefn
\XscaleMat_i^{-\frac{1}{2}} \x_i$. Moreover, for each index \mbox{$i
  \in [\numobs]$,} we define the matrices
\begin{align*}
\Z_{i}^\top \mydefn \begin{bmatrix} \z_1 & \z_2 & \cdots & \z_i
\end{bmatrix}, \quad
\mbox{and} \quad \W_{i} \mydefn \begin{bmatrix} \w_{1} & \w_2 & \cdots
  & \w_i
\end{bmatrix}
\end{align*}
We also define $\W_{0} = 0$ and $\Z_0 = 0$.  With these definitions,
the vectors $\Nseq{\w}{i}$ are obtained recursively by solving the
following convex program
\begin{subequations}
\label{eqn:Wn-expression-full}  
  \begin{align}
    \label{eqn:Wn-expression-original}    
    \w_{i} & \mydefn \arg \min _{\w \in \real^\Dim} \left \{
    \frobnorm{ \MyIdDim - \W_{i - 1} \Z_{i - 1} - \w \z_i^\top }^{2} +
    \frac{\tuneParScaled_\numobs}{2} \|\w\|_2^2 \right \} \quad
    \mbox{for $i = 1, 2, \ldots, \numobs$.}
  \end{align}
Conveniently, this optimization problem has the explicit solution
\begin{align}    
  \label{eqn:Wn-expression}
\w_i & = \frac{(\MyIdDim - \W_{i - 1} \Z_{i -
    1})\z_i}{(\tuneParScaled_\numobs/2) + \| \z_i \|_2^2}.
  \end{align}
\end{subequations}
%%%%%%%%%%%%%%%%%%%%%%%%%%%%%%%%%%%%%%%%%%%%%%%%%%%%%%%%%%%%%%%%%%%%%%%%%%%%%%

\section{Main results}
\label{SecMain}

Having motivated and introduced the online debiasing approach, we now
turn to some theoretical guarantees that can be given for these
methods.  We begin in~\Cref{sec:normality} by providing sufficient
conditions for the online debiasing estimator of~\Cref{sec:W-decorr}
to exhibit asymptotically Gaussian behavior
(\Cref{thm:asymp-normality}). In~\Cref{subsec:asympexactci}, we
provide an asymptotically exact confidence region for $\thetastar$ as
well as an asymptotically exact confidence interval
(Proposition~\ref{prop:gen-confinv}) for $\dir^\top \thetastar$, where
$\dir$ is an arbitrary fixed direction $\dir \in \real^\dims$.
In~\Cref{SecLower}---in particular,
see~\Cref{thm:Minimax-Lowerbound}---we complement these results by
providing minimax lower bounds on a family of Mahalanobis errors and
the length of confidence intervals.  These lower bounds apply to any
estimator for the stochastic regression model which does not know the
true value of the target parameter $\thetastar$ but may have the full
knowledge of how the data was collected. Finally,
in~\Cref{SecSufficientThree} we provide general strategies which can
be used to verify the conditions of~\Cref{thm:asymp-normality}.  All
of our asymptotic statements assume that the dimension $\Dim$ is fixed
(constant) while the sample size $\numobs$ grows.

%%%%%%%%%%%%%%%%%%%%%%%%%%%%%%%%%%%%%%%%%%%%%%%%%%%%%%%%%%%%%%%%%%%%%%%%%%%%%%%%%%%%%%%%%%%%%%%%

\subsection{Asymptotic normality guarantees}
\label{sec:normality}
The main result of this section is an asymptotic normality guarantee
for the proposed estimator~\eqref{eqn:ThetaDecorr-defn}, where the
weight vectors are defined via the
recursion~\eqref{eqn:Wn-expression-full}.

We begin by stating our assumptions and providing some intuition about
their role in the theorem.
\subsubsection*{Assumption A}
\begin{enumerate}[label=(A\arabic*)]
  \item \label{assumption:A1} There are positive scalars $\sigma$ and
    $\mompar$ such that the noise sequence $\Nseq{\error}{i}$ satisfies
    the conditions $\Exs[\error_i \mid \filtration_{i - 1} ]
    = 0$ and $\Exs[\error_i^2 \mid \filtration_{i - 1}] = \noisesd^2$ for all
    $i \in [\numobs]$ and moreover
    \begin{align*}
\max_{i \in [\numobs]} \; \Exs[\error_i^{2+\mompar} \mid
  \filtration_{i - 1}] < \infty.
\end{align*}
\item
\label{assumption:A2}
The sequence of matrices $\{\SigmaMat_\numobs\}_{\numobs \geq1}$
satisfy the conditions \mbox{$\lambda_{\min}(\SigmaMat_\numobs)
  \almostSurely \infty$} and \mbox{$\frac{\log
    \lambda_{\max}(\SigmaMat_\numobs)}{\lambda_{\min}(\SigmaMat_\numobs)}
  \almostSurely 0$.}
\item
\label{assumption:A3}
For each $\numobs$, the scalar $\tuneParScaled_\numobs > 0$ and
positive semidefinite matrices $\Nseq{\XscaleMat}{i}$ with
\mbox{$\XscaleMat_i \in \filtration_{i - 1}$} are chosen such that:
\begin{subequations}
\begin{align}
\tag*{(a) Asymptotic negligibility:} \max \limits_{i \in [\numobs]}
\;\; \left \{ \frac{1}{\tuneParScaled_\numobs}
\myinprod{\x_i}{\XscaleMat_i^{-1} \x_i} \right \} &
\stackrel{\smallop{p}}{\longrightarrow} 0, \\
  \tag*{(b) Vanishing bias:} \sqrt{\tuneParScaled_\numobs \log
    \lambda_{\max}(\SigmaMat_\numobs) } \cdot \opnorm{\MyIdDim -
    \W_\numobs \X_\numobs \SigmaMat_\numobs^{-\frac{1}{2}}} & \stackrel{\smallop{p}}{\longrightarrow}
  0, \quad \mbox{and} \\
\tag*{(c) Variance stability:} \opnorm{\MyIdDim - \Nsum{i} \w_i
  \x_i^\top \XscaleMat_i^{-\frac{1}{2}}} &
\stackrel{\smallop{p}}{\longrightarrow} 0.
\end{align}
\end{subequations}
\end{enumerate}

\vspace*{0.2in}

Let us provide some intuition for the role of each of these
assumptions in the theorem.  First, Assumption~\ref{assumption:A1} is
quite simple: it imposes relatively mild moment conditions on the
noise variables.  Second, as discussed in the introduction,
Assumption~\ref{assumption:A2} is standard in guaranteeing the
consistency of the least squares estimate.  Both
Assumptions~\ref{assumption:A1} and~\ref{assumption:A2} are viewed as
mild conditions in the stochastic linear regression literature and are
satisfied by many practical models including those studied in the
papers~\citep{lai1979adaptive,lai1982least,lai1994asymptotic,deshpande2017accurate}.
Note that Assumptions~\ref{assumption:A1} and~\ref{assumption:A2}
concern the regression model itself as opposed to the method: in
particular, they do not depend on the algorithm parameters
$\tuneParScaled_\numobs$ and $\Nseq{\XscaleMat}{i}$.

The more subtle requirements for our theorem to apply, which \emph{do
depend} on the algorithm parameters, are stated in
Assumption~\ref{assumption:A3}.  We discuss the technical role of
these conditions in the comments after~\Cref{thm:asymp-normality}, to
be stated momentarily.  In~\Cref{SecSufficientThree} to follow, we
provide concrete choices of the algorithm parameters
$\tuneParScaled_\numobs$ and $\Nseq{\XscaleMat}{i}$ that ensure that
Assumption~\ref{assumption:A3} holds.

Finally, it should be noted that Assumption~\ref{assumption:A3} is
weaker than the stability condition~\eqref{eqn:Cond-Cov-stability}.
Indeed, if the stability condition~\eqref{eqn:Cond-Cov-stability} and
the growth condition~\ref{assumption:A2} are satisfied, then we may
take $\gamma_n = 1$, $\XscaleMat_i = \SigmaLim_\numobs$ and $\w_i =
\SigmaLim_\numobs^{-\frac{1}{2}} \x_i$. With these choices, the
conditions~\ref{assumption:A3} are automatically satisfied; moreover,
in this particular case, the online debiased estimator reduces to
ordinary least squares; see~\Cref{prop:stability-to-debiasing} for
details. \\

With these preliminaries in 
place, we are now equipped to state our main theorem on the online
debiasing estimator $\thetaDecorr$:
\begin{theos}
  \label{thm:asymp-normality}
Let $\widehat{\noisesd}^2$ be any consistent estimator of
$\noisesd^2$.  Then under
\mbox{Assumptions~\ref{assumption:A1}--\ref{assumption:A3}}, we have
\begin{align}
  \label{eqn:asymp-normality}
  \sqrt{\frac{\tuneParScaled_\numobs}{\widehat{\noisesd}^2}} \cdot
  \SigmaMat_\numobs^{\frac{1}{2}}(\thetaDecorr - \thetastar) \indistrb
  \Ncal(0, \MyIdDim).
\end{align}
\end{theos}
\noindent We prove this theorem in~\Cref{proof:thm:asymp-normality}.\\

A few comments on this theorem are in order.  First, needing a
consistent estimate of the error variance $\noisesd^2$ is a mild
requirement.  For instance, under our conditions, the estimator
\begin{align*}
\widehat{\noisesd}^2 = \frac{1}{\numobs} \sum_{i = 1}^\numobs
\left(\y_i - \x_i^\top \thetaLS \right)^2
\end{align*}
is strongly consistent; see Lemma 3 in the paper~\citep{lai1982least}
for details.

A second important fact is that Assumption~\ref{assumption:A3} is
considerably weaker than the stability
condition~\eqref{eqn:Cond-Cov-stability} required for asymptotic
normality of the OLS estimate.  To reinforce this point,
\cref{sec:applications} provides a detailed discussion of three
classes of problems for which OLS fails to be asymptotically normal
but the guarantee~\eqref{eqn:asymp-normality} still holds for the
online debiasing estimator.

Of all the conditions of~\cref{thm:asymp-normality}, verifying the
variance stability condition in part (c) of
Assumption~\ref{assumption:A3} is the most challenging, and our
arguments for doing so vary from problem to problem.  In
Corollaries~\ref{cor:bandits-corr} and~\ref{cor:unit-root-autoreg}, we
verify the variance stability condition for multi-armed bandit
problems and autoregressive time series models, respectively. In
Corollary~\ref{cor:multidim-gen-case}, we verify this condition for a
large class of problems satisfying a sufficient exploration
condition. In Section~\ref{SecSufficientThree} to follow, we argue
that when $\lambda_{\min}(\SigmaMat_\numobs) \geq \log^2(n)$, we can
always find choices of the tuning parameters $\gamma_n$ and
$\{\Gamma_i\}$ such that Assumption~\ref{assumption:A3} is
satisfied. Additionally, in absence of such lower bound on
$\lambda_{\min}(\SigmaMat_\numobs)$, we propose a data-augmentation
strategy that gets rids of this growth condition at the cost of
collecting additional $\log^2(n)$ many data-points.

Let us now discuss how Assumption~\ref{assumption:A3} enters the proof of
Theorem~\ref{thm:asymp-normality}.  Our argument is based on the
decomposition
\begin{subequations}
  \label{eqn:bias-variance-decomp}
\begin{align}  
\sqrt{\tuneParScaled_\numobs} \cdot
  \SigmaMat_\numobs^{\frac{1}{2}}(\thetaDecorr - \thetastar) & =
  \Bias_\numobs + \ZeroMartin_\numobs, \quad \mbox{where} \\
  \Bias_\numobs & \mydefn \sqrt{\tuneParScaled_\numobs} \cdot \left(
  \MyIdDim - \W_\numobs \X_\numobs \SigmaMat_\numobs^{-\frac{1}{2}}
  \right)(\thetaLS - \thetastar) \quad \mbox{and} \\
\ZeroMartin_\numobs & \mydefn \sqrt{\tuneParScaled_\numobs} \cdot
\sum_{i = 1}^\numobs \w_i \error_i.
\end{align}
\end{subequations}
By construction (and suggested by our notation), the term
$\Bias_\numobs$ corresponds to the bias in our estimate, a quantity
that must be shown to vanish in order for our claim to hold. In order
to do so, we first derive an upper bound on the norm
$\|\Bias_\numobs\|_2$.  The ``vanishing bias'' condition stated in
Assumption~\ref{assumption:A3}(b) enters in showing that, via our
choices of the tuning parameters $\tuneParScaled_\numobs$ and
$\Nseq{\XscaleMat}{i}$, this upper bound converges to zero in
probability.

The random vector $\ZeroMartin_\numobs$ defines a zero-mean
martingale, and our proof controls its behavior via a standard
martingale central limit theorem.  Doing so requires a Lindeberg type
condition on the weight vectors $\Nseq{\w}{i}$, as given in part (a)
of Assumption~\ref{assumption:A3}.  Moreover, it requires that the
conditional covariance of the martingale behave suitably, in which
context part (c) of Assumption~\ref{assumption:A3} enters.

%%%%%%%%%%%%%%%%%%%%%%%%%%%%%%%%%%%%%%%%%%%%%%%%%%%%%%%%%%%%%%%

%%%%%%%%%%%%%%%%%%%%%%%%%%%%%%%%%%%%%%%%%%%%%%%%%%%%%%%%%%%%%%%%%%%%%%%%%%%%%%%%%%%%%%%%%%%%%%

\subsection{Obtaining confidence regions and intervals}
\label{subsec:asympexactci}

In this section, we use the online
debiasing procedure to obtain asymptotically exact confidence regions
and intervals.

\subsubsection{Confidence region for $\thetastar$}

First, for some user-defined level $\alpha \in (0,1)$, consider the
problem of finding a confidence region for $\thetastar$---that is, a
(random) set $\CIset_{1 - \alpha}$ that contains $\thetastar$ with
probability at least $1 - \alpha$.  We would like a set that is as
small as possible, asymptotically exact in the sense that its coverage
converges to $1-\alpha$.

Theorem~\ref{thm:asymp-normality} allows us to construct such a set in
the following straightforward way.  For any $\alpha \in (0,1)$,
consider the subset of $\real^\dims$ given by
\begin{align*}
\CIset_{1 - \alpha} = \left\{ \thetaVec\in \real^{\dims} \mid
\frac{\tuneParScaled_\numobs}{\sigmahat^2} \cdot (\thetaDecorr -
\theta)^\top \SigmaMat_\numobs (\thetaDecorr - \theta) \leq
\chisq_{\dims, 1 - \alpha} \right\}
 \end{align*} 
where $\chisq_{\dims, 1 - \alpha}$ denotes the $(1 - \alpha$)-quantile
for a standard chi-squared distribution with degrees of freedom $d$.
From the result  of Theorem~\ref{thm:asymp-normality}, we have the
guarantee
\begin{align*}
  \lim_{\numobs \rightarrow \infty} \Prob(\thetastar\in\CIset_{1 - \alpha}) = 1 - \alpha,
\end{align*}

In many applications, however, instead of a confidence region for the
full vector $\thetastar$, we are instead interested in obtaining a
confidence interval for the scalar quantity $\dir^\top \thetastar$,
where $\dir \in \real^\dims$ is a fixed direction.  It turns out that
Theorem~\ref{thm:asymp-normality} no longer provides a straightforward
answer to this question.  In order to understand why, it is useful to
begin by following a naive line of reasoning that is incorrect
and then show how it can be fixed.

\subsubsection{An incorrect argument}
In order to obtain a confidence interval for $\dir^\top \thetastar$,
it might be tempting to ``directly invert'' the distributional
property~\eqref{eqn:asymp-normality}.  In particular, letting
\mbox{$\NormalQuantile{1 - (\alpha/2)} \mydefn \Phi^{-1}(1 -
  \tfrac{\alpha}{2})$} denote the $1 - \tfrac{\alpha}{2}$ quantile of
the standard Gaussian distribution, we might claim that the interval
\begin{align}
\label{eqn:wrong-CI}
\left[ \myinprod{\dir}{\thetaDecorr} -
  \tfrac{\sigmahat}{\sqrt{\tuneParScaled_\numobs}}
  (\myinprod{\dir}{\SigmaMat_\numobs^{-1} \dir})^{\frac{1}{2}}
  \NormalQuantile{1 - \alpha/2}, \qquad
  \myinprod{\dir}{\thetaDecorr} +
  \tfrac{\sigmahat}{\sqrt{\tuneParScaled_\numobs}} (\myinprod{\dir}
        {\SigmaMat_\numobs^{-1}\dir})^{\frac{1}{2}}
        \NormalQuantile{1 - \alpha/2} \right],
\end{align}
is an asymptotically exact $1 - \alpha$ confidence interval for
$\dir^\top \thetastar$.

Unfortunately, the conclusion~\eqref{eqn:wrong-CI} is based on faulty
logic, namely the assertion that the asymptotic
guarantee~\eqref{eqn:asymp-normality} implies that
\begin{align}
\label{eqn:wrong-ASN}
  \frac{\sqrt{\tuneParScaled_\numobs}}{\sigmahat} \cdot (\thetaDecorr
  - \thetastar) \; - \; \Ncal(0, \SigmaMat_\numobs^{-1}) \indistrb 0.
\end{align} 
It is now interesting to understand when the implication above follows from Theorem~\ref{thm:asymp-normality}. 
% When the covariance matrix satisfies a stability condition, the above statement is true by an application of Slutsky's theorem. In the absence of a stability condition, the sample covariance matrix $\SigmaMat_n$ is random and \emph{may dependent on $\thetaDecorr$}. In this case, we cannot argue that the implication~\eqref{eqn:wrong-ASN} is satisfied.
%
Under the stability condition $\mathbf{C}_n^{-1} \mathbf{S}_n \indistrb \Id$ for a sequence of deterministic matrices $\mathbf{C}_n$, the conclusion \eqref{eqn:wrong-ASN} follows from the Theorem~\ref{thm:asymp-normality} by Slutsky's theorem.
However, in the absence of this stability condition, the sample covariance matrix remains random and \emph{may depend on $\thetaDecorr$}.

The following counterexample shows that when a sequence of random vectors $\mathbf{b}_n$ and matrices $\mathbf{A}_n$ are dependent, $\mathbf{A}_n \mathbf{b}_n \indistrb \mathcal{N}(0,\Id)$ does not imply $\mathbf{b}_n - \mathcal{N}(0,(\mathbf{A}_n^\top \mathbf{A}_n)^{-1}) \indistrb 0$ in general.
% To further understand this point, consider the following example: 
Let $X, Y$ be independent standard Gaussian random variables. Consider the vector $\mathbf{b} = (\textrm{sign}(X), \textrm{sign}(Y))^\top$, and a $2\times2$  matrix $\mathbf{A}$ with $\mathbf{A}_{11} = |X|/\sqrt{2}, \mathbf{A}_{12} = |Y|/\sqrt{2}, A_{21} = - |X|/\sqrt{2}, \mathbf{A}_{22} = |Y|/\sqrt{2}$. Simple calculations yields that $\mathbf{Ab} = (X + Y,\; X - Y)/\sqrt{2} \sim \mathcal{N}\left((0, 0)^\top, \Id \right)$. Note that each entry of $\mathbf{b}$ is $\pm1$ with probability $1/2$. Additionally, 
the variable $\mathcal{N}( (0, 0)^\top (\mathbf{A}^\top \mathbf{A})^{-1})$ is a scale mixture of Gaussians --- a continuous distribution ---- and therefore does not match the distribution of $\mathbf{b}$.
%
% Additionally, note that $\mathcal{N}(0, \mathbf{A}^{-1}\mathbf{A}^{-\top})$ is a scale mixture of Gaussians and hence does not match with the distribution of $\mathbf{b}$.
%
% \notate{instead note that $N(0,(A^\top A)^{-1}$ is a scale mixture of normals and hence a continuous distribution and therefore does not match the distribution of B}
In summary, we conclude that additional justification is needed to guarantee 
the asymptotic validity of the confidence interval~\eqref{eqn:wrong-CI} for the functional $\dir^\top \thetastar$ in general.

% Observe that in absence of a stability condition, the sample covariance matrix $\SigmaMat_n$ is random and is dependent on $\thetaDecorr$. 
% \notate{the next statement is too strong}
% The implication above  from Theorem~\ref{thm:asymp-normality} requires that $\SigmaMat_n$ is independent of $\thetaDecorr$. This independence is certainly true when the stability condition is satisfied --- the scaled $\SigmaMat_n$ converges to a deterministic matrix in this case --- but is not necessarily true otherwise.~\footnote{Let $X, Y$ be independent standard Gaussian random variables. Consider the vector $B = (\textrm{sign}(X), \textrm{sign}(Y))^\top$, and a $2\times2$  matrix $A$ with $A_{11} = |X|/\sqrt{2}, A_{12} = |Y|/\sqrt{2}, A_{21} = - |X|/\sqrt{2}, A_{22} = |Y|/\sqrt{2}$. Simple calculations yields that $AB = (X + Y,\; X - Y)/\sqrt{2} \sim \mathcal{N}\left((0, 0)^\top, \Id \right)$. Note that each entry of $B$ is $\pm1$ with probability $1/2$, and the distribution of B is not Gaussian.} As a result, this argument cannot be used to justify the asymptotic validity of the confidence interval~\eqref{eqn:wrong-CI} for the functional $\dir^\top \thetastar$ in general. 
% %s \emph{need not be} a valid CI for the functional $\dir^\top \thetastar$.

Nonetheless, there are certain special cases in which the
interval~\eqref{eqn:wrong-CI} is a valid CI.  Concretely, suppose that
$\dir = \e_j$ is one of the standard coordinate basis vectors and that
$\SigmaMat_\numobs$ is diagonal as in the multi-armed bandit setting
studied in~\Cref{sec:bandits}.  In this case, the calculations
of~\Cref{app:gen-confinv} show that the interval~\eqref{eqn:wrong-CI}
is valid.  More generally, given an arbitrary direction $\dir$, our
strategy will be to run a variant of online debiasing that effectively
reduces the problem to this favorable case.

%%%%%%%%%%%%%%%%%%%%%%%%%%%%%%%%%%%%%%%%%%%%%%%%%%%%%%%%%%%%%%%%%%

\subsubsection{Correct fixed-direction confidence intervals}
\label{sec:gen-CI}

Let us now describe the variant of online debiasing that can be used
to obtain asymptotically correct confidence intervals for fixed
directions.  Let $\dir \in \real^\dims$ be the direction of interest;
without loss of generality, we assume that $\enorm{\dir} = 1$.  We now
form an orthonormal basis of $\real^\dims$ with $\dir$ as its first
element---that is, a collection of orthonormal vectors $\{\dir_1 =
\dir, \dir_2, \ldots, \dir_\dims \}$.  Let $\Basis$ be the matrix with
$\dir_j^\top$ as its $j^{th}$ row.  Note that we have $\Basis \Basis^\top
= \MyIdDim$ and $\Basis^\top \e_1 = \dir$ by construction.  Using these
two properties, we can rewrite our model as

\begin{align*}
  \dir^\top \thetastar = \e_1^\top \Basis \thetastar \;\;
  \text{and} \;\; \y_i = \inprod{\Basis \x_i}{ \Basis \thetastar} +
  \error_i \quad \text{for all} \;\; i = 1, \ldots, \numobs.
\end{align*}
Consequently, in this new basis, estimating the scalar $\dir^\top
\thetastar$ is same as estimating the first coordinate of transformed
vector $\Basis \thetastar$.

This fact allows us to define a variant of online debiasing that
supports asymptotically exact confidence intervals for $\dir^\top
\thetastar$.  In particular, let us introduce the notation
% \notate{I fear that Sigma is too confusing a letter choice for this quantity, as readers will likely misremember it as representing the empirical covariance when it really represents roughly the inverse of that.  Maybe $\Omega$ or $M$ instead?}
\begin{gather*}
 \x_{\dir, i} = \Basis \x_i, \quad \X_{\dir, \numobs} = \X_\numobs
 \Basis^\top, \quad \SigmaMat_{\dir, \numobs}
  = (\Basis \SigmaMat_\numobs \Basis^\top)^{-1}, \quad \text{and} \quad
  \InvSigmaV_{\dir, \numobs} = \SigmaMat_{\dir, \numobs}^{-1}.
\end{gather*}
Define the block diagonal matrix $\diagCov_{\dir,\numobs}$ as 
by
\begin{align}
\label{eqn:Blk-diag-scaling-defn}
    \diagCov_{\dir, \numobs}^{\frac{1}{2}} = \begin{pmatrix} \omega_{11}^{-\frac{1}{2}} & \mathbf{0}^\top \\ 
    \mathbf{0} & \InvSigmaV_{22}^{-\frac{1}{2}} \end{pmatrix} \qquad  \text{where} \qquad
    \InvSigmaV_{\dir, \numobs} = \begin{pmatrix} \omega_{11} &\InvSigmaV_{12} \\ \InvSigmaV_{21} & \InvSigmaV_{22} \end{pmatrix},
\end{align}
where the matrix $\InvSigmaV_{22}^{-\frac{1}{2}}$ denotes the symmetric square root of the matrix $\InvSigmaV_{22}^{-1}$.  Now consider the estimator
\begin{align}
  \label{eqn:thetaBasis-defn}
  \thetaBasis & \mydefn \thetaBasisLS + \OpTuning_\numobs \cdot \diagCov_{\dir,
    \numobs}^{-\frac{1}{2}} \sum_{i = 1}^{\numobs} \w_i (\y_i -
  \x_{\dir, i}^\top  \thetaBasisLS),
\end{align}
where $\OpTuning_\numobs \mydefn
\opnorm{\diagCov_{\dir,\numobs}^{\frac{1}{2}} \SigmaMat_{\dir,
    \numobs}^{-\frac{1}{2}}}$ is a scalar which is at least one by
definition, and \mbox{$\thetaBasisLS \mydefn \SigmaMat_{\dir,
    \numobs}^{-1} \X_{\dir, \numobs}^\top \y$} is the OLS estimator
using the data $\{\y_i, \Basis \x_i \}_{i = 1}^\numobs$. We analyze
the behavior of $\thetaBasis$ under the following variant of
Assumption~\ref{assumption:A3}.

\subsubsection*{Assumption (A3)$^\prime$}
\begin{enumerate}[label=(A3)$^\prime$]
\item \label{assumption:A3prime} For each $\numobs$, the scalar
  $\tuneParScaled_\numobs > 0$ and positive semidefinite matrices
  $\Nseq{\XscaleMat}{i}$ with \mbox{$\XscaleMat_i \in \filtration_{i - 1}$} are
  chosen such that:
\begin{subequations}
\begin{align}
\tag*{(a) Asymptotic negligibility:} \max \limits_{i \in [\numobs]}
\;\; \left \{ \frac{1}{\tuneParScaled_\numobs}
\myinprod{\x_{\dir, i}}{\XscaleMat_i^{-1} \x_{\dir, i}} \right \} & \stackrel{\smallop{p}}{\longrightarrow} 0, \\
  \tag*{(b) Vanishing bias:} \sqrt{\tuneParScaled_\numobs \log
    \lambda_{\max}(\SigmaMat_{\dir,\numobs}) } \cdot \opnorm{ \tfrac{1}{\OpTuning_\numobs} \cdot \diagCov_{\dir, \numobs}^{\frac{1}{2}} \SigmaMat_{\dir, \numobs}^{-\frac{1}{2}}  -
    \W_\numobs \X_{\dir, \numobs} \SigmaMat_{\dir, \numobs}^{-\frac{1}{2}}} & \stackrel{\smallop{p}}{\longrightarrow}
  0, \quad \mbox{and} \\
\tag*{(c) Variance stability:} \opnorm{\MyIdDim - \Nsum{i} \w_i
  \x_{\dir, i}^\top \XscaleMat_i^{-\frac{1}{2}}}  & \stackrel{\smallop{p}}{\longrightarrow} 0.
\end{align}
\end{subequations}
\end{enumerate} 

\begin{props}
  \label{prop:gen-confinv}
Under \mbox{Assumptions~\ref{assumption:A1},\ref{assumption:A2},
  and~\ref{assumption:A3prime}}, given any consistent estimator
$\widehat{\noisesd}^2$ of $\noisesd^2$, the following interval is an
asymptotically exact $1 - \alpha$ confidence interval for $\dir^\top
\thetastar$
\begin{align}
\label{eqn:general-CI}
\left[ \e_1^\top \thetaBasis -
  \tfrac{ \OpTuning_\numobs \sigmahat }{\sqrt{\tuneParScaled_\numobs}}
  (\myinprod{\dir}{\SigmaMat_\numobs^{-1} \dir})^{\frac{1}{2}}
  \NormalQuantile{1 - \alpha/2}, \  \e_1^\top \thetaBasis +
  \tfrac{ \OpTuning_\numobs \sigmahat}{\sqrt{\tuneParScaled_\numobs}}
  (\myinprod{\dir} {\SigmaMat_\numobs^{-1}\dir})^{\frac{1}{2}}
  \NormalQuantile{1 - \alpha/2} \right].
\end{align}
\end{props}
\noindent See Section~\ref{app:gen-confinv} for the proof of this
claim.

A few comments regarding Proposition~\ref{prop:gen-confinv} are in order.
Observe that 
the length of the confidence intervals~\eqref{eqn:general-CI}
matches the length of the confidence interval~\eqref{eqn:wrong-CI}
up to a multiplicative factor $\OpTuning_\numobs$. Thus, it is interesting to understand 
the value of the scalar $\OpTuning_\numobs$.  Note that $\OpTuning_\numobs^2 =
\lambda_{\max}(\diagCov_{\dir, \numobs}^{\frac{1}{2}} \SigmaMat_{\dir, \numobs}^{-1} \diagCov_{\dir, \numobs}^{\frac{1}{2}})$, and a little calculation yields
\begin{align*}
     \diagCov_{\dir, \numobs}^{\frac{1}{2}} \SigmaMat_{\dir, \numobs}^{-1} \diagCov_{\dir, \numobs}^{\frac{1}{2}} 
     = \Id_{\dims} +   \begin{pmatrix}
        0  & (\InvSigmaV_{22}^{-\frac{1}{2}} \InvSigmaV_{21} \omega_{11}^{-\frac{1}{2}})^\top \\
         \InvSigmaV_{22}^{-\frac{1}{2}} \InvSigmaV_{21} \omega_{11}^{-\frac{1}{2}}  & {\mathbf{0}}_{\dims - 1}
     \end{pmatrix}.
\end{align*}
Thus, we have
\begin{align}
    \label{eqn:OpTuning-upperbound}
    1 \leq \OpTuning_\numobs^2 = 1 + 2 \cdot\enorm{\InvSigmaV_{22}^{-\frac{1}{2}} \InvSigmaV_{21} \omega_{11}^{-\frac{1}{2}}}^2,
\end{align}
which yields $\OpTuning_\numobs \approx 1$ when the vector
$\InvSigmaV_{21} \approx 0$.  To gain further intuition on when
$\OpTuning_\numobs \approx 1$, let us assume $\dir = \e_1$, i.e., we
are interested in obtaining a confidence interval for the coordinate
$\thetastarScalar_1$.  In this case, a natural choice of the basis matrix is
$\Basis = \MyIdDim$, and as a result, we have $\InvSigmaV_{21} =
(\SigmaMat_\numobs^{-1})_{21} = (\SigmaMat^{-1}_{21}, \ldots,
\SigmaMat^{-1}_{\dims1})$. Recall that for $j \neq 1$, the entry
$\SigmaMat^{-1}_{j1}$ is proportional to the (empirical) partial
correlation coefficient between the first and the $j^{th}$ coordinate,
conditioned the remaining $\dims - 2$ coordinates of $\x$; meaning
that $\InvSigmaV_{21} \approx 0$ when the first coordinate of $\x$ has
small correlation with all linear functions of the other $\dims - 1$
coordinates of $\x$.

Finally, we point out that the assumption~\ref{assumption:A3prime}
is not significantly stronger than the original
assumption~\ref{assumption:A3}. To fix ideas, we again assume $\dir = \e_1$
and $\Basis = \Id$.
In that case, assumption~\ref{assumption:A3prime} and~\ref{assumption:A3}
only differ in the vanishing bias condition~(b). Assuming
\mbox{$\sqrt{\tuneParScaled_\numobs \log \lambda_{\max}(\SigmaMat_\numobs)}
= \smalloP(1)$} and condition~\ref{assumption:A3}(b) holds, we have
\begin{align}
  &\phantom{=} \sqrt{\tuneParScaled_\numobs \cdot \log
    \lambda_{\max}(\SigmaMat_{\e_1,\numobs})} \cdot \opnorm{ \tfrac{1}{\OpTuning_\numobs} \cdot
    \diagCov_\numobs^{\frac{1}{2}} \SigmaMat_\numobs^{-\frac{1}{2}}-
    \W_\numobs \X_\numobs \SigmaMat_\numobs^{-\frac{1}{2}} } \notag \\
     &\leq
  \sqrt{\tuneParScaled_\numobs \cdot \log
    \lambda_{\max}(\SigmaMat_\numobs) } \cdot \opnorm{ \tfrac{1}{\OpTuning_\numobs}
     \cdot\MyIdDim-
    \W_\numobs \X_\numobs \SigmaMat_\numobs^{-\frac{1}{2}} } \notag \\ 
    & \quad \quad \quad \quad +  
   \sqrt{\tuneParScaled_\numobs \cdot \log
    \lambda_{\max}(\SigmaMat_\numobs) } \cdot (1 +
  \tfrac{1}{\OpTuning_\numobs} \cdot \opnorm{\diagCov_\numobs^{\frac{1}{2}}
    \SigmaMat_\numobs^{-\frac{1}{2}}}) \notag \\
     &= \smalloP(1) + \smalloP(1) \stackrel{\smallop{p}}{\longrightarrow} 0 
     \label{eqn:A3_to_A3prime} .
\end{align}
The first step  uses the fact $\lambda_{\max}(\SigmaMat_\numobs) = \lambda_{\max}(\SigmaMat_{\dir, \numobs})$ for any basis matrix $\Basis$.
The second step uses the vanishing bias condition~\ref{assumption:A3}(b), 
the fact that the dimension $\dims$ is 
fixed, and the upper
bound \mbox{$\opnorm{\diagCov_\numobs^{\frac{1}{2}}
  \SigmaMat_\numobs^{-\frac{1}{2}}} = \OpTuning_\numobs$}.

\subsubsection{Comparison with least squares confidence intervals}
An interesting consequence of the results so far, is that we can construct a confidence intervals based on any consistent estimator $\thetahat$ of $\thetastar$. To illustrate this, we focus on constructing a confidence interval for the mean of the first arm ($\theta_1$) of a multiarmed bandit problem. Theorem~\ref{thm:asymp-normality} and Proposition~\ref{prop:gen-confinv} ensures that 
$$[\widehat{\theta}^{OD}_{1} - \sqrt{1/\gamma_n} \cdot \frac{z_{1 - \alpha/2}}{\sqrt{\SigmaMat_{11}}}, \;\;\widehat{\theta}^{OD}_1 + \sqrt{1/\gamma_n} \cdot \frac{z_{1 - \alpha/2}}{\sqrt{\SigmaMat_{11}}}]$$
is an asymptotically exact $1 - \alpha$ confidence interval for $\theta^\star$. As an immediate consequence, we have that for any weakly consistent estimator $\widehat{\theta}_1$ for $\theta^\star_1$ the confidence interval  
$$[\widehat{\theta}_{1} -  \sqrt{1/\gamma_n} \cdot \frac{z_{1 - \alpha/2}}{\sqrt{\SigmaMat_{11}}}, \;\;\widehat{\theta}_1 + \sqrt{1/\gamma_n} \cdot \frac{z_{1 - \alpha/2}}{\sqrt{\SigmaMat_{11}}}]$$
is also an asymptotically exact $1 - \alpha$ confidence interval. Accordingly, when the OLS estimator $\widehat{\theta}^{OLS}$ is consistent, the above argument also allows us to construct 
an asymptotically valid confidence interval for the least squares estimator. 

Interestingly, while the width of a valid OLS interval is tightly constrained by its own bias, the width of a valid online debiasing interval can be significantly smaller. In particular, one can construct confidence intervals which are $\sqrt{\log(n)}$ times smaller in width than the ones obtained from lest square estimator; see Appendices~\ref{sec:LS-vs-online} and~\ref{sec:post-debiasing-proof} for details.

%%%%%%%%%%%%%%%%%%%%%%%%%%%%%%%%%%%%%%%%%%%%%%%%%%%%%%%%%%%%%%%%%%%

\subsection{Minimax lower bounds}
\label{SecLower}

Thus far, we have derived two guarantees for online debiasing
procedures: asymptotic normality in Theorem~\ref{thm:asymp-normality}
along with confidence intervals in Proposition~\ref{prop:gen-confinv}.
It is natural to wonder in what sense these guarantees are optimal.
Accordingly, this section is devoted to lower bounds that apply to the
performance of \emph{any} estimator $\thetahat$.  These bounds are
derived within the classical minimax framework and cover two
particular risk measures.

Our first risk measure involves the \emph{Mahalanobis pseudometric}:
given an arbitrary positive semi-definite matrix $\mahal$, possibly
random, this pseudometric\footnote{We parameterize the Mahalnobis
pseudometric slightly differently than standard definitions, using
$\mahal$ as opposed to its inverse for the quadratic form.  This is
only for notational ease when $\mahal$ has a non-trivial null space.}
is given by
\begin{align}
  \label{EqnMahal}
  \lVert\thetahat - \thetastar\rVert_\mahal \mydefn \lVert
  \mahal^{\frac{1}{2}}(\thetahat - \thetastar) \rVert_2,
\end{align}
and we provide lower bounds on the squared form of this pseudometric
in part (a) of Theorem~\ref{thm:Minimax-Lowerbound}, below.  Notably,
our analysis allows for the matrix $\mahal$ to also depend on the
dataset $\{ \x_i, \y_i \}_{i = 1}^\numobs$ itself, so that for example, setting $\mahal =
\SigmaMat_\numobs$ is a valid choice.

Our second risk measure corresponds to the length of a two-sided
confidence interval. For a given vector $\dir \in \real^\Dim$ and
significance level $\alpha \in (0,1)$, let \mbox{$ \CI_{\alpha, \dir}
  \equiv [\ell_\alpha, u_\alpha] \subseteq \real$} be any level
$\alpha$ confidence interval for the scalar
$\myinprod{\dir}{\thetastar}$, so that by definition, we have
\begin{align}
\label{eqn:valid-CI}
  \Prob_{\thetastar} \big[\myinprod{\dir}{\thetastar} \in \CI_{\alpha, \dir}
     \big] & \geq 1 -  \alpha \qquad \text{for
    all } \;\; \thetastar \in \real^\dims.
\end{align}
We are interested in finding the smallest such confidence interval,
and part (b) of Theorem~\ref{thm:Minimax-Lowerbound} provides a lower
bound on its length ${|\CI_{\alpha, \dir}| \mydefn u_\alpha -
  \ell_\alpha}$.

Our bounds apply to any estimator $\thetahat$, meaning a measurable
function of the data as well as the data collection process.  The data
collection process is summarized by a collection of (potentially
randomized) selection algorithms, each of the form
\mbox{$\selection_i: (\reals \times \reals^\Dim)^{i - 1} \to
  \reals^\Dim$,} which take the observed data $\{ (\x_j, \y_j)
\}_{j=1}^{i-1}$ up to time $i$ and output a new observation
$\x_{i}$. With a slight abuse of notation, we refer to
\mbox{$\Selections_\numobs \mydefn (\selection_i)_{i \in [\numobs]}$}
as the \emph{selection algorithm} of the data collection process.

\begin{theos}
\label{thm:Minimax-Lowerbound}
Fix any selection algorithm $\Selections_\numobs$.  Under the linear
model~\eqref{eqn:linear-reg-model} with i.i.d.\ Gaussian noise
\mbox{$\error_i \sim \normal(0, \sigma^2)$} and data collected using
$\Selections_\numobs$, the following claims hold:
\begin{enumerate}[label=(\alph*)]
\itemsep 6pt
\item For any (possibly random) matrix $\mahal$ such that $\Exs[\tr(
  \SigmaMat_\numobs^{-1} \mahal)]$ is finite, we have
\begin{subequations}  
  \begin{align}
\label{EqnMinimaxMSE}    
\inf_{\thetahat} \sup_{\thetastar \in \real^\dims}
\E \lVert
  \thetahat - \thetastar \rVert^{2}_{\mahal} \geq \sigma^2 \Exs[\tr(
    \SigmaMat_\numobs^{-1} \mahal) \big],
\end{align}
where the infimum is taken over any estimator $\thetahat$, potentially
depending on $\Selections_\numobs$, in addition to the data.
\item
There is a universal constant $C > 0$ such that for any pair
$(\numobs, \Dim)$ with $n \geq d^3/C$ and $d \geq 2$, there exists a selection
algorithm $\Psi_\numobs$ and direction $\dir \in \reals^d$ such that
\begin{align*}
\inf_{\cE} \sup_{\thetastar} \E\left[\frac{(\ip{\dir, \thetastar} -
    \cE)^2}{\norm{\dir}^2_{\SigmaMat_n^{-1}}}\right] \geq C \cdot d
\sigma^2 \log\left(n\right)\,,
\end{align*}
where the infimum is taken over all $\filtration_n$-measurable
estimates $\cE$ of the scalar $\ip{\dir,\thetastar}$
\item Suppose that $\Exs[\SigmaMat_n]$ exists and is invertible.  Then
  for any direction $\dir \in \real^\Dim$ and scalar \mbox{$\alpha \in
    (0,1/8)$,} we have
  \begin{align}
\label{EqnMinimaxCI}    
\inf_{\CI_{\alpha, \dir}} \sup_{\thetastar \in \real^\dims} \E \big[
  \lvert \CI_{\alpha, \dir} \rvert \big] \geq 2\noisesd \cdot z_{1-
  \alpha/2} \cdot\Exs\left\{\left(\dir^\top \SigmaMat_n^{-1}
\dir\right)^\frac{1}{2}\right\},
\end{align}
\end{subequations}
where the infimum is taken over any procedure, potentially depending
on $\Selections_\numobs$ in addition to the data, that returns valid
level $1- \alpha$ confidence intervals $\CI_{\alpha, \dir}$ for
$\ip{\dir,\thetastar}$ (cf. definition~\eqref{eqn:valid-CI}).
\end{enumerate}
\end{theos}
\noindent We provide the proofs of parts (a), (b) and (c)
of~\Cref{thm:Minimax-Lowerbound}
in~\Cref{proof:thm:Minimax-Lowerbound},
\Cref{proof:thm:Minimax-Lowerbound-minimax}
and~\Cref{proof:Confidence-interval-lower-bound}, respectively.

\subsection*{Comments on part (a): instance-dependent lower bound on MSE}
In order to gain intuition for the MSE bound in part (a), it is
helpful to begin with the simplest case---that is, the non-adaptive
setting.  Consider the classical problem of fixed design linear
regression, in which the covariates (and hence $\SigmaMat_\numobs$)
are viewed as fixed, and the additive noise is zero-mean Gaussian with
variance $\noisesd^2$.  In this case, the standard OLS estimate
$\thetaLS$ follows the Gaussian distribution $\normal(\thetastar,
\noisesd^2 \SigmaMat_\numobs^{-1})$ for any sample size $\numobs$.
Consequently, for any fixed matrix $\mahal$, we have the equality
\begin{align}
\E \lVert \thetahat - \thetastar \rVert^{2}_{\mahal} = \sigma^2
\Exs[\tr(\SigmaMat_\numobs^{-1} \mahal)].
\end{align}
This simple calculation shows that the lower
bound~\eqref{EqnMinimaxMSE} is unimprovable in general.

Of course, the more substantive content of
Theorem~\ref{thm:Minimax-Lowerbound}(a) lies in the fact that it
allows for adaptive data collection, along with potentially random
choices of $\mahal$.  One interesting choice is the random matrix
$\mahal = \SigmaMat_\numobs$, for which the
bound~\eqref{EqnMinimaxMSE} guarantees that \mbox{$\E[\lVert \thetahat
    - \thetastar\rVert_{\SigmaMat_\numobs}^2] \geq \noisesd^2 \Dim$.}
It is worth comparing this lower bound to~\Cref{thm:asymp-normality}.
From the arguments used to prove this theorem, and under a
\emph{mildly stronger} version of
Assumption~\ref{assumption:A3}---which the convergence in distribution
conditions are replaced by convergence in $L_1$---it can be shown that
\begin{align}
\label{EqnL1-convergence}
\lim_{\numobs \rightarrow +\infty} \tuneParScaled_\numobs
\|\thetaDecorr - \thetastar\|_{\SigmaMat_\numobs}^2 & = \noisesd^2
\dims.
\end{align}
See~\Cref{sec:L1-convergence} for the details of this argument.

As discussed in~\Cref{sec:applications} to follow, in many practical
problems of interest, the tuning parameter $\tuneParScaled_\numobs$
typically scales logarithmically in the sample size $\numobs$, and
also our choice of the tuning parameters ensure that the
aforementioned stronger version of assumption~\ref{assumption:A3} is
satisfied. Consequently, the result~\eqref{EqnL1-convergence}, when
combined with the lower bound~\eqref{EqnMinimaxMSE}, shows that the
online debiasing procedure is instance-optimal up to logarithmic
factors.

%%%%%%%%%%%%%%%%%%%%%%%%%%%%%%%%%%%%%%%%%%%%%%%%%%%%%%%%%%%%%%%%%%%%%%%%

\subsubsection*{Comments on part (b): minimax lower bound on MSE}
Part (b) of Theorem~\ref{thm:Minimax-Lowerbound} provides a minimax
lower bound on the MSE. It shows that in the adaptive setting, the
scaling of $\gamma_n$ needs to at least $\log(\numobs)$, and moreover,
this logarithmic scaling is unavoidable when dimension $d \geq 2$. As
discussed in the last comment, for many problems, we can take
$\gamma_n$ arbtitarily close to $\log(\numobs)$; thus, we conclude
that the online debiased estimator is miniamx optimal when $d \geq
2$. Finally, we point out that the lower bound result is \emph{not
true} when $d = 1$. In this case the logarithmic dependence on $n$
becomes doubly logarithmic, as is consistent with the law of the
iterated logarithm that underlies this behavior in that case.

\subsection*{Comments on part (c): bounds on lengths of CIs}
\Cref{thm:Minimax-Lowerbound}(c) provides a lower bound on the width
of any confidence interval for the scalar $\dir^\top \thetastar$ that
is valid when the data set is collected in an adaptive manner. To the
best of our knowledge, this is the first result providing a lower
bound on the width of confidence intervals in an adaptive setting.

%%%%%%%%%%%%%%%%%%%%%%%%%%%%%%%%%%%%%%%%%%%%%%%%%%%%%%%%%%%%

\subsection{Choices of the tuning parameters}
\label{SecSufficientThree}

Let us now return to the practical issue of choosing the tuning
parameters $\tuneParScaled_\numobs$ and $\Nseq{\XscaleMat}{i}$ of our
debiasing procedures.  In particular, these parameters must be chosen
appropriately so as to ensure that either
Assumption~\ref{assumption:A3}, or its variant in
Assumption~\ref{assumption:A3prime}, is satisfied.

We analyze practical default choices that are based upon on a
deterministic matrix $\SigmaMatLB_\numobs$ that acts as a lower bound
on the sample covariance matrix $\SigmaMat_\numobs$.  Let $\{
\SigmaMatLB_\numobs \}_{\numobs \geq 1}$ be a sequence of
$\Dim$-dimensional diagonal matrices with nonnegative entries
such that
\begin{align}
  \label{eqn:SigmaLB-conditions}
\opnorm{ \SigmaMatLB_\numobs^{\frac{1}{2}}
  \diag(\SigmaMat_\numobs^{-1}) \SigmaMatLB_\numobs^{\frac{1}{2}} } =
\BigoP(1) \quad \text{and} \quad \lambda_{\min}(\SigmaMatLB_\numobs)
\almostSurely \infty.
\end{align} 
For a given $\numobs$, we define a collection of (diagonal) scaling
matrices
% \begin{subequations}
\begin{align}
  \label{eqn:scaling_mat-choice}
  \XscaleMat_{i, \numobs} & \mydefn \max \left \lbrace \diag\left(
  \SigmaMat_i^{-1} \right)^{-1}, \SigmaMatLB_\numobs \right \rbrace,
\end{align}
where $\max\{\cdot, \cdot\}$ denotes the element-wise maximum
operator.\footnote{The choice~\eqref{eqn:scaling_mat-choice} of
scaling matrix is especially easy to understand for multi-armed bandit
problems, where the scaling matrix $\XscaleMat_{i, \numobs}$ can be
written as $\XscaleMat_{i, \numobs} = \max \left\lbrace \SigmaMat_i,
\;\; \SigmaMatLB_\numobs \right\rbrace$.  Assuming that $\SigmaMat_i =
\max\{\SigmaMatLB_\numobs, \SigmaMat_i \}$ for large value of $i$, we
see that the tuning parameter $\XscaleMat_{i, \numobs}$ is the sample
covariance matrix up to time $i$.  This assumption indeed holds for
Corollaries~\ref{cor:bandits-corr}--~\ref{cor:multidim-gen-case} to be
presented in the sequel.}
%%%%%%%%%%%%%%%%%%%%%%%%%%%%%%%%%%%%%%%%%%%%%%%%%%%%%%%%%%%%%%%%%%%%%%%
% Next, we choose a sequence of tuning
% parameters $\Nseq{\tuneParScaled}{i}$ such that
% \begin{align}
%   \label{eqn:tuneparScaled-condition}
%   \max_{i \in [\numobs]} \;\; \frac{1}{\tuneParScaled_\numobs}
%   \myinprod{\x_i}{\SigmaMatLB_\numobs^{-1} \x_i} \stackrel{\smallop{p}}{\longrightarrow} 0.
% \end{align}
% \end{subequations}
We point out that it is relatively straightforward to find a diagonal
matrix $\SigmaMatLB_\numobs$ satisfying the
condition~\eqref{eqn:SigmaLB-conditions} as long as
$\lambda_{\min}(\SigmaMat_\numobs) \rightarrow \infty$ almost
surely. For simplicity, let us assume that the covariates $\x_i$ are
uniformly bounded and that the minimum eigenvalue of the matrix
$\SigmaMat_\numobs$ is lower bounded as \mbox{$\lambda_{\min}
  (\SigmaMat_\numobs) \geq \log^2 \numobs$} with high probability
(see~\Cref{sec:Growth-cond-on-lambda-min} for one sufficient condition
for this bound to hold).  Then, with our recommended default choices
\mbox{$\SigmaMatLB_\numobs = \frac{\log\log(n)}{\gamma_n} \MyIdDim$}
and $\gamma_n = \frac{1}{\log(n)\log\log(n)}$ the
condition~\eqref{eqn:SigmaLB-conditions} is satisfied.

\begin{props}
\label{prop:zero-bias-lemma}
Consider the solutions $\Nseq{\w}{i}$ obtained from the optimization
problem~\eqref{eqn:Wn-expression-original} using parameters
$\NseqNosub{\XscaleMat_{i, \numobs}}{i}$ 
defined in equation~\eqref{eqn:scaling_mat-choice} and (non-random) $\tuneParScaled_\numobs > 0$. Then we have the operator norm bound
\begin{align}
  \label{EqnVanishBiasProp}
  \opnorm{ \Id - \W_\numobs \X_\numobs
    \SigmaMat_\numobs^{-\frac{1}{2}} } = \BigoP(\Dim^2).
\end{align}
In particular, if $\tuneParScaled_\numobs = \smalloP 
\left(\Dim^4\log \lambda_{\max}(\SigmaMat_\numobs) \right)$,
the vanishing bias and asymptotic negligibility conditions in
Assumption~\ref{assumption:A3} are satisfied.
\end{props}
\noindent See~\cref{app-bias-control-lemma} for the proof of this
claim.   \\

\subsubsection{Sharper bound for multi-armed bandits} The dimension
dependence of the upper bound~\eqref{EqnVanishBiasProp} can be removed
in many concrete applications in which we have additional information
about the data generating process.

As one concrete example, in the multi-armed bandit model of the sequel
(\cref{sec:bandits}), the upper bound can be sharpened to
\begin{align}
\label{eqn:bias-bound-improved} 
\opnorm{\Id - \W_\numobs \X_\numobs \SigmaMat_\numobs^{-\frac{1}{2}} }
= \BigoP(1).
\end{align}  
See the end of Appendix~\ref{app-bias-control-lemma} for a proof of
this claim, and see the proofs of the
Corollaries~\ref{cor:bandits-corr},~\ref{cor:unit-root-autoreg},
and~\ref{cor:multidim-gen-case} in the sequel for more details.

\subsubsection{Verifying the growth condition on $\lambda_{\min}(\SigmaMat_\numobs)$}
\label{sec:verifying-growth}
\label{sec:Growth-cond-on-lambda-min}
In the absence of any additional assumption on the data collection
method, it may be difficult to verify the lower bound
\mbox{$\lambda_{\min} (\SigmaMat_\numobs) \geq \log^2(\numobs)$.}

A simple fix to this problem is to collect $\log^2(\numobs)$ many
additional data points with $\x_i$ chosen uniformly at random from a
$d$-dimensional unit sphere and append them to the original
dataset. Then the condition
$\lambda_{\min}(\SigmaMat^{\textrm{new}}_\numobs) \geq
\log^2(\numobs)$ is automatically satisfied for
$\SigmaMat^{\textrm{new}}_\numobs$ --- the sample covariance matrix of
the new augmented dataset.
% 

%%%%%%%%%%%%%%%%%%%%%%%%%%%%%%%%%%%%%%%%%%%%%%%%%%%%%%%%%%%%%%%%%%%%%%%%%%%%%%%%%%%%%%%%%%%%%%%%%%%%%%%%%%

\section{Applications}
\label{sec:applications}

We next illustrate the concrete consequences of our results.
Sections~\ref{sec:bandits} and~\ref{sec:one-dim} are devoted to
multi-armed bandit problems and autoregressive time series models,
respectively, while \cref{sec:suff-exploration} discusses active
learning with exploration.  We end each section with an empirical
evaluation of online debiasing.  Specifically, we compare the
confidence interval (CI) coverage and width of four methods: our
online debiasing estimator~\eqref{eqn:ThetaDecorr-defn},
OLS~\eqref{EqnLSDecompose} with standard but potentially invalid
Gaussian intervals, the \mbox{$W$-decorrelation estimator}
of~\citet{deshpande2017accurate}, and a valid CI based on the
concentration inequality of~\citet{abbasi2011online}.  We highlight
that the CIs for OLS are based on the distributional assumption
\mbox{$\SigmaMat_n^{\frac{1}{2}} (\thetaLS - \thetastar) \sim \Ncal(0,
  \MyIdDim)$}. This property, while true when the covariates $\{\x_i
\}_{i = 1}^n$ are selected in a non-adaptive manner, need not hold for
adaptively collected
covariates~\citep{deshpande2017accurate,zhang2020inference}, and as a
consequence, the corresponding CIs need not give the correct coverage.
Meanwhile, the valid concentration inequality-based
intervals~\citep{abbasi2011online} are guaranteed to provide at least
the nominal coverage but are often unnecessarily wide.

%%%%%%%%%%%%%%%%%%%%%%%%%%%%%%%%%%%%%%%%%%%%%%%%%%%%%%%%%%%%%%%%%%%%%%%%%%%%%%%%%%%%%%%%%%%%%%%%%%%%%%%%%%

\subsection{Multi-armed bandits}
\label{sec:bandits}
Consider a multi-armed bandit with $\Dim$ arms indexed by the set
\mbox{$[\Dim] \mydefn \{1, \ldots, \Dim \}$.}  At each time $i \in
     [\numobs]$, a bandit algorithm selects an arm $k_i \in [\Dim]$
     and observes the reward
\begin{align}
\label{model:bandits}
\y_i = \myinprod{\basis_{k_i}}{\armMeans} + \epsilon_i,
\end{align}   
where $\basis_{k_i}$ is the $k_i^{th}$ basis vector in dimension
$\Dim$ and $\armMeans \in \real^\Dim$ is the vector containing the
mean rewards of $\Dim$ arms.  We assume that the noise sequence
$\Nseq{\error}{i}$ satisfies Assumption~\ref{assumption:A1}.  Notably,
the multi-armed bandit model~\eqref{model:bandits} is a special case
of the adaptive linear regression model~\eqref{eqn:linear-reg-model}
with $\x_i = \basis_{k_i}$ for each $i \in [\numobs]$.

Since the bandit observation model~\eqref{model:bandits} has a simple
linear form, the OLS solution $\thetaLS$ is a standard estimate of the
reward vector $\thetastar \in \real^\Dim$.  As we mentioned earlier,
the behavior of the OLS estimate depends on the stability of the
matrix~$\SigmaMat_\numobs$; see the covariance stability
condition~\eqref{eqn:Cond-Cov-stability}.  In the
paper~\citep{deshpande2017accurate}, the authors conjectured based on
empirical evidence that for various popular data selection algorithms,
including the Upper Confidence Bound (UCB), Thompson Sampling, and
$\greedyEps$-greedy algorithms (see the
book~\citep{lattimore2020bandit}), the stability
condition~\eqref{eqn:Cond-Cov-stability} is not satisfied when there
are multiple optimal arms. In recent work, Zhang et
al.~\citep{zhang2020inference} established the validity of this
conjecture for the two-armed bandit problem: when the two means are
equal, then the OLS estimate fails to have a Gaussian limiting
distribution.

In sharp contrast to these negative results for \OLS,
\cref{cor:bandits-corr} to follow guarantees that the online debiasing
estimator~\eqref{eqn:ThetaDecorr-defn} is asymptotically normal under
a mild assumption on the minimum number of times that each arm is
pulled.  More precisely, for each arm $k \in [\Dim]$ and round $i \in
[\numobs]$, let $\armCount_{k, i}$ denote the number of times $k$ is
pulled in the first $i$ rounds, and define the minimum
\mbox{$\armCount_{\min} \mydefn \min \limits_{k \in [d]} \armCount_{k,
    \numobs}$,} and maximum \mbox{$\armCount_{\max} = \max \limits_{k
    \in [d]} \armCount_{k, \numobs}$} arm counts.  Then the scaled
sample covariance is a $\Dim \times \Dim$ diagonal matrix, in which
the $k^{th}$ diagonal entry corresponds to the number of times that
arm $k$ is pulled within the first $i$ rounds:
\begin{align}
  \label{eqn:bandit-notations}
  \armCountMat_i & = \diag \left( \armCount_{1, i}, \ldots,
  \armCount_{\Dim, i} \right).
\end{align}

We assume a lower bound on the minimum number of times that each arm
is pulled---namely,
\begin{align}
\label{eqn:bandits-minimum-arm-LB}
\armCount_{\min} \geq ( \log \numobs)^2.
\end{align}  
Moreover, we implement the debiasing
estimate~\eqref{eqn:ThetaDecorr-defn} with the choice of tuning
parameters
\begin{align}
  \label{eqn:bandits-tuning}
  \tuneParScaled_\numobs = \frac{1}{(\log \numobs) \cdot \log
    \log(\numobs)} \quad \text{and} \quad \XscaleMat_{i, \numobs} =
  \max \left\lbrace \armCountMat_i, \;\; (\log \numobs)^2 \cdot
  \Id_{\Dim} \right \rbrace,
\end{align}
where $\max \left\lbrace \cdot, \cdot \right\rbrace$ denotes the
element-wise maximum operator.

\begin{cors}
\label{cor:bandits-corr}
Suppose the minimum arm pull
condition~\eqref{eqn:bandits-minimum-arm-LB} and the moment
condition~\ref{assumption:A1} are valid. Then, given any consistent
estimate $\sigmahat^2$ of the error variance $\errorstd^2$, the
estimate $\thetaDecorr$ obtained using the tuning parameter
choices~\eqref{eqn:bandits-tuning} satisfies
\begin{align}
\label{EqnBanditsCorr}  
  (\widehat{\noisesd}^2 \cdot (\log \numobs) \cdot \log\log(\numobs) \big)^{-1/2}
  \cdot \armCountMat_\numobs^{\frac{1}{2}} \left( \thetaDecorr -
  \armMeans \right) \indistrb \Ncal(0, \MyIdDim).
  \end{align}
\end{cors}
\noindent See~\cref{proof:cor-bandits} for the proof of this claim.
\cref{cor:bandits-corr} also enables us to construct asymptotically exact
confidence regions for $\thetastar$. Moreover, the sample covariance matrix $\SigmaMat_n$
is diagonal, and as a result, we can also construct confidence intervals of the coordinates
$\theta_i^*$; see the proof of Proposition~\ref{prop:gen-confinv} for details. Finally,
for a direction $\dir$ which is not a standard basis direction,
we can obtain an asymptotically exact $1 - \alpha$ confidence interval
of $\dir^\top \thetastar$ using Proposition~\ref{prop:gen-confinv}; see the comments following
Corollary~\ref{cor:multidim-gen-case} for further details.

%%%%%%%%%%%%%%%%%%%%%%%%%%%%%%%%%%%%%%%%%%%%%%%%%%%%%%%%%%%%%%%%%%%%%%%%%%%%%%%%%%%%%%%%%%%%%%%%%%%%%%%%%%%

\subsubsection{Numerical experiment}
\label{sec:numericals-bandits-small}
\cref{fig:Thompson-sampling} 
illustrates the performance of online debiasing with bandit tuning~\eqref{eqn:bandits-tuning}.
Here 
we consider a two-armed bandit problem~\eqref{model:bandits} with arm-mean vector  
$\thetastar = (0.3, 0.3)^\top$ and i.i.d.\ standard normal error 
$\{ \error_i \}_{i = 1}^\numobs$.
The covariates $\{ \x_i \}_{i = 1}^\numobs$  were
generated using the Thompson sampling algorithm~\citep{thompson1933likelihood}, and we consider confidence intervals (CIs) for $\theta_1^*$.

We observe first that online debiasing provides appropriate coverage for all confidence levels.
Meanwhile, the \OLS lower tail interval severely undercovers, and  W-decorrelation undercovers for both tails despite having larger widths than online debiasing.
Finally, the concentration CI provides 100\% coverage for all  confidence levels but yields intervals uniformly larger than the online debiasing CIs.  
In Appendix~\ref{sec:bandits-simulation}, we present analogous results for two other popular multi-armed bandit algorithms, the upper confidence bound (UCB) and $\greedy$-greedy algorithms.   

%%%%%%% The least squares page %%%%%%
\begin{figure}[H]%
    \centering
    %\begin{subfigure}{\linewidth}
    \includegraphics[width=.33\textwidth]{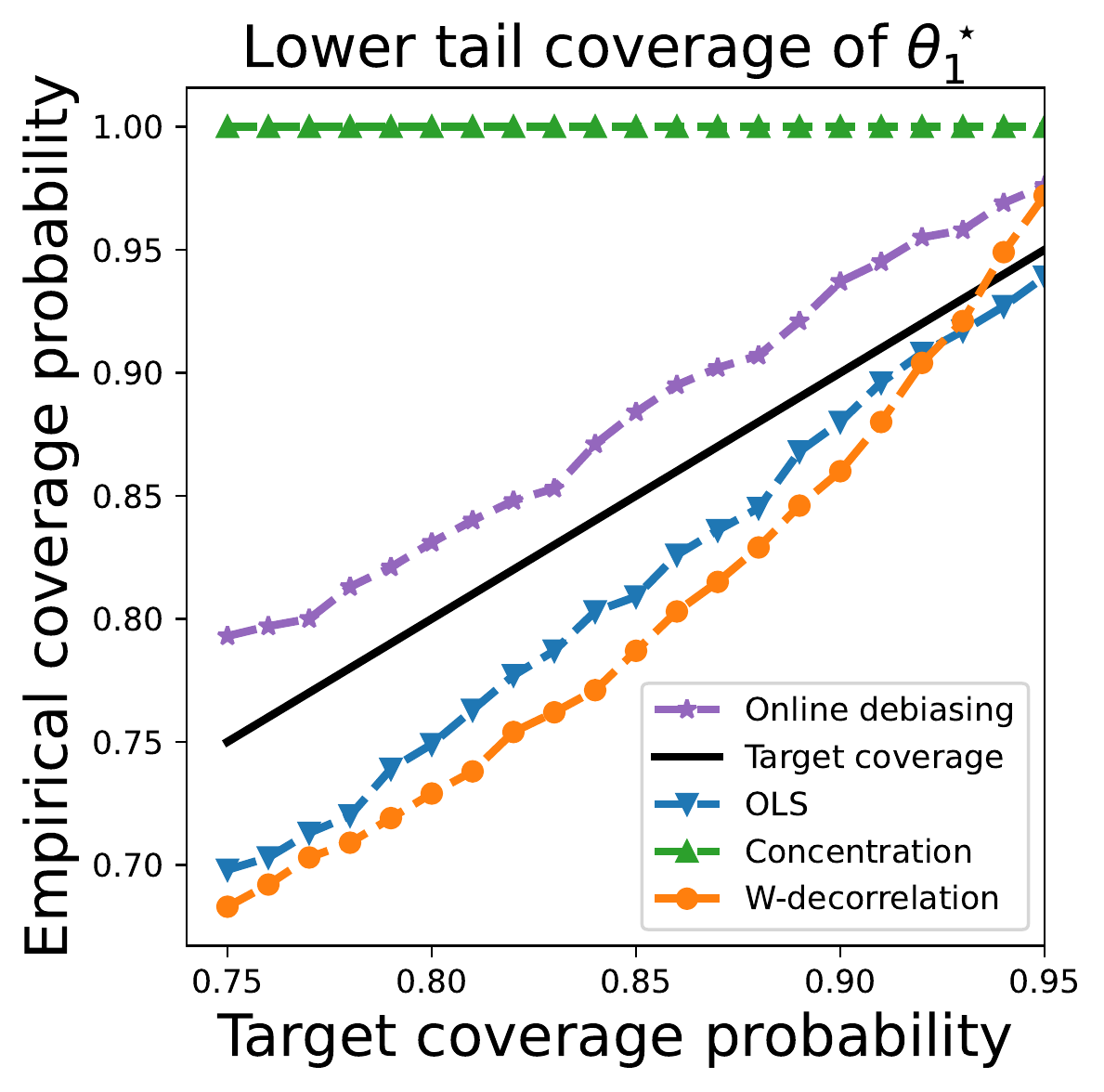}%
        \includegraphics[width=.33\textwidth]{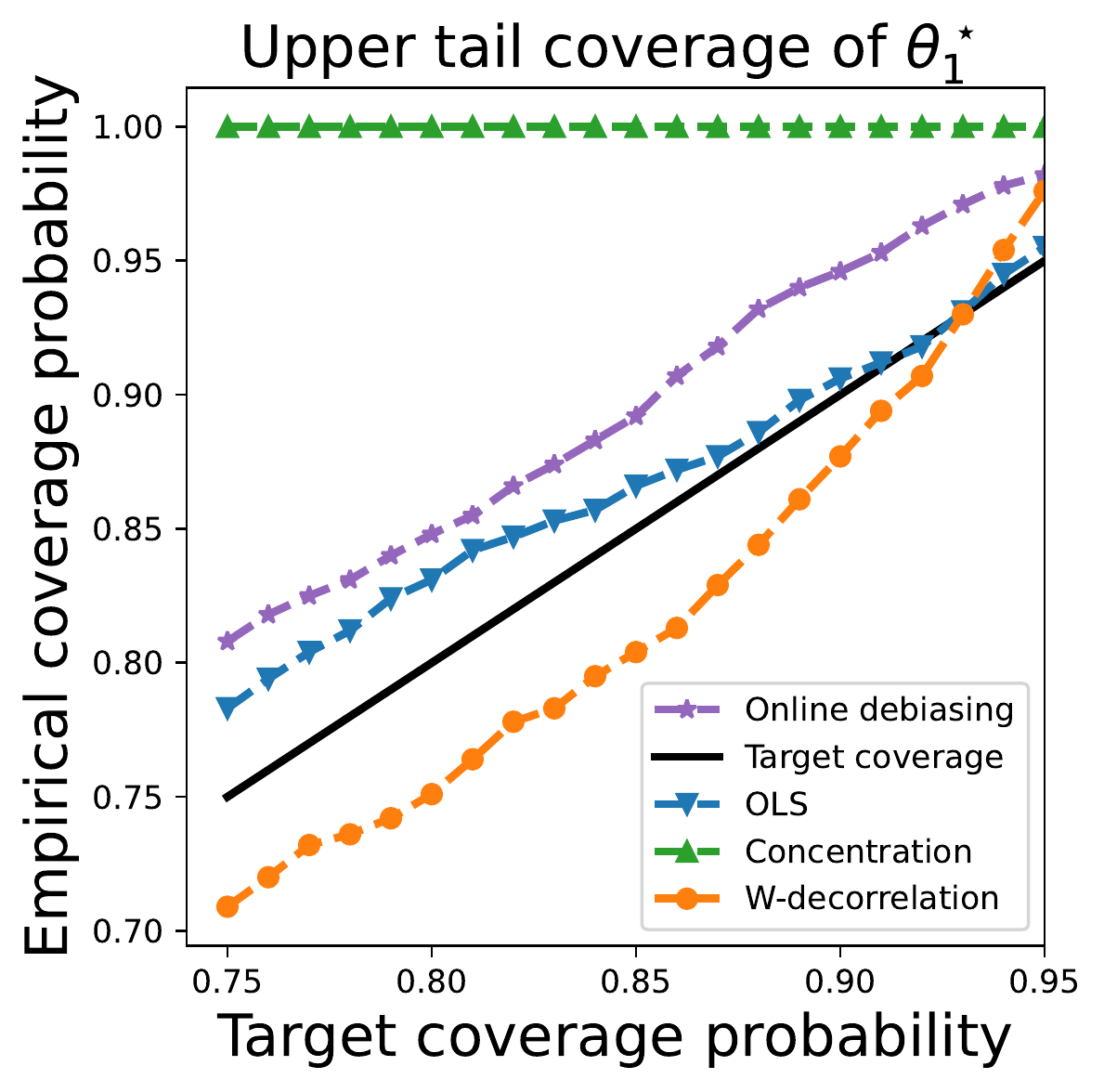}%
    \includegraphics[width=.35\textwidth]{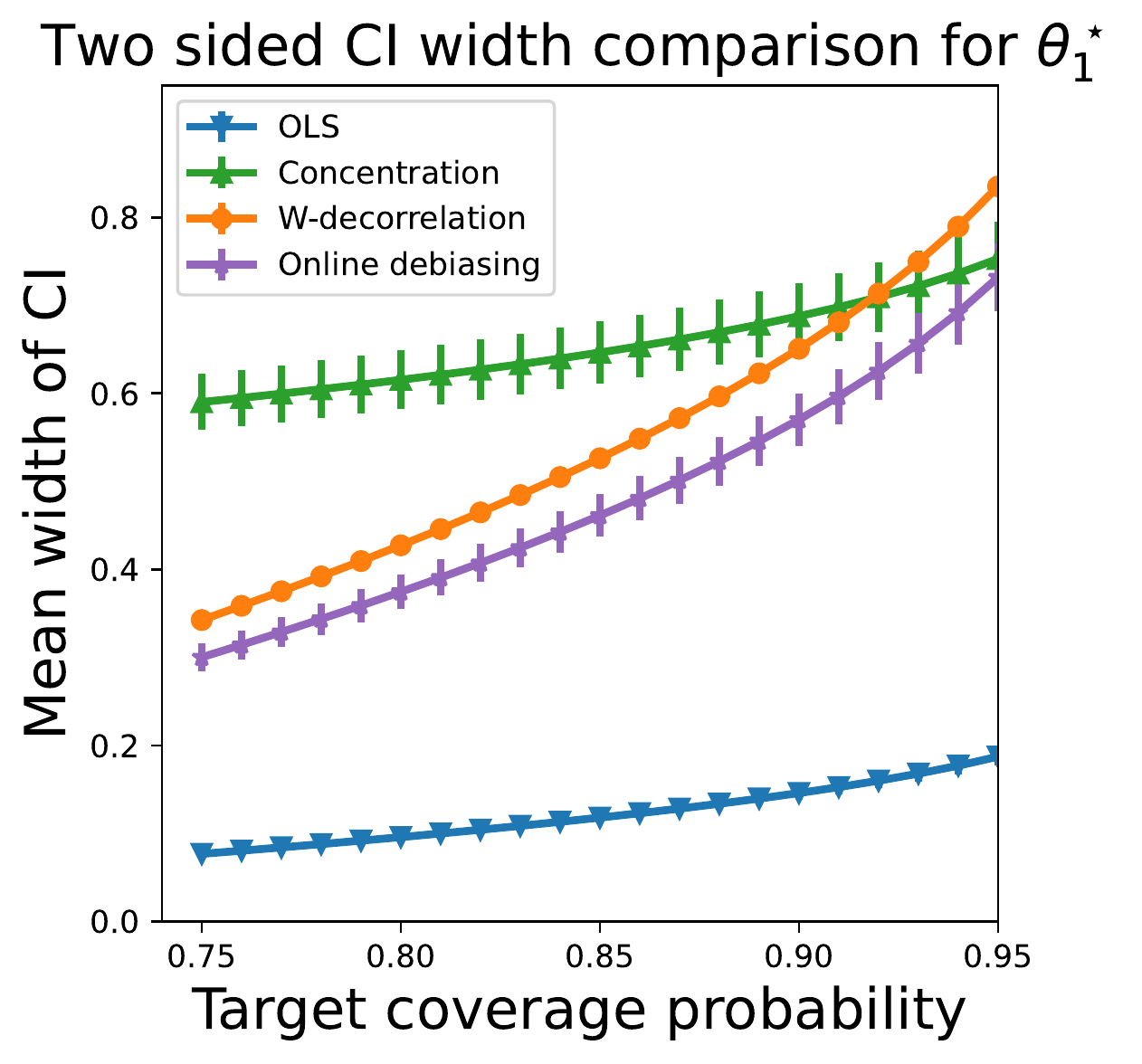} \\
     \includegraphics[width=.33\textwidth]{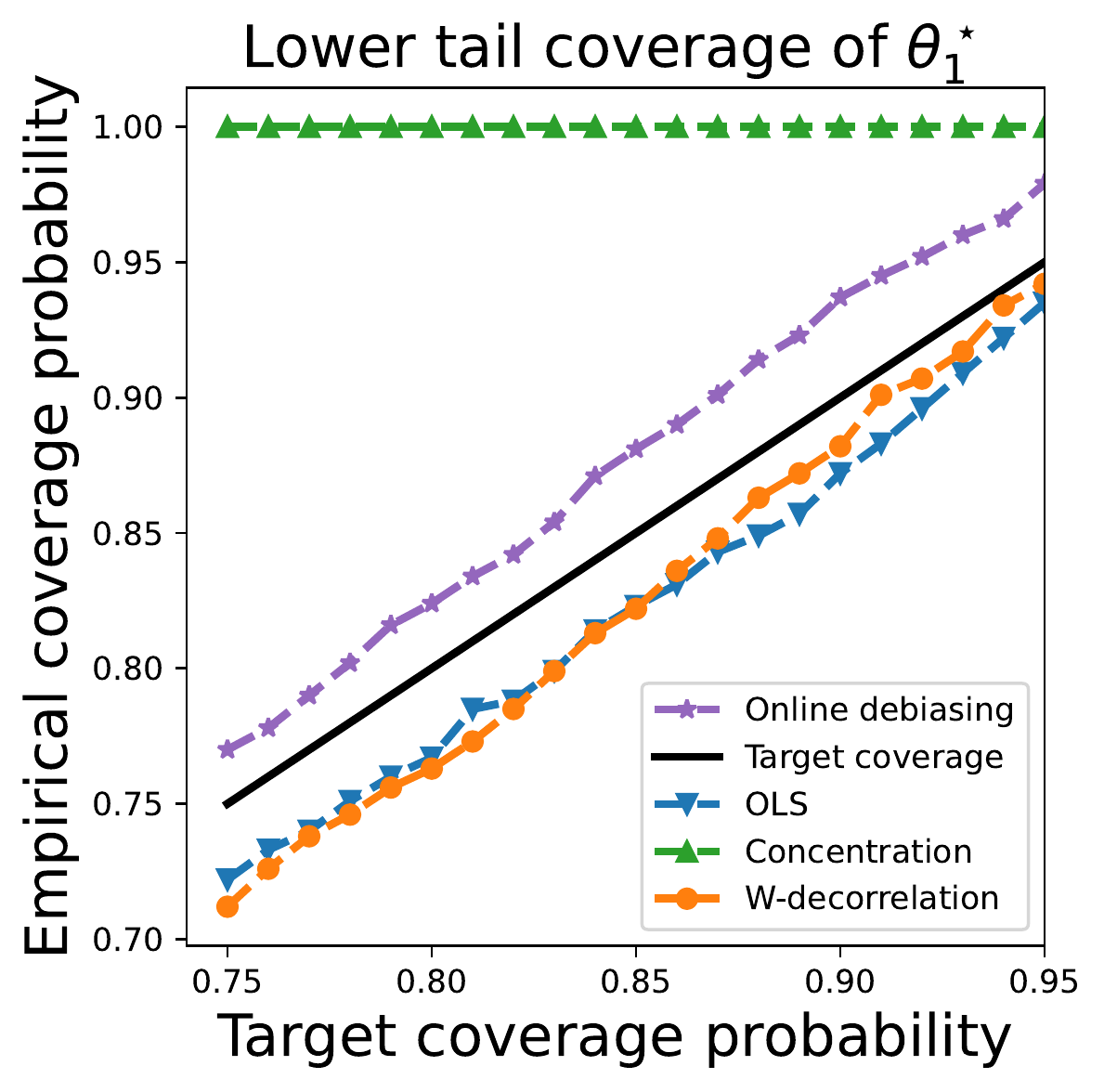}%
        \includegraphics[width=.33\textwidth]{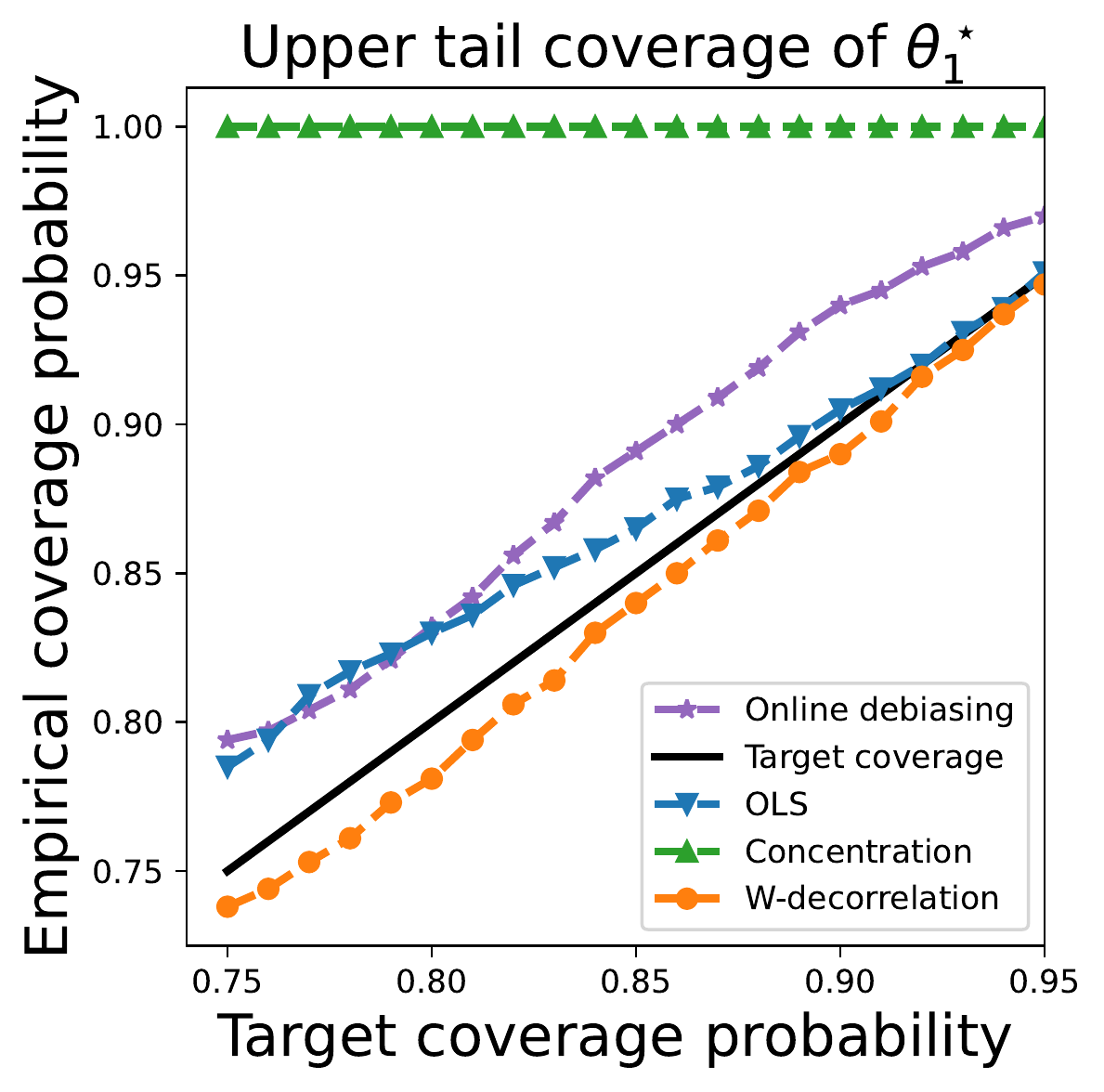}%
    \includegraphics[width=.35\textwidth]{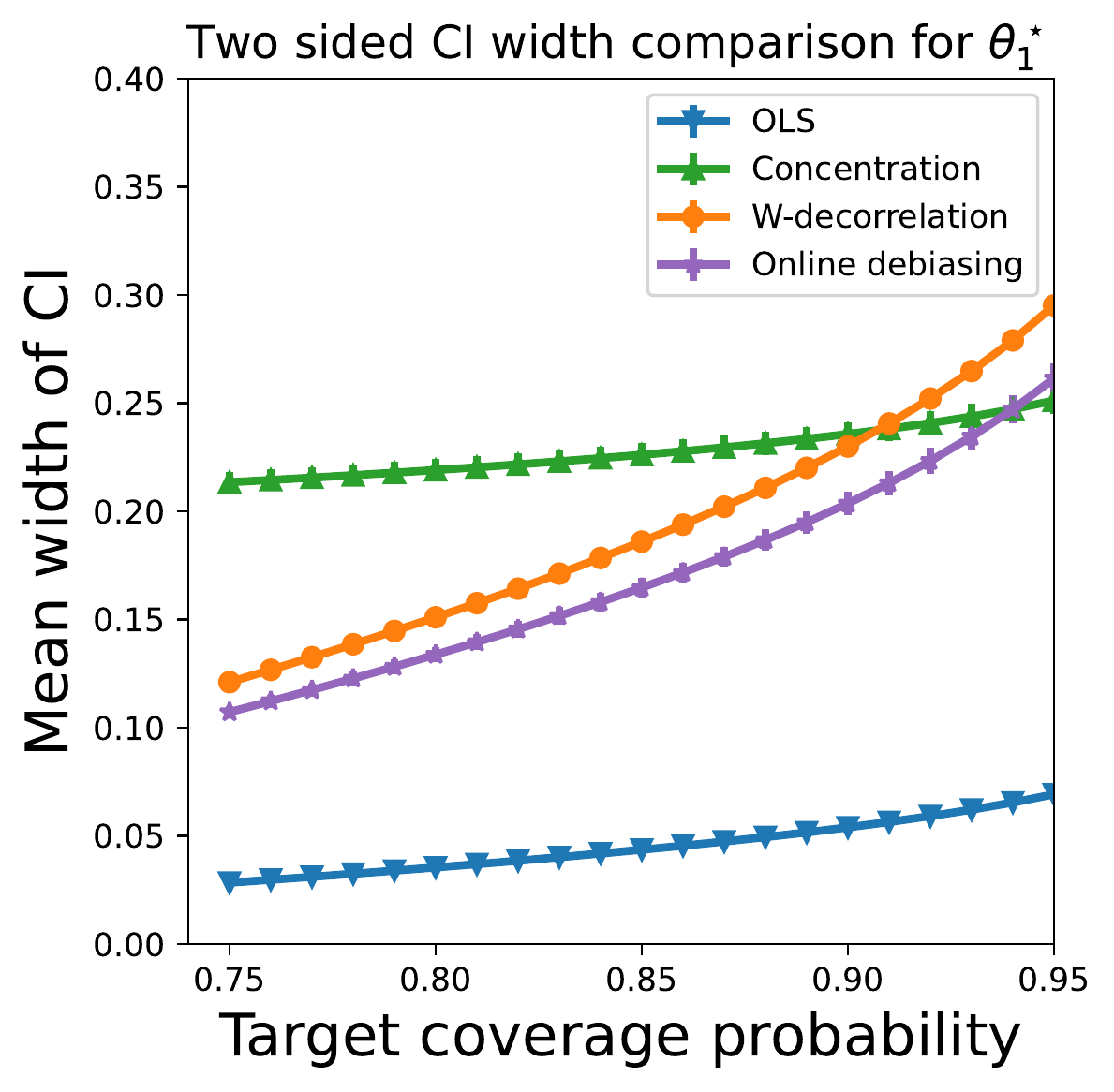}
    %\end{subfigure}
    %
    \caption{
    Average coverage and width of confidence intervals for $\theta_1^*$ across 1000 independent replications  of a multi-armed bandit experiment~\eqref{model:bandits} with $\thetastar \equiv (\theta_1^*, \theta_2^*) =  (0.3, 0.3)^\top$. The covariates $\{\x_i\}_{i = 1}^{1000}$
    were selected using the UCB algorithm~\citep{lattimore2020bandit}, and the error bars represent $\pm 1$ standard error.
    \textbf{Left} and \textbf{Center:} Coverage of one-sided $1 - \alpha$
    intervals for $\theta_1^*$. \textbf{Right:} Width of two-sided $1 - \alpha$ intervals for $\theta_1^*$.
    The first row corresponds to data generated following the UCB algorithm, the second row corresponds to performances when additional data of size $\log^2(n)$ is added to the data; see Section~\ref{sec:verifying-growth}. We observe that the online debiased estimator provide qualitatively similar performances in the both cases. 
    See Section~\ref{sec:numericals-bandits-small} for additional simulations.}
    \label{fig:Thompson-sampling}
\end{figure}

\subsection{Autoregressive time series model}
\label{sec:one-dim}

Our next example involves estimating the parameters of an autoregressive
time series model.  It is well-known that the \OLS estimate can exhibit non-Gaussian limit behavior for
versions of such processes that are unstable \citep{lai1982least}.  In order to focus attention on the key
issues, we restrict ourselves here to the simple case of a scalar
autoregressive process.

More precisely, given the initial point $\y_0 = 0$ and an unknown
scalar $\thetastarScalar \in (-1, 1]$, consider a stochastic process
  generated by the first-order autoregression
\begin{align}
\label{eqn:autoreg-model}
  \y_i = \thetastarScalar \y_{i - 1} + \error_i \quad \mbox{for $i =
    1, \ldots, \numobs$.}
\end{align} 
We assume that the noise sequence $\Nseq{\error}{i}$ consists of
i.i.d.\ standard normal random variables.  Note that the
autoregression~\eqref{eqn:autoreg-model} is a special case of the
stochastic linear regression model~\eqref{eqn:linear-reg-model}, in
particular one with $x_i = \y_{i - 1}$ for all $i \in [\numobs]$. An
especially interesting instantiation of the
autoregression~\eqref{eqn:autoreg-model} is obtained by setting
$\thetastarScalar = 1$. Such a process is a special case of a
\emph{unit root autoregression}, a class of models that play an
important role in econometric time series analysis~\citep{box2015time}.

With the choice $\thetastarScalar = 1$, the
process~\eqref{eqn:autoreg-model} is a random walk and so has
a variance that grows linearly with time.  
%Some calculations show that the stability condition~\eqref{eqn:Cond-Cov-stability} fails to hold.
Moreover, by an application of Donsker's theorem (cf. Example
3 in the paper~\citep{lai1982least}), we have
\begin{align}
\label{eqn:Var-scaling-AR1}  
 \frac{1}{\numobs^2} \Nsum{i} x_i^2 \mydefn \frac{1}{\numobs^2}
 \Nsum{i} \y_{i - 1}^2 &\indistrb \int_{0}^1 \wiener^2(t) dt, \qquad
 \mbox{and} \\
\sqrt{\Nsum{i} y_{i - 1}^2} \cdot (\thetaLSScalar - \thetastarScalar)
& \indistrb \frac{\wiener^2(1) - 1}{2 \int_{0}^1 \wiener^2(t) dt},
\notag
\end{align}
where $\wiener(t)$ denotes the standard Wiener process (see the
paper~\citep{white1958limiting} for details).  Put simply, in the autoregressive time series 
model~\eqref{eqn:autoreg-model} with $\thetastar = 1$ the
stability condition~\eqref{eqn:Cond-Cov-stability} is not satisfied,
and the distribution of the OLS estimate $\thetaLSScalar$ is not
asymptotically normal.

In contrast to this negative result for the OLS estimate, we can show
that the debiasing estimate $\thetaDecorrScalar$, after suitable
centering and scaling, does indeed converge in distribution to a
standard Gaussian.  Our result is based on the tuning parameters
and scaling matrices chosen as
\begin{align}
  \label{eqn:autoregressive-tuning}
\tuneParScaled_\numobs = \frac{1}{(\log \numobs)\cdot \log\log(\numobs)}, \quad
\mbox{and} \quad {\XscaleMat_{i, \numobs} = \max \left \{ (\log \numobs)^2 y_{i - 1}^2, \;\; \sum_{j =1}^{i-1} y_j^2 \right \}}.
\end{align}

\begin{cors}
\label{cor:unit-root-autoreg}
Given a sequence $\Nseq{\y}{i}$ generated from the autoregressive
model~\eqref{eqn:autoreg-model}, the estimate
$\thetaDecorrScalar$~\eqref{eqn:ThetaDecorr-defn} obtained with the
tuning parameters~\eqref{eqn:autoregressive-tuning} satisfies
\begin{align}
\label{EqnAutoRegClima}  
  \sqrt{\frac{\Nsum{i} \y_{i - 1}^2}{(\log \numobs) \cdot \log\log(\numobs)}}
  \cdot ( \thetaDecorrScalar - \thetastarScalar ) \indistrb \Ncal(0,
  1).
\end{align}
\end{cors}
\noindent See~\cref{proof-unit-root-autoreg} for the proof of this
claim. 
\cref{cor:unit-root-autoreg} enables us to construct asymptotically exact
confidence intervals for $\thetastarScalar$. We also reiterate that the above result holds for any ${\thetastar \in (-1, 1]}$. 

\subsubsection{Numerical experiment}
\label{sec:simulation-time-series-small}
\cref{fig:Time-series-small} illustrates the performance of online
debiasing with autoregression
tuning~\eqref{eqn:autoregressive-tuning}.  Here our data is generated
from the time series model~\eqref{eqn:autoreg-model} with $\thetastar
= 1$.  We again find that online debiasing provides appropriate
coverage for all confidence levels.  Meanwhile, the \OLS lower tail
interval exhibits
severe undercoverage, and W-decorrelation exhibits ranges of undercoverage for both tails.  Finally, the concentration-based CI again
provides 100\% coverage for all confidence levels, at the expense of
interval lengths that are uniformly longer than the online debiasing
CIs.
\begin{figure}[H]%
     \centering
    %\begin{subfigure}{\linewidth}
    \includegraphics[width=.34\textwidth]{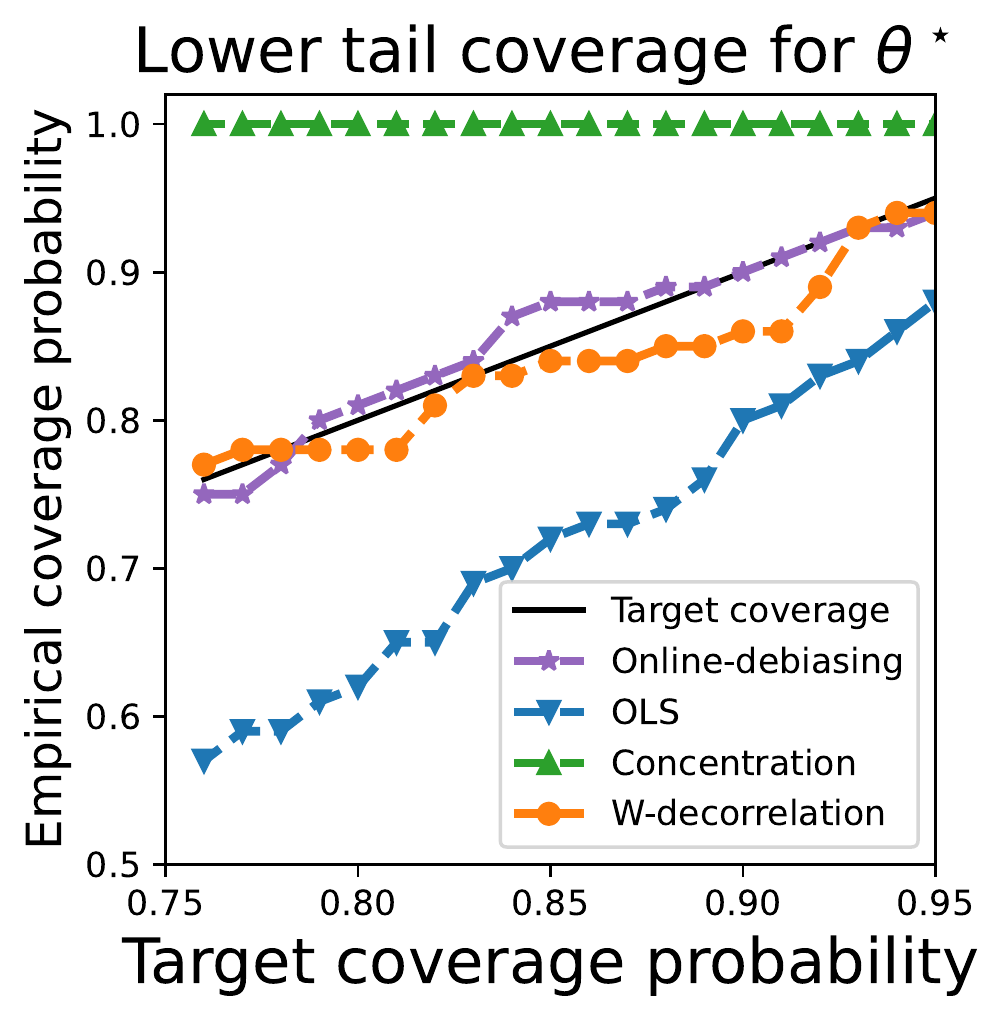}%
        \includegraphics[width=.34\textwidth]{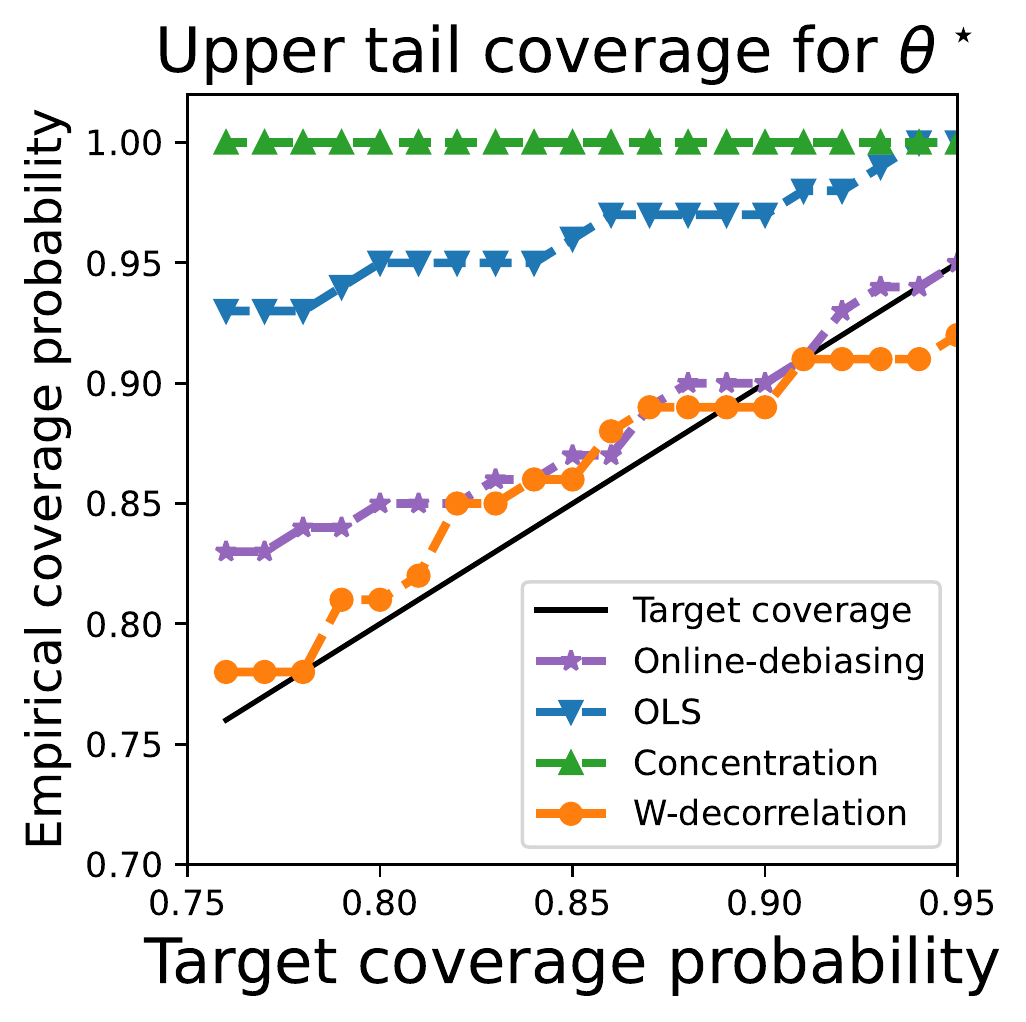}%
        \includegraphics[width=.34\textwidth]{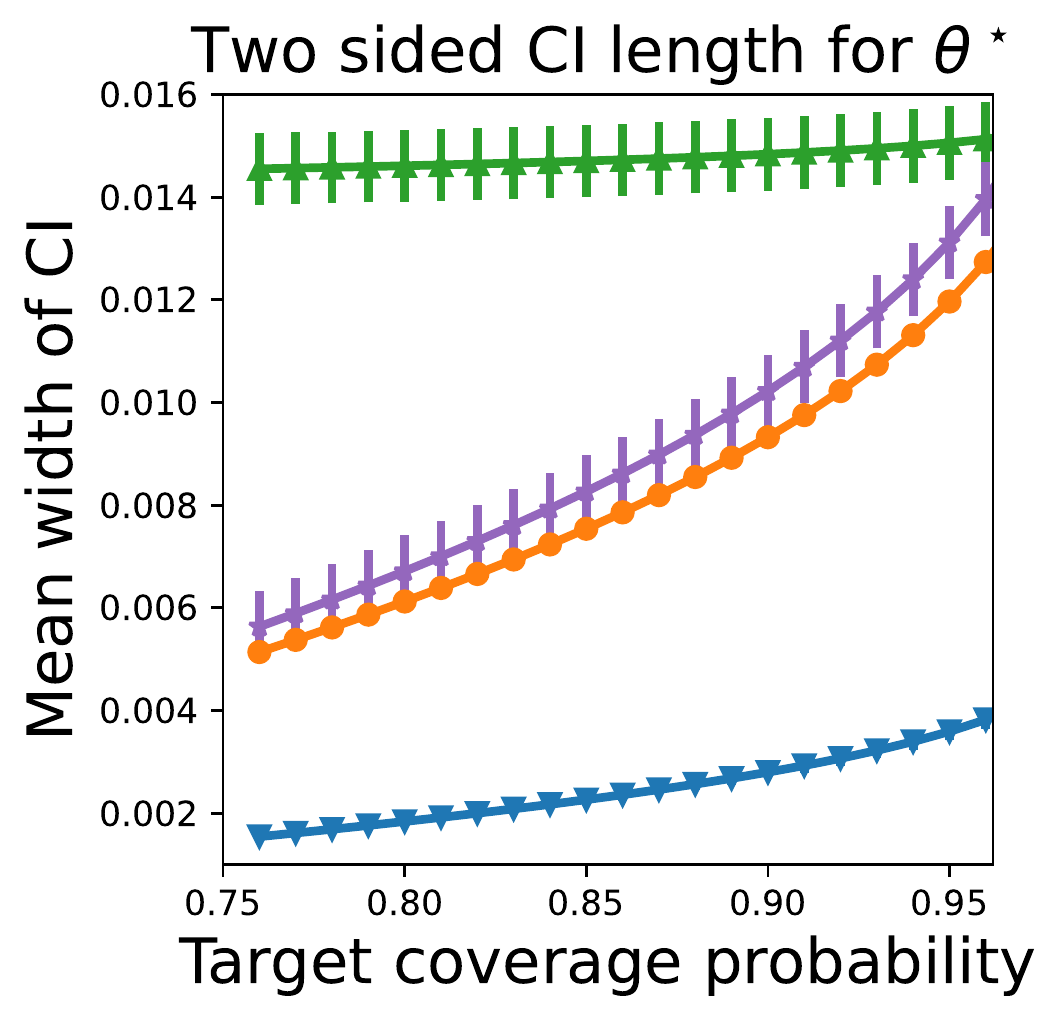} 
    %\caption{1D Autoregression}
    %\end{subfigure}
    %
    \caption{
    Average coverage and width of confidence intervals for $\thetastar=1$ across 1000 independent replications of an autoregressive time series experiment~\eqref{model:bandits}. The error bars represent $\pm1$ standard error.
    \textbf{Left} and \textbf{Center:} Coverage of one-sided $1 - \alpha$
    intervals for $\thetastar$. \textbf{Right:} Width of two-sided $1 - \alpha$ intervals for $\thetastar$.
    See Section~\ref{sec:simulation-time-series-small} for details.}
\label{fig:Time-series-small}
\end{figure}

%%%%%%%%%%%%%%%%%%%%%%%%%%%%%%%%%%%%%%%%%%%%%%%%%%%%%%%%%%%%%%%%%%%%%%%%%%%%%%%%%%%%%%%%%%%%%%%%%%%%

\subsection{Active learning with exploration}
\label{sec:suff-exploration}

In our third example, we focus on the case where the covariates
$\Nseq{\x}{i}$ are generated using any algorithm satisfying a sufficient exploration property.

%%%%%%%%%%%%%%%%%%%%%%%%%%%%%%%%%%%%%%%%%%%%%%%%%%%%%%%%%%%%%%%%%%%%%%%%%%%%%%%%%%%%%%%%%%%%%%%%%%%%%%%

\begin{definition}[Selection algorithms with $\greedy$-exploration]
We say that a selection algorithm $\Selections_\numobs =
\Nseq{\selection}{i}$ admits a $\greedy$-exploration property
if
\begin{align}
\label{eqn:selection-with-exploration}
  \x_i \mydefn
  \begin{cases}
    \algoDir_i & \text{with probability $1 - \greedyEps_i$ for some $\algoDir_i \in
      \filtration_{i - 1}$, and} \\
    \randDir_i & \mbox{with probability $\greedyEps_i$ for some $\randDir_i$
      independent of $\filtration_{i - 1}$.}
   \end{cases} 
\end{align}
Here the exploration probability sequence $\Nseq{\greedyEps}{i}$
consists of nonnegative scalars in the interval $(0,1)$, and the
vectors $\Nseq{v}{i}$ are i.i.d.\ random vectors such that
\begin{align}
\label{eqn:exploration-lower-bound}
  \Exs[\randDir_i \randDir_i^\top] \succeq \ExsXsqLB
  \qquad \text{where}  \quad \ExsXsqLB \succeq \mathbf{0}
\end{align}

In words, the selection algorithm $\selection_i$ behaves as follows:
with probability $1 - \greedyEps_i$, it chooses vector $\algoDir_i$
based on the previous data points $\{(\x_j, \y_j)\}_{j =1}^{i-1}$, and
with probability $\greedyEps_i$, it chooses a random direction
$\randDir_i$, independent of the previous data points.
\end{definition}

%%%%%%%%%%
\subsubsection*{{\bf{Example: $\greedy$-greedy linear bandits}}}
Let us briefly consider a concrete instance of a selection algorithm
$\Nseq{\selection}{i}$ that is of the $\greedyEps$-greedy type.  In
the linearly parameterized bandit problem, at each time $i \in
[\numobs]$, an algorithm $\selection_i$ chooses an action vector
$\x_i$, usually lying within some bounded set $\actionset_i$, and
obtains a reward \mbox{$\y_i = \myinprod{\x_i}{\thetastar} + \err_i$}.
A popular and simple strategy for regret minimization is a special
case of the $\greedyEps$-greedy selection
algorithm~\citep{lattimore2020bandit}.  For linearly parameterized
bandits, the selection algorithm $\selection_i$ chooses
\begin{align}
\label{eqn:greedy-selection}
  \x_i \;\; &
  \begin{cases} 
       \in \arg \max \limits_{\x\in \actionset_i} \quad
       \myinprod{\x_i}{\thetaRidge^{(i - 1)}} & \mbox{with probability
         $1 - \greedyEps_i$, and} \\
       \sim \Unif(\actionset_i) & \mbox{with probability
         $\greedyEps_i$.}
   \end{cases} 
\end{align}
where $\thetaRidge^{(i - 1)}$ denotes the ridge regression estimator
based on all data observed up to stage $i - 1$, i.e., the collection
of covariate-response pairs $\{ (\x_j, \y_j)\}_{j=1}^{i-1}$.  Put
simply, with probability $1-\greedyEps_i$, the selection algorithm
$\selection_i$ chooses an optimal action given data collected so far
(exploitation), and with probability $\greedyEps_i$, the algorithm
randomizes uniformly amongst its choices (exploration).  In the more
general setting~\eqref{eqn:selection-with-exploration} considered
here, it is not necessary to select the optimal action in the
exploitation step.  Rather, our result holds also when an arbitrary,
$\filtration_{i-1}$-measurable choice is made in the first part
of~\cref{eqn:greedy-selection}, as in the EXP3 or UCB algorithms with
exploration.  See the book~\citep{lattimore2020bandit} for more
details.  \hfill $\diamondsuit$

\medskip

Returning to our general
setting~\eqref{eqn:selection-with-exploration}, we now state a
guarantee for selection algorithms with $\greedyEps$-exploration.  As
is standard in the bandit literature, we assume that the covariates
are uniformly bounded, so that there exists a scalar $K$ satisfying
\begin{subequations}
\begin{align}
\label{eqn:bounded-covariates}
\| \x_i \|_2 \leq K \quad \mbox{for all $i \in [\numobs]$.}
\end{align}
See our discussion following the corollary for how this condition can
be relaxed. In addition, we impose a \emph{sufficient exploration}
condition, meaning a lower bound on the magnitude of the exploration
probabilities, of the form
\begin{align}
\label{eqn:suff-exploration}
  \sum_{i = 1}^\numobs \greedyEps_i \geq \frac{ \Exs[ \max_{i \in
        [\numobs]} \| \x_i \|^2_2] }{\lambda_{\min}(\ExsXsqLB)} (\log
  \numobs)^2 % \quad \mbox{for some $\pardelta > 0$,}
\end{align}
where the reader should recall that the matrix $\ExsXsqLB$ was defined
in~\cref{eqn:exploration-lower-bound}.  We implement the debiasing
estimate~\eqref{eqn:ThetaDecorr-defn} with the choice of tuning
parameters
\begin{align}
  \label{EqnTuningGreedy}
\XscaleMat_{i, \numobs} = \sum_{j = 1}^\numobs \greedyEps_j \ExsXsqLB
\quad \qtext{and} \quad \tuneParScaled_\numobs =
\frac{1}{(\log\numobs) \cdot \log\log(\numobs)}.
\end{align}
\end{subequations}

\begin{cors}
\label{cor:multidim-gen-case}
Suppose that Assumptions~\ref{assumption:A1} and \ref{assumption:A2}
hold, the covariates satisfy the bound~\eqref{eqn:bounded-covariates},
and the exploration conditions~\eqref{eqn:exploration-lower-bound}
and~\eqref{eqn:suff-exploration} both hold. Then given any consistent
estimator $\sigmahat^2$ of the error variance $\errorstd^2$, the
estimator $\thetaDecorr$ with tuning
parameters~\eqref{EqnTuningGreedy} satisfies

\begin{subequations}
\begin{align}
\label{eqn:gen-case-asymp-normality}
  \big(\log\numobs \cdot \log\log(\numobs) \cdot
  \sigmahat^2\big)^{-1/2} \cdot \SigmaMat_\numobs^{\frac{1}{2}} \left(
  \thetaDecorr - \thetastar \right) \indistrb \Ncal(0, \MyIdDim).
\end{align}
Moreover, for any $\dir \in \real^\dims$, the following is an
asymptotically exact $1 - \alpha$ confidence intervals for $ \dir^\top
\thetastar$
\begin{align}
\label{eqn:gen-case-simple-CI}
        \left[ \e_1^\top \thetaBasis - \tfrac{\OpTuning_\numobs
            \sigmahat}{\sqrt{\tuneParScaled_\numobs}}
          (\myinprod{\dir}{\SigmaMat_\numobs^{-1} \dir})^{\frac{1}{2}}
          \NormalQuantile{1 - \alpha/2}, \qquad \e_1^\top \thetaBasis
          +
          \tfrac{\OpTuning_\numobs\sigmahat}{\sqrt{\tuneParScaled_\numobs}}
          (\myinprod{\dir} {\SigmaMat_\numobs^{-1}\dir})^{\frac{1}{2}}
          \NormalQuantile{1 - \alpha/2} \right],
\end{align}
\end{subequations}
where $\OpTuning_\numobs = \opnorm{\diagCov_{\dir,
    \numobs}^{-\frac{1}{2}} \SigmaMat_{\dir, \numobs}^{\frac{1}{2}}}$
and the estimator $\thetaBasis$ was calculated
using~\eqref{eqn:thetaBasis-defn} and with the tuning
parameters~\eqref{EqnTuningGreedy}.
\end{cors}

\noindent See~\cref{sec:proof-multidim-gen-case} for the proof.

\vspace{20pt}

It is worth noting that the bounded covariate
condition~\eqref{eqn:bounded-covariates} can be relaxed.  For
instance, in absence of the condition~\eqref{eqn:bounded-covariates},
one may obtain a result similar to the part (a) of
Corollary~\ref{cor:multidim-gen-case} under the following assumptions:
\begin{align*}
  \tuneParScaled_\numobs = \smalloP(\log \lambda_{\max}
  (\SigmaMat_\numobs)), \quad \text{and} \quad \sum_{i = 1}^\numobs
  \greedyEps_i = \frac{ \max_{i \in [\numobs]} \Exs[\| \x_i \|^2_2]
  }{\lambda_{\min}(\ExsXsqLB) \cdot \smalloP(\tuneParScaled_\numobs)}.
\end{align*}

Finally, as a special case, Corollary~\ref{cor:multidim-gen-case}
allows us to construct confidence interval for $\dir^\top \thetastar$
for multi-armed bandit problems that we discussed in
Section~\ref{sec:bandits}.  The
condition~\eqref{eqn:bounded-covariates} is readily satisfied for
multi-armed bandit problems, but the
conditions~\eqref{eqn:exploration-lower-bound}
and~\eqref{eqn:suff-exploration} are mildly stronger than the
analogous condition~\eqref{eqn:bandits-tuning}.

\subsubsection{Ensuring sufficient exploration via data augmentation}

It is natural to ask if we can obtain online debiased method when the
sufficient exploration condition is either difficult to verify or is
not satisfied. In such settings, one simple fix is the following is
based on the data-augmentation technique disscued in
Section~\ref{sec:verifying-growth}. Indeed, if we may collect
$\log^2(n)$ many data points with $\x_i$ chosen uniformly at random
froma $d$-dimensional unit sphere and append it to the new data. The
new data-set has $n + \log^2(n)$ many data points, and the new dataset
satisfy the exploration
condition~\eqref{eqn:selection-with-exploration} with
\begin{align*}
\greedyEps_i = 0 \;\; \text{for all} \;\; 1 \leq i \leq \numobs \qquad
\greedyEps_i = 1 \;\; \text{for} \;\; i = \numobs + 1, \ldots, \numobs
+ \log^2(\numobs).
\end{align*}     
In summary, the growth condiiton~\eqref{eqn:suff-exploration} is
satisfied in this case up to a factor $d$.  Hence, the result from
Corollary~\ref{cor:multidim-gen-case} holds true in this case.

%%%%%%%%%%%%%%%%%%%%%%%%%%%%%%%%%%%%%%%%%%%%%%%%%%%%%%%%%%%%%%%%%%%%%%%%%

\subsubsection{Numerical simulation}
\label{sec:LinBanditsSimSmall}
\begin{figure}[!ht]%
    \centering
    \includegraphics[width=.33\textwidth]{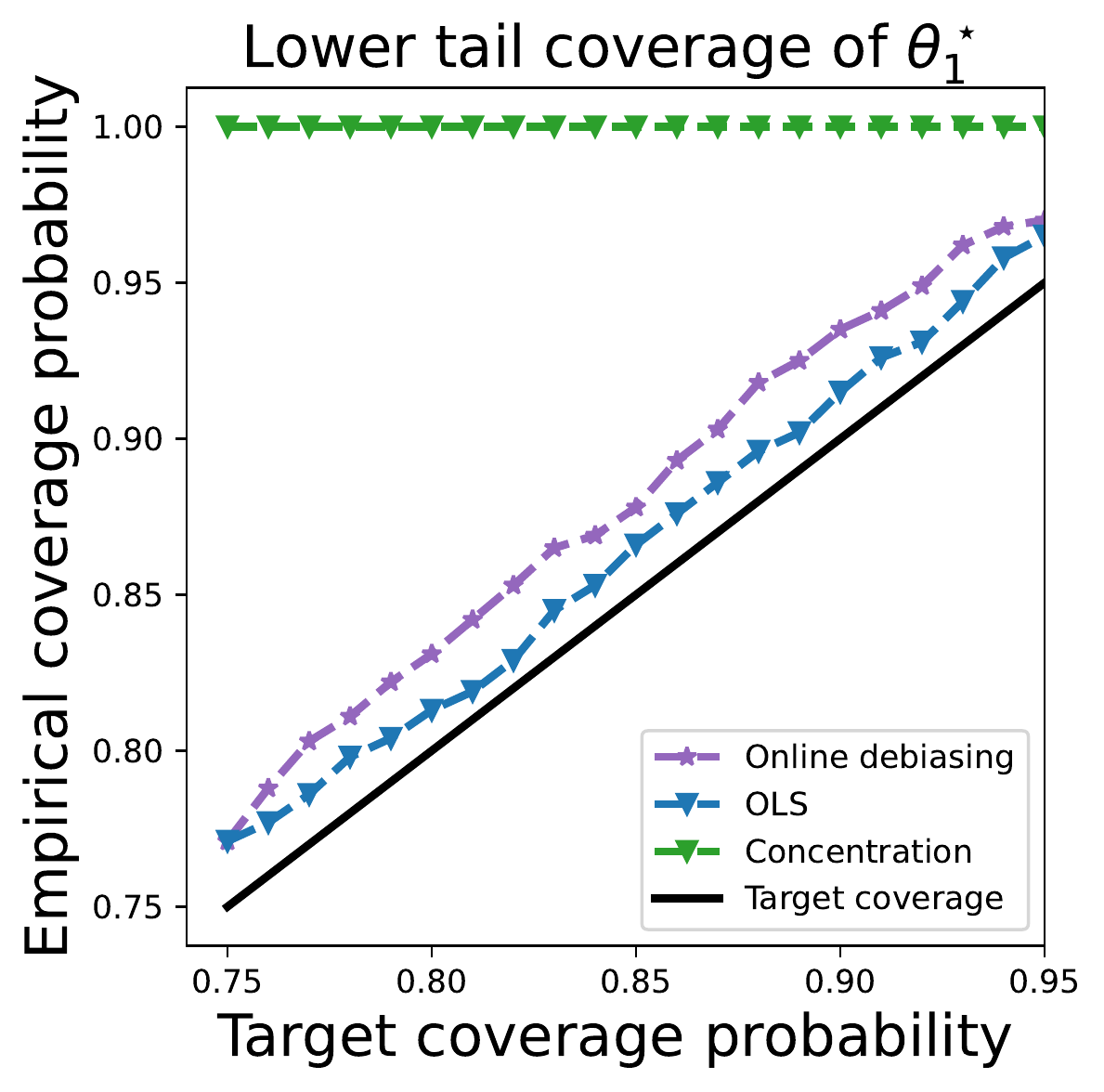}%
        \includegraphics[width=.33\textwidth]{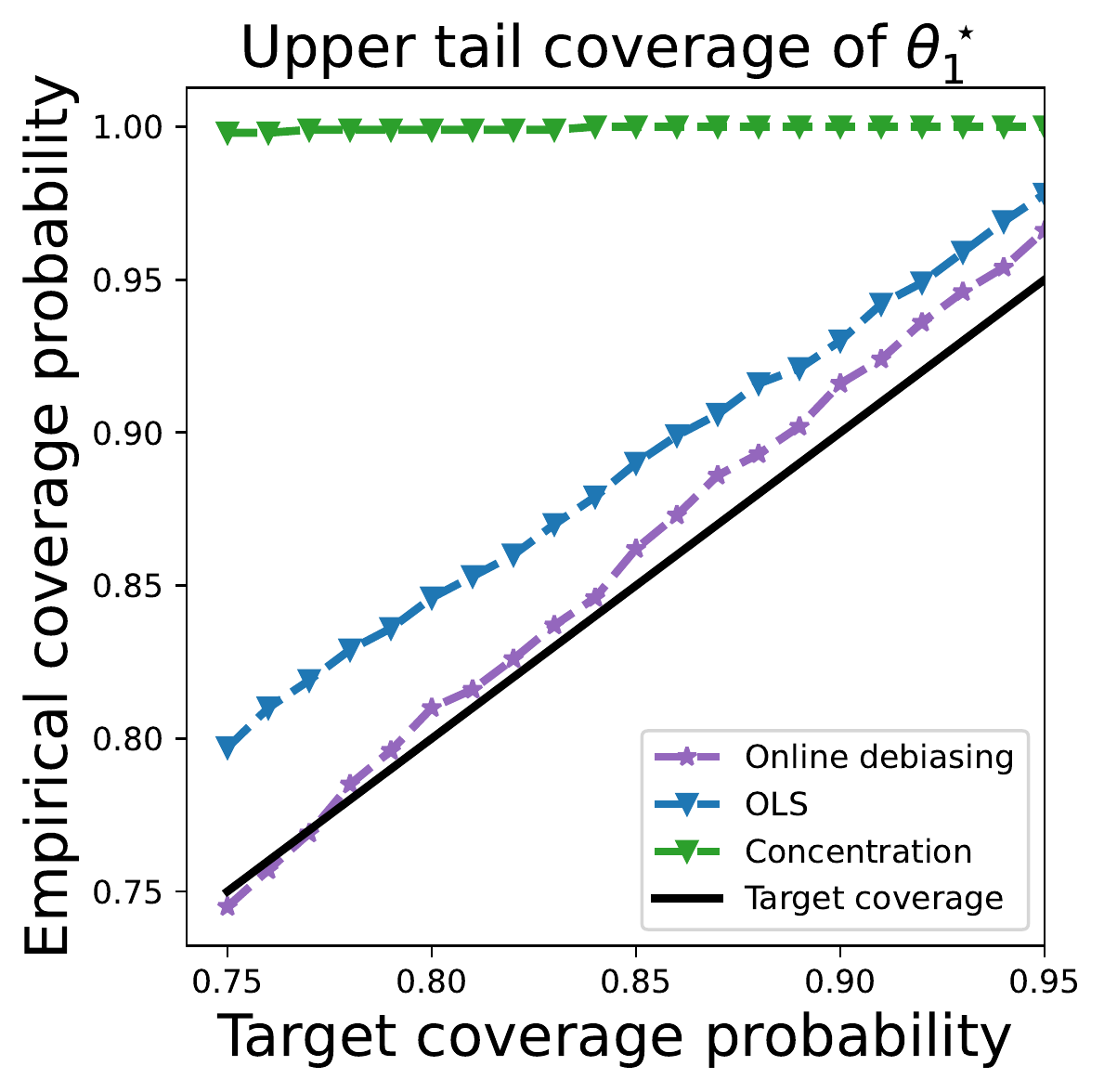}%
    \includegraphics[width=.33\textwidth]{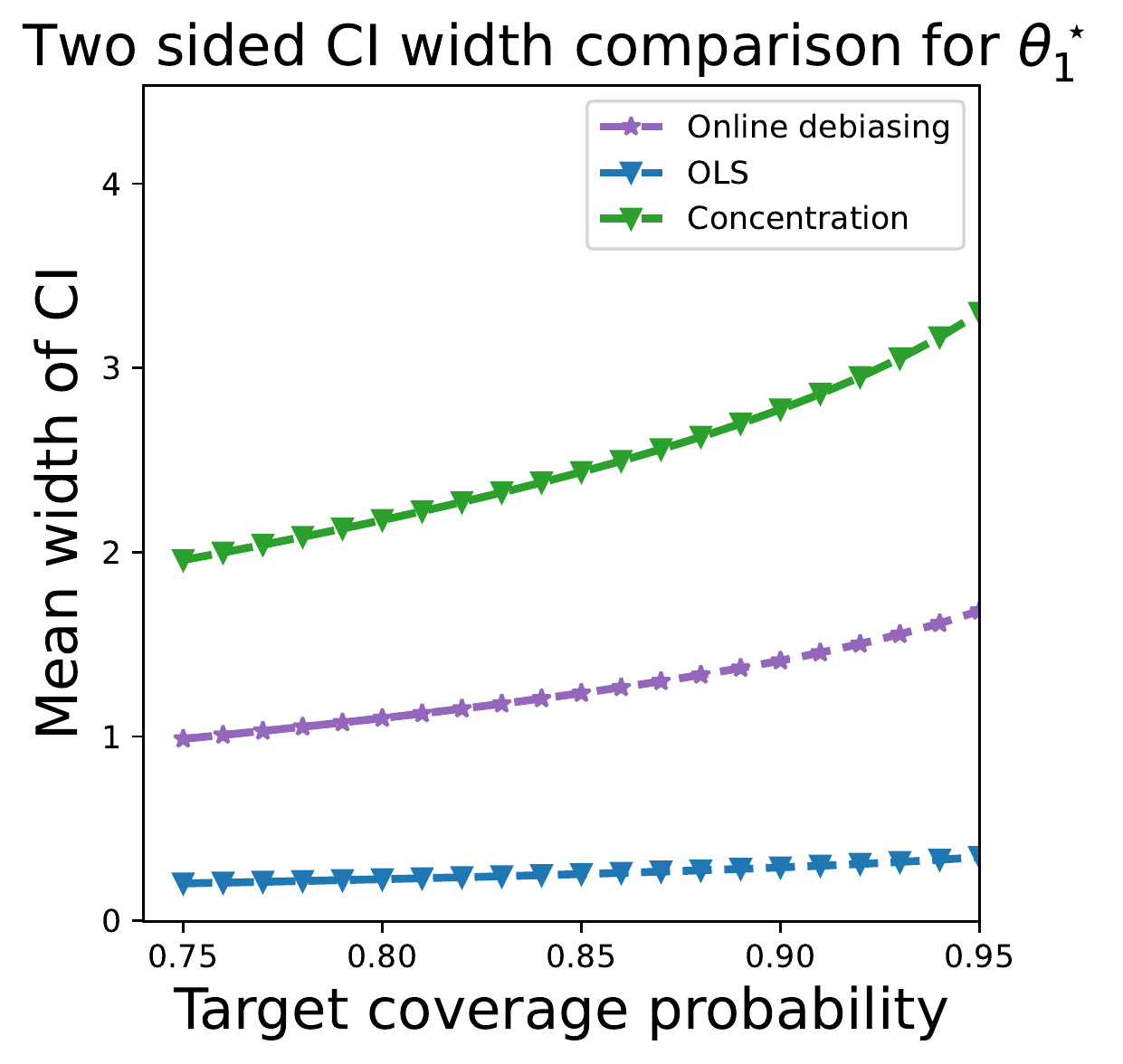} \\
    \includegraphics[width=.33\textwidth]{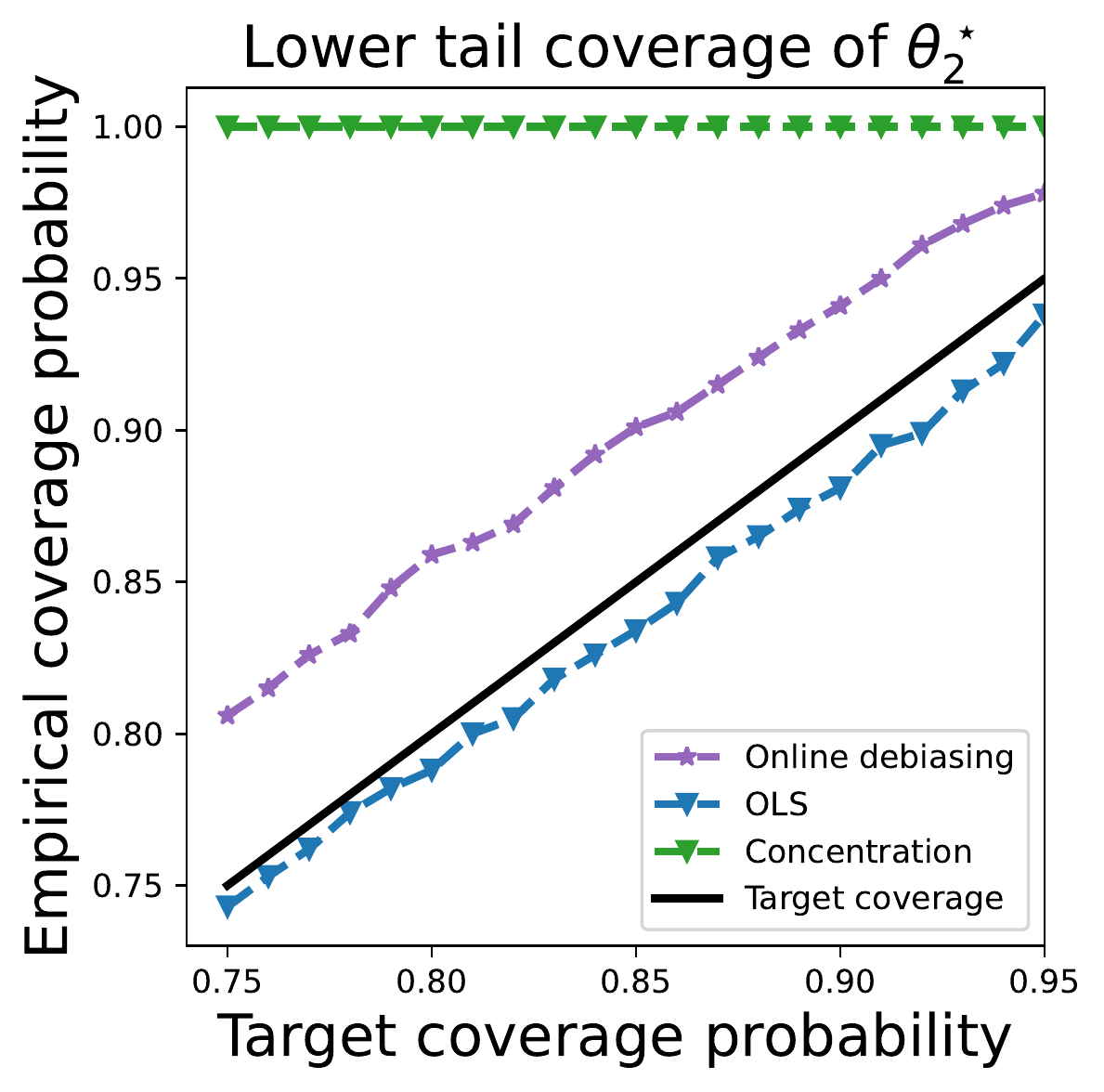}%
    \includegraphics[width=.33\textwidth]{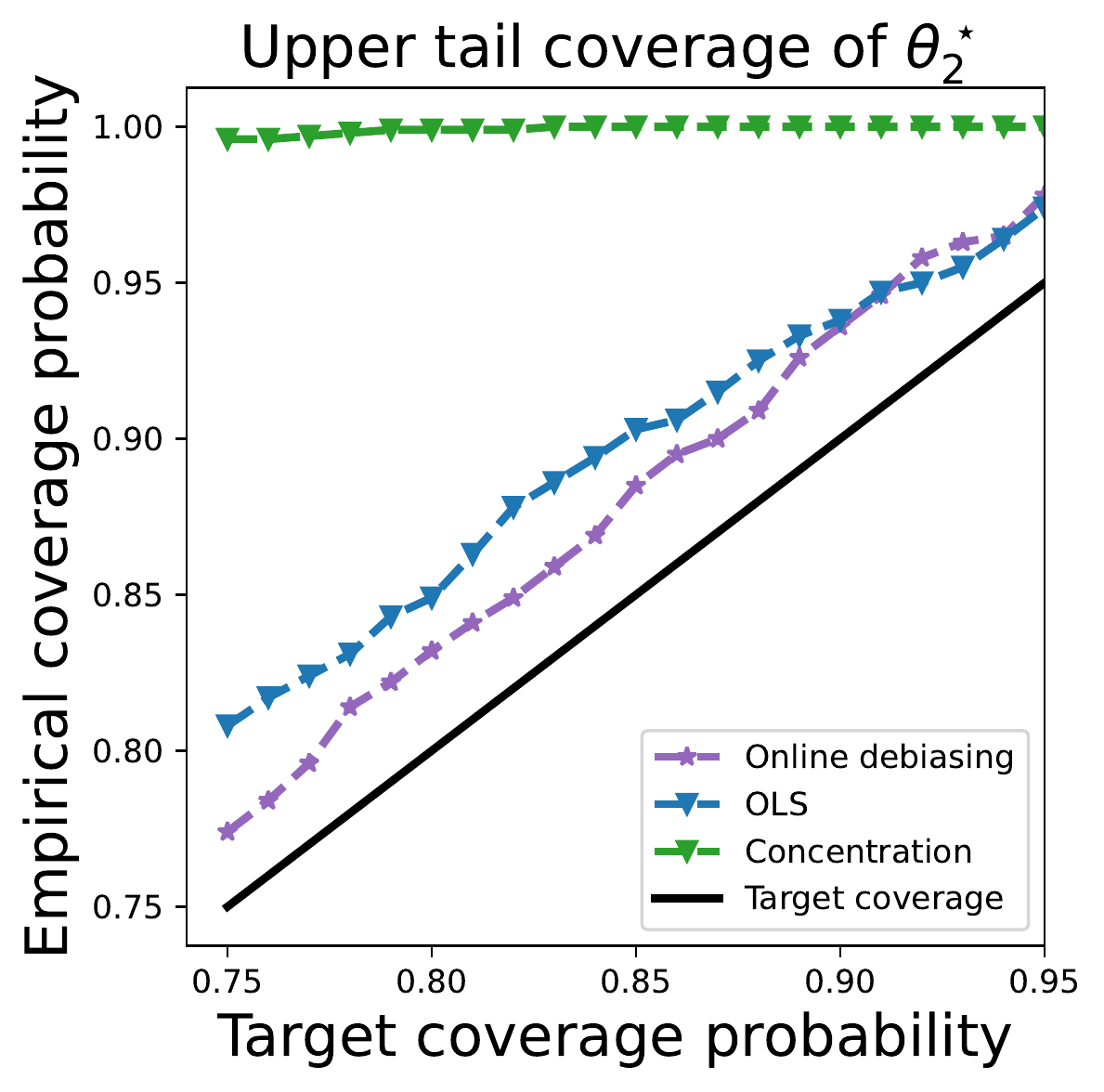}%
    \includegraphics[width=.33\textwidth]{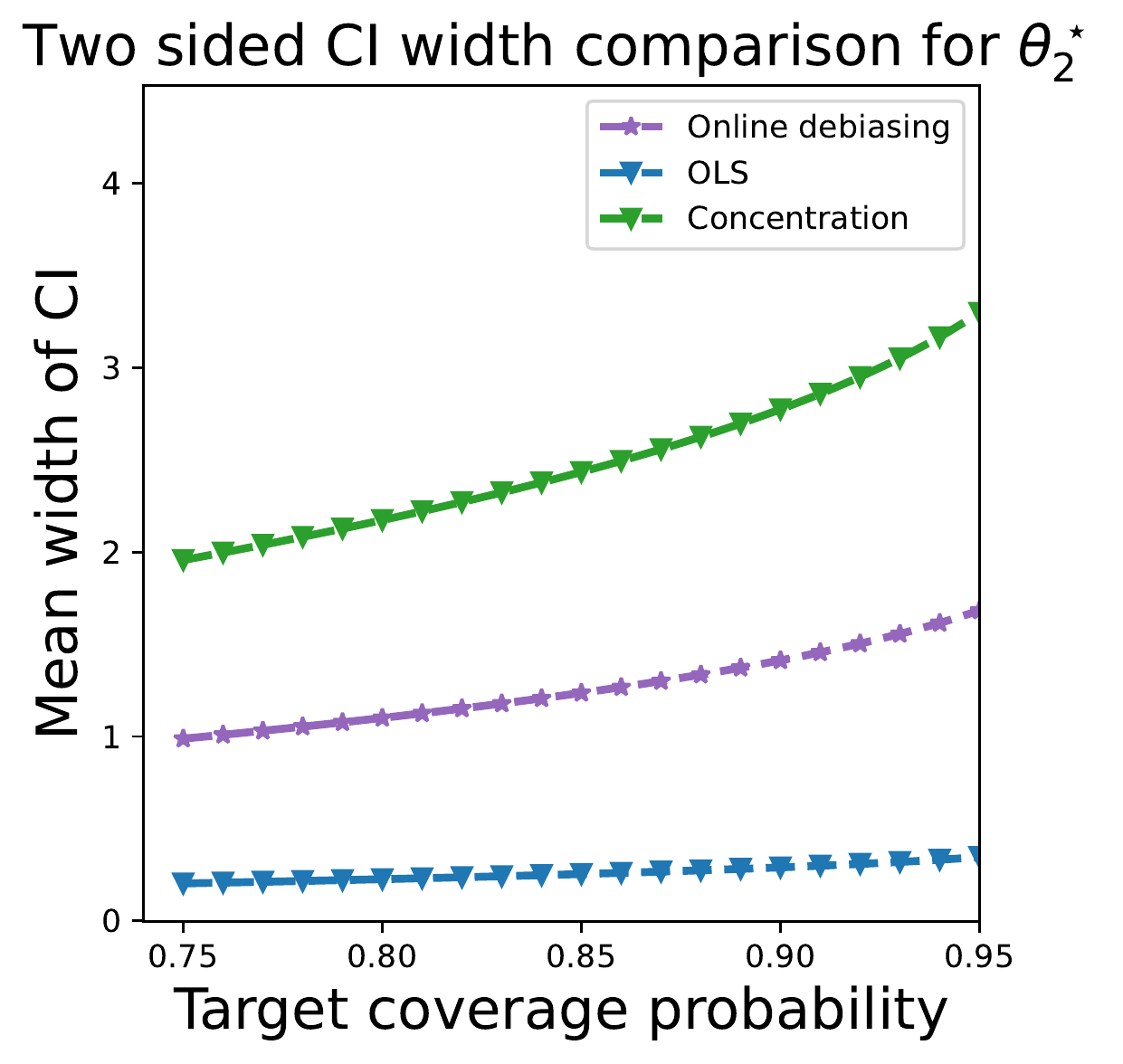}
    \caption{Average coverage and width of confidence intervals for
      $\theta_1^*$ and $\theta_2^*$ across 1000 independent
      replications of a linear bandits
      experiment~\eqref{eqn:greedy-selection} with $\thetastar \equiv
      (\theta_1^*, \theta_2^*) = (0.3, 0.3)^\top$. The covariates
      $\{\x_i\}_{i = 1}^{1000}$ were selected using the
      $\greedyEps$-greedy linear bandits
      algorithm~\eqref{eqn:greedy-selection}, and the error bars
      represent $\pm 1$ standard error.  \textbf{Left} and
      \textbf{Center:} Coverage of one-sided $1 - \alpha$ intervals
      for $\theta_1^*$ and $\theta_2^*$. \textbf{Right:} Width of
      two-sided $1 - \alpha$ intervals for $\theta_1^*$ and
      $\theta_2^*$.  See \cref{sec:LinBanditsSimSmall} for details.}
    \label{fig:LinBandits-small}
\end{figure}

\cref{fig:LinBandits-small} illustrates the performance of online
debiasing with the active learning tuning~\eqref{EqnTuningGreedy}.
Here we consider a linear bandits problem with $\thetastar = (0.3,
0.3)^\top$ and i.i.d.\ standard normal error $\{ \error_i \}_{i =
  1}^\numobs$.  The covariates $\{ \x_i \}_{i = 1}^\numobs$ were
generated using the $\greedyEps$-greedy linear bandits
algorithm~\eqref{eqn:greedy-selection}, where, for each stage, the
context set $\actionset_i$ consisted of the same $50$ vectors drawn
and uniformly from the unit sphere in dimension $2$.  For this
problem, the exploration lower bound
\cref{eqn:exploration-lower-bound} is satisfied with $\mathbf{G} =
\frac{1}{|\actionset|} \cdot \sum_{\action_i \in \actionset} \action_i
\action_i^\top$.  In this setting, \citet{abbasi2011online} only
provide concentration-based CIs based on ridge regression estimators,
rather than \OLS.  Here we report the CIs from ridge regression with
regularization parameter $\lambda_{\text{Ridge}} = 0.1$ (which closely
approximates the \OLS solution) and display analogous results for
alternative regularization parameters in
Appendix~\ref{sec:LinBanditsFullSimulation}. We computed the
confidence intervals for $\theta_1^*$ and $\theta_2^*$ using
Corollary~\ref{cor:multidim-gen-case}.

We observe once more that online debiasing provides appropriate
coverage for all confidence levels, while the \OLS lower tail interval
consistently undercovers. Meanwhile, the concentration CI provides high
coverage for all confidence levels but yields intervals typically
larger than the online debiasing CIs.

%%%%%%%%%%%%%%%%%%%%%%%%%%%%%%%%%%%%%%%%%%%%%%%%%%%%%%%%%%%%%%%%%%%%%%%%%

\section{Proofs of the theorems}
\label{SecProofs}
In this section, we provide the proofs of our two main results. We
prove Theorem~\ref{thm:asymp-normality} in
Section~\ref{proof:thm:asymp-normality}, and
Theorem~\ref{thm:Minimax-Lowerbound} in
Section~\ref{sec:lower-bound-proof}.
%%%%%%%%%%%%%%%%%%%%%%%%%%%%%%%%%%%%%%%%%%%%%%%%%%%%%%%%%%%%%%%%%%%%%%%%%%%%%%%%%%%%%%

\subsection{Proof of Theorem~\ref{thm:asymp-normality}}
\label{proof:thm:asymp-normality}

Using the condition $\lambda_{\min}(\SigmaMat_\numobs) \almostSurely \infty$
from assumption~\ref{assumption:A2}, thus we may assume without loss of generality that $\SigmaMat_\numobs$
is invertible. We claim that it suffices to show that
\mbox{$\sqrt{\tuneParScaled_\numobs} \cdot \SigmaMat_\numobs^{\frac{1}{2}}
(\thetaDecorr - \thetastar)$} converges in distribution to $\normal(0,
\sigma^2 \MyIdDim)$.  Indeed, when this claim holds, then since
\mbox{$\sigmahat^2 \stackrel{\smallop{p}}{\longrightarrow} \sigma^2$} by assumption, Slutsky's theorem
implies the claim of the theorem.

Recall from equation~\eqref{eqn:bias-variance-decomp} that the random
vector $\sqrt{\tuneParScaled_\numobs} \cdot
\SigmaMat_\numobs^{\frac{1}{2}} (\thetaDecorr - \thetastar)$ can be
decomposed into the sum $\Bias_\numobs + \ZeroMartin_\numobs$. Based
on this decomposition, we see that it is sufficient to prove that
\mbox{$\Bias_\numobs \stackrel{\smallop{p}}{\longrightarrow} 0$} and \mbox{$\ZeroMartin_\numobs \indist \normal(0,
  \sigma^2 \MyIdDim)$.}  The remainder of our proof is devoted to
establishing these two claims.

%%%%%%%%%%%%%%%%%%%%%%%%%%%%%%%%%%%%%%%%%%%%%%%%%%%%%%%%%%%%%%%%%%%%%%%%%%%%%%%%%%%%

\myparagraph{Analysis of $\Bias_\numobs$}

By definition of the operator norm, we have the upper bound
\begin{align}
  \label{EqnFirstTriangle}
\| \Bias_\numobs \|_2 & \leq \sqrt{\tuneParScaled_\numobs} \, \opnorm{
  \MyIdDim - \W_\numobs \X_\numobs \SigmaMat_\numobs^{-\frac{1}{2}} }
\big \| \SigmaMat_\numobs^{\frac{1}{2}}(\thetaLS - \thetastar)
\big\|_2.
\end{align}
Lemma 1 from the paper~\citep{lai1982least} guarantees that
\begin{subequations}
\begin{align}
  \label{eqn:ls-error}  
  \big\| \SigmaMat_\numobs^{\frac{1}{2}}(\thetaLS - \thetastar)
  \big\|_2 & = \Bigo \left(\sqrt{ \log
    \lambda_{\max}(\SigmaMat_\numobs) } \right) \quad \text{almost
    surely}.
\end{align}
On the other hand, the vanishing bias condition from
Assumption~\ref{assumption:A3}(b) guarantees that
\begin{align}
\label{eqn:biasterm}
\sqrt{\tuneParScaled_\numobs \log \lambda_{\max}(\SigmaMat_\numobs) }
\cdot \opnorm{ \MyIdDim - \W_\numobs \X_\numobs
  \SigmaMat_\numobs^{-\frac{1}{2}} } \stackrel{\smallop{p}}{\longrightarrow} 0.
\end{align}
\end{subequations}
Applying the bounds~\eqref{eqn:ls-error} and~\eqref{eqn:biasterm} to
the right-hand side of the inequality~\eqref{EqnFirstTriangle}
shows that $\|\Bias_\numobs\|_2 \stackrel{\smallop{p}}{\longrightarrow} 0$.

%%%%%%%%%%%%%%%%%%%%%%%%%%%%%%%%%%%%%%%%%%%%%%%%%%%%%%%%%%%%%%%%%%%%%%%%%%%%%%%%%%%%%%%%%%%%

\myparagraph{Analysis of $\ZeroMartin_\numobs$}

In order to control the second term, we seek to apply a classical
martingale central limit theorem (cf. Theorem 2.2 in the
paper~\citep{dvoretzky1972asymptotic}).  We begin by observing that
$\NseqNosub{\sqrt{\tuneParScaled_\numobs} \w_i \err_i}{i}$ is a
martingale difference sequence with respect to the sigma-field
$\Nseq{\filtration}{i}$. Noting that the tuning parameter
$\tuneParScaled_\numobs$ is non-random, it follows that the sum
$\Nsum{i}\sqrt{\tuneParScaled_\numobs} \w_i \err_i$ has zero mean and
moreover that
\begin{align}
\Nsum{i} \Cov \left[\sqrt{\tuneParScaled_\numobs} \w_i \err_i \mid
  \filtration_{i - 1}\right] &= \tuneParScaled_\numobs \Nsum{i} \w_i
\w_i^\top.
\end{align}
Consequently, in order to apply the martingale CLT so as to obtain the
stated claim, we need to show that
\begin{align*} \tuneParScaled_\numobs \Nsum{i} \w_i
\w_i^\top \stackrel{\smallop{p}}{\longrightarrow} \noisesd^2 \Id_{\Dim}.
\end{align*}

Doing so requires the following auxiliary lemma, which characterizes
the behavior of the weight vector sequence $\Nseq{\w}{i}$ constructed
in equation~\eqref{eqn:Wn-expression}.
\begin{lems}
\label{lem:stability-lemma}
Under the Assumption~\ref{assumption:A3} parts (a) and (c), the
sequence of vectors $\Nseq{\w}{i}$ obtained from
equation~\eqref{eqn:Wn-expression} has the following properties:
\begin{align*}
 \text{(Stability:)} &\qquad \tuneParScaled_\numobs \sum_{i = 1
 }^{\numobs} \w_i \w_i^\top \stackrel{\smallop{p}}{\longrightarrow} \Id_p, \quad \text{and}
 \\ \text{(Vanishing norm:)} &\qquad \max_{i \in [\numobs]}
 \sqrt{\tuneParScaled_\numobs} \|\w_i\|_2 \stackrel{\smallop{p}}{\longrightarrow} 0.
\end{align*}
\end{lems}
\noindent See Appendix~\ref{AppProofStabilityLemma} for the proof of
this lemma.\\

With the above lemma in hand, we now apply a standard martingale
central limit theorem\footnote{Concretely, by applying Theorem 2.2
  from the paper~\citep{dvoretzky1972asymptotic}, we first show that
  for any unit vector $u$, the inner product $\tfrac{1}{\noisesd}
  \myinprod{u}{ \Nsum{i} \sqrt{\tuneParScaled_\numobs} \w_i \err_i}$
  converges to a standard Gaussian.} to conclude that
\begin{align*}
\sum_{i = 1}^{\numobs} \sqrt{\tuneParScaled_\numobs} \w_i \err_i
\indistrb \Ncal(0, \noisesd^2 \Id_\Dim).
\end{align*}  
Putting together the pieces, we conclude that
\begin{align*}
  \sqrt{\tuneParScaled_\numobs} \cdot \SigmaMat_\numobs^{\frac{1}{2}}
  (\thetaDecorr - \theta) = \Bias_\numobs + \sum_{i = 1}^{\numobs}
  \sqrt{\tuneParScaled_\numobs} \w_i \err_i \indistrb \Ncal(0,
  \noisesd^2 \Id_\Dim),
\end{align*}
which completes the proof of Theorem~\ref{thm:asymp-normality}.

\subsubsection{Proof of claim~\eqref{EqnL1-convergence}:}
\label{sec:L1-convergence}
For simplicity, let us assume $\noisesd$ is known. Recalling the decomposition~\eqref{eqn:bias-variance-decomp}
we have 
\begin{align*}
  \tuneParScaled_\numobs \cdot 
  \| \thetaDecorr - \thetastar \|_{\SigmaMat_\numobs}^2
  &= \enorm{\Bias_\numobs}^2 + 
  2 \myinprod{\Bias_\numobs}{\ZeroMartin_\numobs} + \enorm{\ZeroMartin_\numobs}^2 \\
  &\leq \enorm{\Bias_\numobs}^2
  + 2 \enorm{\Bias_\numobs} \cdot \enorm{\ZeroMartin_\numobs} 
  + \enorm{\ZeroMartin_\numobs}^2
\end{align*}
Invoking the condition $\sqrt{\tuneParScaled_\numobs \log
    \lambda_{\max}(\SigmaMat_\numobs) } \cdot \opnorm{\MyIdDim -
    \W_\numobs \X_\numobs \SigmaMat_\numobs^{-\frac{1}{2}}}  \inEllOne
  0$ we immediately have $\enorm{\Bias_\numobs}^2
\inEllOne 0$. 
It suffices to show that $\enorm{\ZeroMartin_\numobs}^2 \leq \dims$
and $\Exs[\enorm{\ZeroMartin_\numobs}^2] \rightarrow \dims$. Observe that 
\begin{align*}
  \Exs[\enorm{\ZeroMartin_\numobs}^2]
  &= \tuneParScaled_\numobs \cdot \sum_{i = 1}^{\numobs} \Exs [\error_i^2 \enorm{\w_i}^2]
  + \sum_{i \neq j} \tuneParScaled_\numobs \cdot \Exs [\error_i \error_j \w_i^\top \w_j] \\
  & \stackrel{(i)}{=} \sum_{i = 1}^{\numobs} \tuneParScaled_\numobs \cdot \noisesd^2 \cdot \Exs [\enorm{\w_i}^2]
  +  \sum_{i \neq j} \tuneParScaled_\numobs \cdot \Exs [\error_i \error_j \w_i^\top \w_j]
\end{align*}
The last line above follows from the assumption $\Exs[\error_i^2 \mid
  \filtration_{i - 1}] = \noisesd^2$ and the fact (by construction)
that $\w_i \in \filtration_{i - 1}$. Taking trace on both sides of
equation~\eqref{eqn:Id-minus-qqT-expression} we have that
\mbox{$\tuneParScaled_\numobs \cdot \enorm{\w_i}^2 \leq\dims$}, and
using the stronger $L_1$ version of assumption~\ref{assumption:A3}
along with the proof techniques of Lemma~\ref{lem:stability-lemma} we
have $ \tuneParScaled_\numobs \cdot \sum_{i = 1}^\numobs
\Exs[\enorm{\w_i}^2] \rightarrow \noisesd^2 \dims$. It remains to show
that $\Exs [\error_i \error_j \w_i^\top \w_j] = 0$ for all $i \neq
j$. Without loss of generality, assume $i < j$. By construction of
$\w_i$ and the martingale assumption~\ref{assumption:A1} of the noise
$\error_i$, we have $\{ \w_i, \w_j, \error_i\} \in \filtration_{j -
  1}$. As a result, we conclude that
\begin{align*}
 \Exs [\error_i \error_j \w_i^\top \w_j] = \Exs \Big[ \error_i \cdot
   \w_i^\top \w_j \Exs[ \error_j \mid \filtration_{j - 1}]\Big] = 0,
\end{align*}    
thereby completing the proof of the claim~\eqref{EqnL1-convergence}.

%%%%%%%%%%%%%%%%%%%%%%%%%%%%%%%%%%%%%%%%%%%%%%%%%%%%%%%%%%%%%%%%%%%%%%%%%%%%%%%%%%%%%%%%%%%%%%%%%

%%%%%%%%%%%%%%%%%%%%%%%%%%%%%%%%%%%%%%%%%%%%%%%%%%%%%%%%%%%%%%%%%%%%%%%%%%%%%%%%%%%%

\subsection{Proof of Theorem~\ref{thm:Minimax-Lowerbound}}
\label{sec:lower-bound-proof}
We prove part (a) of Theorem~\ref{thm:Minimax-Lowerbound} in
Section~\ref{proof:thm:Minimax-Lowerbound} and part (b) of
Theorem~\ref{thm:Minimax-Lowerbound} in
Section~\ref{proof:Confidence-interval-lower-bound}.

\subsubsection{Proof of Theorem~\ref{thm:Minimax-Lowerbound}(a)}
\label{proof:thm:Minimax-Lowerbound}
Throughout the proof, we use $\thetahat$ to denote a generic estimator
for $\thetastar$.  We assume that the estimator $\thetahat$ is a
function only of the $\numobs$ datapoints $\{(\x_i, \y_i) \}_{i =
  1}^{\numobs}$ and family of selection algorithms
\mbox{$\Selections_\numobs \mydefn (\selection_i)_{i \in [\numobs]}$},
one for each $i \in [\numobs]$; of course, the estimator $\thetahat$
does not know the value of the true parameter $\thetastar$.  Consider
any positive semidefinite and potentially data-dependent matrix
$\mahal \in \reals^{d\times d}$, and define the nonnegative scalar
loss function
\begin{align}
\label{eqn:Loss}
\Loss(\thetahat, \thetastar) & \mydefn (\thetahat - \thetastar)^\top
\mahal (\thetahat - \thetastar).
\end{align} 

\subsection*{\underline{From minimax to Bayes risk}}
In terms of the above notations, Theorem~\ref{thm:Minimax-Lowerbound}
(a) posits a lower bound on the minimax risk:
\begin{align}
\label{eqn:minimax-risk}
  \inf_{\thetahat} \sup_{\thetastar \in \real^\dims} \;
  \Exs[\Loss (\thetahat, \thetastar) \mid \thetastar],
\end{align}
where in the above expression, we have taken an expectation of the
loss $\Loss(\thetahat, \thetastar)$ over the randomness in the data
$(\X_\numobs, \Y_\numobs)$ conditioned on $\thetastar$. We establish
the lower bound Theorem~\ref{thm:Minimax-Lowerbound}(a) on the minimax
risk~\eqref{eqn:minimax-risk} via the standard avenue of first lower
bounding the minimax risk by the Bayes risk, and then providing a
lower bound on the Bayes risk.  In order to do so, we make use of the
inequality
\begin{align}
\label{eqn:Minimax-to-Bayes}
  \inf_{\thetahat} \sup_{\thetastar} \;
  \Exs[
  \Loss(\thetahat, \thetastar) \mid \thetastar]
  \geq \inf_{\thetahat} \; \Exs_{\thetastar, \X_\numobs, \Y_\numobs}
  \Loss(\thetahat, \thetastar),
\end{align}
where the expectation $\Exs_{\thetastar, \X_\numobs, \Y_\numobs}$
above is taken with respect the joint distribution on $(\thetastar,
\X_\numobs, \Y_\numobs)$. Note that this joint distribution is defined
by \emph{choosing} a prior distribution over the parameter
$\thetastar$, and this choice of prior is a design parameter in our
proof.

%%%%%%%%%%%%%%%%%%%%%%%%%%%%%%%%%%%%%%%%%%%%%%%%%%%%%%%%%%%%%%%%%%%%%%%%%

\subsection*{\underline{Main argument}}

We claim that it suffices to prove that for any \mbox{estimator
  $\thetahat$}
\begin{align}
\label{eqn:bayes-risk-bound}
  \Exs[\Loss(\thetahat, \thetastar) \mid \X_\numobs, \Y_\numobs] \geq
  \sigma^2 \tr(\mahal \SigmaMat_\numobs^{-1}).
\end{align}
where the expectation $\Exs[\cdot \mid \X_n, \y_n]$ is taken with
respect to the conditional distribution of \mbox{$\thetastar \mid
  \X_n, \y_n$.} Indeed, taking expectation over $\X_n, \y_n$ yields
the desired bound:
\begin{align}
\Exs_{\thetastar, \X_\numobs, \Y_\numobs} \Loss(\thetahat, \thetastar)
& = \Exs_{\X_\numobs, \Y_\numobs} \Exs[\Loss(\thetahat, \thetastar)
  \mid \X_\numobs, \Y_\numobs] \geq \sigma^2 \Exs_{\X_\numobs,
  \Y_\numobs} \tr(\mahal \SigmaMat_\numobs^{-1}).
\end{align}
Accordingly, it remains to prove the
bound~\eqref{eqn:bayes-risk-bound}.

\subsubsection*{Proof of bound~\eqref{eqn:bayes-risk-bound}}
We complete the proof of this bound by first computing the conditional
distribution of $\thetastar \mid \X_\numobs, \Y_\numobs $, and then
lower bounding the conditional expectation of the loss
$\Loss(\thetahat, \thetastar)$ given the data $(\X_n, \y_n)$.
Concretely, we show that under the prior distribution $\thetastar \sim
\Ncal(0, \rho^2 \Id_\Dim)$, we have
\begin{align}
  \label{eqn:Conditional-dist-of-thetastar}
\thetastar \mid \X_\numobs, \Y_\numobs & \sim \Ncal( \mu_n,
\PopSigma_n), \quad \text{where} \notag \\ \mu_n = \PopSigma_n
\X_\numobs \Y_\numobs \quad &\text{and} \quad \PopSigma_n =
(\SigmaMat_\numobs/\sigma^2 + \Id_\Dim/\rho^2)^{-1}.
\end{align}
A simple calculation using these distributional properties yields that
for any \mbox{positive semidefinite} matrix $\mahal$---one that may
depend on the data $(\X_n, \y_n)$---the function \mbox{$\theta \mapsto
  \Exs[\Loss(\theta, \thetastar) \mid \X_\numobs, \Y_\numobs]$} is
minimized\footnote{Here we have assumed that the prior distribution
$\thetastar \sim \Ncal(0, \rho^2)$ and the error variance $\sigma^2$
are known to the estimator $\thetahat$; this assumption is justified
since without the knowledge of the prior distribution on $\thetastar$
and error-variance $\sigma^2$, the minimum value of the expected loss
$\Exs[\Loss(\thetahat, \thetastar) \mid \X_\numobs, \Y_\numobs]$ can
only increase, which yields a (possibly) stronger lower bound.}  by
the choice \mbox{$\thetahat \mydefn \PopSigma_n \X_\numobs
  \Y_\numobs$}.  Moreover, this choice of estimator yields the minimum
value $\sigma^2\tr(\mahal \PopSigma_n)$.  Finally, we are free to
choose the value of the prior error variance $\rho^2$; in particular,
taking the limit $\rho^2 \rightarrow \infty$ yields the
claim~\eqref{eqn:bayes-risk-bound}. \\

\medskip
\noindent It remains to prove the auxiliary
claim~\eqref{eqn:Conditional-dist-of-thetastar}.

\subsection*{Proof of
  claim~\eqref{eqn:Conditional-dist-of-thetastar}} We proceed via
induction on the number of datapoints $\numobs$.
\paragraph*{Base case} For $n = 0$, we have 
\begin{align}
  \thetastar \mid \X_0, \Y_0 \equiv \thetastar \sim \Ncal(0, \rho^2
  \Id_\Dim),
\end{align}
using the facts that $\thetastar \sim \Ncal(0, \rho^2 \Id_\Dim)$ by
our choice of prior, and the triple $(\X_0, \Y_0, \SigmaMat_0)$ are
defined as zeros of respective dimensions. This proves the
statement~\eqref{eqn:Conditional-dist-of-thetastar} for $n = 0$, and
$\mu_0 = 0$, and $\PopSigma_0 = \rho^2 \Id_\Dim$.

\paragraph*{Induction step}
Given some \mbox{$n \geq 1$,} assume that the
claim~\eqref{eqn:Conditional-dist-of-thetastar} holds for $\numobs -
1$.  Here we show that the statement then holds for $n$.  Recall that
the query algorithm \mbox{$\selection_i: (\reals \times
  \reals^\Dim)^{i - 1} \to \reals^\Dim$} is oblivious to the true
value~$\thetastar$; thus, the conditional distribution $\x_n \mid
\filtration_n$ is independent of $\thetastar$ (see the discussion
before Theorem~\ref{thm:Minimax-Lowerbound}). Furthermore, from the
model~\eqref{eqn:linear-reg-model}, it follows that the conditioned
random variable $y_n \mid \filtration_{n - 1}, \x_n, \thetastar$
follows a $\Ncal(\x_\numobs^\top \thetastar, \sigma^2)$ distribution,
and using the induction
hypothesis~\eqref{eqn:Conditional-dist-of-thetastar}, we conclude that
\mbox{$\thetastar \mid \X_{n - 1}, \y_{n - 1} \sim \Ncal(\mu_{n - 1},
  \PopSigma_{n - 1})$}.

Now let $\frac{d \Prob (\thetastar \mid \X_n, \Y_n)}{d
  \lebesgue^{d}(\thetastar)}$ denote the Radon-Nikodym derivative of
the conditional distribution defind by $\thetastar \mid \X_n, \y_n$
with respect to the Lebesgue measure $\lebesgue^{d}$ on $\real^\dims$.
With the last three observations in hand, an application of Bayes'
rule yields
\begin{align*}
 \frac{d \Prob (\thetastar \mid \X_n, \Y_n)}{d
   \lebesgue^{d}(\thetastar)} & \propto \exp \left\lbrace
 -\frac{1}{2}(\thetastar - \mu_{n - 1})^\top \PopSigma_{n - 1}^{-1}
 (\thetastar - \mu_{n - 1})^\top \right\rbrace \times \exp
 \left\lbrace -\frac{1}{2\sigma^2}(\y_n - \x_n^\top \thetastar)^2
 \right\rbrace \\ &\propto \exp \left\lbrace -\frac{1}{2}(\thetastar -
 \mu_{n})^\top \PopSigma_{n} (\thetastar - \mu_{n})^\top
 \right\rbrace,
 \end{align*} 
where the pair $(\mu_n, \PopSigma_n)$ are given by
\begin{align*} 
\PopSigma_{n}^{-1} = \PopSigma_{n - 1}^{-1} + \frac{\x_n
  \x_n^\top}{\sigma^2} \quad \text{and} \quad \mu_n =
\frac{1}{\sigma^2} \PopSigma_{n} \sum_{i = 1}^{n} \x_i \y_i.
 \end{align*}
This completes the proof of the inductive step, and putting together
the pieces yields the claim of part (a)
of~\Cref{thm:Minimax-Lowerbound}.

\subsubsection{Proof of Theorem~\ref{thm:Minimax-Lowerbound}(b)}
\label{proof:thm:Minimax-Lowerbound-minimax}
This proof and construction follows by discretizing the construction in \cite{Lat23}, which establishes a similar result in a kernelized version of the problem in continuous time. Regrettably, the error terms that arise as a consequence of the discretization lead to a rather unpleasant calculation. We may assume without loss of generality that $\sigma = 1$. The
general result with $\sigma > 0$ can be obtained by a rescaling
argument. For simplicity we also assume that $n$ is divisible by $d-1$, which can be relaxed by correctly rounding the indices of the many sums that appear in the calculations that follow. Let $d > 1$ and $\dir = e_d =
(\bm{0}, 1)^\top \in \reals^d$.
Our proof follows a standard Bayesian argument.
Consider a randomly generated vector $\thetastar$ that is equal to
$\zero$ with probability $1/2$, and otherwise sampled from a multivariate
Gaussian distribution with mean $\mu = (\bm{0}, 1)^\top \in \reals^d$ and degenerate
covariance
\begin{align*}
\Sigma = \begin{pmatrix} \MyIdDim & \zero \\ \zero^\top & 0 \end{pmatrix}\,.
\end{align*}
Lower bounding the supremum by an expectation over this prior, the
minimax risk can be lower bounded as
\begin{align}
\label{eq:lower-1}  
\inf_{\cE} \sup_{\thetastar} \E_{\thetastar}\left[\frac{(\ip{\dir,
      \thetastar} - \cE)^2}{\norm{\dir}^2_{\SigmaMat_n^{-1}}}\right] \geq
\inf_{\cE} \E\left[\frac{(\ip{\dir, \thetastar} -
    \cE)^2}{\norm{\dir}^2_{\SigmaMat_n^{-1}}}\right] \,.
\end{align}
where the second expectation integrates over randomness in
$\thetastar$ as well as the observations $\{(\x_i, \y_i)\}_{i \in
  [n]}$.

We now provide a sequential definition of the selection algorithm that
yields the claimed lower bound. Each covariate $\x_i$ is supported on
the last coordinate as well as one of the first $d-1$ coordinates, chosen in round-robin fashion. 
Let $u_i = 1 + (i \operatorname{mod} (d-1))$ and $v_i = \ceil{i/(d-1)}$, which are chosen so that
\begin{align*}
\{u_i\}_{i\geq 1} = \{1,2,\ldots,d-1,1,2,\ldots\} \text{ and } \{v_i\}_{i \geq 1} = \{\underbrace{1,1,\ldots,1}_{d-1 \text{ times}},2,2,\ldots\} \,.
\end{align*}
In other words, $u_i$ is the index of the first non-zero coordinate in $\x_i$
and $v_i$ is the number of times coordinate $u_i$ was non-zero in rounds $j \leq i$.
The first $d-1$ coordinates of the covariate process are deterministic and the last coordinate is chosen adaptively to maximize the difficulty of estimation. Precisely, $\{\x_i\}_{i
  \geq 1}$ is given by 
\begin{align*}
\x_i \mydefn b_{v_i} e_{u_i} + a_{u_i,v_i} e_d \in \reals^d \,,
\end{align*}
where
$b_v \mydefn v^{-1/4} / \sqrt{d}$, and the random sequence $\{a_{u,v}\}_{u \in \{1,\ldots,d-1\},v \geq 1}$ is to 
defined momentarily. 
Let $\y_{u,w} \mydefn \y_{u+(w-1)(d-1)}$, which is the observed response in the round $i$ where coordinate $u$ was non-zero for the $v$th time. Define
\begin{align*}
m_{u,v} \mydefn \sum_{w=1}^v b_w (\y_{u,w} - a_{u,w}), \quad \mbox{and} \quad d_v
\mydefn 1 + \sum_{w=1}^v b_w^2.
\end{align*}
Then $a_{u,v} \mydefn -b_v m_{u,v-1} / d_{v-1}$, noting that $a_{u,1} = 0$ for all $u$. 

To provide some intuition, our construction is designed so as to make
estimation challenging. Let $\error_{u,v} = \error_{u+(v-1)(d-1)}$. On the event that $\thetastar \neq \zero$, we
have the equality $\y_{u,v} = a_{u,v} + \thetastarScalar_u b_v + \epsilon_{u,v}$, and
the ratio $m_{u,v} / d_v$ is the ridge regression estimate of
$\thetastarScalar_u$. For any vector $\thetastar \neq 0$, the choice $\x_i$ ensures that
\begin{align*}
\ip{\x_i, \thetastar} = a_{u_i,v_i} + b_{v_i} \thetastarScalar_{u_i} \approx 0 = 
\ip{\x_i, \zero}\,,
\end{align*}
so that the observed responses are extremely similarly under either of
the events $\{\thetastar = \zero\}$ or $\{\thetastar \neq \zero\}$.
But $\ip{\dir,\thetastar} = \bm{1}(\thetastar \neq \zero)$, which means that any estimator of $\ip{\dir,\thetastar}$ must have large
error in expectation. What is missing is to formalize the above claims
and show that $\norm{\dir}^2_{\SigmaMat^{-1}_n}$ shrinks suitably fast.

Returning the proof, since $\norm{\dir}^2_{\SigmaMat_n^{-1}}$ is
$\filtration_n$-measurable, the infimum on the right-hand side of
equation~\eqref{eq:lower-1} is achieved by the estimator
\begin{align*}
\cE_\numobs \mydefn \E[\ip{\dir, \thetastar}|\filtration_n] =
\Prob\left(\thetastar \neq \zero|\filtration_n\right).
\end{align*}
Therefore, introducing the event $\Event = \{\thetastar \neq \zero\}$,
we have the lower bound
\begin{align*}
\textrm{Risk} & \geq \E \Biggr[ \frac{\big(\Prob \big[\thetastar \neq
      \zero \mid \filtration_\numobs \big] - \ip{e_d,
      \thetastar}\big)^2}{\norm{e_d}^2_{\SigmaMat_\numobs^{-1}}}
  \Biggr] = \E \Biggr[
  \frac{\cE_\numobs(1-\cE_\numobs)}{\norm{e_d}^2_{\SigmaMat_n^{-1}}}
  \Biggr] \geq \frac{1}{2} \E \Biggr[
  \frac{\cE_\numobs(1-\cE_\numobs)}{\norm{e_d}^2_{\SigmaMat_n^{-1}}}
  \Bigg| \Event \Biggr] \,.
\end{align*}

In the remainder of the proof, we study the laws of $\cE_\numobs$ and
$\smash{\norm{e_d}^2_{\SigmaMat_n^{-1}}}$ under the measure
$\Prob(\cdot \mid \Event)$.  For a sequence of random variables $X_n$,
we use the notation $X_\numobs = \Omega_p(a_n)$ to mean that $\sup_n
\Prob(X_n / a_n < C_\epsilon \mid \Event) < \epsilon$. We will show below that
\begin{subequations}
\begin{align}
\label{eq:lower-a}  
\cE_\numobs(1-\cE_\numobs) &= \Omega_p(1), \qquad \mbox{and} \\
\label{eq:lower-b}
\frac{1}{\norm{e_d}_{\SigmaMat_n^{-1}}^2} 
&= \Omega_p(d \log(n/d^3)) \,.
\end{align}
\end{subequations}
Therefore, there exists a universal constant $C > 0$ such that for $n \geq d^3/C$,
\begin{align*}
\Prob\left(1/\norm{e_d}_{\SigmaMat_n^{-1}}^2 \geq C d \log(n)\right) \geq 3/4 \text{ and }
\Prob\left(\cE_\numobs(1 - \cE_\numobs) \geq C\right) \geq 3/4\,.
\end{align*}
By a union bound and the positivity in the integrand of the risk,
\begin{align*}
\textrm{Risk} &\geq \frac{C^2 d \log(n/d^3)}{4} \,.
\end{align*}

\medskip

\noindent It remains to show that the lower bounds~\eqref{eq:lower-a}
and \eqref{eq:lower-b} hold, which we prove in Appendix~\ref{app:minimax-log-bound-proofs}.  

% \mjwcomment{START: These calculations need to be edited.  Notation
%   made consistent; checked for correctness; steps explained and
%   writing made smoother.}

% \mjwcomment{Possibly we should move the proofs of these claims to the
%   Appendix --- fairly technical in nature, and we will also have
%   length constraints on the length of the main body for AOS.}

%%%%%%%%%%%%%%%%%%%%%%%%%%%%%%%%%%%%%%%%%%%%%%%%%%%%%%%%%%%%%%%%%%%%%%%%%%%%%

% \mjwcomment{END: These calculations need to be edited.  Notation
%   made consistent; checked for correctness; steps explained and
%   writing made smoother.}

\subsubsection{Proof of~\Cref{thm:Minimax-Lowerbound}(c)}
\label{proof:Confidence-interval-lower-bound}

% \mjwcomment{I did a fair bit of notation change, editing, and
%   reorganization here to try and make things clearer.. Please check.}
Let $\ExsData$ denote expectation over a data set drawn from the
distribution indexed by $\thetabf$, and define $\ProbData$ as the
analogous probability.  For a given dataset based on $\numobs$
samples, consider a confidence interval of the form $[\lowhat_\numobs,
  \uphat_\numobs]$.  Introducing the shorthand $\smallv^2 \mydefn v^T
\SigmaMat_\numobs^{-1} v$, we then define the minimax risk
\begin{align*}
\CIrisk & \mydefn \inf_{\hackCI} \sup_{\thetastar} \Exs_\thetastar
\bigg\{ \frac{\uphat_\numobs - \lowhat_\numobs}{\smallv} \bigg\},
\end{align*}
where we take the infimum over all estimators $\hackCI$ such that
$\Prob_{\thetastar} \big( [\lowhat_\numobs, \uphat_\numobs] \ni v^\top \thetastar \big)
\geq 1 - \alpha$ for each value of $\thetastar$.

By the usual Bayesian argument, for any prior distribution $\pi$ on $\thetastar$, we have the lower
bound
\begin{align*}
\CIrisk \geq \inf_{\hackCI} \E_{\thetastar \sim \pi} \Exs_\thetastar \bigg[
  \frac{\uphat_\numobs - \lowhat_\numobs}{\smallv} \bigg].
\end{align*}
We now obtain a further lower bound by enlarging the space of possible
estimators $\hackCI$, in particular requiring only that $\hackCI$
belong to the set
\begin{align*}
  \mathcal{A} & \mydefn \Big \{ \hackCI \; \mid \; \Exs_{\thetastar \sim
    \pi} \Prob_\thetastar \big( [\lowhat_\numobs, \uphat_\numobs] \ni
  v^\top \thetastar\big) \geq 1 - \alpha \Big \}.
\end{align*}
Since this allows for a larger collection of possible estimators, we
have the lower bound
\begin{align*}
\CIrisk & \ge \inf_{\hackCI \in \mathcal{A}} \E_{\thetastar \sim \pi}
\Exs_\thetastar \bigg\{ \frac{\uphat_\numobs - \lowhat_\numobs}{\smallv}
\bigg\}.
\end{align*}

We are now free to choose the prior.  In particular, we set $\pi$
equal to the density $\phi_\rho$ of the Gaussian random vector
$\Ncal(0, \rho^2 \Id_d)$.  From our previous calculations~\eqref{eqn:Conditional-dist-of-thetastar} we have that conditional on the observed
data, the random vector $\thetastar$ is Gaussian with
covariance $(\SigmaMat_n/\sigma^2 +\Id_d/\rho^2)^{-1}$. With this
choice, the random variable $v^\top\thetastar$, conditioned on
the observed data, is a Gaussian random variable with variance $\smallvtil^2
\mydefn v^T (\SigmaMat_n/\sigma^2 + \Id_d/\rho^2)^{-1} v$.  Therefore, the width of 
any confidence interval  $\hackCI \in \mathcal{A}$ is lower bounded by 
$2 z_{1-\alpha/2} \smallvtil$, and we have
\begin{align*}
\CIrisk & \ge 2 z_{1-\alpha/2} \E_{\thetastar \sim \pi} \Exs_\thetastar \Big[
  \frac{\smallvtil}{\smallv} \Big].
\end{align*}
It remains to show that $\lim_{\rho \to \infty} \E_{\thetastar\sim \pi}
\Exs_\thetastar \big[ \frac{\smallvtil}{\smallv} \big] = 1$.  Recalling that
$\phi_\rho$ denotes the Gaussian density with zero mean and covariance
$\rho^2 I_d$, we have
\begin{align*}
\E_{\thetastar \sim \pi} \Exs_\thetastar \Big[ \frac{\smallvtil}{\smallv} -1
  \Big] & = \int \Exs_\thetastar \Big[\frac{\smallvtil}{\smallv} - 1 \Big]
\phi_\rho(\thetastar) d \thetastar,
\end{align*}
By the bounded convergence theorem, we have
\begin{align*}
\lim_{\rho \to \infty} \Exs_\thetastar \Big[\frac{\smallvtil}{\smallv} - 1
  \Big] & = 0,
\end{align*}
pointwise for each $\thetabf$.  Consequently, the quantity in the
integral converges point-wise to zero, so that applying the bounded
convergence theorem again yields
\begin{align*}
\lim_{\rho \to \infty} \E_{\thetastar \sim \pi} \Exs_\thetastar \Big[
  \frac{\smallvtil}{\smallv} -1 \Big] = \lim_{\rho \to \infty} \int
\Exs_\thetastar \Big[ \frac{\smallvtil}{\smallv} - 1 \Big] \phi_\rho(\thetastar)
d \thetastar & = 0,
\end{align*}
which completes the proof of part (c).
%
%%%%%%%%%%%%%%%%%%%%%%%%%%%%%%%%%%%%%%%%%%%%%%%%%%%%%%%%%%%%%%%%%%%%%%%%%%%%%%%

\section{Proofs of corollaries} 
\label{app:proof-of-corollaries}

We now turn to the proofs of our three corollaries, with
Sections~\ref{proof:cor-bandits},~\ref{proof-unit-root-autoreg},
and~\ref{sec:proof-multidim-gen-case} devoted to the proofs of the
Corollaries~\ref{cor:bandits-corr},~\ref{cor:unit-root-autoreg}
and~\ref{cor:multidim-gen-case}, respectively.

%%%%%%%%%%%%%%%%%%%%%%%%%%%%%%%%%%%%%%%%%%%%%%%%%%%%%%%%%%%%%%%%%%%%%%

\subsection{Proof of Corollary~\ref{cor:bandits-corr}}
\label{proof:cor-bandits}

In light of Theorem~\ref{thm:asymp-normality}, it suffices to verify
Assumptions~\ref{assumption:A1}--\ref{assumption:A3}.  The assumptions
stated in Corollary~\ref{cor:bandits-corr} ensure that the error
sequence $\Nseq{\error}{i}$ satisfies Assumption~\ref{assumption:A1}.
The growth conditions in Assumption~\ref{assumption:A2} are satisfied
due to the minimum arm-pull
assumption~\eqref{eqn:bandits-minimum-arm-LB}.  It remains to verify
the three conditions in Assumption~\ref{assumption:A3}.

Beginning with the asymptotic negligibility condition, we have
\begin{align*}
   \max \limits_{i \in [\numobs]} \;\;
   \frac{1}{\tuneParScaled_\numobs} \x_i^T \XscaleMat_i^{-1} \x_i \leq
   \frac{1}{\tuneParScaled_\numobs} \frac{ \max_{i \in [\numobs]}
     \|\x_i^2 \|}{(\log \numobs)^2} & = \frac{\log\log(\numobs)}{(\log
     \numobs)} \rightarrow 0.
\end{align*}
The first inequality above uses the bound $\XscaleMat_i^{-1} \preceq
\frac{1}{\log(\numobs)^2} \cdot \Id_\Dim$ (see the
definition~\eqref{eqn:bandits-tuning}); the second equality uses $\|
\x_i \|_2^2 = 1$, and the final step follows by substituting
${\tuneParScaled_\numobs = 1/((\log\numobs) \cdot \log\log(\numobs))}$.

\vspace{10pt}

Turning to the vanishing bias condition in~\ref{assumption:A3}, we invoke
the operator norm bound~\eqref{eqn:bias-bound-improved} on the matrix
$\Id_\Dim - \W_\numobs \X_\numobs \SigmaMat_\numobs^{-\frac{1}{2}}$ to
find that
\begin{align*}
  \sqrt{\tuneParScaled_\numobs \log\lambda_{\max}(\SigmaMat_\numobs)}
  \cdot \opnorm{ \Id_\Dim - \W_\numobs \X_\numobs
    \SigmaMat_\numobs^{-\frac{1}{2}} } & \leq
  \sqrt{\tuneParScaled_\numobs \log \numobs} \cdot \BigoP(1) \\
  & =  \sqrt{\frac{\log\log(n)}{(\log \numobs)}} \cdot \BigoP(1)
    \stackrel{\smallop{p}}{\longrightarrow} 0,
\end{align*}
where we have used the bound \mbox{$\lambda_{\max}(\SigmaMat_\numobs)
  \leq \tr(\SigmaMat_\numobs) = \numobs$} in the above derivation.

\vspace{10pt}

Finally, we verify the variance stability condition in
Assumption~\ref{assumption:A3} with the help of the following lemma
\begin{lems}[Commutative guarantee]
  \label{lems:Commutative-lemma}
For any collection of matrices $\;\;$ \mbox{$\NseqNosub{ \XscaleMat_i^{-\frac{1}{2}}
  \x_i \x_i^\top \XscaleMat_i^{-\frac{1}{2}}}{i}$} that commute with
each other, we have
\begin{align*}
  \opnorm{ \Id - \Nsum{i} \w_i \x_i^\top \XscaleMat_i^{- \frac{1}{2}} }
  \leq \exp \left( - \frac{\lambda_{\min}(\Nsum{i}
    \XscaleMat_i^{-\frac{1}{2}} \x_i \x_i^\top
    \XscaleMat_i^{-\frac{1}{2}})}{\tuneParScaled_\numobs} \right).
\end{align*}
\end{lems}
\noindent  See the end of this subsection for the proof of this claim. \\

Let us complete the proof of~\cref{cor:bandits-corr}
using~\cref{lems:Commutative-lemma}.  In the multi-armed bandit setting
of Corollary~\ref{cor:bandits-corr}, the matrices
$\NseqNosub{\XscaleMat_i^{-\frac{1}{2}} \x_i \x_i^\top
  \XscaleMat_i^{-\frac{1}{2}}}{i}$ are all diagonal, and hence they
commute. Thus, invoking the operator norm bound from
Lemma~\ref{lems:Commutative-lemma} yields
\begin{align*}
    \opnorm{ \Id_\Dim - \Nsum{i} \w_i \x_i^\top
      \XscaleMat_i^{-\frac{1}{2}} } & \leq \exp \left( - \frac{
      \lambda_{\min}(\Nsum{i} \XscaleMat_i^{-\frac{1}{2}} \x_i
      \x_i^\top \XscaleMat_i^{-\frac{1}{2}}) }{\tuneParScaled_\numobs}
    \right).
\end{align*}
Recall that in the bandits model~\eqref{model:bandits}, the matrices $\SigmaMat_\numobs$ and $\x_i \x_i^\top$ are diagonal.
By construction~\eqref{eqn:bandits-tuning} and the minimum arm-pull
condition~\eqref{eqn:bandits-minimum-arm-LB}, the tuning matrix $\XscaleMat_i$ is also diagonal with diagonal entries upper bounded by the corresponding diagonal entries of the (diagonal) matrix $\SigmaMat_\numobs$. Combining these two observations we have that 
$\SigmaMat_\numobs^{-\frac{1}{2}} \x_i \x_i^\top
\SigmaMat_\numobs^{-\frac{1}{2}} \preceq \XscaleMat_i^{-\frac{1}{2}}
\x_i \x_i^\top \XscaleMat_i^{-\frac{1}{2}}$.
Consequently, we find that
\begin{align*}
- \lambda_{\min}(\Nsum{i} \XscaleMat_i^{-\frac{1}{2}} \x_i \x_i^\top
\XscaleMat_i^{-\frac{1}{2}}) & \stackrel{(i)}{\leq} -
\lambda_{\min}(\Nsum{i} \SigmaMat_\numobs^{-\frac{1}{2}} \x_i
\x_i^\top \SigmaMat_\numobs^{-\frac{1}{2}}) 
     = -1,
\end{align*}
where the final equality follows from the definition of
$\SigmaMat_\numobs$.  Substituting the value $\tuneParScaled_\numobs =
1 / (\log(\numobs) \cdot \log\log(\numobs))$ yields
 \begin{align*}
   \opnorm{\Id_\Dim - \Nsum{i} \w_i \x_i^\top
     \XscaleMat_i^{-\frac{1}{2}} } \leq \exp(-1/\tuneParScaled_\numobs)
   \; \leq \; \frac{1}{\numobs} \rightarrow 0.
 \end{align*}
This verifies the variance stability condition from
Assumption~\ref{assumption:A3}, and applying
Theorem~\ref{thm:asymp-normality} yields
Corollary~\ref{cor:bandits-corr}. \\

\noindent The only remaining detail is to prove
Lemma~\ref{lems:Commutative-lemma}.

%%%%%%%%%%%%%%%%%%%%%%%%%%%%%%%%%%%%%%%%%%%%%%%%%%%%%%%%%%%%%%%%%%%%%%%%%%%%%%%%%%%%%%%

\subsection*{Proof of Lemma~\ref{lems:Commutative-lemma}} 

For notational convenience, we use the shorthands \mbox{$\z_i \mydefn
  \x_i \XscaleMat_i^{-\frac{1}{2}}$} and $\Z_i^\top
\mydefn \begin{bmatrix} \z_1 & \cdots & \z_i \end{bmatrix}$, as
previously introduced in Section~\ref{sec:W-decorr}. Substituting the
formula for the weight vector $\w_i$ from~\cref{eqn:Wn-expression},
and performing some algebra yields
\begin{align*}
  \left( \Id - \W_\numobs \Z_\numobs\right)^\top \left(\Id -
  \W_\numobs \Z_\numobs\right) & = \prod \limits_{j = 1}^{\numobs}
  \left( \Id_\Dim - \frac{\z_{\numobs + 1 -j} \z_{\numobs + 1 -
      j}^\top}{\tuneParScaled_\numobs + \|\z_{\numobs +1 -j}\|^2}
  \right) \prod \limits_{i = 1}^{\numobs} \left( \Id_\Dim - \frac{\z_i
    \z_i^\top}{\tuneParScaled_\numobs + \|\z_i\|^2} \right) \\
&=
\begin{multlined}[t]
 \exp \left[ \sum_{j=1}^\numobs \log \left( \Id_\Dim -
   \frac{\z_{\numobs + 1- j} \z_{\numobs + 1-
       j}^\top}{\tuneParScaled_\numobs + \|\z_{\numobs + 1 - j}\|^2}
   \right)\right. \\ + \left.\Nsum{i} \log \left( \Id_\Dim -
   \frac{\z_{i} \z_{i}^\top}{\tuneParScaled_\numobs + \|\z_{i}\|^2}
   \right) \right]
\end{multlined}\\
& \stackrel{(i)}{\preccurlyeq} \exp \left( - \Nsum{j}
  \frac{\z_{\numobs + 1 -j} \z_{\numobs + 1
      -j}^\top}{\tuneParScaled_\numobs + \|\z_{\numobs + 1 -j}\|^2} - \Nsum{i}
  \frac{\z_i \z_i^\top}{\tuneParScaled_\numobs  + \|\z_{i}\|^2} \right) \\
& \preccurlyeq \exp \left( - 2\cdot\Nsum{i} \frac{\z_i
    \z_i^\top}{\tuneParScaled_\numobs} \right),
\end{align*}
where step (i) above uses the fact that $\exp(\log(1 - a)) \leq
\exp(-a)$ for any scalar $a < 1$ and that the matrices $\{\z_i
\z_i^\top\}_{i \in [\numobs]}$ commute.  Via an inductive argument, it
can be verified that the entries of the matrix $\Id_\Dim -
\frac{\z_{i} \z_{i}^\top}{\tuneParScaled_\numobs + \|\z_{i}\|^2}$ are
all upper bounded by $1$.  Putting together the pieces, we conclude
that the operator norm satisfies the bound
\begin{align*}
  \opnorm{ \Id - \W_\numobs \Z_\numobs } \leq \exp \left( -
  \frac{\lambda_{\min}(\Nsum{i} \z_i
    \z_i^\top)}{\tuneParScaled_\numobs} \right),
\end{align*}
as claimed. 

%%%%%%%%%%%%%%%%%%%%%%%%%%%%%%%%%%%%%%%%%%%%%%%%%%%%%%%%%%%%%%%%%%%%%%

\subsection{Proof of Corollary~\ref{cor:unit-root-autoreg}}
\label{proof-unit-root-autoreg}

The proof of this claim is similar to that of
Corollary~\ref{cor:bandits-corr}; in particular, we need to verify
Assumptions~\ref{assumption:A1}--\ref{assumption:A3}.  Recall that the
time series model~\eqref{eqn:autoreg-model} in
Corollary~\ref{cor:unit-root-autoreg} is a special case of the
stochastic linear regression model~\eqref{eqn:linear-reg-model} with
$(x_i, y_i) \equiv (y_{i - 1}, y_i)$; thus the covariance term based
on the data $\{(x_i, y_i) \}_{i \in [\numobs]}$ is given by $\Nsum{i}
x_{i}^2 \equiv \Nsum{i} y_{i - 1}^2$. Here we have used the convention
$y_0 = 0$.

\vspace{10pt}

The moment condition~\ref{assumption:A1} is satisfied since the
additive noise $\error_i$ in the autoregressive
model~\eqref{eqn:autoreg-model} is assumed to have a standard Gaussian
distribution. Before we verify the remaining conditions, it is helpful
to deduce a few bounds regarding the sample covariance term $\Nsum{i}
y_{i - 1}^2$. In particular, we show that for any $\thetastarScalar
\in (-1, 1]$, the sample covariance term satisfies the following
  relations
\begin{subequations}
  \begin{align}
\label{eqn:AR1-lower-bound}    
\Nsum{i} y_{i - 1}^2 &\rightarrow \infty \quad \text{almost surely},
\quad \\
\label{eqn:AR1_var_scaling}
\log ( \Nsum{i} y_{i - 1}^2 ) &= \BigoP(\log \numobs), \quad
\text{and} \quad (\log\numobs)^2  y_{\numobs}^2 =
\BigoP ( \Nsum{i} y_{i - 1}^2 ),
\end{align}
\end{subequations}
We prove
these bounds at the end of this sub-section, but let us complete the
proof of the Corollary using these bounds.

First, observe that the condition~\ref{assumption:A2} follows from the
growth condition~\eqref{eqn:AR1-lower-bound}, and the asymptotic
negligibility condition in~\ref{assumption:A3} is satisfied by noting
that
\begin{align*}
  \max_{i \in [\numobs]} \frac{1}{\tuneParScaled_\numobs} \cdot
  \frac{y_{i - 1}^2}{ \max\{ (\log \numobs)^{2} y_{i -
      1}^2 , \;\; \sum_{j =1}^{i-1} y_j^2 \} } & \leq \frac{(\log
    \numobs) \cdot \log\log(n) }{(\log \numobs)^{2}}
  \rightarrow 0.
\end{align*}

Next, in order to verify the vanishing bias condition in
Assumption~\ref{assumption:A3}, doing a calculation similar to
Proposition~\ref{prop:zero-bias-lemma} we find that (see the arguments
leading up to
bounds~\eqref{eqn:first-term-max-norm-bound}--\eqref{eqn:second-term-max-norm-bound}
and their proofs)
\begin{align*}
\sqrt{\tuneParScaled_\numobs \log (\Nsum{i} y_{i - 1}^2 )} \cdot
\left| 1 - \Nsum{i} \frac{w_i y_{i -1}}{\XscaleMat_n^{\frac{1}{2}}}
\right | & \stackrel{(i)}{=} \BigoP \left( \frac{1}{\sqrt{\log\log(n)}} \right) \cdot \BigoP \left( 1 + \sqrt{
  \frac{\XscaleMat_\numobs}{\Nsum{i} y_{i - 1}^2 } } \right) \\
& \stackrel{(ii)}{=} \BigoP \left( \frac{1}{\sqrt{\log\log(n)}} \right)
\cdot \BigoP(1) \\
& \stackrel{\smallop{p}}{\longrightarrow} 0,
\end{align*}
where step (i) follows by invoking the first part of the
bound~\eqref{eqn:AR1_var_scaling} and step (ii) uses the second part
of~\cref{eqn:AR1_var_scaling}.

Finally, we verify the variance stability condition
in~\ref{assumption:A3} with the help of
Lemma~\ref{lems:Commutative-lemma}, as previously stated and proved in
the proof of Corollary~\ref{cor:bandits-corr}.  Note that in dimension
$\Dim = 1$, the commutativity condition in
Lemma~\ref{lems:Commutative-lemma} holds trivially.  Consequently, we
may apply Lemma~\ref{lems:Commutative-lemma} to the one-dimensional
autoregressive model~\eqref{eqn:autoreg-model} so as to obtain the
bound
\begin{align*}
  \left | 1 - \Nsum{i} \frac{w_i y_{i - 1}}{ \XscaleMat_i^\frac{1}{2}}
  \right| & \leq \frac{1}{\numobs}.
\end{align*}
See the calculations following the statement of
Lemma~\ref{lems:Commutative-lemma} in the proof of
Corollary~\ref{cor:bandits-corr} for details on this step.

This verifies the variance stability condition from
Assumption~\ref{assumption:A3}, and applying
Theorem~\ref{thm:asymp-normality} yields
Corollary~\ref{cor:unit-root-autoreg}. \\

\noindent The only remaining detail is to prove the
\mbox{bounds~\eqref{eqn:AR1-lower-bound}--\eqref{eqn:AR1_var_scaling}.}

%%%%%%%%%%%%%%%%%%%%%%%%%%%%%%%%%%%%%%%%%%%%%%%%%%%%%%%%%%%%%%%%%%%%%%%%%%%%%%%%%%%%%%%%%%%%%%%%%%%%%

\subsection*{Proofs of the
  bounds~\eqref{eqn:AR1-lower-bound}--\eqref{eqn:AR1_var_scaling}}

The proof of the first part of the bound \eqref{eqn:AR1_var_scaling}
follows by invoking Theorem 2 part (i) from the
paper~\citep{lai1982least}.  Concretely, in the
paper~\citep{lai1982least}, the authors showed that when
$|\thetastarScalar| \leq 1$, then there is some constant $a > 0$ such
that $y_\numobs = \BigoP(\numobs^a)$. Thus, we have the relation $\Nsum{i} y_{i -
  1}^2 = \BigoP(\numobs^{2a + 1})$, and first part of the
bound~\eqref{eqn:AR1_var_scaling} follows.

We divide the proof of the remaining bounds into two parts, depending
on the value of $\thetastarScalar$.

\subsection*{\underline{Case 1}} First, suppose that $\thetastarScalar = 1$.
Recall that in~\cref{eqn:Var-scaling-AR1} we argued that
\begin{align*}
\frac{1}{\numobs^2} \Nsum{i} \y_{i - 1}^2 &\indistrb \int_{0}^1
\wiener^2(t) dt.
 \end{align*} 
In light of the last relation, the growth
condition~\eqref{eqn:AR1-lower-bound} is immediate.  For the remaining
bounds, note that \mbox{$y_\numobs \mydefn \sum_{i \in [\numobs - 1]}
  \error_i \sim \Ncal(0, \numobs - 1)$}; thus we have \mbox{$\frac{1}{\numobs^2} \cdot (\log\numobs)^{2} y_\numobs^2 \stackrel{\smallop{p}}{\longrightarrow} 0$}, and we conclude that
\begin{align*}
  \log(\numobs)^{2}  \y_{\numobs}^2 = \BigoP
  \left( \Nsum{i} \y_{i - 1}^2 \right),
\end{align*}
as claimed.

\subsection*{\underline{Case 2}} Otherwise, we may assume that $|\thetastarScalar| 
 < 1$, in which case the term $\Nsum{i} \y_{i - 1}^2$ stabilizes~\citep{lai1982least};
concretely, we have
\begin{align*}
  \frac{1}{\numobs} \Nsum{i} \y_{i - 1}^2 \almostSurely c, \quad
  \text{where} \quad c>0 \;\; \text{is a non-random scalar}.
\end{align*}
The growth condition~\eqref{eqn:AR1-lower-bound} follows directly from
the above relation.  Moreover, we have $y_{\numobs - 1} = \Nsum{i}
{\thetastarScalar}^{i} \error_{\numobs - i} \sim \Ncal \left(0,
\frac{1}{1 - {\thetastarScalar}^2} \right)$.  Putting these two pieces
together yields ${(\log\numobs)^{2} \cdot y_{\numobs - 1}^2 = \BigoP
  \left( \Nsum{i} \y_{i - 1}^2 \right)}$.

%%%%%%%%%%%%%%%%%%%%%%%%%%%%%%%%%%%%%%%%%%%%%%%%%%%%%%%%%%%%%%%%%%%%%%%%%

\subsection{Proof of Corollary~\ref{cor:multidim-gen-case}}
\label{sec:proof-multidim-gen-case}
We obtain the first claim of the Corollary~\ref{cor:multidim-gen-case}
by applying Theorem~\ref{thm:asymp-normality}, and the second part of
the Corollary~\ref{cor:multidim-gen-case} follows
from~\cref{prop:gen-confinv}.  We prove these two parts separately.

\vspace{8pt}

\textbf{\underline{Proof of claim~\eqref{eqn:gen-case-asymp-normality}}:} \hspace{4pt} 
In order to apply Theorem~\ref{thm:asymp-normality} to the setup of
Corollary~\ref{cor:multidim-gen-case} it suffices to verify the
Assumption~\ref{assumption:A3}. Recall that our choice of scaling
$\XscaleMat_i = \Nsum{j} \greedyEps_j \ExsXsqLB$ matrices does not
actually vary as a function of the round $i$.  For this reason, we simply
write $\XscaleMat$ from here onwards.
We begin by verifying the asymptotic negligibility condition
in~\ref{assumption:A3}.  Observe that
\begin{align}
\label{eqn:asymp-negligibility-suff-explore}
\Exs \left\lbrace \max_{i \in [\numobs]} \;\;
\frac{1}{\tuneParScaled_\numobs} \x_i^\top \XscaleMat^{-1} \x_i
\right\rbrace \leq \frac{1}{\tuneParScaled_\numobs} \cdot \frac{ \Exs
  \left [ \max \limits_{i \in [\numobs]} \| \x_i \|_2^2 \right ]}{
  \lambda_{\min}(\ExsXsqLB) \; (\Nsum{i} \greedyEps_i ) } \rightarrow
0,
\end{align}
where the first inequality above follows by substituting the value of
the scaling matrix $\XscaleMat$, and the second step follows by
invoking the sufficient exploration
condition~\eqref{eqn:suff-exploration}.

Next, we verify the variance stability and vanishing bias conditions
in~\ref{assumption:A3}.  In doing, we make use of the following
auxiliary result:
\begin{lems}
\label{lem:suff-exploration}
Under the sufficient exploration
condition~\eqref{eqn:suff-exploration}, for any tuning parameter
$\tuneParScaled_\numobs \in (0, 1/ (\log(\numobs) \cdot \log\log(\numobs))]$ and
  a sufficient large sample size $\numobs$, we have
\begin{align*}
  \Exs \big[ \frobnorm{ \MyIdDim - \W_\numobs \X_\numobs
      \XscaleMat^{-\frac{1}{2}} }^2 \big] & \leq \frac{\Dim}{K\numobs}.
  \end{align*}
\end{lems}

\noindent 
See Section~\ref{sec:Proof-of-suff-exploration} for the proof of this
lemma. \\

Taking \cref {lem:suff-exploration} as given, we now complete the
proof of Corollary~\ref{cor:multidim-gen-case}.  Note that the
variance stability condition in~\ref{assumption:A3} follows directly
from the Frobenius norm bound in Lemma~\ref{lem:suff-exploration} and
by letting the number of datapoints $n \rightarrow \infty$, keeping the
dimension $\Dim$ fixed.

In order to prove the vanishing bias condition in~\ref{assumption:A3},
we first bound the operator norm of the matrix $\MyIdDim - \W_\numobs
\X_\numobs \SigmaMat_\numobs^{-\frac{1}{2}}$:
\begin{align}
 \opnorm{ \MyIdDim - \W_\numobs \X_\numobs
   \SigmaMat_\numobs^{-\frac{1}{2}}} & \leq 1 + \opnorm{ \W_\numobs
   \X_\numobs \XscaleMat^{-\frac{1}{2}} } \cdot \opnorm{
   \XscaleMat^{\frac{1}{2}} \SigmaMat_\numobs^{-\frac{1}{2}} } \notag \\
& = \BigoP(1) 
\label{eqn:WNS-gen-bound},
\end{align}
where the derivation above uses the Frobenius norm upper bound from
Lemma~\ref{lem:suff-exploration} and the fact that $\opnorm{
  \XscaleMat^\frac{1}{2} \SigmaMat_\numobs^{-\frac{1}{2}}} = \BigoP(1)$
by the choice of the tuning parameter $\XscaleMat$; see the
bound~\eqref{eqn:exploration-lower-bound} for instance. Using the last
bound on $\opnorm{ \MyIdDim - \W_\numobs \X_\numobs
  \SigmaMat_\numobs^{-\frac{1}{2}}}$, we then find that
\begin{align*}
  \sqrt{\tuneParScaled_\numobs \log \lambda_{\max}(\SigmaMat_\numobs)}
  \cdot \opnorm{ \MyIdDim - \W_\numobs \X_\numobs
    \SigmaMat_\numobs^{-\frac{1}{2}} } \leq
  \sqrt{\tuneParScaled_\numobs \log \lambda_{\max}(\SigmaMat_\numobs)}
  \cdot \BigoP(1) \stackrel{\smallop{p}}{\longrightarrow} 0,
\end{align*}
where the last step above utilizes the choice $\tuneParScaled_\numobs
= \smalloP(\log K\numobs)$ and the bound $
\lambda_{\max}(\SigmaMat_\numobs) = \BigoP(\log K\numobs )$. (Recall
the uniform boundedness assumption~\eqref{eqn:bounded-covariates}.)
All together, we have verified the assumptions of
Theorem~\ref{thm:asymp-normality}, so that
Corollary~\ref{cor:multidim-gen-case} follows. \\

\noindent It remains to prove Lemma~\ref{lem:suff-exploration}.

%%%%%%%%%%%%%%%%%%%%%%%%%%%%%%%%%%%%%%%%%%%%%%%%%%%%%%%%%%%%%%%%%%%%%%%%%%%%%%%%%%%%%%%%%%%%%%%%%%%%%%%%

\subsubsection{Proof of Lemma~\ref{lem:suff-exploration}}
\label{sec:Proof-of-suff-exploration}

Throughout this proof, we use the shorthands \mbox{$\z_i = \XscaleMat^{-\frac{1}{2}} \x_i
  $}, $\Z_i^\top = [\z_1, \z_2, \ldots \z_i]$,
$\W_i = [\w_1, \ldots, \w_i]$ and ${\DelMat_i \mydefn \MyIdDim - \W_i
  \Z_i}$.  Substituting the expression~\eqref{eqn:Wn-expression} for
the weight vector $\w_i$ we find that
\begin{align}
\frobnorm{ \DelMat_{i-1} }^{2} - \frobnorm{ \DelMat_{i} }^{2} & =
\frac{\tuneParScaled_\numobs + \| \z_{i}\|_{2}^{2}}{\left(
  \tuneParScaled_\numobs/2 + \| \z_{i} \|_{2}^{2}\right)^{2}} \tr
\left\{ \DelMat_{i-1} \z_{i} \z_{i}^{\top}
\DelMat_{i-1}^{\top}\right\} \nonumber \\
\label{eqn:delta_i_decay}
& \geq \frac{1}{ \tfrac{\tuneParScaled_\numobs}{2} + \|\z_{i}
  \|_{2}^{2}} \tr \left\{\DelMat_{i-1} \z_{i} \z_{i}^{\top}
\DelMat_{i-1}^{\top}\right\}.
\end{align}
In equation~\eqref{eqn:asymp-negligibility-suff-explore}, we proved
that the random variable \mbox{$\frac{1}{\tuneParScaled_\numobs}
  \max_{i \in [\numobs]}\| \z_i\|_2^2$} converges to zero in
probability; consequently, we may assume that
\begin{align*}
  \Prob \big[
    \max_{i \in \numobs} \| \z_i\|_2^2 \leq \tuneParScaled_\numobs/2
    \big] \geq \frac{1}{2}
\end{align*}
for all sufficiently large values of the sample size
$\numobs$. Keeping this in mind, taking expectations conditional on
the sigma-field $\filtration_{i-1}$ on both sides in the
inequality~\eqref{eqn:delta_i_decay}, and using the fact that
$\DelMat_{i} \in \filtration_{i - 1}$, we have
\begin{align*}
\Exs \left [ \frobnorm{\DelMat_{i-1} }^{2} \mid \filtration_{i-1}
  \right ] - \Exs \left [ \frobnorm{\DelMat_{i}}^{2} \mid
  \filtration_{i-1} \right ] \geq \frac{\greedyEps_i}{2
  \tuneParScaled_\numobs \Nsum{i} \greedyEps_i} \Exs \left [
  \frobnorm{\DelMat_{i - 1}}^{2} \mid \filtration_{i-1}\right].
\end{align*}
Rearranging the last inequality and using the upper bound $(1 - t)
\leq \exp(-t)$ for $t \geq 0$ we obtain
\begin{align*}
  \Exs \left [ \frobnorm{\DelMat_{i}}^{2} \mid \filtration_{i-1}
    \right] & \leq \exp \left(\frac{-\greedyEps_i}{2
    \tuneParScaled_\numobs \Nsum{i} \greedyEps_i }\right)
  \frobnorm{\DelMat_{i-1}}^{2}.
\end{align*}
Iterating the last bound $\numobs$ times and removing the conditioning
on the sigma filed $\filtration_{i -1}$, we find that
\begin{align*}
  \Exs \left\{ \frobnorm{\DelMat_{\numobs}}^{2} \right\} \leq \Dim
  \exp \left(-\frac{1}{2\tuneParScaled_\numobs} \right) \stackrel{(i)}{\leq}
  \Dim \exp \left(-\frac{1}{2} \cdot \log(\numobs) \cdot \log\log(\numobs)
  \right) \stackrel{(ii)}{\leq} \frac{\Dim}{K\numobs}.
\end{align*}
Here step (i) follows by using $\tuneParScaled_\numobs \leq
\tfrac{1}{\log(\numobs) \cdot \log\log(\numobs)}$, and step (ii) holds since
the sample size is assumed $\numobs$ to be sufficiently large. This
completes the proof of the claim~\eqref{eqn:gen-case-asymp-normality}.

\vspace{8pt}

\textbf{\underline{Proof of
    claim~\eqref{eqn:gen-case-simple-CI}:}} \hspace{4pt} In order to
apply~\cref{prop:gen-confinv}, it suffices to verify condition~\ref{assumption:A3prime} for $\{\Basis \x_i \}_{i = 1}^\numobs$
with the
choice of tuning parameters~\eqref{EqnTuningGreedy}. Note that in the proof
of~\eqref{eqn:gen-case-asymp-normality}, we already verified that conditions~\eqref{eqn:exploration-lower-bound}, \eqref{eqn:bounded-covariates}, and~\eqref{eqn:suff-exploration}
ensures that assumption~\ref{assumption:A3} is satisfied. Fortunately, these three conditions 
are not affected by the change of basis transformation, and are readily satisfied by the regressors $\{ \Basis \x_i \}_{i = 1}^\numobs$. 

Indeed, for any orthonormal basis matrix $\Basis$, via linearity of expectation, we have
\begin{align*}
  \Exs[\Basis \randDir_i \randDir_i^\top \Basis] = \Basis \Exs[ \randDir_i \randDir_i^\top] \Basis^\top 
  \succeq \Basis \ExsXsqLB \Basis^\top.   
\end{align*}
Moreover, for any orthonormal basis matrix $\Basis$, we have
\begin{align*}
\lambda_{\min}(\Basis \ExsXsqLB \Basis^\top) = \lambda_{\min}
(\ExsXsqLB) \qquad \text{and} \qquad \enorm{\Basis \x_i} =
\enorm{\x_i} \leq K.
\end{align*}
Thus, following a proof similar
to~\eqref{eqn:gen-case-asymp-normality} we have that the
assumption~\ref{assumption:A3prime} parts (a), (c), and
condition~\ref{assumption:A3} part (b), modified for $\{ \Basis \x_i
\}_{i = 1}^{\numobs}$, are satisfied. Finally, from the bounded
covariates condition~\eqref{eqn:bounded-covariates} we have
$\lambda_{\max}(\SigmaMat_{\dir, \numobs}) \leq K \numobs$, and as a
result, \mbox{$\tuneParScaled_\numobs \cdot
  \log(\lambda_{\max}(\SigmaMat_\numobs)) = \smalloP(1)$}. Combining
this observation with a calculation similar to
equation~\eqref{eqn:A3_to_A3prime}, we deduce that
condition~\ref{assumption:A3prime}(b) holds, thereby completing the
proof of the claim~\eqref{eqn:gen-case-simple-CI}.

%%%%%%%%%%%%%%%%%%%%%%%%%%%%%%%%%%%%%%%%%%%%%%%%%%%%%%%%%%%%%%%%%%%%%%

\section{Discussion}
\label{SecDiscussion}

In this paper, we proposed a family of online debiasing estimators for
adaptive linear regression and analyze their asymptotic properties. We
introduced an online debiasing estimator, and proved that it admits a
Gaussian limit under considerably weaker conditions than the \OLS
estimator. We highlighted its practical behavior using examples from
multi-armed bandits, time series modeling, and active learning in
which online debiasing yields asymptotic normality while \OLS does
not.  We also proved a minimax lower bound for the adaptive linear
regression model; in conjunction with our upper bounds, our results
reveal that the online debiasing estimator is minimax optimal.
 
This work opens up a number of directions for future research.  For
example, it would be interesting to characterize the non-asymptotic
behavior of estimators based on online debiasing.  Concretely, we
would like to investigate the rate of distributional convergence of
the online debiasing estimators to the appropriate Gaussian
distributions.

\subsection*{Acknowledgements}  This work was partially supported by
a BAIR-Microsoft research grant to MJW and LM, as well as DOD ONR
Office of Naval Research N00014-21-1-2842, National Science Foundation
DMS grant 2015454, and National Science Foundation CCF grant 1955450 to MJW.

%%%%%%%%%%%%%%%%%%%%%%%%%%%%%%%%%%%%%%%%%%%%%%%%%%%%%%%%%%%%%%%%%%%%%%%%%%%%%%%

\appendix

%%%%%

\section{Proofs of the propositions}

This section provides the proofs of the two propositions stated in
this paper.  ~\Cref{app:gen-confinv} is devoted to the proof
of~\Cref{prop:gen-confinv}, whereas~\Cref{app-bias-control-lemma} is
devoted to the proof of~\Cref{prop:zero-bias-lemma}.

%%%%%%%%%%%%%%%%%%%%%%%%%%%%%%%%%%%%%%%%%%%%%%%%%%%%%%%%%%%%%%%%%%%%%%%

\subsection{Proof of~\Cref{prop:gen-confinv}}
\label{app:gen-confinv}
Our proof is based on the following auxiliary result that
characterizes the asymptotic behavior of $\thetaBasis$.
\begin{lems}
\label{lemma:confinv-asn}
Under \mbox{Assumptions~\ref{assumption:A1}, \ref{assumption:A2},
  and~\ref{assumption:A3prime}}, given any consistent estimator
$\widehat{\noisesd}^2$ of $\noisesd^2$, we have
\begin{align}
\label{eqn:confinv-asn}
\sqrt{\frac{\tuneParScaled_\numobs}{ \OpTuning_\numobs^2
    \widehat{\noisesd}^2}} \cdot
\diagCov_\numobs^{\frac{1}{2}}(\thetaBasis - \Basis \thetastar)
\indistrb \Ncal(0, \MyIdDim).
\end{align}
\end{lems}

We prove this claim shortly, but let us complete the proof of
Proposition~\ref{prop:gen-confinv} using
Lemma~\ref{lemma:confinv-asn}.  Now, by construction we have $\Basis
\Basis^\top = \MyIdDim$ and $\Basis^\top\e_1 = \dir$.  Using these two
properties, we can write
\begin{align*}
  \e_1^\top \Basis \thetastar = \dir^\top \thetastar  \;\;
  \text{and} \;\; \y_i = \inprod{\Basis \x_i}{ \Basis \thetastar} +
  \error_i \quad \text{for all} \;\; i = 1, \ldots, \numobs.
\end{align*}
Consequently, in this new basis, estimating the scalar $\dir^\top
\thetastar$ is same as estimating the first coordinate of transformed
vector $\Basis \thetastar$. Next, by construction of the matrix $\Basis$, we have 
\begin{align}
\label{eqn:Cov-diag-entry}
  \e_1^\top \diagCov_{\dir, \numobs}^{-1} \e_1 = \dir^\top \SigmaMat_{\numobs}^{-1} \dir
  \qquad \text{and} \qquad \OpTuning_\numobs = \opnorm{\diagCov_{\dir, \numobs}^{-\frac{1}{2}} \SigmaMat_{\dir,\numobs}^{\frac{1}{2}}}.
\end{align} 

Thus, we deduce
\begin{align*}
  \e_1^\top \diagCov_{\dir,\numobs}^{\frac{1}{2}}(\thetaBasis - \Basis \thetastar)
  & = \sqrt{(\diagCov_\numobs)_{11}} \cdot ( (\thetaBasis)_1 - \dir^\top \thetastar) \\
  & = \sqrt{\frac{1}{(\diagCov_\numobs^{-1})_{11}}} \cdot (\e_1^\top \thetaBasis -  \dir^\top \thetastar) \\
  & =  \sqrt{\frac{1}{\dir^\top \SigmaMat_\numobs^{-1} \dir} } \cdot 
  ( \e_1^\top \thetaBasis -  \dir^\top \thetastar),
\end{align*}
The first equality above follows since the first row of the matrix $\diagCov_{\dir, \numobs}$ is proportional 
to $\e_1$ by construction and the fact that $\e_1 \Basis \thetastar = \dir^\top \thetastar$. 
The last line follows from the relation~\eqref{eqn:Cov-diag-entry}. Thus, from  property~\eqref{eqn:confinv-asn} 
 we deduce   
\begin{align}
\label{eqn:one-dim-asn}
  \sqrt{\frac{\tuneParScaled_\numobs}{\OpTuning_\numobs^2 \noisesd^2}} \cdot \sqrt{\frac{1}{\dir^\top \SigmaMat_\numobs^{-1} \dir} } \cdot 
  ( \e_1^\top \thetaBasis -  \dir^\top \thetastar) \indistrb \Ncal(0, 1). 
\end{align}
Define, the set $\CIset_{\dir, 1 - \alpha} \subseteq \real$ as
% \notate{why does the set refer to the OD estimator instead of v, diagOD?  also, should some of the e1's in the final expression be vs?}
%
\begin{align*}
  \CIset_{\dir, 1 - \alpha}  &\mydefn \left\{ \theta \in \real \mid  -\NormalQuantile{1 - \alpha/2} \leq \sqrt{\frac{\tuneParScaled_\numobs}{ \OpTuning_\numobs^2 \sigmahat^2}} \cdot \sqrt{\frac{1}{\dir^\top \SigmaMat_\numobs^{-1} \dir} } \cdot 
  ( \e_1^\top \thetaBasis -  \theta) \leq \NormalQuantile{1 - \alpha/2} \right\} \\
  &\equiv \Big[ 
   \e_1^\top \thetaBasis -
      \tfrac{\OpTuning_\numobs \cdot \sigmahat}{\sqrt{\tuneParScaled_\numobs}}
      (\myinprod{\e_1}{\SigmaMat_\numobs^{-1} \e_1})^{\frac{1}{2}}
      \NormalQuantile{1 - \alpha/2},
       \quad 
      \e_1^\top \thetaBasis +
      \tfrac{\OpTuning_\numobs \cdot \sigmahat}{\sqrt{\tuneParScaled_\numobs}}
      (\myinprod{\e_1}{\SigmaMat_\numobs^{-1} \e_1})^{\frac{1}{2}}
      \NormalQuantile{1 - \alpha/2} \Big]
\end{align*}
where, $\NormalQuantile{1 - \alpha/2}$ is the $1 - \alpha/2$ quantile of
the standard Gaussian random variable.
From the equation~\eqref{eqn:one-dim-asn} we have $ \lim_{\numobs
  \rightarrow \infty} \Prob(\dir^\top \thetastar \in \CIset_{\dir, 1 -
  \alpha}) = 1 - \alpha$, i.e., $\CIset_{\dir, 1 - \alpha}$ is an
asymptotically exact $1 - \alpha$ confidence intervals for $\dir^\top
\thetastar$. This completes the proof of the Proposition~\ref{prop:gen-confinv}.
It remains to prove Lemma~\ref{lemma:confinv-asn}. 

\subsection*{Proof of Lemma~\ref{lemma:confinv-asn}}
Observe that 
\begin{align*}
    \y_i = \inprod{\Basis \x_i}{ \Basis \thetastar} +
  \error_i \quad \text{for all} \;\; i = 1, \ldots, \numobs.
\end{align*}
The proof of the Lemma~\ref{lemma:confinv-asn}
is similar to the proof of Theorem~\ref{thm:asymp-normality} but modified 
for the data $\{ \Basis \x_i, \y_i \}_{i = 1}^\numobs$ and with $\thetastar$
replaced by $\Basis \thetastar$. Without loss of generality, we
assume that $\noisesd$ is known; thus, it suffices to 
prove \mbox{$\frac{\sqrt{\tuneParScaled_\numobs}}{\OpTuning_\numobs} \cdot
\diagCov_{\dir, \numobs}^{\frac{1}{2}}(\thetaBasis - \thetastar) \indistrb
\Ncal(0, \noisesd^2 \MyIdDim)$.}
Recalling the expression for $\thetaBasis$ from the definition~\eqref{eqn:thetaBasis-defn} we have
\begin{align*}
  \frac{\sqrt{\tuneParScaled_\numobs}}{\OpTuning_\numobs} \cdot
  \diagCov_{\dir,\numobs}^{\frac{1}{2}}(\thetaBasis - \Basis\thetastar) &=
    \sqrt{\tuneParScaled_\numobs} \cdot \sum_{i = 1}^\numobs \w_i \error_i\\
   &  \quad \quad  + 
     \sqrt{\tuneParScaled_\numobs} \cdot
   \left( \tfrac{1}{\OpTuning_\numobs} \cdot \diagCov_{\dir,\numobs}^{\frac{1}{2}} \SigmaMat_{\dir,\numobs}^{-\frac{1}{2}} - 
   \W_\numobs \X_{\dir,\numobs} \SigmaMat_{\dir,\numobs}^{-\frac{1}{2}}  \right)  \SigmaMat_{\dir,\numobs}^{\frac{1}{2}} (\thetaBasisLS - \Basis\thetastar) \\
  &= \ZeroMartin_\numobs + \Bias_\numobs
\end{align*}
It remains to prove $\Bias_\numobs \stackrel{\smallop{p}}{\longrightarrow} 0$ and $\ZeroMartin_\numobs \indist \Ncal(0, \noisesd^2 \MyIdDim)$. Observe that
\begin{align*}
  \enorm{\Bias_\numobs} &\leq \sqrt{\tuneParScaled_\numobs} \cdot \opnorm{ \tfrac{1}{\OpTuning_\numobs} \cdot  \diagCov_{\dir, \numobs}^{\frac{1}{2}} \SigmaMat_{\dir,\numobs}^{-\frac{1}{2}}- 
    \W_\numobs \X_{\dir,\numobs} \SigmaMat_{\dir,\numobs}^{-\frac{1}{2}}} \cdot 
  \enorm{\SigmaMat_{\dir, \numobs}^{\frac{1}{2}}(\thetaBasisLS - \Basis\thetastar)} \\
  &\leq \sqrt{\tuneParScaled_\numobs \log
    \lambda_{\max}(\SigmaMat_{\dir,\numobs})} \cdot \opnorm{ \tfrac{1}{\OpTuning_\numobs} \cdot \diagCov_\numobs^{\frac{1}{2}} \SigmaMat_{\dir,\numobs}^{-\frac{1}{2}}- 
    \W_\numobs \X_{\dir,\numobs} \SigmaMat_{\dir,\numobs}^{-\frac{1}{2}}}  \\
    &\stackrel{\smallop{p}}{\longrightarrow} 0,
\end{align*}
where, the second inequality uses Theorem 1 from the paper~\cite{lai1982least}, and the last step uses the 
vanishing bias condition~\ref{assumption:A3prime}(b).

The analysis of the martingale term $\ZeroMartin_\numobs$ is exactly same as that of Theorem~\ref{thm:asymp-normality} proof. This completes the proof of the Lemma~\ref{lemma:confinv-asn}.

\subsection{Proof of Proposition~\ref{prop:zero-bias-lemma}} 
\label{app-bias-control-lemma}

Recalling that $\maxnorm{\mymatrix{M}} \mydefn \max_{i,j}
|\mymatrix{M}_{ij}|$ denotes the maximum absolute entry of a matrix,
we claim that it suffices to show that $\maxnorm{\W_\numobs \X_\numobs
  \XscaleMat_\numobs^{-\frac{1}{2}}} \leq 4$. Indeed, when this claim
holds, we have
\begin{align*}
\opnorm{\Id - \W_\numobs \X_\numobs \SigmaMat_\numobs^{-\frac{1}{2}}}
& \leq 1 + \opnorm{ \W_\numobs \X_\numobs
  \XscaleMat_\numobs^{-\frac{1}{2}} } \opnorm{
  \XscaleMat_\numobs^{\frac{1}{2}} \SigmaMat_\numobs^{-\frac{1}{2}}} \\
& \leq 1 + \BigoP(\sqrt{\Dim}) \opnorm{ \W_\numobs \X_\numobs
  \XscaleMat_\numobs^{-\frac{1}{2}} } \\
& \leq 1 + \BigoP(\Dim^2) \maxnorm{\W_\numobs \X_\numobs
  \XscaleMat_\numobs^{-\frac{1}{2}} }
\end{align*}
The second last inequality above follows by noting that the diagonal
entries of the matrix $\XscaleMat_\numobs^{\frac{1}{2}}
\SigmaMat_\numobs^{-1} \XscaleMat_\numobs^{\frac{1}{2}}$ is of the
order $\BigoP(1)$; this bound uses the expression of the scaling
matrix $\XscaleMat_\numobs$ from the
definition~\eqref{eqn:scaling_mat-choice}, and the operator-norm bound
$\| \SigmaMatLB_\numobs^{\frac{1}{2}} \diag(\SigmaMat_\numobs^{-1})
\SigmaMatLB_\numobs^{\frac{1}{2}} \|_{op} = \BigoP(1)$ from
assumption~\eqref{eqn:SigmaLB-conditions}.  The last inequality above
follows from the fact that $\opnorm{ \mymatrix{A} } \leq
\Dim^{\frac{3}{2}} \maxnorm{ \mymatrix{A} }$, for any
$\Dim$-dimensional matrix $\mymatrix{A}$. This completes the proof of
Proposition~\ref{prop:zero-bias-lemma}. The remainder of the proof is
devoted to establishing an upper-bound on the max-norm of the matrix
$\W_\numobs \X_\numobs \XscaleMat_\numobs^{-\frac{1}{2}}$. We do so by
proving the following upper bounds
\begin{subequations}
  \begin{align}
    \label{eqn:first-term-max-norm-bound}    
    \maxnorm{ \sum_{i =1 }^k \w_i \x_i^\top \XscaleMat_i^{-\frac{1}{2}}
    } & \leq 2 \quad \text{for} \;\; k = 1, \ldots \numobs, \qquad
    \text{and} \\
    \label{eqn:second-term-max-norm-bound}    
    \maxnorm{ \Nsum{i} \w_i \x_i^\top \XscaleMat_i^{-\frac{1}{2}} (\Id
      - \XscaleMat_i^{\frac{1}{2}} \XscaleMat_\numobs^{-\frac{1}{2}} ) }
    & \leq 2.
  \end{align}
\end{subequations}
\noindent
Note that a combination of these two bounds implies that \mbox{$
  \maxnorm{ \W_\numobs \X_\numobs \XscaleMat_\numobs^{-\frac{1}{2}}}
  \leq 4$}. \\

\noindent Accordingly, the remainder of our proof is devoted to
establishing the bounds~\eqref{eqn:first-term-max-norm-bound}
and~\eqref{eqn:second-term-max-norm-bound}.

%%%%%%%%%%%%%%%%%%%%%%%%%%%%%%%%%%%%%%%%%%%%%%%%%%%%%%%%%%%%%%%%%%%%%%%%%%%%%%%%%

\subsection*{\textbf{Proof of bound~\eqref{eqn:first-term-max-norm-bound}}}

Using the expression for the weight vector $\w_i$ from
equation~\eqref{eqn:Wn-expression}, we have
\begin{align*}
  \Id - \sum_{i=1}^k \w_i \x_i \XscaleMat_i^{-\frac{1}{2}} = \prod_{i =
    1}^k \left( \Id - \frac{\XscaleMat_i^{-\frac{1}{2}} \x_i \x_i^\top
    \XscaleMat_i^{-\frac{1}{2}} }{\tfrac{\tuneParScaled_\numobs}{2} +
    \|\XscaleMat_i^{-\frac{1}{2}} \x_i\|^2} \right).
\end{align*}

We claim
\begin{align*}
\left\| \Id - \frac{\XscaleMat_i^{-\frac{1}{2}} \x_i \x_i^\top
    \XscaleMat_i^{-\frac{1}{2}} }{\tfrac{\tuneParScaled_\numobs}{2} +
    \|\XscaleMat_i^{-\frac{1}{2}} \x_i\|^2} \right\|_{\mathrm{op}} \leq 1 \quad \text { for all } i \in[n] .    
\end{align*}
It suffices to show that the maximum absolute eigenvalue of the symmetric matrix $\Id - \frac{\XscaleMat_i^{-\frac{1}{2}} \x_i \x_i^\top
    \XscaleMat_i^{-\frac{1}{2}} }{\tfrac{\tuneParScaled_\numobs}{2} +
    \|\XscaleMat_i^{-\frac{1}{2}} \x_i\|^2}$ is upper bounded by 1. Indeed, for any $u \in \real^d$ with $\|u\|_2 = 1$
\begin{align*}
 0 \leq    u^\top \left( \Id - \frac{\XscaleMat_i^{-\frac{1}{2}} \x_i \x_i^\top
    \XscaleMat_i^{-\frac{1}{2}} }{\tfrac{\tuneParScaled_\numobs}{2} +
    \|\XscaleMat_i^{-\frac{1}{2}} \x_i\|^2}  \right) u = 1 - \frac{(\x_{i}^{\top} \XscaleMat_{i}^{-\frac{1}{2}} u)^2}{\gamma_{n} / 2+\left\|\XscaleMat_{i}^{-\frac{1}{2}} \x_{i}\right\|_{2}^{2}} \leq 1
\end{align*}
The above conclusions follow from the fact that for $\|u\|_2 \leq 1$
\begin{align*}
    |\x_{i}^{\top} \XscaleMat_{i}^{-\frac{1}{2}} u| \leq
    \left\|\XscaleMat_{i}^{-\frac{1}{2}} \x_{i}\right\|_{2}
\end{align*}
Thus we conclude that for all $k \in[n]$ we have the bound
\begin{align*}
\left\|\sum_{i=1}^{k} \w_{i} \x_{i} \XscaleMat_{i}^{-\frac{1}{2}}\right\|_{\max } \leq\left\|\sum_{i=1}^{k} \w_{i} \x_{i} \XscaleMat_{i}^{-\frac{1}{2}}\right\|_{\mathrm{op}} \leq 2,    
\end{align*}
where, in the last derivation we used the fact that the max-norm of a matrix is upper bounded by the operator norm of that matrix. This completes the proof of the bound~\eqref{eqn:first-term-max-norm-bound}.

% \newpage 
% Invoking the lower bound
% assumption~\eqref{eqn:tuneparScaled-condition} and doing simple
% algebra, we find that
% \begin{align*}
% \bigopnorm { \Id - \tfrac{\XscaleMat_i^{-\frac{1}{2}} \x_i \x_i^\top
%     \XscaleMat_i^{-\frac{1}{2}} }{\tuneParScaled_\numobs/2 +
%     \|\XscaleMat_i^{-\frac{1}{2}} \x_i\|_2^2} } \leq 1 \quad \mbox{for
%   all $i \in [\numobs]$.}
% \end{align*}
% Consequently, for all $k \in [\numobs]$ we
% have the bound
% \begin{align}
% \label{eqn:max-norm-bound}
% \maxnorm{\sum_{i=1}^k \w_i \x_i \XscaleMat_i^{-\frac{1}{2}} } \leq
% \opnorm{ \sum_{i =1}^k \w_i \x_i \XscaleMat_i^{-\frac{1}{2}}} \leq 2,
% \end{align}
% where, in the last derivation we used the fact that the max-norm of a matrix 
% is upper bounded by the operator norm of that matrix. 
% This completes the proof of the bound~\eqref{eqn:first-term-max-norm-bound}.
% %

\subsection*{\textbf{Proof of bound~\eqref{eqn:second-term-max-norm-bound}}}

The proof is this bound exploits the following auxiliary lemma:
\begin{lems}
  \label{lem:real-sum-product-bound}
Consider a non-increasing sequence of nonnegative real numbers
$\Nseq{\delta}{i}$ and a sequence of real numbers $\Nseq{a}{i}$ for
which there exists a constant $\MaxVal$ such that $\max_{k \in
  [\numobs]} |\sum_{i=1}^k  a_i| \leq \MaxVal$.  Then we have
  \begin{align}
    |\Nsum{i} a_i \delta_i| \leq \MaxVal \delta_1.
  \end{align}
\end{lems}
\noindent 
We prove this lemma at the end of this subsection. \\

Taking \cref{lem:real-sum-product-bound} as given, let us prove the
bound~\eqref{eqn:second-term-max-norm-bound}.  The
bounds~\eqref{eqn:first-term-max-norm-bound} guarantee that
\begin{align*}
\maxnorm{\sum_{i =1}^k \w_i \x_i^\top \XscaleMat_i^{-\frac{1}{2}}} &
\leq 2 \quad \mbox{for each $k \in [\numobs]$.}
\end{align*}
Moreover, by construction, the diagonal entries of the matrix $(\Id -
\XscaleMat_i^{-\frac{1}{2}} \XscaleMat_\numobs^{\frac{1}{2}})$, for $i =
1, \ldots, \numobs$, are positive and non-increasing.  Thus, we can
apply Lemma~\ref{lem:real-sum-product-bound} with the sequence
$\Nseq{a}{i}$ as the entries of the matrix $\w_i \x_i^\top
\XscaleMat_i^{-\frac{1}{2}}$ and $\delta_i$ as the diagonal entries of
the (diagonal) matrix $\Id - \XscaleMat_i^{\frac{1}{2}}
\XscaleMat_\numobs^{-\frac{1}{2}}$.  Invoking
Lemma~\ref{lem:real-sum-product-bound} yields
\begin{align*}
\maxnorm{ \Nsum{i} \w_i \x_i^\top \XscaleMat_i^{-\frac{1}{2}} (\Id -
  \XscaleMat_i^{\frac{1}{2}} \XscaleMat_\numobs^{-\frac{1}{2}} ) } \leq
2 \cdot \maxnorm{ \Id - \XscaleMat_1^{\frac{1}{2}}
  \XscaleMat_\numobs^{-\frac{1}{2}} } \leq 2,
\end{align*} 
where, the last inequality above uses the property that the diagonal
matrices $\XscaleMat_1$ and $\XscaleMat_\numobs$, by construction,
satisfy a positive semidefinite ordering $\XscaleMat_1 \preceq
\XscaleMat_\numobs$.  This concludes the proof of
bound~\eqref{eqn:second-term-max-norm-bound}. \\

\noindent It remains to prove the
Lemma~\ref{lem:real-sum-product-bound}.

\subsection*{\textbf{Proof of Lemma~\ref{lem:real-sum-product-bound}}}
\label{sec:Proof-real-sum-prod-lemma}

Let $s_k \mydefn \sum_{i =1}^k a_i$ denote the $k^{th}$ partial sum of
the sequence $\Nseq{a}{i}$.  The sum $\Nsum{i} a_i \delta_i$ can be
represented in terms of these partial sums as
\begin{align*}
| \Nsum{i} a_i \delta_i | & = | \sum_{i =1}^{\numobs-1}q
(\delta_{\numobs - i} - \delta_{\numobs - i + 1}) s_{\numobs - i} +
\delta_\numobs s_\numobs | \\
& \stackrel{(i)}{\leq} \MaxVal \cdot [ \delta_\numobs + \sum_{i
    =1}^{\numobs-1} (\delta_{\numobs - i} - \delta_{\numobs - i + 1})
] \\
& = \MaxVal \delta_1,
\end{align*}
where inequality (i) uses the bound $|s_{\numobs - i}| \leq \MaxVal$
and the ordering \mbox{$\delta_{\numobs - i} \geq \delta_{\numobs - i
    + 1}$}.  This completes the proof of
Lemma~\ref{lem:real-sum-product-bound}.

%%%%%%%%%%%%%%%%%%%%%%%%%%%%%%%%%%%%%%%%%%%%%%%%%%%%%%%%%%%%%%%%%%%%%%%%%%

\subsection*{\textbf{Proof of bound~\eqref{eqn:bias-bound-improved}}}

Note that in the setting of multi-armed bandits (cf.
Section~\ref{sec:bandits}), the covariance matrix $\SigmaMat_\numobs$
is diagonal, and consequently, the
definition~\eqref{eqn:scaling_mat-choice} simplifies to $\XscaleMat_i =
\max\{ \SigmaMat_i, \SigmaMatLB_\numobs\}$. Moreover, a simple
argument, using the method of induction on the integer index $i$,
reveals that the matrix $\W_i \Z_i$ is a diagonal matrix with
nonnegative entries.  In particular, we have $\maxnorm{ \W_\numobs
  \Z_\numobs } = \opnorm{ \W_\numobs \Z_\numobs}$. By combining these
facts, we see that
\begin{align}
  \opnorm{ \Id - \W_\numobs \X_\numobs
    \SigmaMat_\numobs^{-\frac{1}{2}} } & \stackrel{(i)}{\leq} 1 +
  \opnorm{ \W_\numobs \X_\numobs \XscaleMat_\numobs^{-\frac{1}{2}} }
  \cdot \opnorm{ \max\{ \MyIdDim, \; \SigmaMatLB_\numobs^{\frac{1}{2}}
    \SigmaMat_\numobs^{-1} \SigmaMatLB_\numobs^{\frac{1}{2}} \} } \notag \\
  \label{EqnHamster}
& \stackrel{(ii)}{\leq} 1 + \maxnorm { \W_\numobs \X_\numobs
      \XscaleMat_\numobs^{-\frac{1}{2}} } \cdot \BigoP(1),
\end{align}
where step (i) uses the fact that in multi-armed bandit problems the
covariance matrix $\SigmaMat_\numobs$ is diagonal, and the matrix
takes the form $\XscaleMat_\numobs = \max\{ \SigmaMat_\numobs,
\SigmaMatLB_\numobs \}$; and step (ii) follows from
assumption~\eqref{eqn:SigmaLB-conditions} on the matrix
$\SigmaMatLB_\numobs$ and the fact that max-norm equals the operator
norm for diagonal matrices.

By combining the bounds~\eqref{eqn:first-term-max-norm-bound}
and~\eqref{eqn:second-term-max-norm-bound}, we see that
\mbox{$\maxnorm{ \W_\numobs \X_\numobs
    \XscaleMat_\numobs^{-\frac{1}{2}} } \leq 4$.}  Combining this bound
with inequality~\eqref{EqnHamster} yields
\begin{align*}\opnorm{ \Id -
    \W_\numobs \X_\numobs \SigmaMat_\numobs^{-\frac{1}{2}} } =
  \BigoP(1),
\end{align*}
as claimed in the bound~\eqref{eqn:bias-bound-improved}.

%%%%%%%%%%%%%%%%%%%%%%%%%%%%%%%%%%%%%%%%%%%%%%%%%%%%%%%%%%%%%%%%%%%%%%%%%%%%%%%%%%%%%

\section{Proof of stability Lemma~\ref{lem:stability-lemma}}
\label{AppProofStabilityLemma}

For notational convenience, we use the shorthand notation
\begin{align*}
\z_i \mydefn \x_i \XscaleMat_i^{-\frac{1}{2}}, \quad \Z_i^\top \mydefn
  [\z_1, \ldots, \z_i], \quad \mbox{and} \quad \W_i = [\w_1, \ldots,
    \w_i],
\end{align*}
as previously introduced in Section~\ref{sec:W-decorr}.

\subsection*{\textbf{Verifying the stability condition}}

The proof of the stability condition is based on a recursion relation
that connects the terms $\DelMat_{i} \mydefn \Id-\W_{i} \Z_{i}$ and
\mbox{$\DelMat_{i - 1} \mydefn \Id-\W_{i - 1} \Z_{i - 1}$}.
Substituting the expression~\eqref{eqn:Wn-expression} for the vector
$\w_i$ yields
\begin{align}
\DelMat_i \DelMat_i^\top & = (\DelMat_{i - 1} - \w_i \z_i^\top)
(\DelMat_{i - 1} - \w_i \z_i^\top)^\top \notag \\
& = \DelMat_{i - 1} \DelMat_{i - 1}^\top - \DelMat_{i - 1}(\w_i
\z_i^\top)^\top - \w_i \z_i^\top \DelMat_{i - 1}^\top + \w_i
\myinprod{\z_i} {\z_i} \w_i^\top \notag \\
\label{eqn:last-line}
& = \DelMat_{i - 1} \DelMat_{i - 1}^\top - (\tuneParScaled_\numobs +
\|\z_i\|^2)\w_i \w_i^\top,
\end{align}
Summing the last recursion from $i = 1$
to $i = \numobs$ and using the initial condition $\W_0 = 0$ yields
\begin{align}
  \label{eqn:Id-minus-qqT-expression}  
\MyIdDim - \sum _{i = 1}^{\numobs} \tuneParScaled_\numobs \w_i
\w_i^\top = \underbrace{\sum _{i = 1}^{\numobs} \| \z_i \|_2^2 \w_i
  \w_i^\top}_{\A_\numobs} + \underbrace{(\MyIdDim - \W_\numobs
  \Z_\numobs)(\MyIdDim - \W_\numobs \Z_\numobs)^\top}_{\B_\numobs}.
\end{align}
Equipped with the last relation, it suffices to verify $\opnorm{
  \A_\numobs } \stackrel{\smallop{p}}{\longrightarrow} 0$ and $ \opnorm{\B_\numobs} \stackrel{\smallop{p}}{\longrightarrow} 0$.  We
begin by observing that
\begin{align*}
  \Prob[\opnorm{\A_\numobs} > \epsilon] & \leq \Prob[\tr(\A_\numobs) >
    \epsilon ] \\
& \leq \Prob \left[\frac{\max_{i \in [\numobs]}
  \|\z_i\|_2^2}{\tuneParScaled_\numobs} \Nsum{i}
    \tuneParScaled_\numobs \| \w_i \|_2^2 > \epsilon \right].
\end{align*}
Now from ~\cref{eqn:Id-minus-qqT-expression}, we have the upper bound
\begin{align*}
  \sum_{i = 1}^\numobs \tuneParScaled_\numobs \| \w_i\|_2^2 = \Dim -
\tr(\A_\numobs) - \tr(\B_\numobs) \leq \Dim.
\end{align*}
Thus, we have
\begin{align*}
    \Prob[\opnorm{\A_\numobs} > \epsilon] & \leq \Prob \left[
      \frac{\max_{i \in [\numobs]}
        \|\z_i\|_2^2}{\tuneParScaled_\numobs} > \frac{\epsilon}{\Dim}
      \right].
\end{align*}
Combined with the asymptotic negligibility assumption
in~\ref{assumption:A3}, this bound implies that $\opnorm{ \A_\numobs}
\stackrel{\smallop{p}}{\longrightarrow} 0$, as desired.
j

On the other hand, using the operator-norm bound on the matrix
\mbox{$\MyIdDim - \W_\numobs \Z_\numobs$} from the variance stability
condition in~\ref{assumption:A3}, we have
\begin{align*}
\opnorm{\B_\numobs} & = \opnorm{\MyIdDim - \W_\numobs \Z_\numobs }^2
\stackrel{\smallop{p}}{\longrightarrow} 0.
\end{align*}
Putting together the pieces we conclude $\tuneParScaled_\numobs
\Nsum{i} \w_i \w_i^\top \stackrel{\smallop{p}}{\longrightarrow} \MyIdDim$ as claimed.

\subsection*{\textbf{Verifying the vanishing norm condition}}

Using the expression~\eqref{eqn:Wn-expression} for the weight vector
$\w_i$, we find that
\begin{align*}
  \|\w_i\|^2_2 \leq \frac{1}{(\tuneParScaled_\numobs/2 +
    \|\z_i\|_2^2)^2} \cdot \opnorm{\MyIdDim-\W_{i-1} \Z_{i - 1} }^2 \,
  \|\z_i\|_2^2.
\end{align*}
Doing a calculation similar to the derivation~\eqref{eqn:last-line} we
have \mbox{$ \opnorm{\parDelta_{i} }^2 \leq \opnorm{\parDelta_{0}}^2 =
  1$}. Combining the last two observations with the asymptotic
negligibility assumption \mbox{$\frac{1}{\tuneParScaled_\numobs}
  \max_{i \in [\numobs]} \|\z_i\|_2^2 \stackrel{\smallop{p}}{\longrightarrow} 0$} yields
\begin{align*}
  \tuneParScaled_\numobs \max_{i \in [\numobs]} \; \| \w_i \|_2^2 \leq
  \frac{4}{\tuneParScaled_\numobs} \cdot \max_{i \in [\numobs]} \|
  \z_{i} \|_2^2 \stackrel{\smallop{p}}{\longrightarrow} 0,
\end{align*}
as claimed. This completes the proof of
Lemma~\ref{lem:stability-lemma}.

%%%%%%%%%%%%%%%%%%%%%%%%%%%%%%%%%%%%%%%%

\section{Numerical experiment supplement}
In this section, we present the results of additional experiments complementing those in \cref{sec:applications}.

\subsection{Multi-armed bandits:}
\label{sec:bandits-simulation}
%

%%%%%%%%%%%%%%%%%% ARM 1 %%%%%%%%%%%%%%%%%%%%%%%%%%

In this section, we repeat the experiment
of~\Cref{sec:bandits-simulation} using covariates $\{ \x_i \}_{i =
  1}^\numobs$ generated by the following three bandit algorithms:
\begin{enumerate}[label=(\alph*)]
    \item the Thompson sampling
      algorithm~\citep{thompson1933likelihood}.
     \item a standard $\greedy$-greedy
       algorithm~\citep{lattimore2020bandit}.
    \item the upper confidence bound (UCB) strategy based on the
      paper~\citep{jamieson2014lil}
\end{enumerate}
As shown in Figures~\ref{fig:Bandits-all-plots-arm-1}
and~\ref{fig:Bandits-all-plots-arm-2}, online debiasing provides
appropriate coverage for all confidence levels, all bandit algorithms,
and both of the coordinates $\theta_1^*$ and $\theta_2^*$.  In
contrast, the lower tail estimates based on the \OLS estimate severely
undercover for all bandit algorithms and parameters, whereas the
$W$-decorrelation procedure undercovers for several configurations
despite having uniformly larger widths than online debiasing in all
experiments.  Finally, the concentration CIs lead to 100\% coverage
for all confidence levels, but this coverage is based on intervals
that are substantially and uniformly larger than the CIs returned by
online debiasing.

\subsection{Linear bandits}
\label{sec:LinBanditsFullSimulation}
In this section, we repeat the experiment of \cref{sec:LinBanditsSimSmall} with alternative settings of the ridge regression regularization parameter $\lambdaRidge \in \{1, 10\}$ for the concentration inequality CIs. 
%Concretely, we compare the confidence intervals based the~\OLS estimator and the online debaised estimator with a CIs based on the concentration based CIs for ridge regerssion estimator with penalty parameter $\lambdaRidge \in \{1, 10\}$. 
Recall that given a dataset $\{\x_i, \y_i \}_{i = 1}^{n}$ from the model~\eqref{eqn:linear-reg-model},
the ridge regression estimate \thetaRidge is defined as
\begin{align}
   \thetaRidge \in \arg\max_{\theta} \; \left\lbrace \sum_{i = 1}^n(\y_i - \x_i^\top \theta)^2 + 
   \lambda_{\text{ridge}}  \cdot \| \thetaVec\|_2^2 \right\rbrace
 \end{align} 
Here, $\lambda_{\text{ridge}} > 0$ is the regularization parameter for the ridge regression, and 
$\| \thetaVec\|_{2}$ denotes the $\ell_{2}$ norm of the vector $\thetaVec$. 
In Figure~\ref{fig:LinBanditsFullSim}, we observe that the concentration based CIs always 
provide appropriate coverage but are uniformly larger than the online debiasing CIs for both $\lambdaRidge = 1$ and $\lambdaRidge = 10$ and for both parameters $\theta_1^*$ and $\theta_2^*$.  

%%%%%%%%%%%%%%%%%%%%%%%  Figures %%%%%%%%%%%%%%%%%%%%

\begin{figure}[H]
\begin{subfigure}{\linewidth}
\includegraphics[width=.32\textwidth]{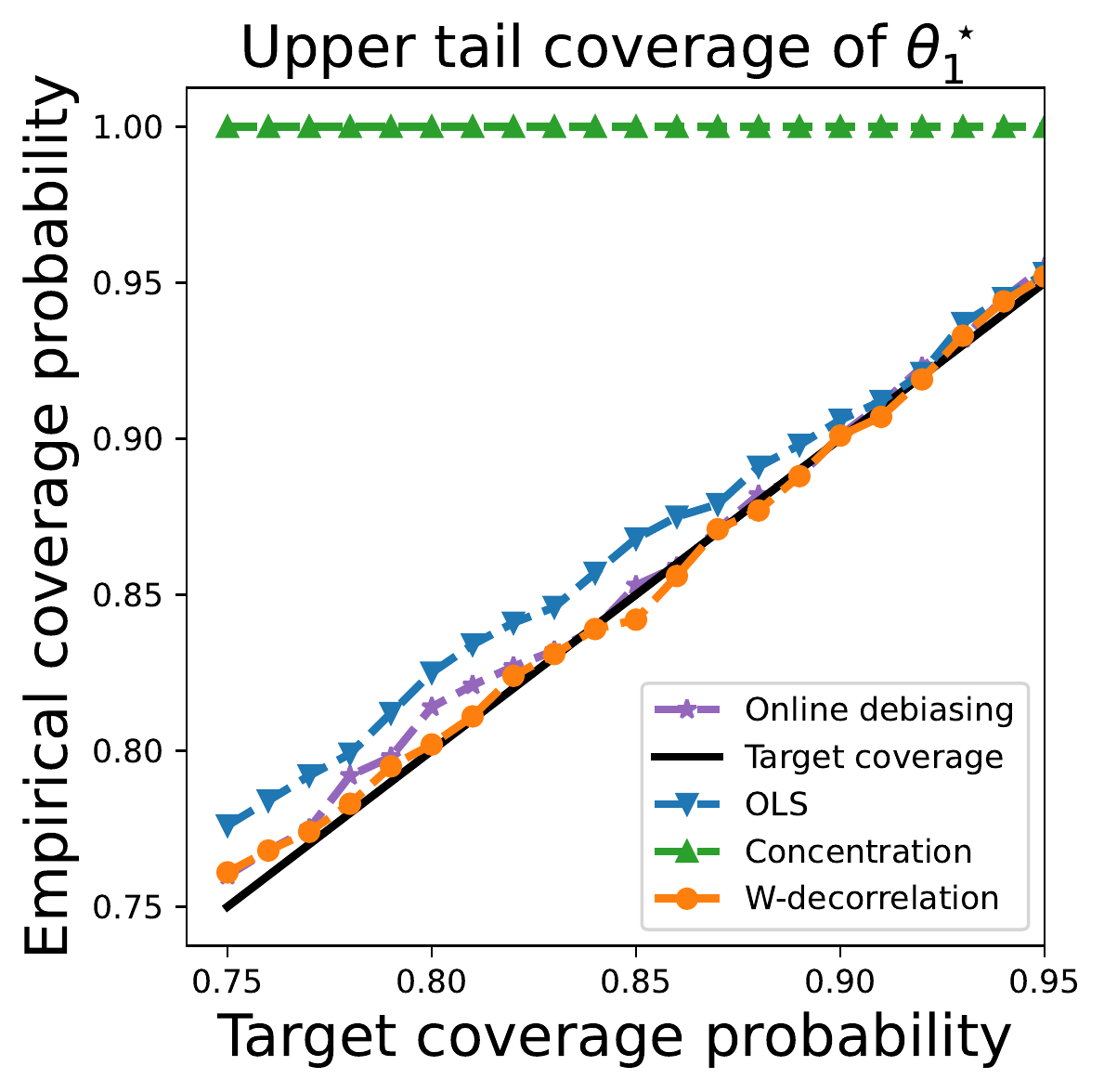}
\includegraphics[width=.32\textwidth]{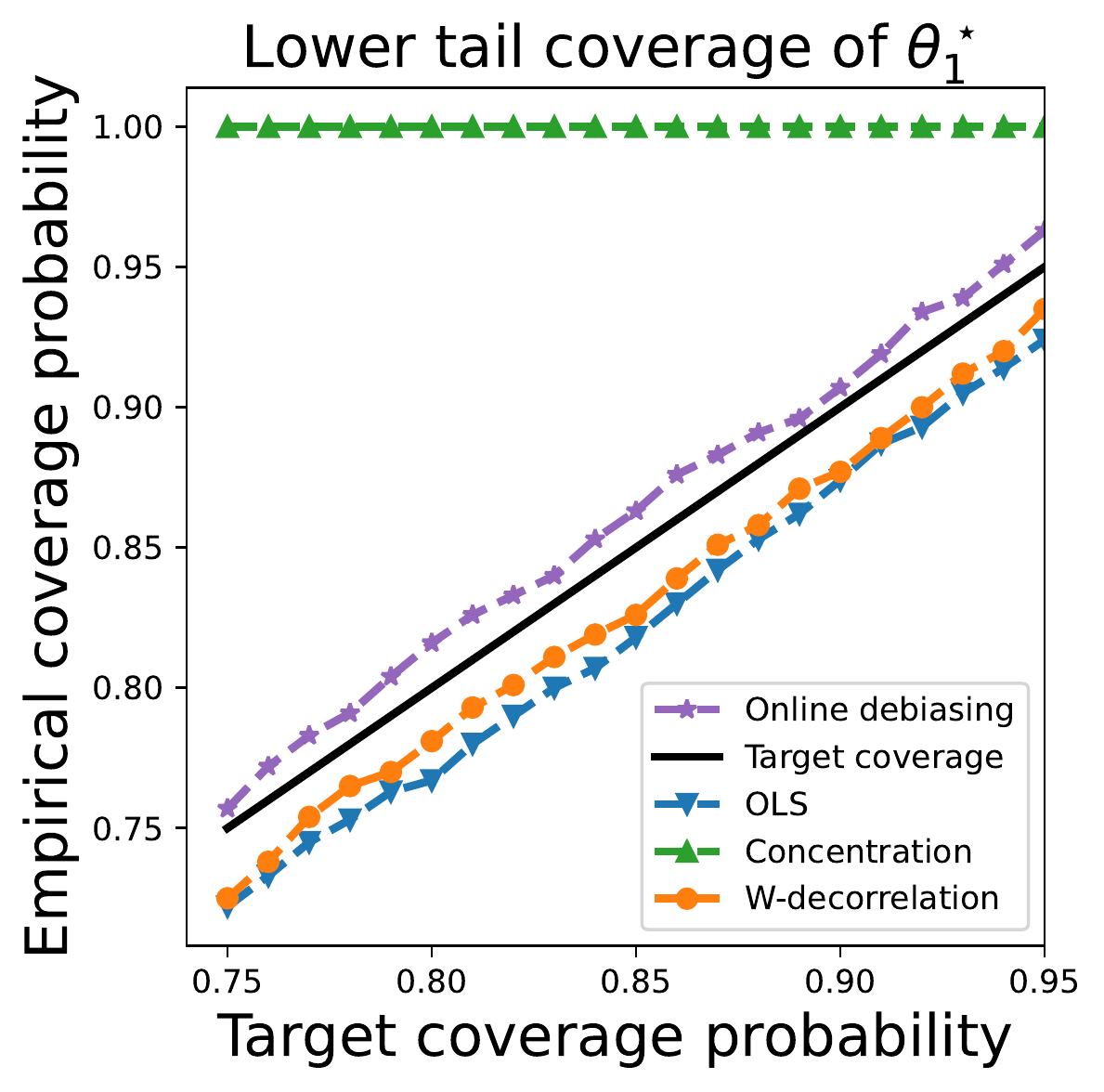}
\includegraphics[width=.32\textwidth]{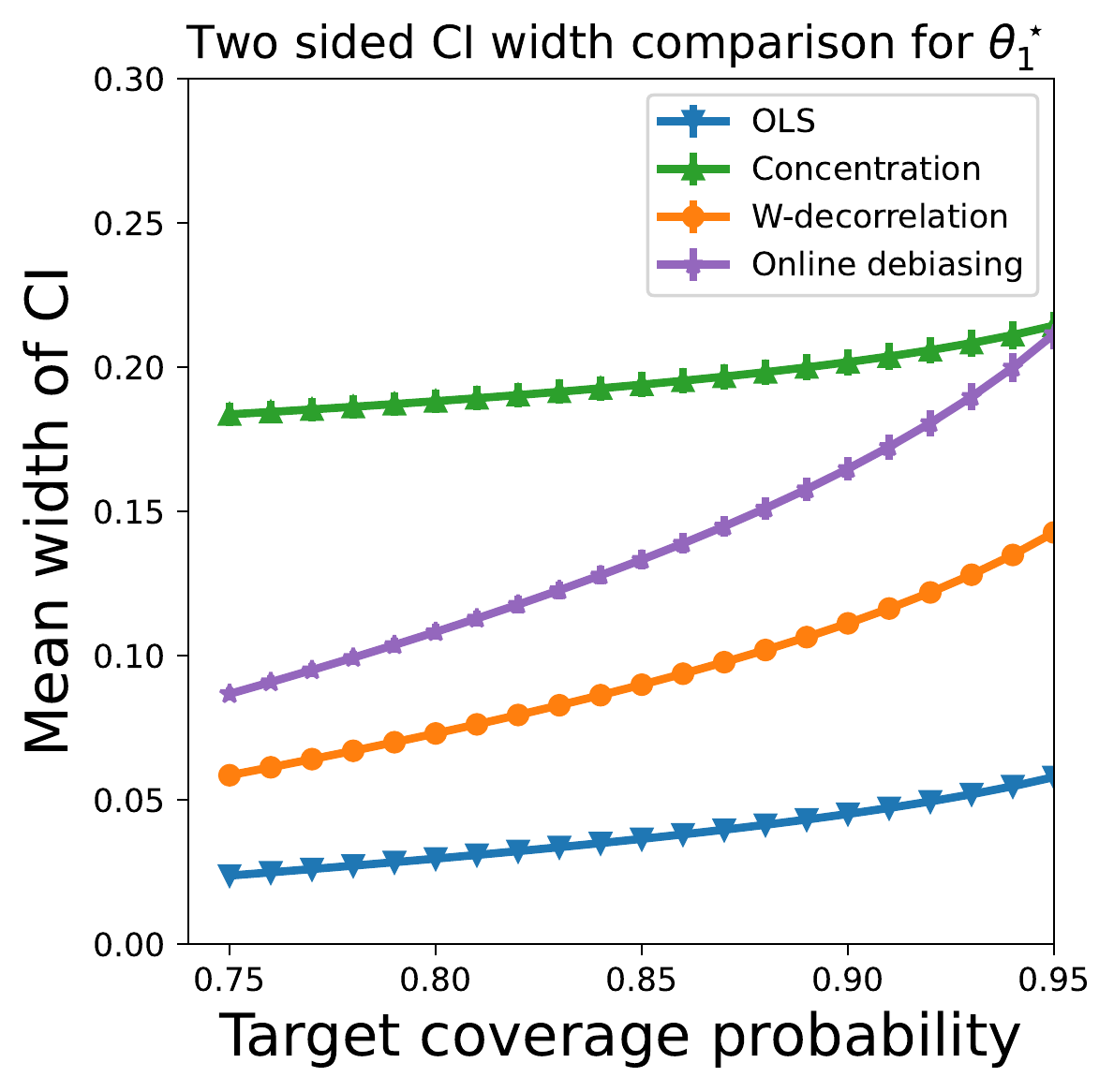}
\caption{Thompson sampling algorithm}
\end{subfigure}
\\
\begin{subfigure}{\linewidth}
\includegraphics[width=.32\textwidth]{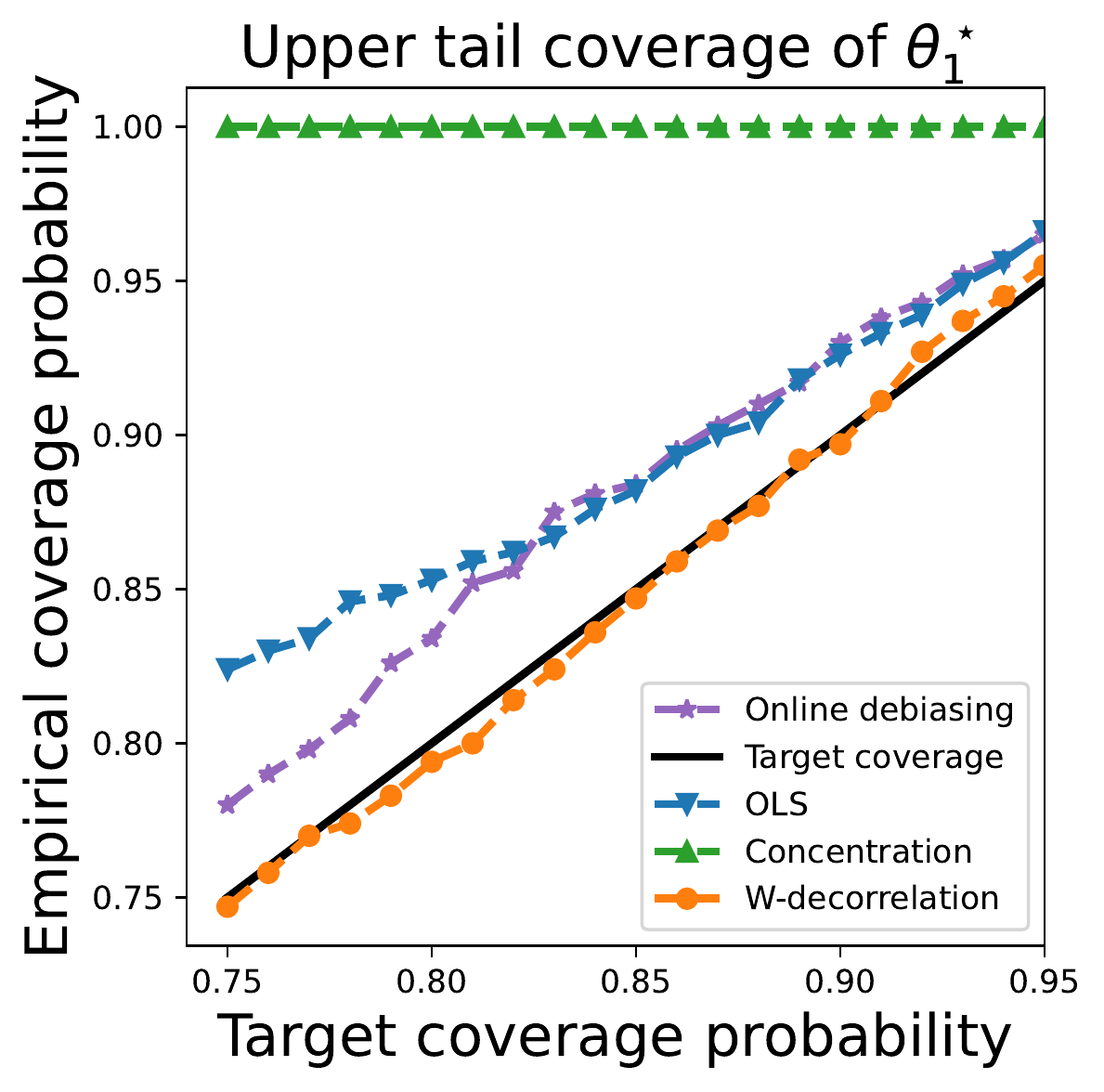}
\includegraphics[width=.32\textwidth]{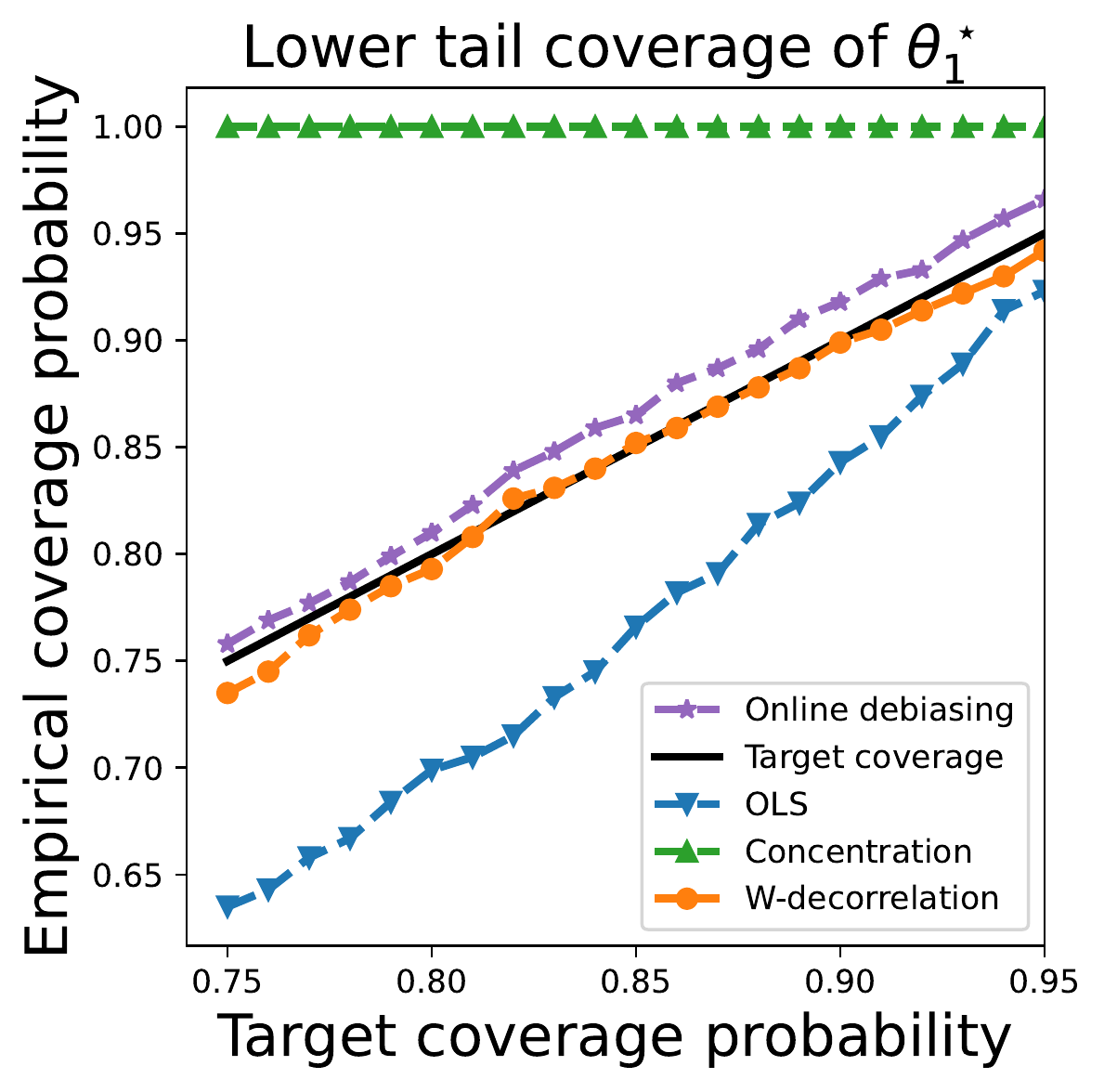}
\includegraphics[width=.32\textwidth]{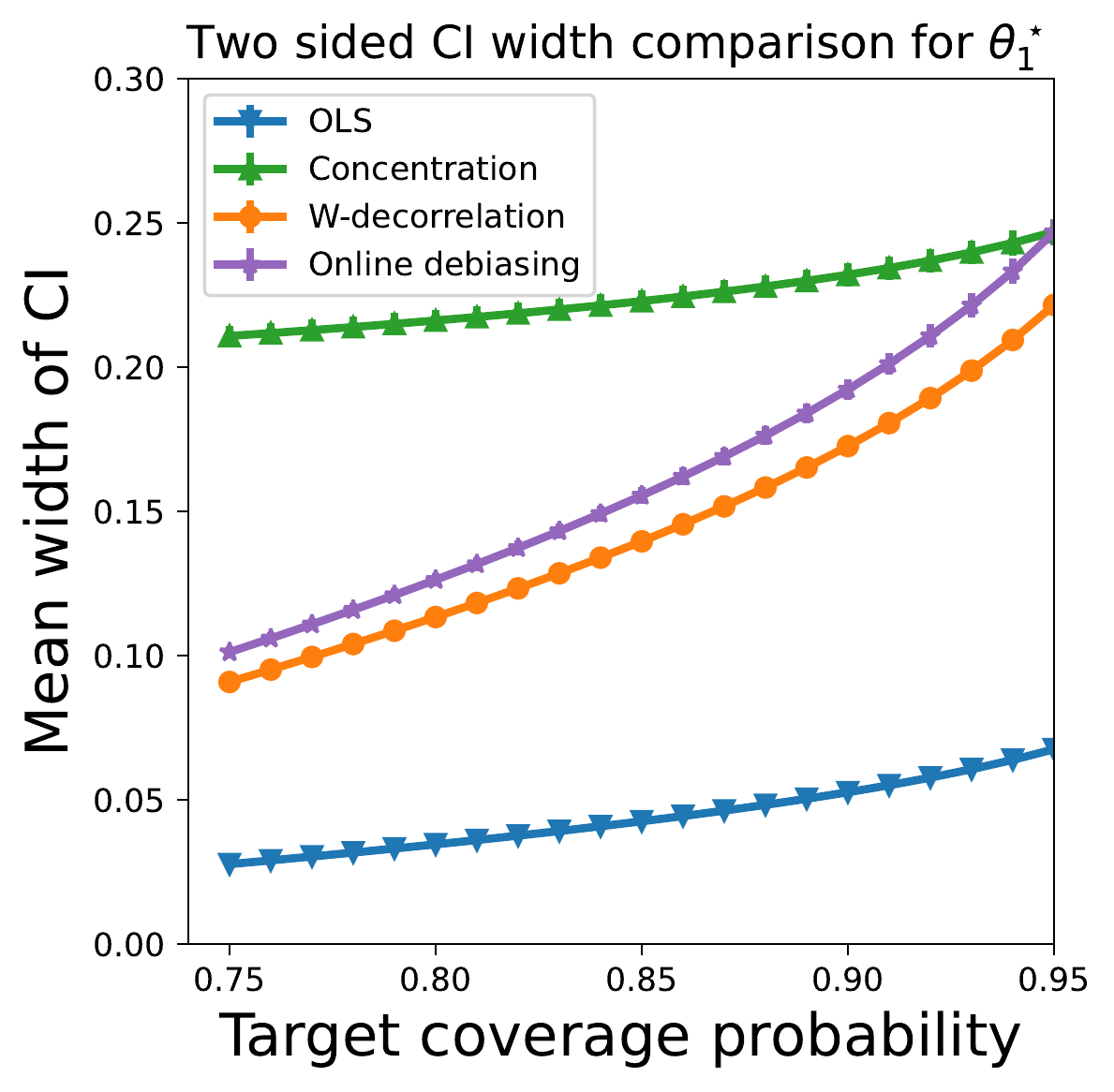}
\caption{$\greedy$-greedy algorithm}
\end{subfigure}
\begin{subfigure}{\linewidth}
\includegraphics[width=.33\textwidth]{figs/bandits_figs/theory_choice_plots/Fig_UCB_upper}
\includegraphics[width=.33\textwidth]{figs/bandits_figs/theory_choice_plots/Fig_UCB_lower}
\includegraphics[width=.35\textwidth]{figs/bandits_figs/theory_choice_plots/Fig_UCB_CI_width}
\caption{Upper confidence bound (UCB) algorithm}
\end{subfigure}
\caption{
Average coverage and width of confidence intervals for $\theta_1^*$ across 1000 independent replications  of a multi-armed bandit experiment~\eqref{model:bandits} with $\thetastar \equiv (\theta_1^*, \theta_2^*) =  (0.3, 0.3)^\top$. The covariates $\{\x_i\}_{i = 1}^{1000}$
    were selected using a) Thompson sampling~\citep{thompson1933likelihood},
    (b) the $\greedy$-greedy algorithm~\citep{lattimore2020bandit},  and (c) the upper confidence bound algorithm (UCB)~\cite{jamieson2014lil}.
    The error bars represent $\pm 1$ standard error. 
    \textbf{Left} and \textbf{Center:} Coverage of one-sided $1 - \alpha$
    intervals for $\theta_1^*$. \textbf{Right:} Width of two-sided $1 - \alpha$ intervals for $\theta_1^*$. 
    See Appendix~\ref{sec:bandits-simulation} for details.}
\label{fig:Bandits-all-plots-arm-1}
\end{figure}

%%%%%%%%%%%%%%%%%%%%%   ARM1    %%%%%%%%%%%%%%%%%%%%%

\begin{figure}[H]
\begin{subfigure}{\linewidth}
\includegraphics[width=.32\textwidth]{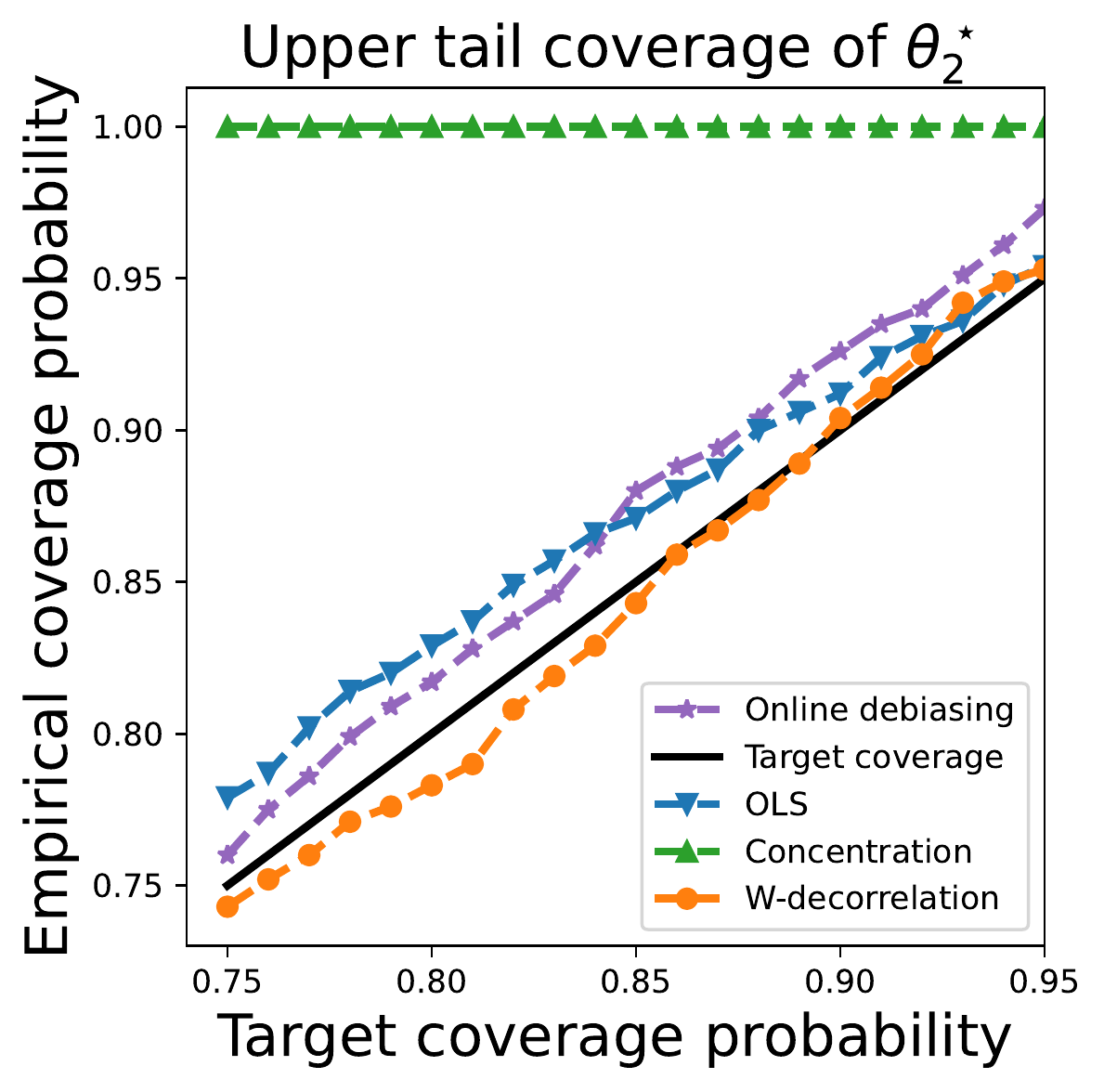}
\includegraphics[width=.32\textwidth]{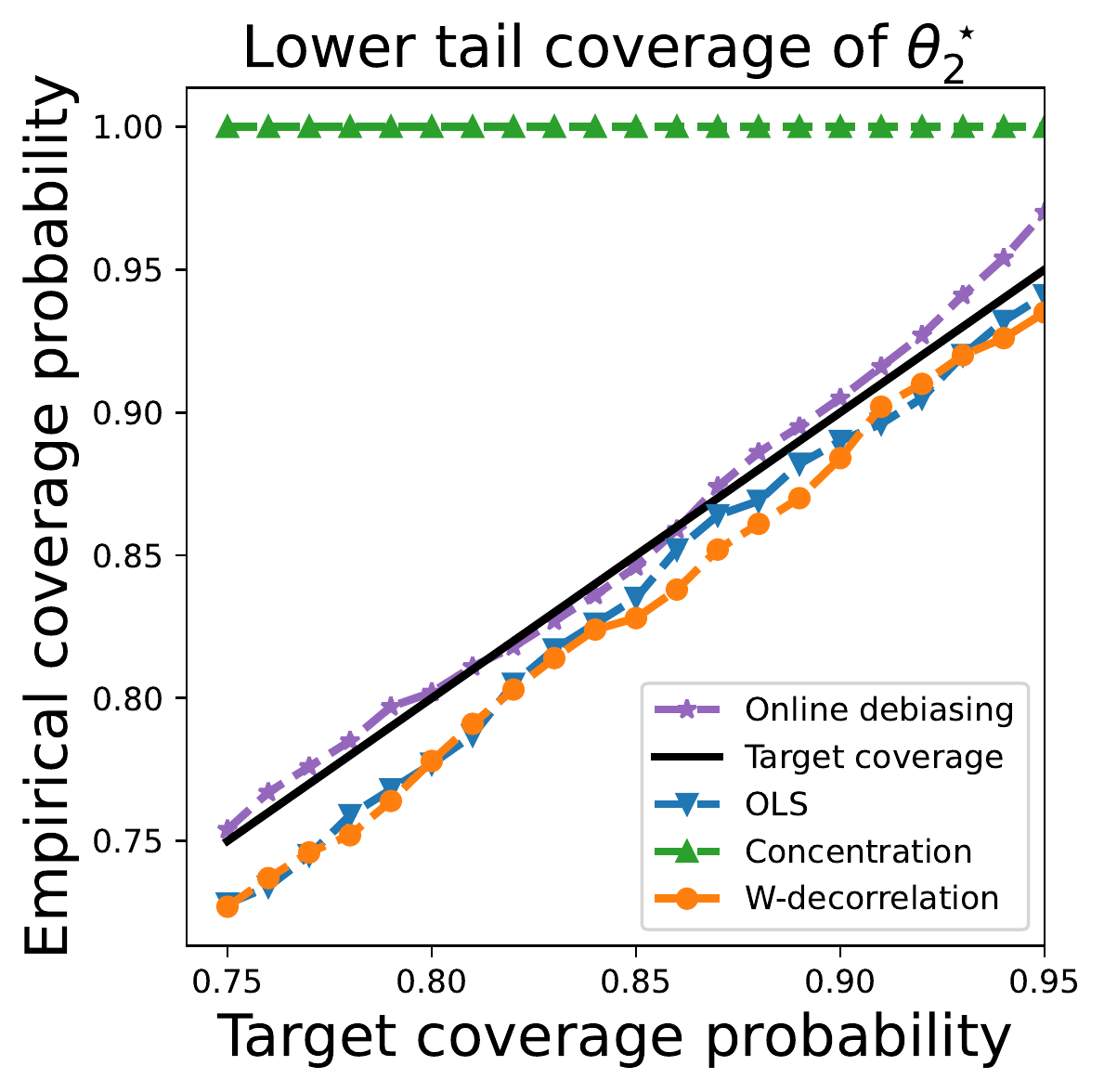}
\includegraphics[width=.32\textwidth]{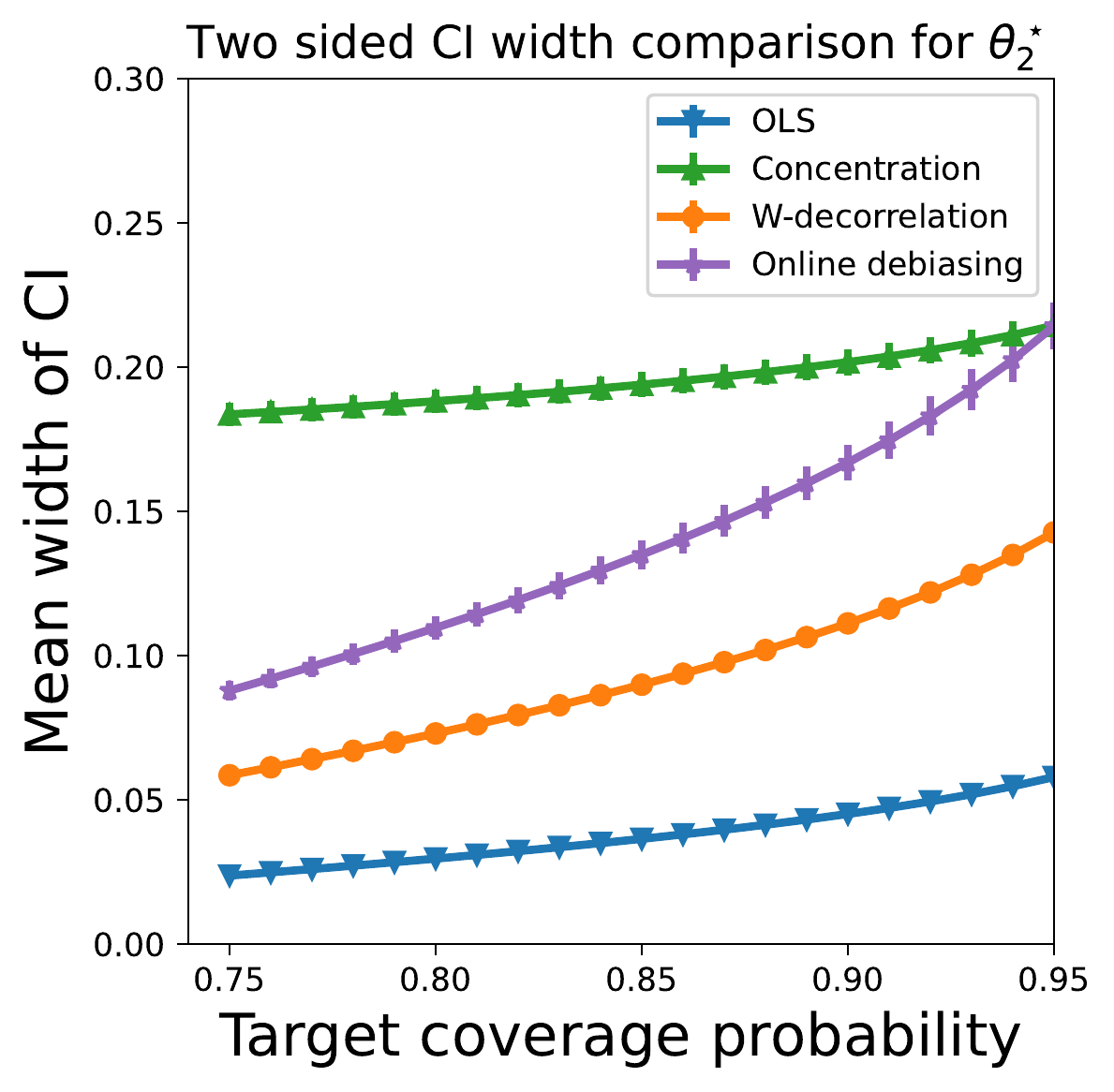}
\caption{Thompson sampling algorithm}
\end{subfigure}
\\
\begin{subfigure}{\linewidth}
\includegraphics[width=.32\textwidth]{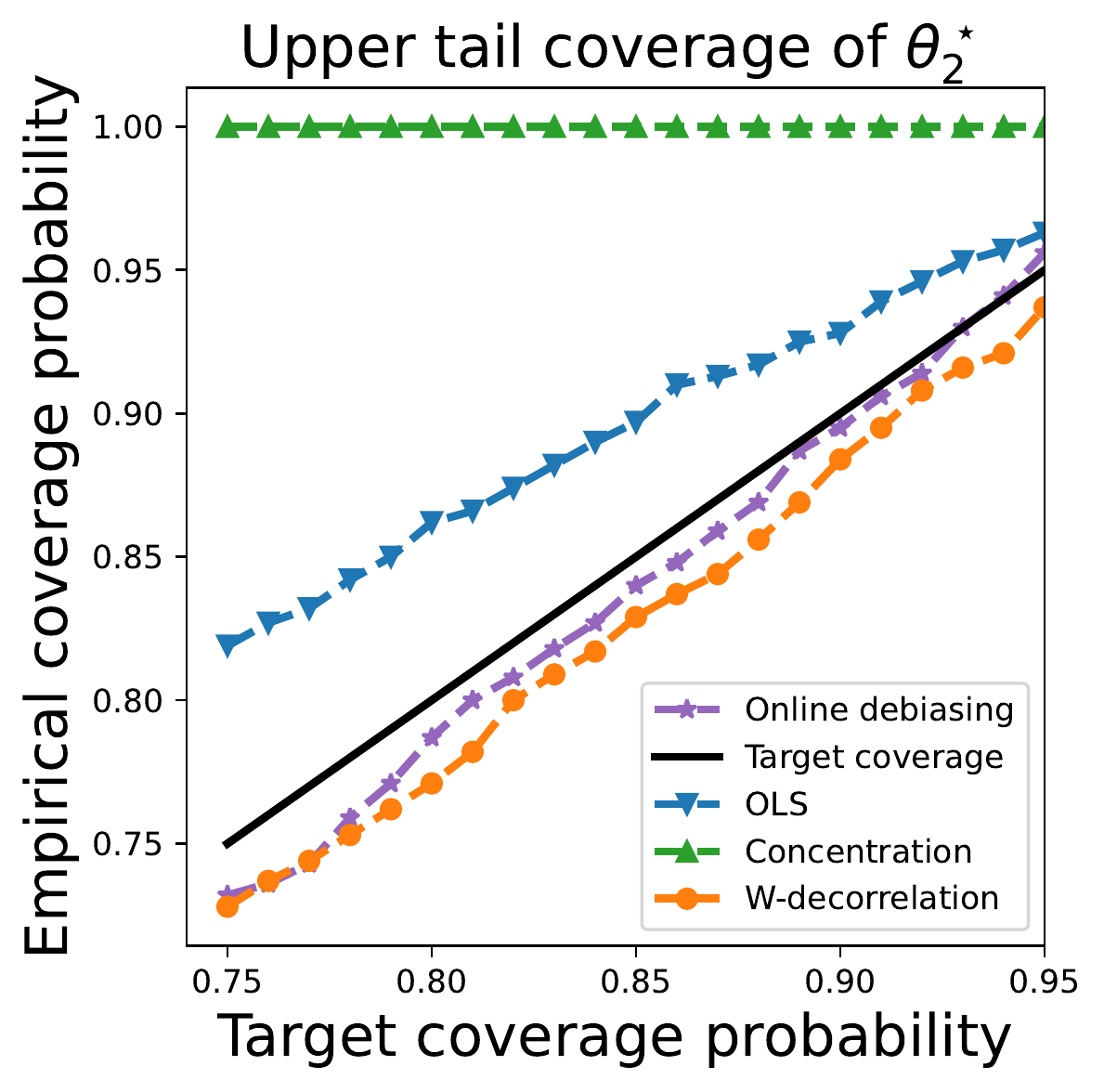}
\includegraphics[width=.32\textwidth]{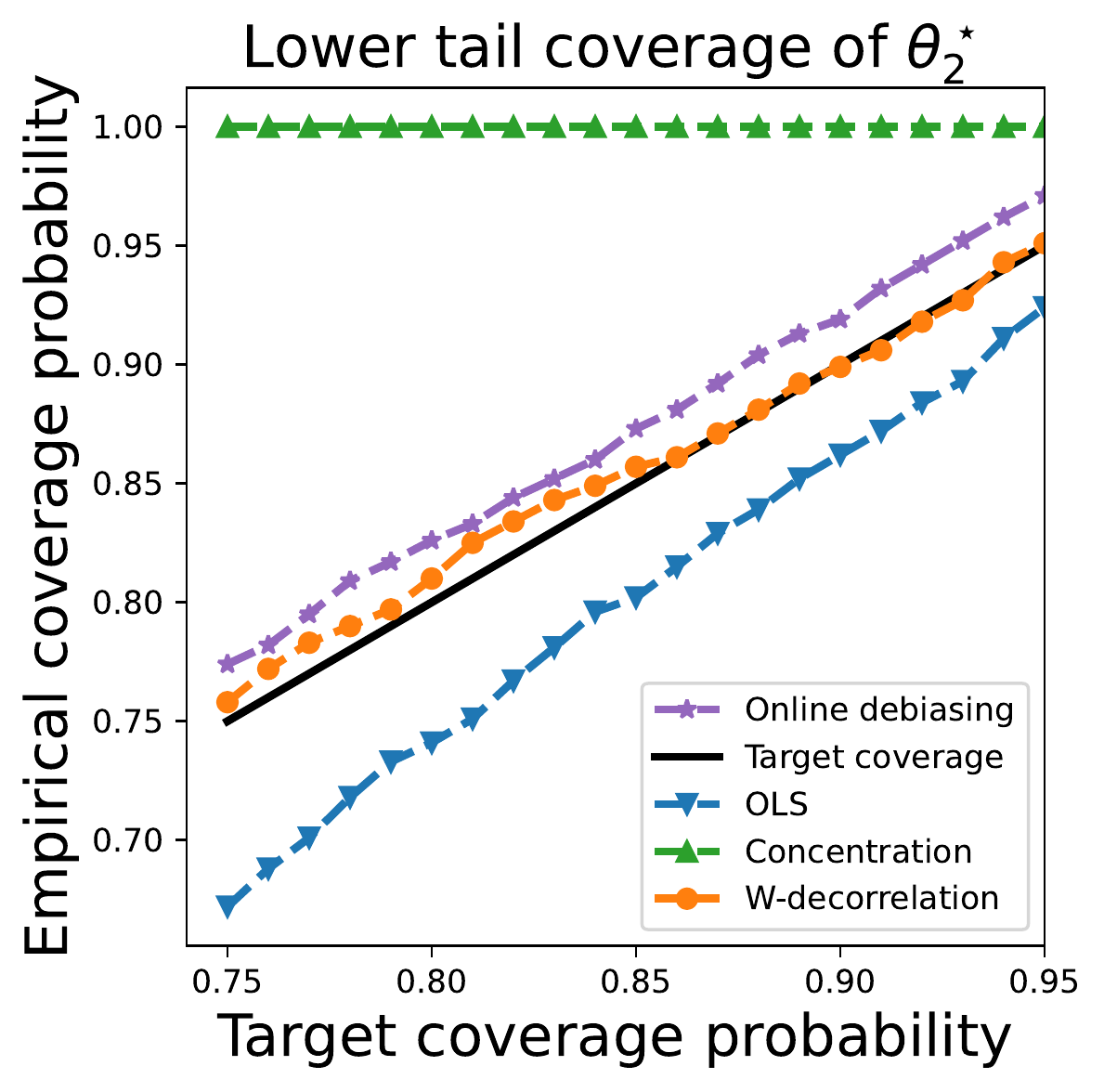}
\includegraphics[width=.32\textwidth]{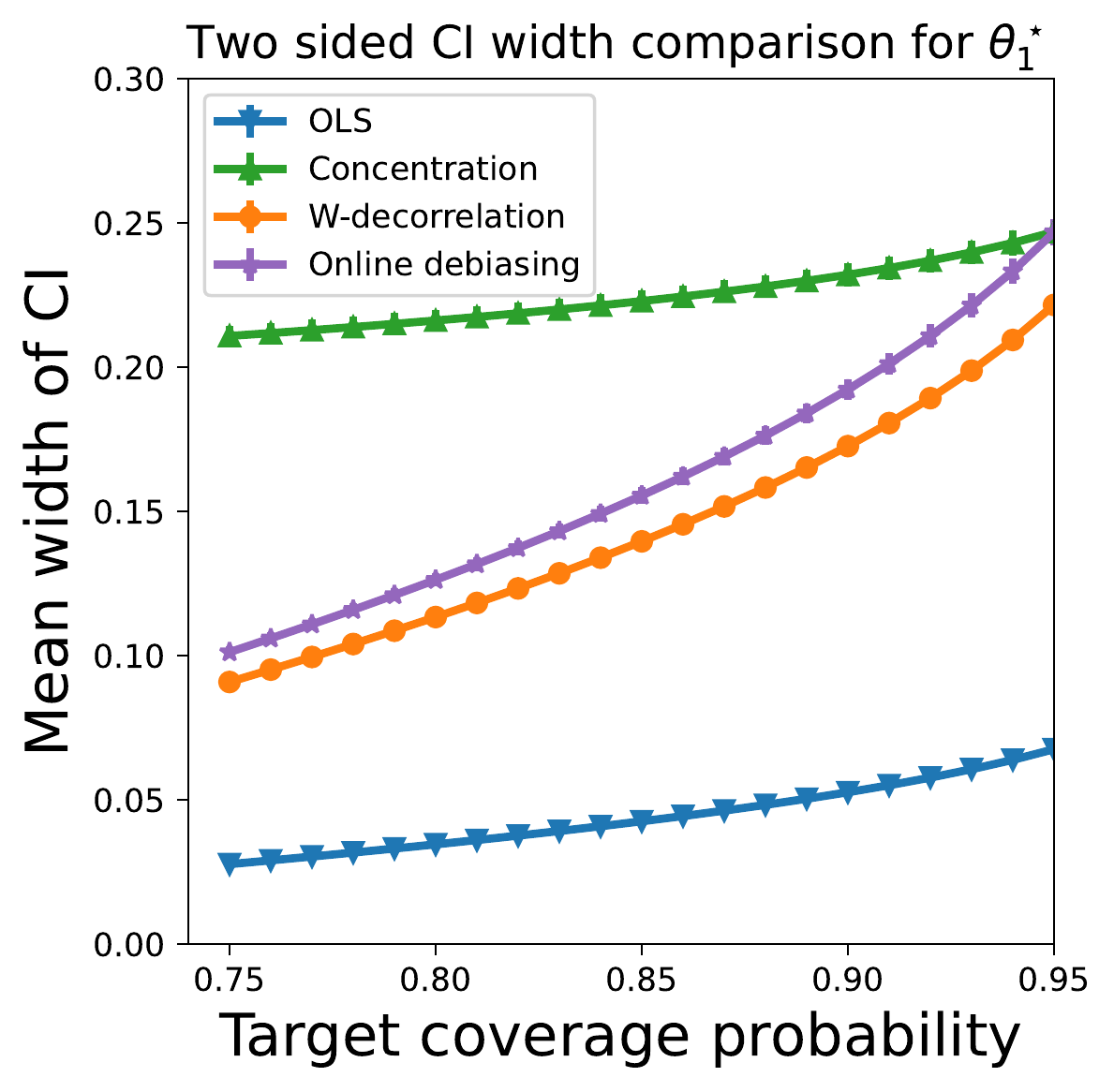}
\caption{$\greedy$-greedy algorithm}
\end{subfigure}
\begin{subfigure}{\linewidth}
\includegraphics[width=.32\textwidth]{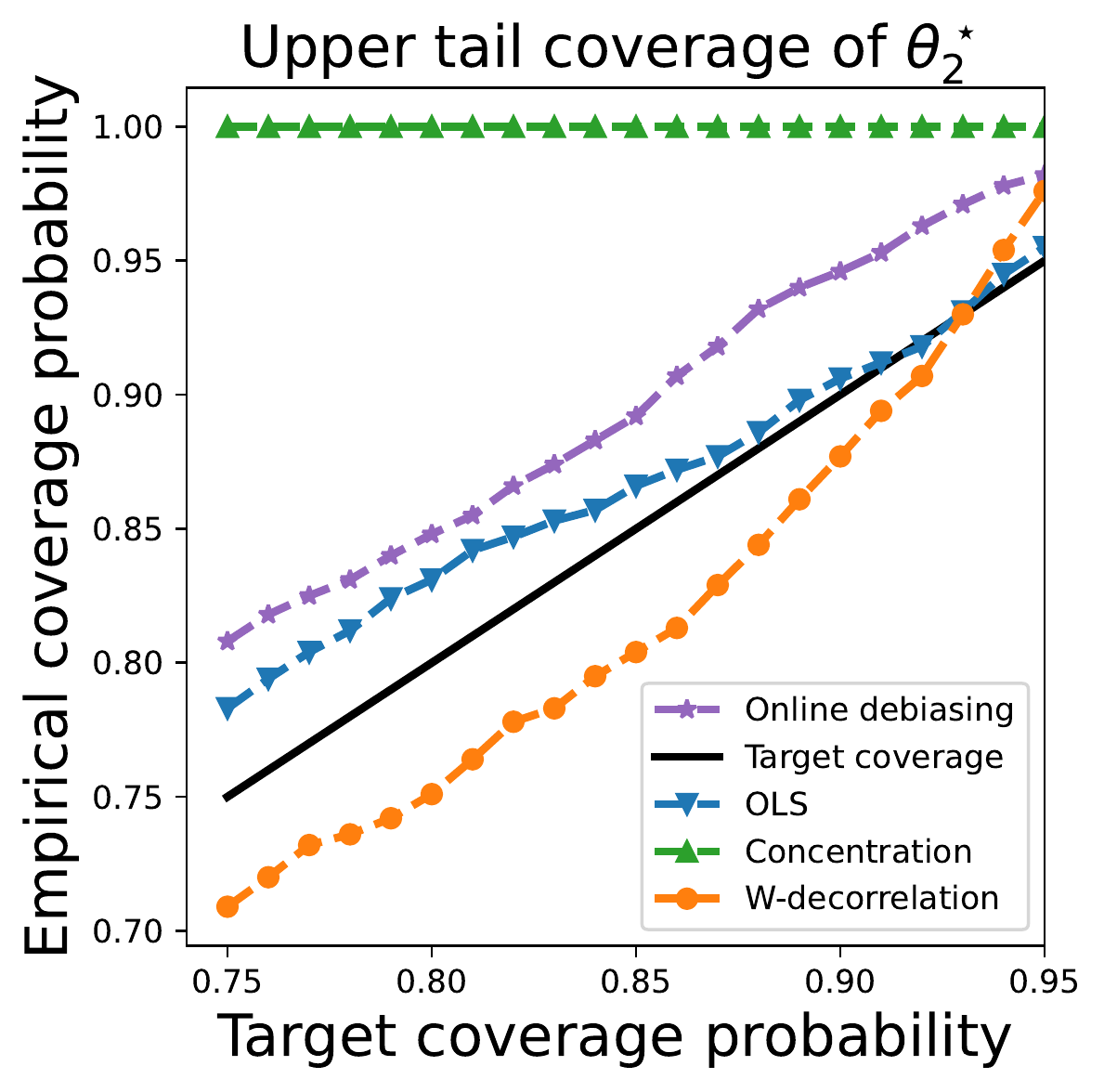}
\includegraphics[width=.32\textwidth]{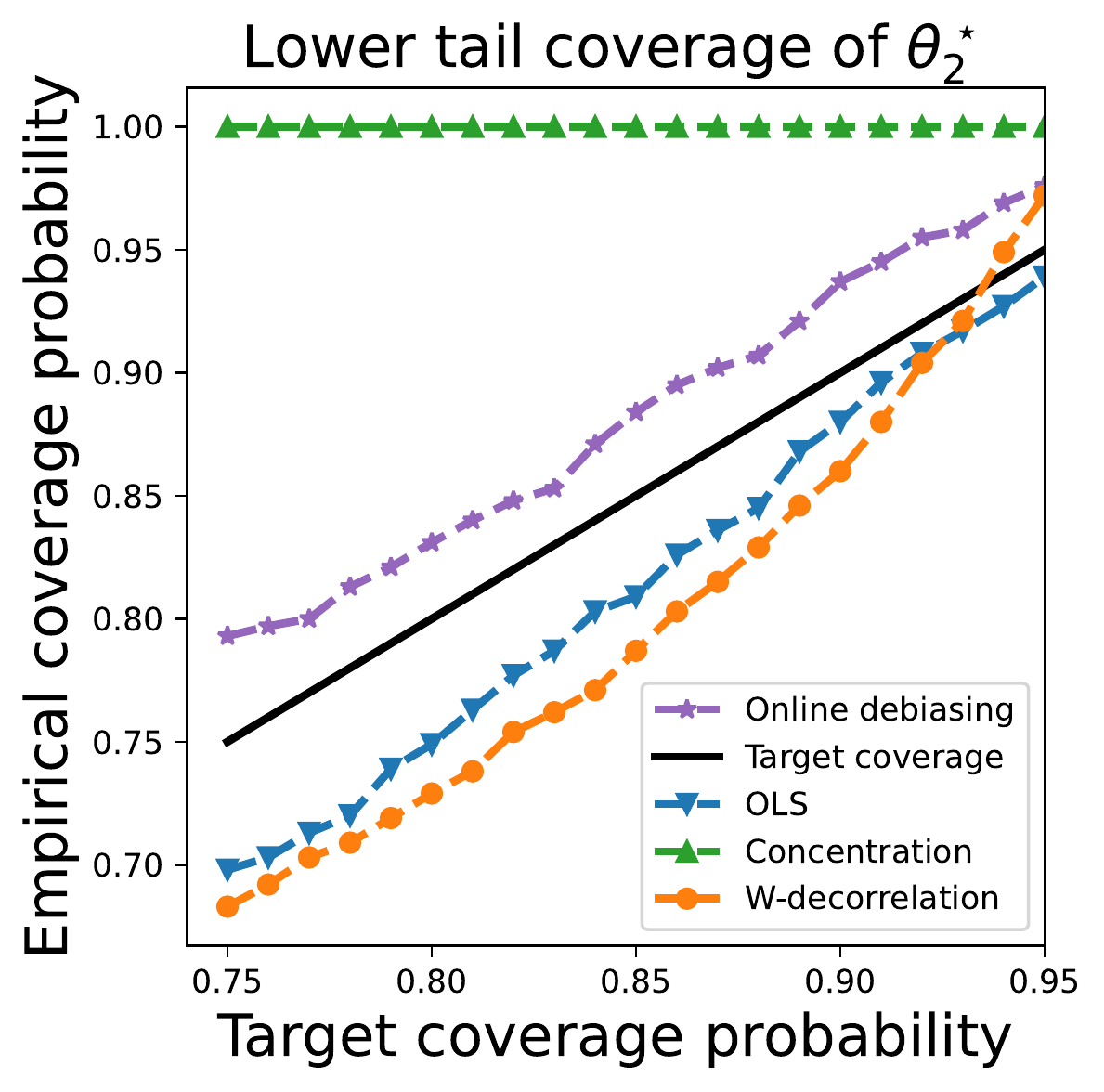}
\includegraphics[width=.32\textwidth]{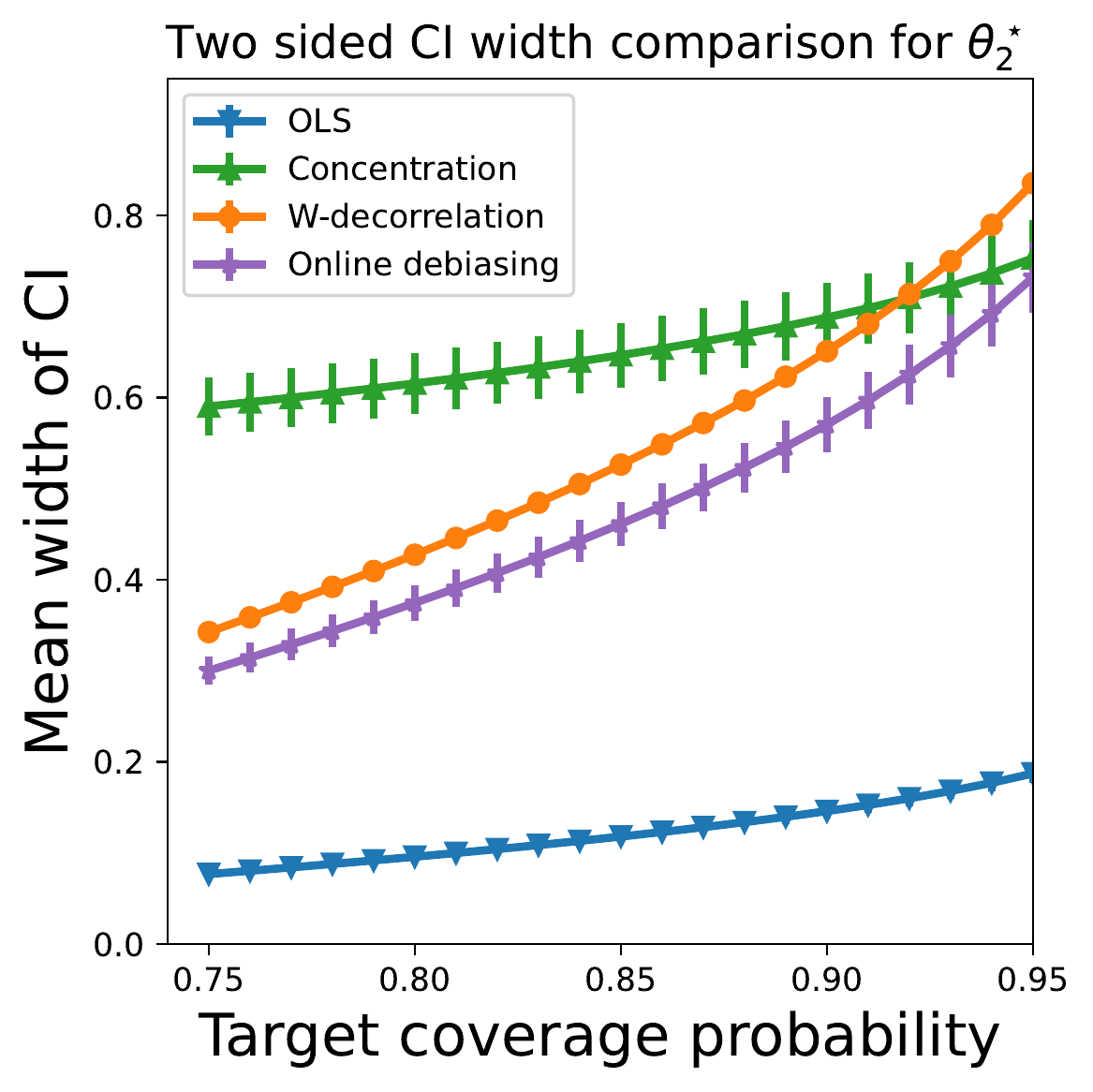}
\caption{Upper confidence bound (UCB) algorithm}
\end{subfigure}
\caption{
Average coverage and width of confidence intervals for $\theta_2^*$ across 1000 independent replications  of a multi-armed bandit experiment~\eqref{model:bandits} with $\thetastar \equiv (\theta_1^*, \theta_2^*) =  (0.3, 0.3)^\top$. The covariates $\{\x_i\}_{i = 1}^{1000}$
    were selected using a) Thompson sampling~\citep{thompson1933likelihood},
    (b) the $\greedy$-greedy algorithm~\citep{lattimore2020bandit},  and (c) the upper confidence bound algorithm (UCB)~\cite{jamieson2014lil}.
    The error bars represent $\pm 1$ standard error. 
    \textbf{Left} and \textbf{Center:} Coverage of one-sided $1 - \alpha$
    intervals for $\theta_2^*$. \textbf{Right:} Width of two-sided $1 - \alpha$ intervals for $\theta_2^*$.  
    See Appendix~\ref{sec:bandits-simulation} for details. }
\label{fig:Bandits-all-plots-arm-2}
\end{figure}

\begin{figure}[H]
\begin{subfigure}{\linewidth}
    \includegraphics[width=.33\textwidth]{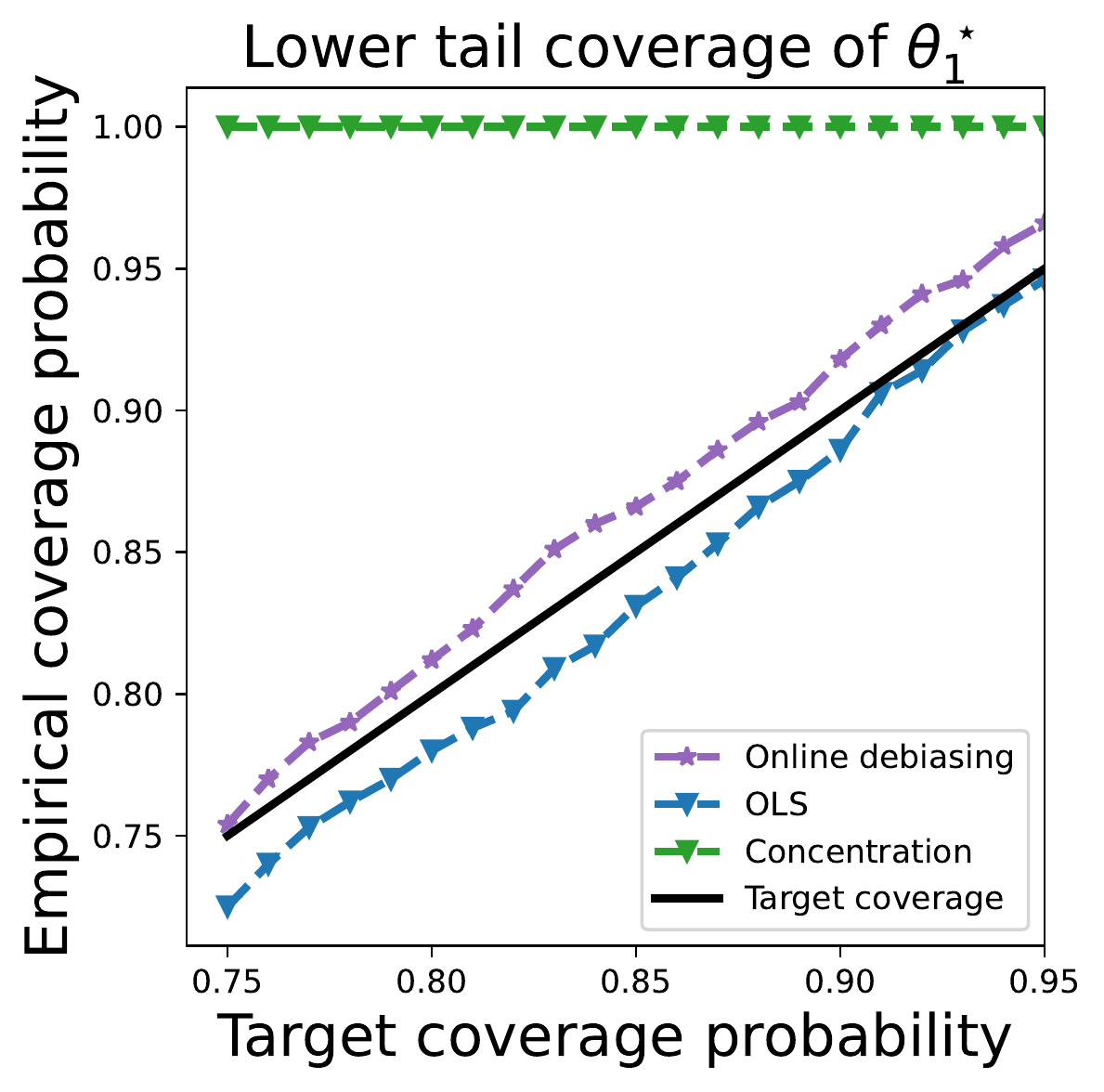}%
        \includegraphics[width=.33\textwidth]{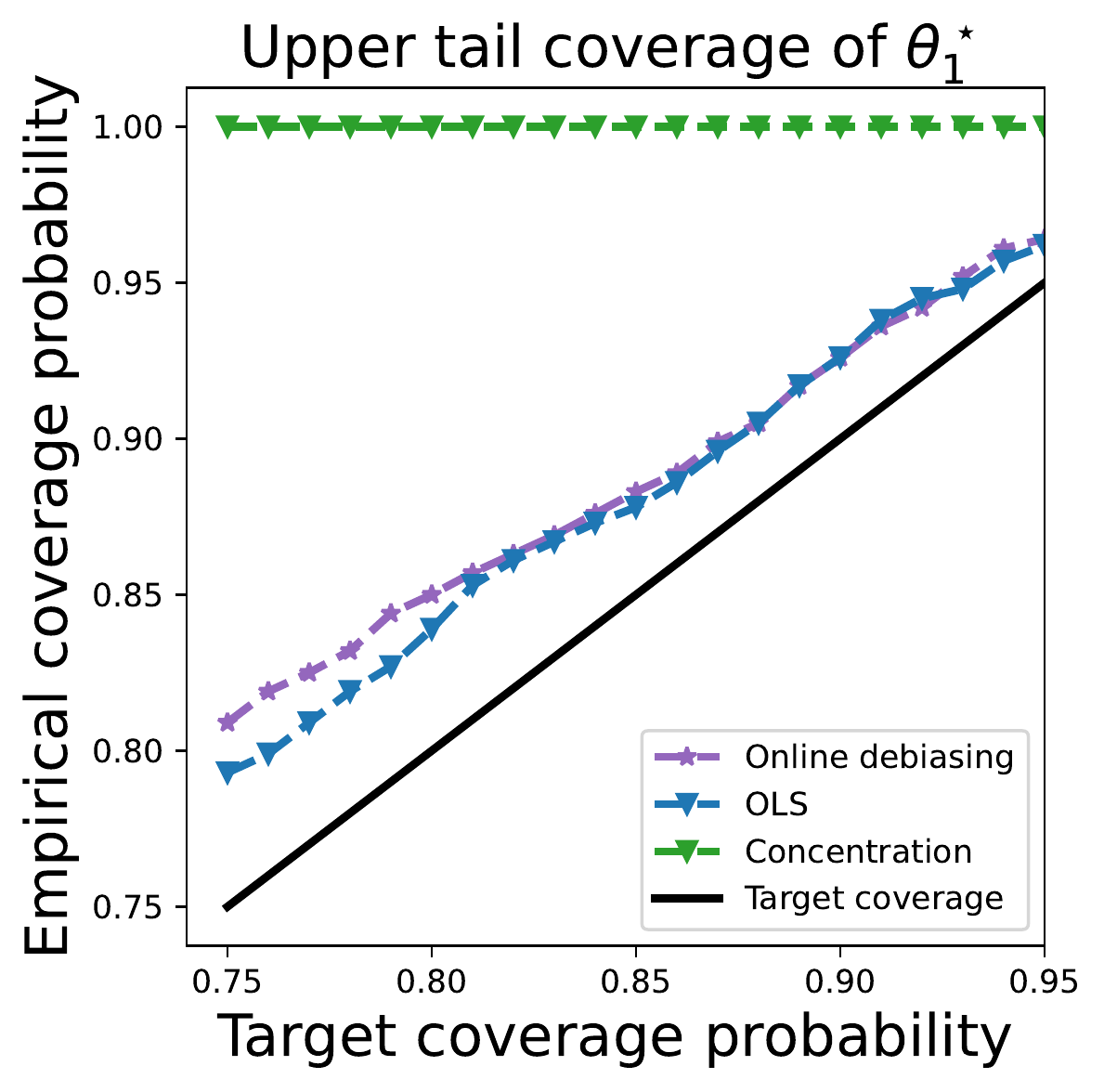}%
    \includegraphics[width=.33\textwidth]{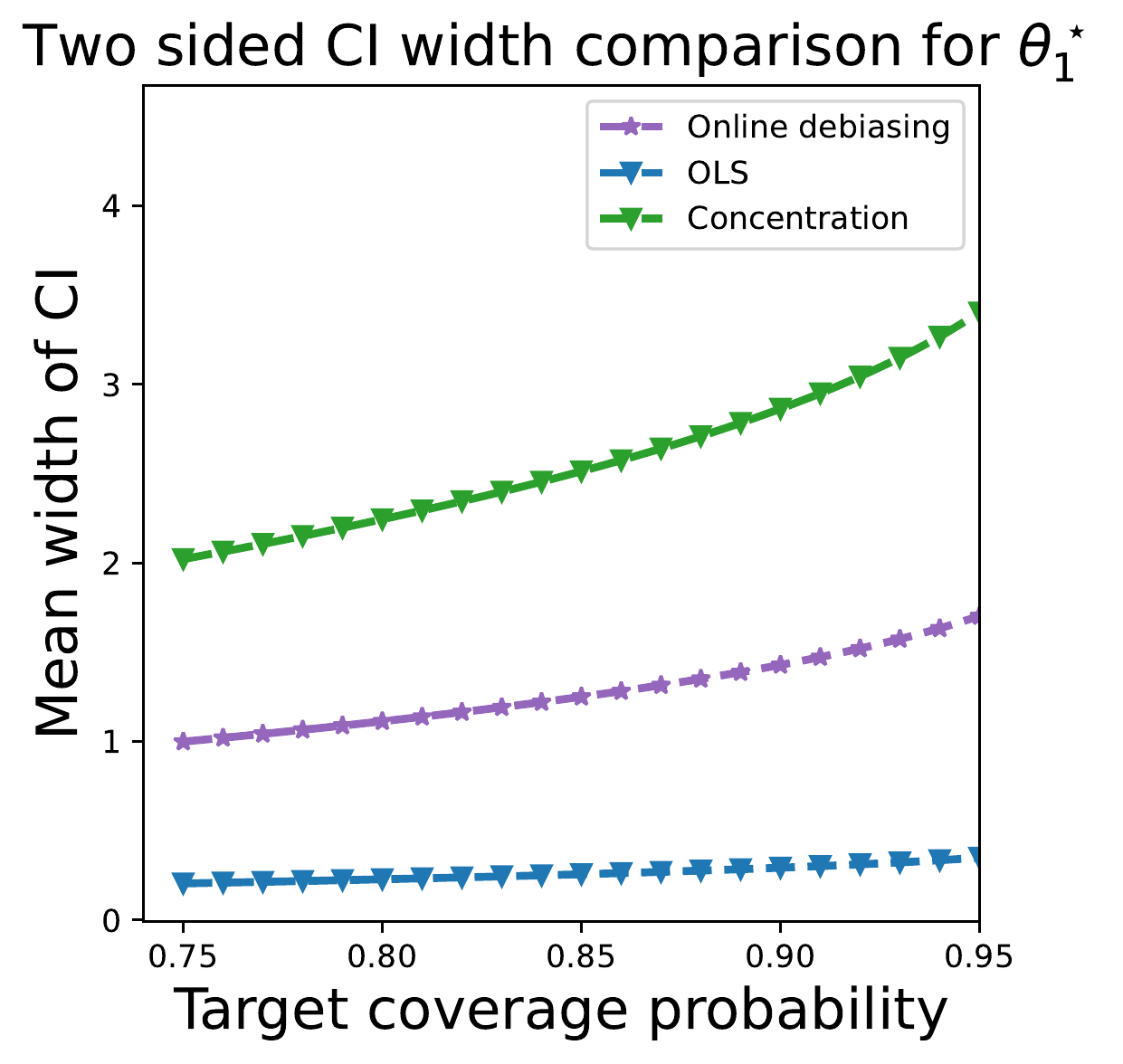} \\
    \includegraphics[width=.33\textwidth]{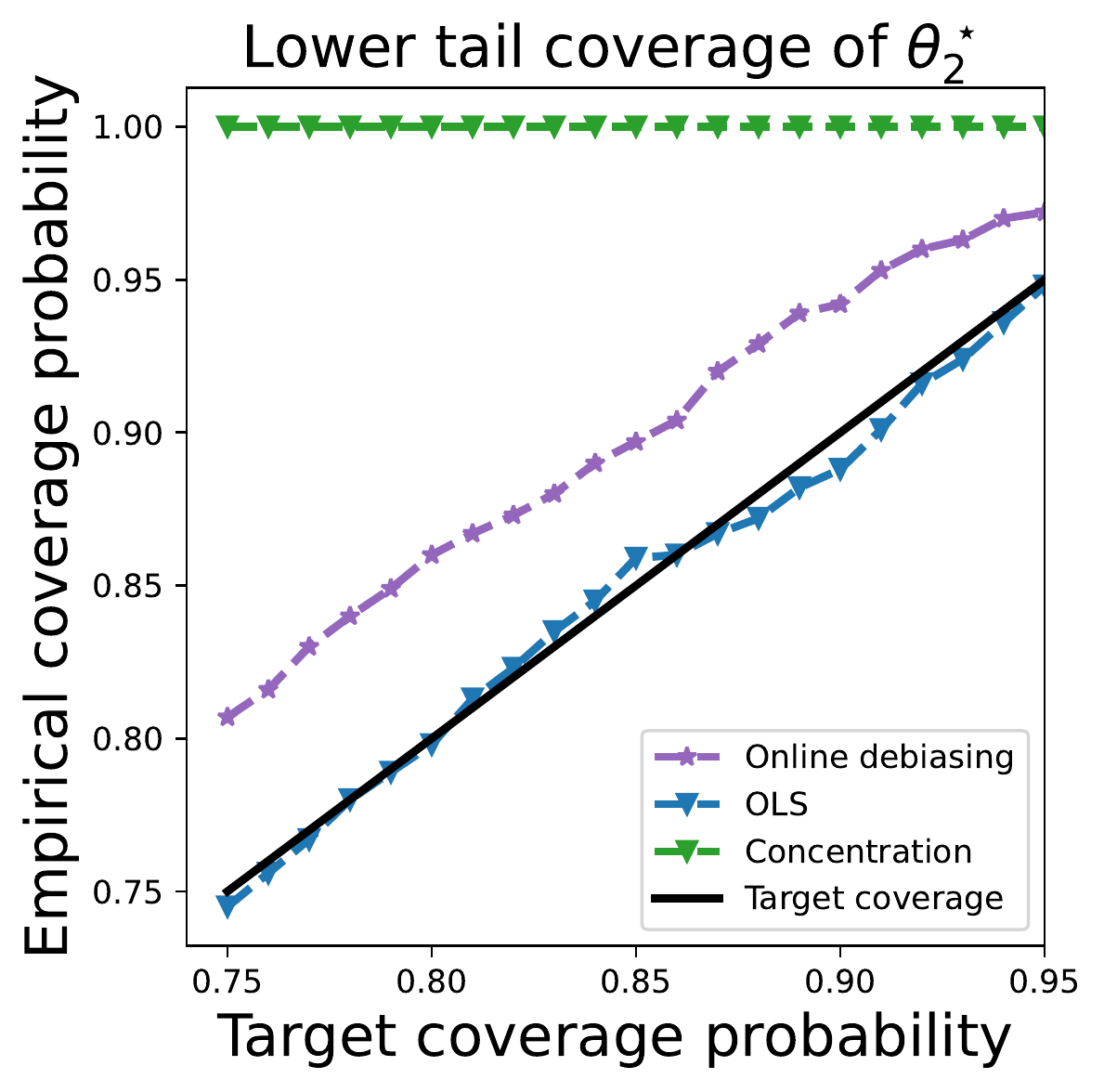}%
        \includegraphics[width=.33\textwidth]{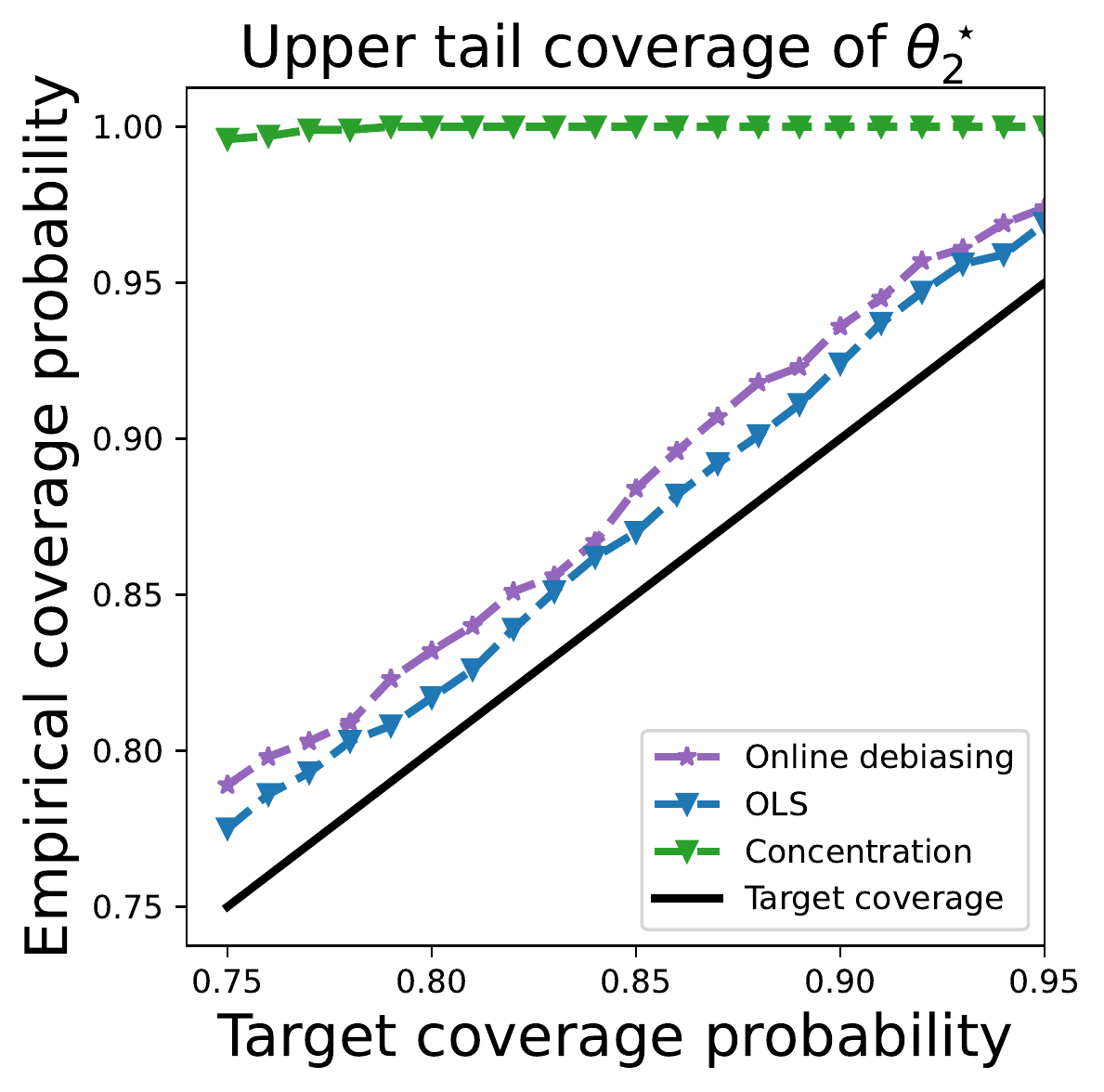}%
    \includegraphics[width=.33\textwidth]{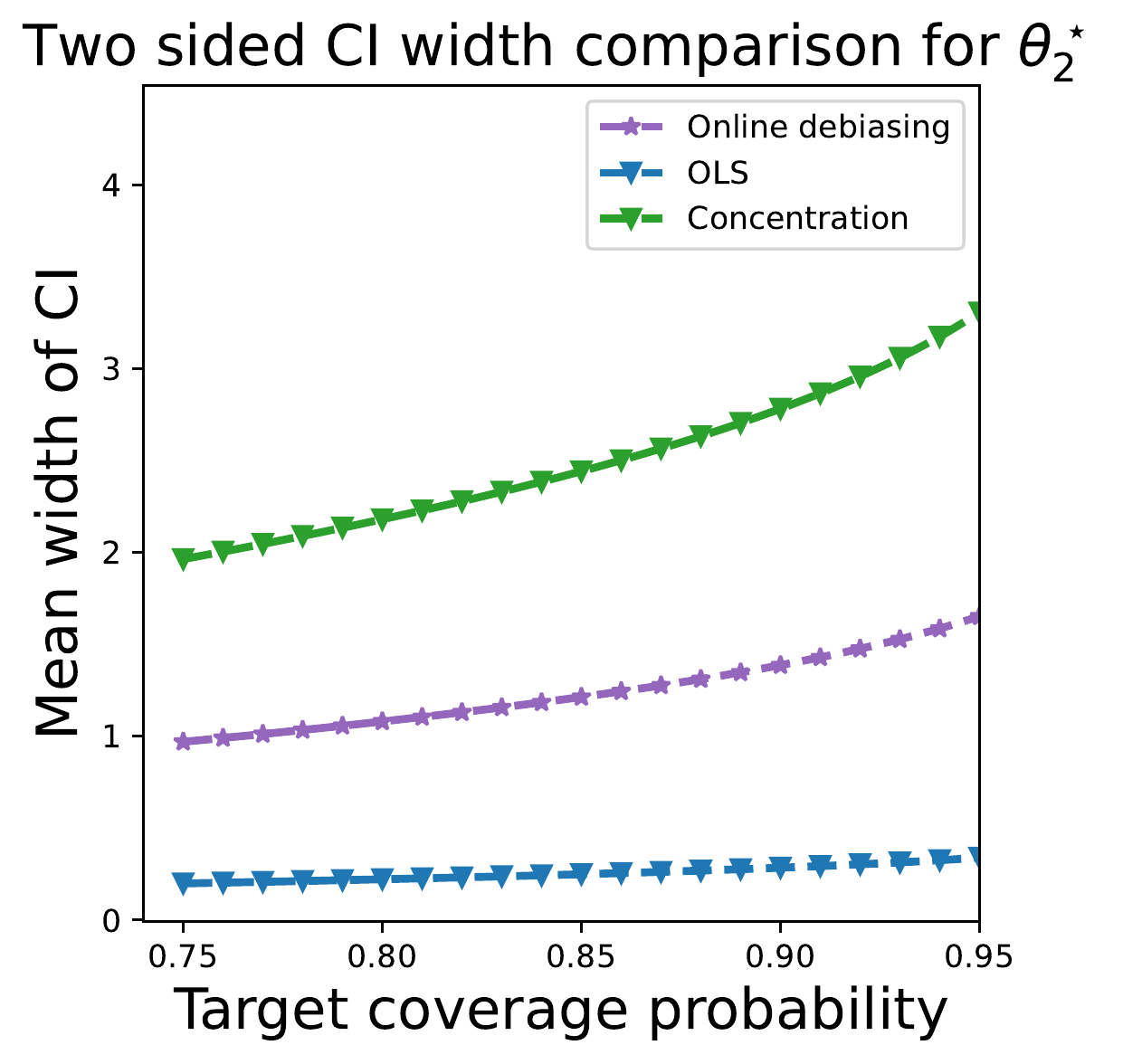}
    \caption{Plots for $\lambda_{\text{ridge}} = 1$}
\end{subfigure}
\begin{subfigure}{\linewidth}
    \includegraphics[width=.33\textwidth]{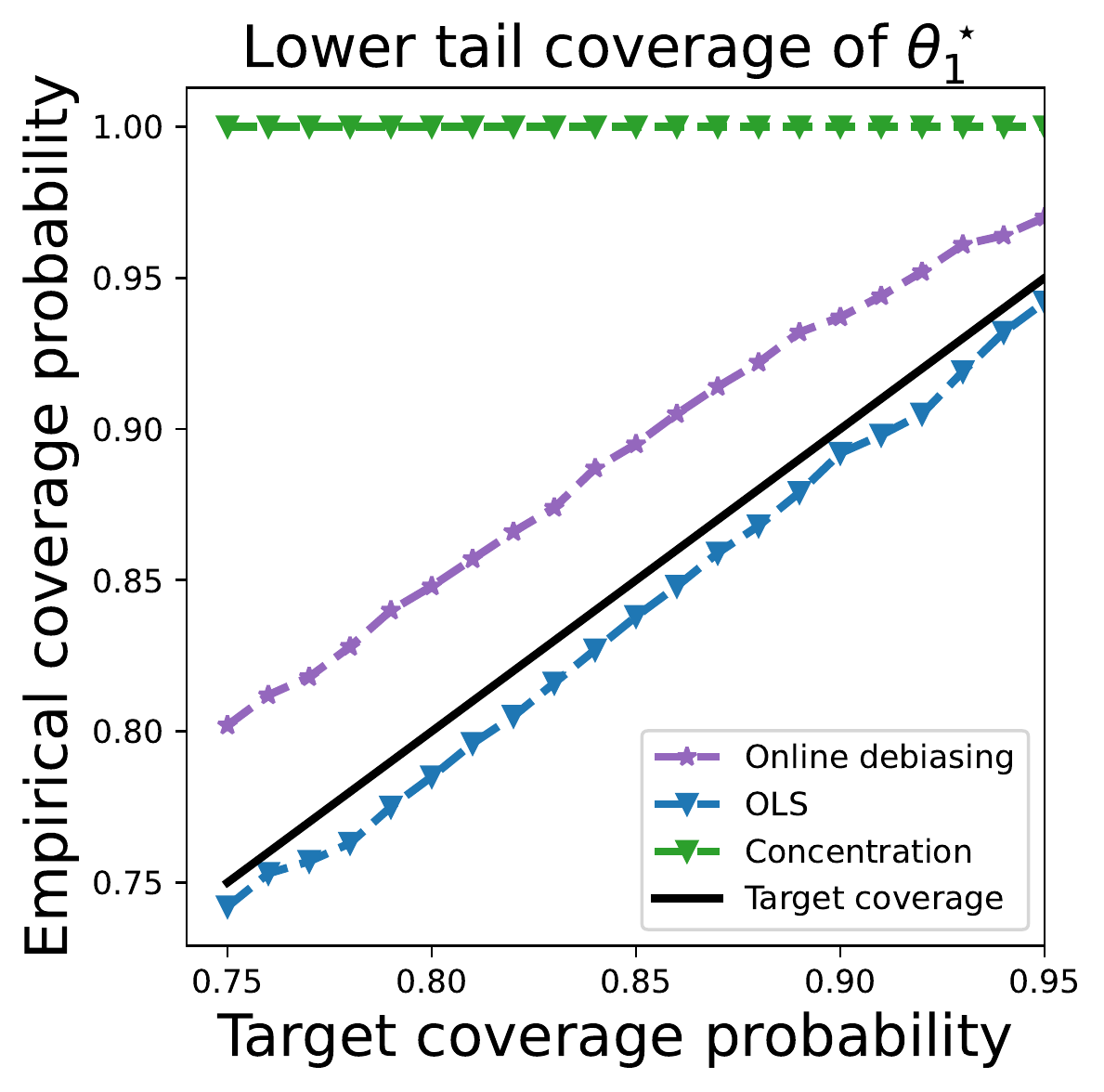}%
        \includegraphics[width=.33\textwidth]{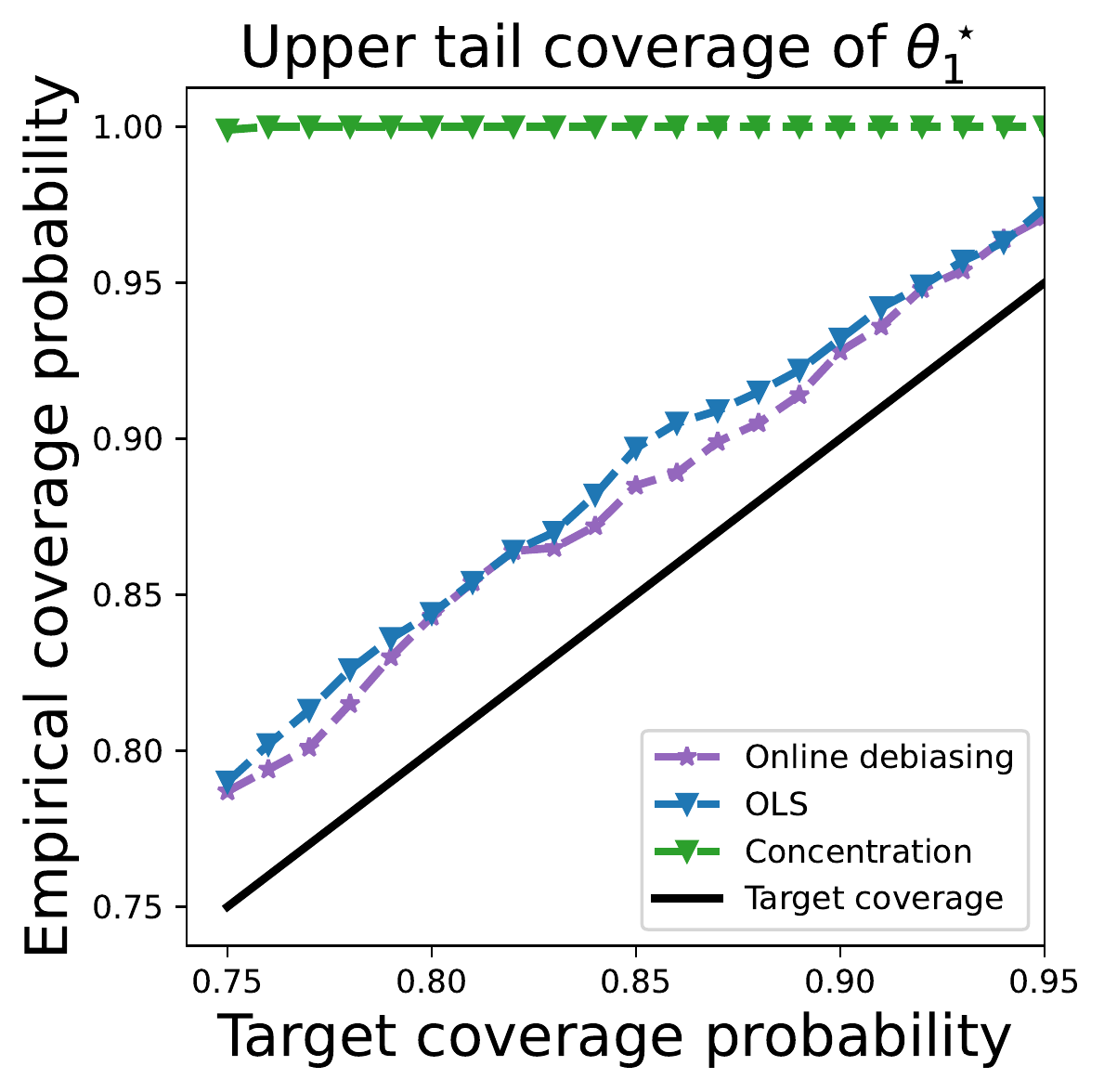}%
    \includegraphics[width=.33\textwidth]{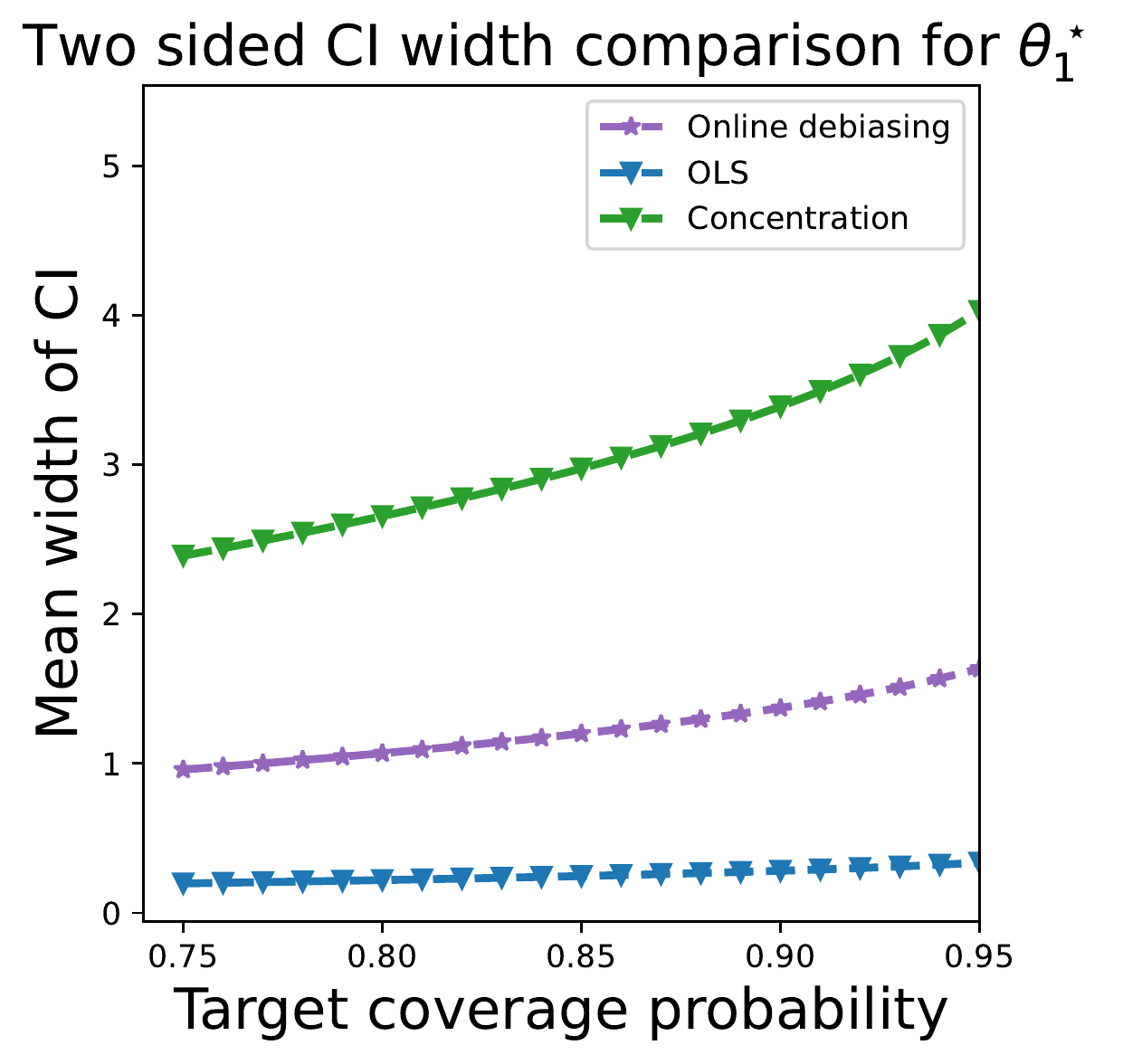} \\
    \includegraphics[width=.33\textwidth]{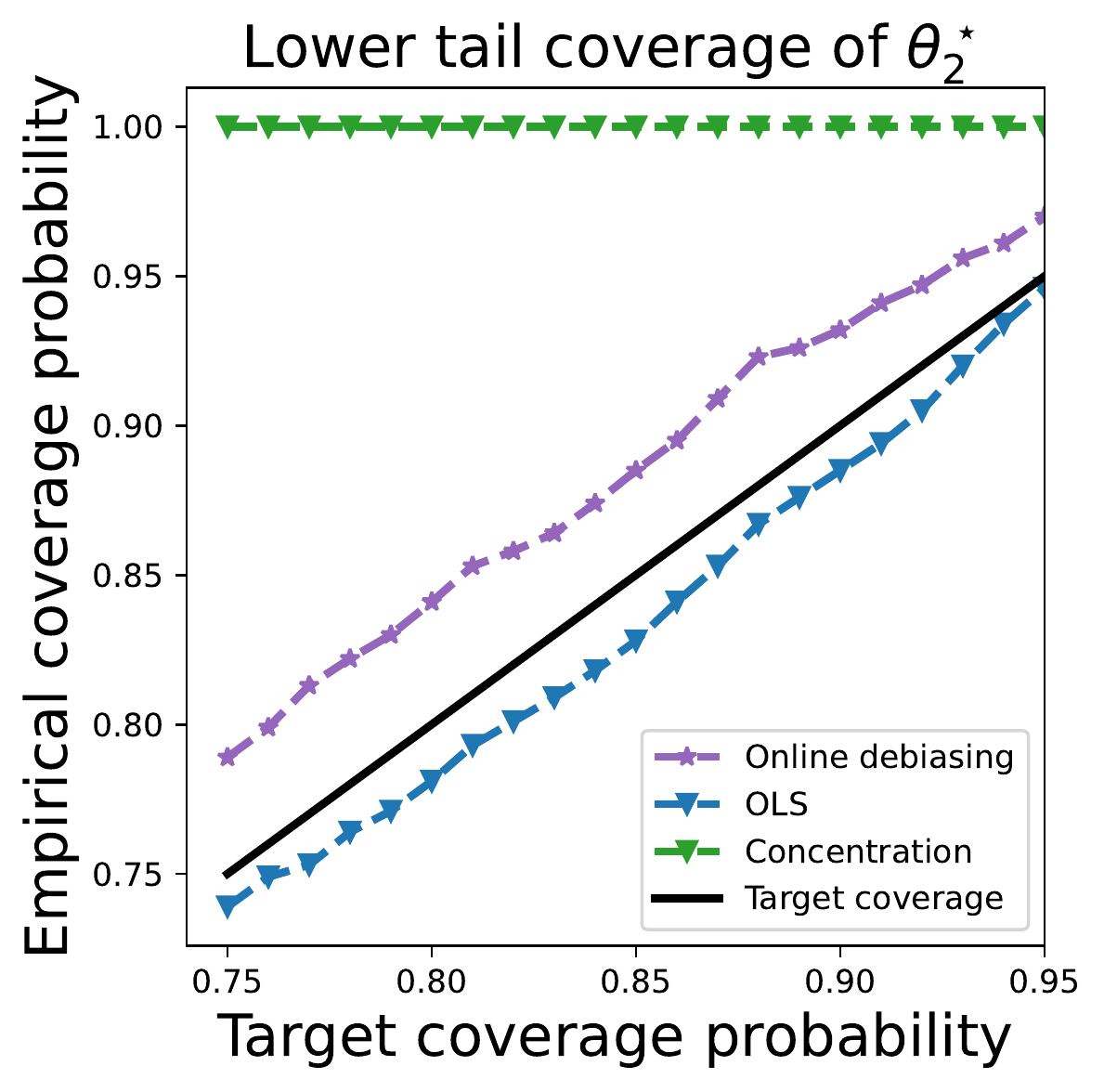}%
        \includegraphics[width=.33\textwidth]{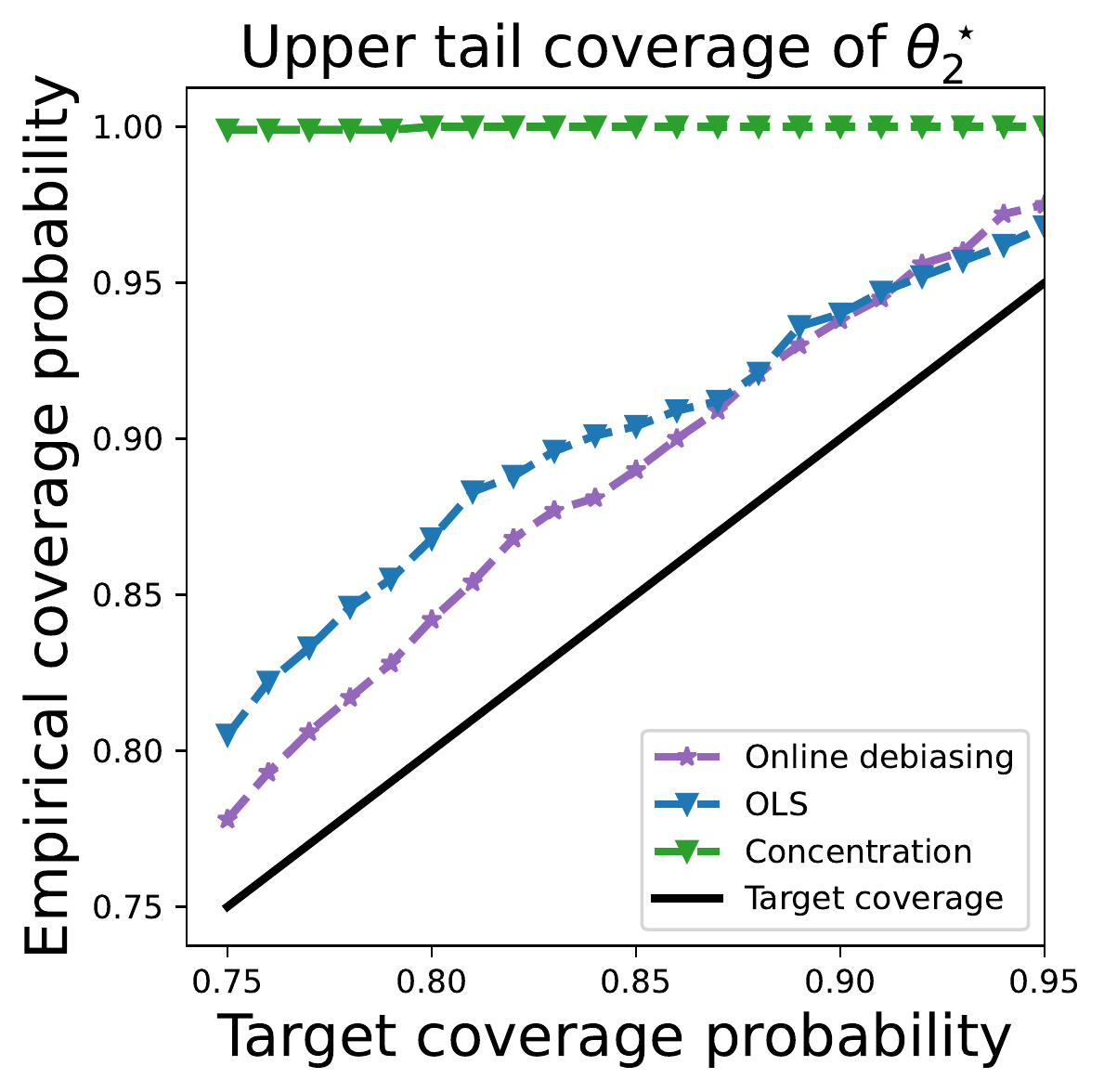}%
    \includegraphics[width=.33\textwidth]{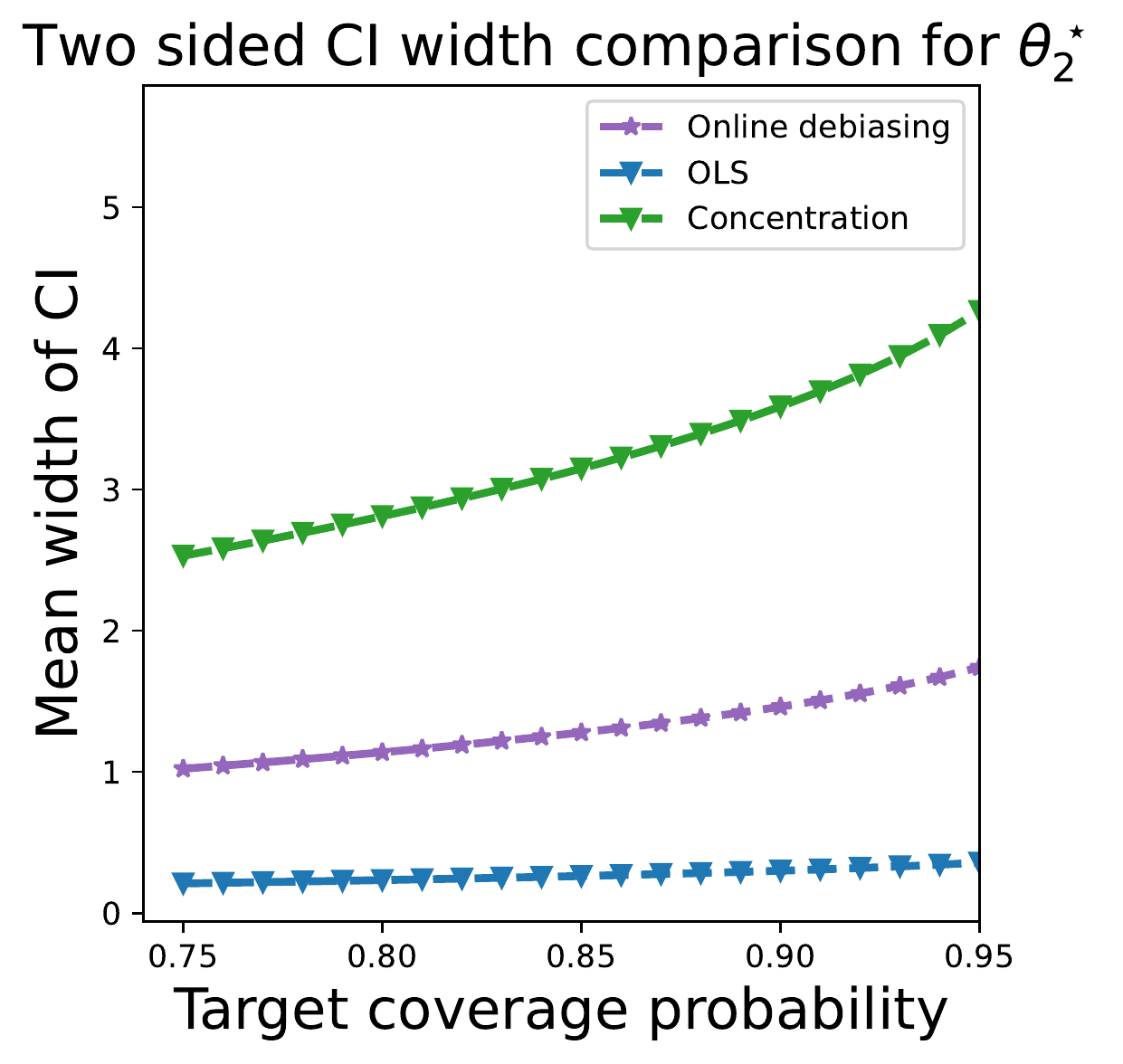}
    \caption{Plots for $\lambda_{\text{ridge}} = 10$}
    \end{subfigure}
    \caption{Average coverage and width of confidence intervals for $\theta_1^*$ and $\theta_2^*$ across 1000 independent replications  of linear bandits experiment~\eqref{eqn:greedy-selection} with $\thetastar \equiv (\theta_1^*, \theta_2^*) =  (0.3, 0.3)^\top$. The covariates $\{\x_i\}_{i = 1}^{1000}$
    were selected using the $\greedyEps$-greedy linear bandits algorithm~\eqref{eqn:greedy-selection},
    and the error bars represent $\pm 1$ standard error.
    \textbf{Left} and \textbf{Center:} Coverage of one-sided $1 - \alpha$
    intervals \textbf{Right:} Width of two-sided $1 - \alpha$ intervals.  
    See Appendix~\ref{sec:LinBanditsFullSimulation} for details.}
    \label{fig:LinBanditsFullSim}
\end{figure}
%%%%%%%%%%%%%%%%%%%%%%%%%%%%%%%%%%%%%%%%%%%%%%%%%%%%%%%%%%%%%%%%%%%%%%%%

\section*{Acknowledgments} 
\label{sec:acknowledgments}
This work was partially supported by the Microsoft-Berkeley BAIR collaboration.

\section{Comparison to the least squares estimator}
\label{sec:LS-vs-online}
In this section, we provide a more fine-grained comparison between the least squares estimator and the online debiased estimator. We start with a lower bound on the MSE of the ordinary least squares estimator, which is based on the law of iterated logarithm.

\subsection*{Law of the iterated logarithm lower bound}
Let $\epsilon_1,\epsilon_2,\ldots$ be a sequence of independent
standard Gaussian random variables and and $d = 1$ and $\thetastar =
0$, which means the least squares estimator is
\begin{align*}
\hat \theta_n = \frac{\sum_{i=1}^n \x_i \epsilon_i}{\sum_{i=1}^n \x_i^2}\,.
\end{align*}
We will choose an adaptive design with $\x_i \in \{0, 1\}$ and for which
\begin{align*}
\limsup_{n\to\infty} \Prob(\mathcal E_n > 1) \geq \frac{1}{2} \,\, \qquad \text{with} \qquad 
\mathcal E_n = \frac{\SigmaMat_n^{1/2} \hat \theta_n}{\sqrt{\frac{1}{2} \log \log n}} \,.
\end{align*}
The construction is essentially based on using the adaptive covariates to derandomize the law of the iterated logarithm.
Let $M_t = \sum_{i=1}^t \epsilon_i$ be a Gaussian random walk.
Define a sequence of intervals $\{m_k, m_k + \ell_k - 1\}$ of length $\ell_k$ inductively as follows.
Let $m_0 = 0$ and $m_k = m_{k-1} + \ell_k$ where
\begin{align*}
\ell_k = \min\left\{\ell : \Prob\Big(\text{exists } 1 \leq s \leq \ell : M_s \geq \sqrt{(m_{k-1}+s) \log \log m_{k-1}}\Big) \geq \frac{1}{2}\right\}\,.
\end{align*}
By the same tedious argument used in the proof of the law of the iterated logarithm one can also show
that $\log \log m_{k+1} \leq 2 \log \log m_k$ for suitably large $k$.
By symmetry, 
\begin{align*}
\Prob\left(\text{exists } m_{k-1} + 1 \leq t \leq m_{k-1} + \ell_k : |M_t| \geq \sqrt{t \log \log m_{k-1}} \Big| M_{m_{k-1}}\right) 
\geq \frac{1}{2} \,\, a.s.\,.
\end{align*}
Let $\mathbb N = \cup_{k=1}^\infty I_k$ where $\{I_k : k \in \mathbb N\}$ are disjoint consecutive intervals and
$I_k$ has length $2 \ell_k$.
We choose an adaptive design so that $\x_i \in \{0,1\}$ for all $i$ and $\sum_{i \in I_k} \x_i = \ell_k$ for all $k$. That is,
$\x_i = 1$ in exactly half of the time in each interval.
The adaptivity within each interval is as follows. The feature $\x_i = 1$ for either $\ell_k$ time-steps or until $\cE_{i-1} \geq 1$.
Then $\x_i = 0$ until a final block of the interval where $\x_i = 1$ to guarantee that $\sum_{i \in I_k} \x_i = \ell_k$ holds.
Formally,
\begin{align*}
\x_i &= \begin{cases}
0 & \text{if } m(i-1) = m_k \\
1 & \text{if } \mathcal E_{i-1} \leq 1 \\
1 & \text{if } m(i-1) + |I_k \setminus \{1,2,\ldots,i-1\}| = m_k \\
0 & \text{otherwise}\,.
\end{cases}
\end{align*}
By construction, if $s_k$ is the $\ell_k$th time-step in interval $I_k$, then
\begin{align*}
\Prob\left(\cE_{s_k} > 1\right)
&= \Prob\left(\sum_{i=1}^{s_k} \x_i \epsilon_i \geq \sqrt{\frac{m(s_k)}{2} \log \log s_k}\right) \\
&\geq \Prob\left(\sum_{i=1}^{s_k} \x_i \epsilon_i \geq \sqrt{m(s_k)\log \log m_{k-1}}\right) \\
&\geq \frac{1}{2}\,.
\end{align*}

%%%%%%% Correcting for the bias %%%%%%%%%%
%

%%%%%%%%%% Bias Reduction %%%%%%%%% 

\subsection*{Reducing bias via post debiasing correction} 

In this section, we consider a modified version of the online debiased estimator. Throughout this section, we assume $\Exs[\error_i^2] = 1$. We show that this modified estimator can have less bias for certain multiarmed bandit problems, and thereby improving upon the performance of the least squares estimators. We also show that this improved estimator has the same performnace as online debiased estimator~\eqref{eqn:ThetaDecorr-defn} in the worst case. We start with the following modified version of the online debiased estimtor
\begin{align}
\label{eqn:New-OD-estimator}
\thetaODNew = \thetaLS + \sqrt{\frac{\gamma_n}{\gammahat_n}} \cdot \sum_{i = 1}^{n} \SigmaMat_n^{-\frac{1}{2}} \w_i \left(y_i - \x_i^\top \thetaLS \right)   
\end{align}
where
\begin{align}
\label{eqn:gammahat-defn}
\sqrt{\gammahat_n} = \max \left\{ \banditdir^\top  \sqrt{\gamma_n} \sum_{i = 1}^n \w_i \x_i^\top \SigmaMat_n^{-\frac{1}{2}}  \banditdir \;,\; \frac{1}{\log(\numobs)\cdot \log\log(\numobs)} \right\}
\end{align}
Focusing on a k-armed bandit problem, we assume that 
the tuning parameters $\gamma_n$, and $\{\XscaleMat_{i,n}\}_{i \leq n}$ satisfy the following conditions
\begin{subequations}
\begin{align} 
\label{eqn:modified-gamma-n-conditions}
&\gamma_n \rightarrow 0, \qquad \text{and} \qquad \frac{\gamma_n}{\log(\numobs)\cdot \log\log(\numobs)} \rightarrow \infty.  \\
& \frac{1}{\gamma_n} \max_{i \in [n]} \cdot \x_i^\top \Gamma_{i,n}^{-1} \x_i   \convprob 0 \qquad \text{and} \qquad \SigmaMat_n \succ \Gamma_{i,n} \succ 0.
\label{eqn:conditions-on-samplecov-bandits}
\end{align}
\begin{align}
\label{eqn:bounded-cov-variation}
  \banditdir^\top \SigmaMat_n^{-1} \XscaleMat_{i,n} \banditdir > c_1 > 0
  \qquad \text{for all} \quad n \geq 1 \;\; \text{and} \;\; i \leq n.    
\end{align}
\end{subequations}
\begin{theos}
\label{thm:post-deboasing-correction}
Given a target direction $\banditdir$,  suppose  the covariance matrix $\SigmaMat_n$ of a $k$-armed bandits problem satisfies conditions~\eqref{eqn:modified-gamma-n-conditions},~\eqref{eqn:conditions-on-samplecov-bandits} and~\eqref{eqn:bounded-cov-variation}. Then the online debiased estimator from~\eqref{eqn:New-OD-estimator} satisfies 
\begin{align}
\label{eqn:normality-and-gammanhat}
\sqrt{\gammahat_n} \banditdir^\top \cdot \SigmaMat^{\frac{1}{2}} (\thetaODNew - \thetastar) \indistrb \Ncal(0, 1)
\quad \text{and} \quad \lim\sup \sqrt{\frac{\gammahat_n}{\gamma_n}} \geq O_p(\sqrt{c_1}). 
\end{align}

Moreover, we have 
\begin{align}
    \Exs \left(\sqrt{\gammahat_n} \banditdir^\top \cdot \SigmaMat^{\frac{1}{2}} (\thetaODNew - \thetastar) \right)^2
    =  \gamma_n  \Exs\left(\sum_{i = 1}^\numobs \banditdir^\top \w_i \w_i^\top \banditdir \right) \leq 1.  
\end{align}
\end{theos}

See Appendix~\ref{sec:post-debiasing-proof} for a proof of Theorem~\ref{thm:post-deboasing-correction}. 

% \subsubsection*{Comments on Theorem}
% \notate{this section is empty}
% % \kkcomment{To be added} 
% \newpage 

\section{Proof of Theorem~\ref{thm:post-deboasing-correction} and related results}
\label{sec:post-debiasing-proof}
In this section, we provide a proof of Theorem~\ref{thm:post-deboasing-correction}. We show shortly that 
under the conditions of Theorem~\ref{thm:post-deboasing-correction}
\begin{align}
\label{eqn:condition-one-1}
\sqrt{\gammahat_n} =  \banditdir^\top  \sqrt{\gamma_n} \sum_{i = 1}^n \w_i \x_i^\top \SigmaMat_n^{-\frac{1}{2}}  \banditdir 
\end{align}
We come back to the proof of this statement in~Lemma~\ref{lem:bounded-cov-variation}, but let us complete the proof assuming this cindition is true.
A simple calculation using the definition of $\gammahat_n$ from~\eqref{eqn:gammahat-defn}  yields
\begin{align*}
\banditdir^\top \sqrt{\gammahat_n} \banditdir^\top \cdot \SigmaMat^{\frac{1}{2}} (\thetaODNew - \thetastar) 
&= \banditdir^\top \left( \sqrt{\gammahat_n} \cdot \Id - \sqrt{\gamma_n} \sum_{i = 1}^\numobs \w_i \x_i^\top \SigmaMat_n^{-\frac{1}{2}}  \right) \SigmaMat_n^{\frac{1}{2}}(\thetaLS - \thetastar) \\
  &\qquad \qquad + \banditdir^\top \sqrt{\gamma_n} \sum_{i = 1}^\numobs \w_i \error_i \\
&= \banditdir^\top \sqrt{\gamma_n} \sum_{i = 1}^\numobs \w_i \error_i. 
\end{align*}
In order to prove asumptotic normality for $\sqrt{\gammahat_n} \banditdir^\top \cdot \SigmaMat^{\frac{1}{2}} (\thetaODNew - \thetastar)$ it now suffices to find weights $\{ \w_i \}_{i = 1}^n$ such that the weights are stabilized. Concretely, 
we require
\begin{subequations}
\begin{align}
\label{eqn:stabilizing-weights-new}
\gamma_n \sum_{i = 1}^n \w_i \w_i^\top & \stackrel{\smallop{p}}{\longrightarrow} \Id \\
\max_{i \in [n]} \;\; \sqrt{\gamma_n} \| \w_i \|_2 & \stackrel{\smallop{p}}{\longrightarrow} 0. 
\label{eqn:weights-Lindeberg-new} 
\end{align}
\end{subequations}
Putting together the pieces we conclude that for any set tuning parameters 
$\gamma_n$ and $\{ \w_i \}_{i \geq 1}$ satisfying conditions~\eqref{eqn:stabilizing-weights-new} and~\eqref{eqn:weights-Lindeberg-new} we have 
\begin{align*}
\sqrt{\gammahat_n} \banditdir^\top \cdot \SigmaMat^{\frac{1}{2}} (\thetaODNew - \thetastar) \indistrb \Ncal(0, 1). 
\end{align*}
The proof of the second part is immediate from the definition of $\gammahat_n$ and from the fact that 
$\gamma_n \sum_{i = 1}^\numobs \w_i \w_i^\top = \Id - A_n - B_n$ for appropriate positive semidefinite matrices
 $A_n, B_n$; see the the decomposition~\eqref{eqn:Id-minus-qqT-expression} for instance. We prove part (ii) of the claim~\eqref{eqn:normality-and-gammanhat} in Lemma~\ref{lem:bounded-cov-variation}. 
  It now remains to derive conditions which guarantee the two properties~~\eqref{eqn:stabilizing-weights-new}~and~\eqref{eqn:weights-Lindeberg-new}. Towards this end we prove the following lemma.

% Write down the conditions for stability // this puts conditions on the \gamma_n 
%
% \subsection{Conditions for stabilizing weights~\eqref{eqn:stabilizing-weights-new}~-~\eqref{eqn:weights-Lindeberg-new}}
%
%
% A3-a is easy to satisfy by choosing large scaling matrix \Gamma_i. // Values of \Gamma_i which are significantly different from Sn with make \gammahat_n small. More on this in the next section. 
% 
%
%
\begin{lems}
\label{lems:Bandits-stability-refined}
Suppose the sample covariance matrix $\SigmaMat_n$ of a $k$-armed bandit problem satisfies the condition~\ref{eqn:conditions-on-samplecov-bandits} for some non-random scalars $\gamma_n$ and diagonal matrices  $\{\Gamma_{i,n}\}_{i, n \geq 1}$. Then conditions~\eqref{eqn:stabilizing-weights-new}~and~\eqref{eqn:weights-Lindeberg-new} are satisfied. 
\end{lems}
\begin{proof}
Invoking Lemma~\ref{lem:stability-lemma} we see that we only need to verify 
Assumptions (A3)~(a) and (c).  Condition (A3)~(a) is readily satisfied by the second condition in~\eqref{eqn:conditions-on-samplecov-bandits}. To verify condition (A3)~(c), invoking Lemma~\ref{lems:Commutative-lemma} we have 
\begin{align*}
\opnorm{ \Id_\Dim - \Nsum{i} \w_i \x_i^\top
      \XscaleMat_{i,n}^{-\frac{1}{2}} } & \leq \exp \left( - \frac{
      \lambda_{\min}(\Nsum{i} \XscaleMat_{i,n}^{-\frac{1}{2}} \x_i
      \x_i^\top \XscaleMat_{i,n}^{-\frac{1}{2}}) }{\tuneParScaled_\numobs}
    \right) \\
    & \stackrel{(i)}{\leq} \exp \left( - \frac{
      \lambda_{\min}(\Nsum{i} \SigmaMat_n^{-\frac{1}{2}} \x_i
      \x_i^\top \SigmaMat_n^{-\frac{1}{2}}) }{\tuneParScaled_\numobs}
    \right) \\
    &= \exp(-1/\gamma_n) \rightarrow 0.    
\end{align*}
The inequality~(i) above follows from the third condition in~\eqref{eqn:conditions-on-samplecov-bandits} and the last deduction follows from the asusmption that $\gamma_n \rightarrow 0$.
This completes the proof of Lemma~\ref{lems:Bandits-stability-refined}. 
\end{proof}

% Now just take largest possible \gamma_n // Note that we do not need to worry about bias here because of the new definition of the online debiased estimator. 

\subsection*{Control on $\gammahat_n$:}
Next we discuss the scaling of $\gammahat_n$. Note that the scaling of $\gammahat_n$ controls the asymptotic variance of the online debiased estimator $\thetaODNew$ from equation~\eqref{eqn:New-OD-estimator}; smaller values of $\gammahat_n$ leads to a larger asymptotic variance.  The scaling of $\gammahat_n$ is controlled by the following lemma. 
\begin{lems}
\label{lem:bounded-cov-variation}
Give a target direction $\banditdir$, suppose the sample covariance matrix $\SigmaMat_n$ of a $k$-armed bandit problem satiafies condition~\eqref{eqn:bounded-cov-variation}, and assume that the conditios of Lemma~\ref{lems:Bandits-stability-refined} are in force. Then the weights $\w_i$ produced by any $\gamma_n \rightarrow 0$ and scaling matrix 
$\XscaleMat_i$ from~\eqref{eqn:bounded-cov-variation} ensures
\begin{align*}
\lim\sup \sqrt{\frac{\gammahat_n}{\gamma_n}} \geq O_p(\sqrt{c_1})    
\end{align*}
Consequently, the condition~\eqref{eqn:condition-one-1} is satisfied. 
\end{lems}
\begin{proof}
Invoking the proof of Lemma~\ref{lems:Commutative-lemma} we know that the matrix $\sum_{i \leq n} \w_i \x_i^\top \XscaleMat_{i,n}^{-\frac{1}{2}}$ is diagonal. Furthermore, conditions of Lemma~\ref{lems:Bandits-stability-refined} ensures that the condition A3 (c) is satisfied and we have 
\begin{align*}
\banditdir^\top \sum_{i \leq n} \w_i \x_i^\top \XscaleMat_{i,n}^{-\frac{1}{2}} \convprob \banditdir^\top.  
\end{align*} 
Combing the last step with the condition~\eqref{eqn:bounded-cov-variation} and the fact that the matrices $\XscaleMat_{i, n}$, $\SigmaMat_n$
and $\sum_{i \leq n} \w_i \x_i^\top \XscaleMat_{i,n}^{-\frac{1}{2}}$ are diagonal yields
\begin{align*}
\sqrt{\frac{\gammahat_n}{\gamma_n}} &=  \banditdir^\top \sum_{i = 1}^n \w_i \x_i^\top \SigmaMat_n^{-\frac{1}{2}} \banditdir \\ 
&= \banditdir^\top \sum_{i = 1}^n \w_i \x_i^\top \XscaleMat_{i,n}^{-\frac{1}{2}} \XscaleMat_{i,n}^{\frac{1}{2}} \SigmaMat_n^{-\frac{1}{2}} \banditdir \\
& \geq \sqrt{c_1} \cdot  \banditdir^\top \sum_{i = 1}^n \w_i \x_i^\top \XscaleMat_{i,n}^{-\frac{1}{2}} \banditdir\\ 
& \geq \sqrt{c_1} \cdot O_p(1).  
\end{align*}
The last line combined with condition~\eqref{eqn:modified-gamma-n-conditions} ensures that the condition~\eqref{eqn:condition-one-1} is satisfied.
This completes the proof of Lemma~\ref{lem:bounded-cov-variation}. 
\end{proof}

\subsection{Worst case performance:}
Next, we consider the choice of $\gammahat_n$ when the
condition~\eqref{eqn:condition-one-1} is not satisfied. A simple
calculation using the fact that the matrices $\SigmaMat_n$ and $\w_i
\x_i^\top$ are diagonal for a multiarmed bandit problem, we have
\begin{align}
\sqrt{\gammahat_n} \cdot \banditdir^\top \cdot \SigmaMat^{\frac{1}{2}}
(\thetaODNew - \thetastar) &= \banditdir^\top \cdot \left(
\sqrt{\gammahat_n} \cdot \Id - \sqrt{\gamma_n} \sum_{i = 1}^\numobs
\w_i \x_i^\top \SigmaMat_n^{-\frac{1}{2}} \right)
\SigmaMat_n^{\frac{1}{2}}(\thetaLS - \thetastar) + \sqrt{\gamma_n}
\sum_{i = 1}^\numobs \w_i \error_i \notag \\ & = \sqrt{\gammahat_n}
\cdot \banditdir^\top \cdot \SigmaMat_n^{\frac{1}{2}}(\thetaLS -
\thetastar) - \sqrt{\gamma_n} \banditdir^\top \cdot \sum_{i =
  1}^\numobs \w_i \x_i^\top \SigmaMat_n^{-\frac{1}{2}} \banditdir
\cdot \banditdir^\top \SigmaMat_n^{\frac{1}{2}} (\thetaLS -
\thetastar) \notag \\ & \qquad \qquad + \sqrt{\gamma_n} \sum_{i =
  1}^\numobs \banditdir^\top \cdot \w_i \error_i
  \label{eqn:modified-online-debiased}
\end{align}
The third term above converges in distribution to $\mathcal{N}(0, 1)$
following the arguments in the proof of
Theorem~\ref{thm:post-deboasing-correction}.  The first two terms in
the last line correspnds to the bias-term in the modified online
debiased estimator, and we show that they converge to zero in
probablity. The first term in~\eqref{eqn:modified-online-debiased}
above can be bounded as
\begin{align*}
\|\sqrt{\gammahat_n} \cdot \SigmaMat_n^{\frac{1}{2}}(\thetaLS -
\thetastar) \|_2 &= \sqrt{\gammahat_n} \cdot \BigoP(\log(\numobs)) \\
& = \smalloP(1/\sqrt{\log(\numobs)}) \cdot
\BigoP(\sqrt{\log(\numobs)}) = \smalloP(1).
\end{align*}  
Focusing on the second term in~\eqref{eqn:modified-online-debiased}
and using the fact that condition~\eqref{eqn:condition-one-1} is
violated
\begin{align*}
&|\sqrt{\gamma_n} \banditdir^\top \cdot \sum_{i = 1}^\numobs \w_i
  \x_i^\top \SigmaMat_n^{-\frac{1}{2}} \banditdir \cdot
  \banditdir^\top \SigmaMat_n^{\frac{1}{2}} (\thetaLS - \thetastar)|
  \\
 &\stackrel{(i)}{\leq} |\sqrt{\gamma_n} \banditdir^\top \cdot \sum_{i = 1}^\numobs \w_i \x_i^\top \SigmaMat_n^{-\frac{1}{2}} \banditdir| 
 \cdot |\banditdir^\top \SigmaMat_n^{\frac{1}{2}} 
  (\thetaLS - \thetastar)|  \\
&\leq \frac{1}{\sqrt{\log(\numobs)\cdot \log\log(\numobs)}}
\cdot \BigoP(\sqrt{\log(\numobs)})  = \smalloP(1).  
\end{align*} 
where step (i) utilizes the fact that
condition~\eqref{eqn:condition-one-1} is violated, and hence
\begin{align*}
0 \leq \banditdir^\top \sqrt{\gamma_n} \sum_{i = 1}^n \w_i \x_i^\top
\SigmaMat_n^{-\frac{1}{2}} \banditdir \leq
\frac{1}{\sqrt{\log(\numobs) \cdot \log\log(\numobs)}}
\end{align*}
Recall that by construction, in the special case of multi-armed
bandits, the matrix $\w_i\x_i^\top$ is a diagonal matrix with
non-negative entries.

Putting together the pieces we conclude that for a multiarm bandit
problem, under the conditions
\begin{align*}
\gamma_n \rightarrow 0 \qquad \text{and} \qquad \frac{1}{\gamma_n}
\max_{i \in [n]} \cdot \x_i^\top \Gamma_{i,n}^{-1} \x_i \convprob 0,
\end{align*}   
we have
\begin{align*}
\sqrt{\gammahat_n} \banditdir^\top \cdot \SigmaMat^{\frac{1}{2}}
(\thetaODNew - \thetastar) \indistrb \Ncal(0, 1).
\end{align*}

\begin{props}
\label{prop:stability-to-debiasing}
Assume that the satbility condition~\ref{eqn:Cond-Cov-stability} is
satisfied, $\| \x_i\| \leq 1$, and conditions~\ref{assumption:A1}
and~\ref{assumption:A2} are in force. Then condition~(A3) is satisfied
with $\gamma_n = 1$ and $\XscaleMat_i = B_n$.
\end{props}

%%%%%%%%%%%%%%%%%%%%%%%%%%%%%%%%%%%%%%%%%%%%%%%%

\begin{proof}
Combining condition~\ref{assumption:A2} and the stability
condition~\ref{eqn:Cond-Cov-stability} we have $\lambda_{\min}(B_n)
\convprob 0$. This combined with $\|\x_i\| \leq 1$ensures that part
(a) of~\ref{assumption:A3} is satisfied. Next, under the stability
condition~\ref{eqn:Cond-Cov-stability} and using $\w_i = \x_i
\XscaleMat_i^{-\frac{1}{2}} = \x_i B_n^{-\frac{1}{2}}$ we have
\begin{align*}
\W \X \SigmaMat_n^{-\frac{1}{2}}= B_n^{-\frac{1}{2}} \sum_{i = 1}^n
\x_i \x_i^\top \SigmaMat_n^{-\frac{1}{2}} = B_n^{-\frac{1}{2}}
\SigmaMat_n^{\frac{1}{2}} \convprob \Id.
\end{align*}
Observe that the psd matrix $B_n$ is is invertible by
assumption~\ref{assumption:A2} and the stability
condition~\ref{eqn:Cond-Cov-stability}. This verifies part (b) of
assumption~\ref{assumption:A3}. Finally, we note that
\begin{align*}
\sum_{i = 1}^n \w_i \x_i \XscaleMat_i^{-\frac{1}{2}} =
B_n^{-\frac{1}{2}} \sum_{i = 1}^n \x_i \x_i^\top B_n^{-\frac{1}{2}} =
B_n^{-\frac{1}{2}} \SigmaMat_n B_n^{-\frac{1}{2}} \convprob \Id
\end{align*}
This verifies part (c) of Assumption~\ref{assumption:A3}.
\end{proof}

\section{Proofs related to Theorem~\ref{thm:Minimax-Lowerbound} part (b)}
\label{app:minimax-log-bound-proofs}
In this appendix, we prove the bounds~\eqref{eq:lower-a} and~\eqref{eq:lower-b}  that were used in the proof of part (b) of Theorem~\ref{thm:Minimax-Lowerbound}

\subsubsection*{Proof of the bound~\eqref{eq:lower-a}}
Let $v = n/(d-1)$. By Bayes' rule,
\begin{align*}
\frac{\cE_\numobs}{1 - \cE_\numobs}
&= \prod_{u=1}^{d-1}\frac{1}{\sqrt{2\pi}} \int_\reals \exp\left(-\frac{1}{2} \sum_{w=1}^v (a_{u,w} + b_w \varphi - \y_{u,w})^2 + \frac{1}{2} \sum_{i=1}^n \y_{u,w}^2 - \frac{\varphi^2}{2}\right) d\varphi \\
&= \prod_{u=1}^{d-1}\exp\left(\sum_{w=1}^v a_{u,w} \y_{u,w} - \frac{1}{2} \sum_{w=1}^v a_{u,w}^2 + \frac{m_{u,v}^2}{2 d_v} - \frac{1}{2} \log d_v\right) \\ 
&\mydefn \exp\left(\sum_{u=1}^{d-1} Z_{u,w}\right)\,.
\end{align*}
Note that $\cE_\numobs(1 - \cE_\numobs) = \Omega_p(1)$ is equivalent
to $|\sum_{u=1}^{d-1} Z_{u,w}| = O_p(1)$, which we now establish.  To begin,
\begin{align*}
\frac{m_{u,w+1}^2}{d_{w+1}}
%&= \frac{(m_{u,w} + b_{w+1}(y_{u,w+1} - a_{u,w+1}))^2}{d_{w+1}} \\
%&= \frac{m_{u,w}^2}{d_w} + m_{u,w}^2\left(\frac{1}{d_{w+1}} - \frac{1}{d_w}\right) + \frac{b_{w+1}^2(y_{u,w+1} - a_{u,w+1})^2}{d_{w+1}} + \frac{2 m_{u,w} b_{w+1}(y_{u,w+1} - a_{u,w+1})}{d_{w+1}} \\
%&= \frac{m_{u,w}^2}{d_w} + m_{u,w}^2\left(\frac{1}{d_{w+1}} - \frac{1}{d_w}\right) + \frac{b_{w+1}^2(y_{u,w+1} - a_{u,w+1})^2}{d_{w+1}} - \frac{2 a_{u,w+1} d_w}{d_{w+1}} (y_{u,w+1} - a_{u,w+1}) \\
%&= \frac{m_{u,w}^2}{d_w} - \frac{m_{u,w}^2b_{w+1}^2}{d_w d_{w+1}} + \frac{b_{w+1}^2(y_{u,w+1} - a_{u,w+1})^2}{d_{w+1}} - \frac{2 a_{u,w+1} d_w}{d_{w+1}} (y_{u,w+1} - a_{u,w+1}) \\
%&= \frac{m_{u,w}^2}{d_w} - \frac{a_{u,w+1}^2 d_w}{d_{w+1}} + \frac{b_{w+1}^2(y_{u,w+1} - a_{u,w+1})^2}{d_{w+1}} - \frac{2 a_{u,w+1} d_w}{d_{w+1}} (y_{u,w+1} - a_{u,w+1}) \\ 
&= \frac{m_{u,w}^2}{d_w} + \frac{a_{u,w+1}^2 d_w}{d_{w+1}} + \frac{b_{w+1}^2(y_{u,w+1} - a_{u,w+1})^2}{d_{w+1}} - \frac{2 a_{u,w+1} d_w}{d_{w+1}} y_{u,w+1} \,. 
\end{align*}
Telescoping shows that
\begin{align*}
Z_{u,v} = 
\underbracket{-\frac{1}{2} \sum_{w=1}^v \frac{b_w^2 a_{u,w}^2}{d_w} }_{\textrm{A}_{u,v}} 
+ \underbracket{\sum_{w=1}^v \frac{b_w^2 a_{u,w} y_{u,w}}{d_w}}_{\textrm{B}_{u,v}} 
+ \underbracket{\frac{1}{2} \sum_{w=1}^v \frac{b_w^2 (y_{u,w} - a_{u,w})^2}{d_w} - \frac{1}{2} \log d_v}_{\textrm{C}_{u,v}} \,. 
\end{align*}
By definition, $b_w = w^{-1/4} / d^{1/2}$ and $d_w = \Omega(w^{1/2}/d)$. Furthermore,
\begin{align*}
    \E[a_{u,w}^2] = b_w^2 \left(\left(\frac{d_{w-1}-1}{d_{w-1}}\right)^2 + \frac{d_{w-1}-1}{d_{w-1}^2}\right) = O(b_w^2)
    \text{ and } \E[y_{u,w}^2] = O(1)\,.
\end{align*}
Therefore, by Cauchy-Schwarz and Chebyshev's inequality, $|\sum_{u=1}^{d-1} \textrm{A}_{u,v}| = O_p(1)$ 
and $|\sum_{u=1}^{d-1} \textrm{B}_{u,v}| = O_p(1)$ and
\begin{align*}
\left|\sum_{u=1}^{d-1}\textrm{C}_{u,v}\right| 
&= \frac{1}{2} \left|\sum_{u=1}^{d-1}\sum_{w=1}^v \frac{b_w^2 (y_{u,w} - a_{u,w})^2}{d_w} - \log d_v\right| \\
&= \frac{1}{2} \left|\sum_{u=1}^{d-1}\sum_{w=1}^v \frac{b_w^2 (b_w + \error_{u,w})^2}{d_w} - \log d_v\right| \\
&= \frac{1}{2} \left|\sum_{u=1}^{d-1}\sum_{w=1}^v \frac{b_w^4 + b_w^2 \error_{u,w}^2 + 2b_w^3 \error_{u,w}}{d_w} - \log d_v\right| \\
&= \frac{1}{2} \left|\sum_{u=1}^{d-1}\sum_{w=1}^v \frac{b_w^2}{d_w} - \log d_v\right| + O_p(1) \\
&= O_p(1)\,,
\end{align*}
where the second last equality follows from Chebyshev's inequality and the last by an elementary integral approximation.
Therefore $|\sum_{u=1}^{d - 1} Z_{u,v}| = O_p(1)$, which
completes the proof of (\ref{eq:lower-a}).
\subsubsection*{Proof of bound \ref{eq:lower-b}} 
By definition,
\begin{align*}
\SigmaMat_n
= \left[\begin{matrix}
  \sum_{w=1}^v b_w^2 & 0 & \cdots & \sum_{w=1}^v a_{1,w} b_w \\
  0 & \ddots & 0 & \vdots \\
  \vdots & 0 & \sum_{w=1}^v b_w^2 & \sum_{w=1}^v a_{d-1,w} b_w \\
  \sum_{w=1}^v a_{1,w} b_w & \cdots & \sum_{w=1}^v a_{d-1,w} b_w & \sum_{u=1}^{d-1} \sum_{w=1}^v a_{u,v}^2 
\end{matrix}\right]
\end{align*}
Using the formula for the inverse of a block matrix,
\begin{align}
1/\norm{e_d}^2_{\SigmaMat_n^{-1}}
&=1/(\SigmaMat_n^{-1})_{d,d} \nonumber \\
&= \sum_{u=1}^{d-1} \sum_{w=1}^v a_{u,w}^2 - \sum_{u=1}^{d-1} \frac{\left(\sum_{w=1}^v a_{u,w} b_w\right)^2}{\sum_{w=1}^v b_w^2} \nonumber \\ 
&= \sum_{u=1}^{d-1} \sum_{w=1}^v b_w^2\left(\thetastarScalar_u + a_{u,w}/b_w\right)^2 - \sum_{u=1}^{d-1} \left(\frac{\sum_{w=1}^v b_w^2(\thetastarScalar_u + a_{u,w}/b_{u,w})}{\sqrt{\sum_{i=1}^n b_w^2}} \right)^2 \,. 
\label{eq:lower:2-part}
\end{align}
By the definition of $a_{u,w}$,
\begin{align*}
\thetastarScalar_u + \frac{a_{u,w}}{b_w} 
= \frac{\thetastarScalar_u - \sum_{s=1}^{w-1} b_s \error_{u,s}}{d_{w-1}} \,.
\end{align*}
Therefore, 
\begin{align}
\sum_{u=1}^{d-1} \sum_{w=1}^v b_w^2(\thetastarScalar_u + a_{u,w}/b_w)^2 
&= \sum_{u=1}^{d-1} \sum_{w=1}^v \frac{b_w^2}{d_{w-1}^2} \left(\thetastarScalar_u - \sum_{s=1}^{w-1} b_s \epsilon_{u,s}\right)^2 \nonumber \\
&\geq 
\sum_{u=1}^{d-1} \sum_{w=1}^v \frac{b_w^2}{d_{w-1}^2} \left(\sum_{s=1}^{w-1} b_s \epsilon_{u,s}\right)^2 
- 2 \sum_{u=1}^{d-1} \thetastarScalar_u \sum_{w=1}^v \frac{b_w^2}{d_{w-1}^2} \sum_{s=1}^{w-1} b_s \epsilon_{u,s} \nonumber \\ 
&= \sum_{u=1}^{d-1} \sum_{w=1}^v \frac{b_w^2}{d_{w-1}^2} \left(\sum_{s=1}^{w-1} b_s \epsilon_{u,s}\right)^2 - 2 \sum_{u=1}^{d-1} \theta^*_u \sum_{s=1}^{v-1} b_s \epsilon_{u,s} \sum_{w=s+1}^v \frac{b_w^2}{d_{w-1}^2} \nonumber \\
&\geq \sum_{u=1}^{d-1}\sum_{w=1}^v \frac{b_w^2}{d_{w-1}^2} \left(\sum_{s=1}^{w-1} b_s \epsilon_{u,s}\right)^2 - O_p(d)\,,
\label{eq:lower3}
\end{align}
where the final inequality follows because $\theta^*_2$ is distributed like a standard Gaussian and
\begin{align}
    \sum_{u=1}^{d-1} \sum_{s=1}^{v-1} b_s \epsilon_{u,s} \sum_{w=s+1}^n \frac{b_w^2}{d_{w-1}^2} &\stackrel{d}=
    \mathcal N\left(0, \beta^2\right)\,, &
    \beta^2 &= \sum_{u=1}^{d-1} \sum_{s=1}^{v-1} b_s^2 \left(\sum_{w=s+1}^n \frac{b_w^2}{d_{w-1}^2}\right)^2 = O(d) \,.
    \label{eq:lower:beta}
\end{align}
The last calculation needed to lower bound the first term in \cref{eq:lower:2-part} is to bound
the sum in \cref{eq:lower3}:
\begin{align}
    \sum_{u=1}^{d-1} \sum_{w=1}^v \frac{b_w^2}{d_{w-1}^2} \left(\sum_{s=1}^{w-1} b_s \epsilon_{u,s}\right)^2
    &= \sum_{u=1}^{d-1} \sum_{w=1}^{v-1} \sum_{s=1}^{n-1} \epsilon_{u,w} \epsilon_{u,s} A_{w,s}\,,
    \label{eq:chaos}
\end{align}
where
\begin{align*}
A_{w,s} = b_w b_s \sum_{t=\max(w,s)+1}^v \frac{b_t^2}{d_{t-1}^2}\,.
\end{align*}
Viewing $A$ as an $(n-1)\times(n-1)$ matrix, the expression in
\cref{eq:chaos} is a sum of Gaussian chaos'. A simple calculation
using the facts that $b_i = i^{-1/4} / \sqrt{d}$ and $d_{i-1}^2 = \Omega(i/d^2)$ shows that
\begin{align*}
    \tr(A) &= \Omega(\log(v/d^2)) &
    \norm{A}_F &= O(\sqrt{\log(v)}) \,.
\end{align*}
where $\norm{A}_F$ is Frobenius norm of matrix $A$. By the Hansen--Write inequality \citep[Theorem 6.2.1]{vershynin2018high},
\begin{align*}
\E\left[\sum_{w=1}^{v-1} \sum_{s=1}^{v-1} \epsilon_{u,w} \epsilon_{u,s} A_{w,s}\right] &= \tr(A_{w,s}) = \Omega(\log(v/d^2)) \\
  \V\left[\sum_{w=1}^{v-1} \sum_{s=1}^{v-1} \epsilon_{u,w} \epsilon_{u,s} A_{w,s}\right] &= O(\norm{A}_F^2) = O(\log(v)) \,.
\end{align*}
Therefore, by Chebyshev's inequality, the right-hand side of \cref{eq:chaos} is $\Omega_P(d \log(v/d^2))$.
Moving to the second term in \cref{eq:lower:2-part},
\begin{align*}
\sum_{u=1}^{d-1} \left(\frac{\sum_{w=1}^v b_w^2(\thetastarScalar_u +
  a_{u,w}/b_w)}{\sqrt{\sum_{w=1}^v b_w^2}}\right)^2 
&= \sum_{u=1}^{d-1} \left(\frac{\sum_{w=1}^v \frac{b_w^2}{d_{w-1}}\left(\thetastarScalar_u
  - \sum_{s=1}^{w-1} b_s \error_{u,s}\right)}{\sqrt{\sum_{w=1}^v
    b_w^2}}\right)^2 \\
&\leq \frac{2 \sum_{u=1}^{d-1} (\theta_u^*)^2 \left(\sum_{w=1}^v \frac{b_w^2}{d_{w-1}}\right)^2
+ 2 \left(\sum_{w=1}^v \frac{b_w^2}{d_{w-1}} \sum_{s=1}^{w-1} b_s \epsilon_{u,s}\right)^2}{\sum_{w=1}^v b_w^2} \\
&= O_p(d)\,,
\end{align*}
where in the final inequality we used the fact that $\theta^*_u$ is distributed like a standard
Gaussian for $u \in \{1,\ldots,d-1\}$ and $(\sum_{w=1}^v b_w^2/d_{w-1})^2 / (\sum_{w=1}^v b_w^2) = O(1)$ to bound the first
term while the second term was bounded using the same calculation as in \cref{eq:lower:beta}.
Therefore
\begin{align*}
    1/\norm{e_d}^2_{\SigmaMat_n^{-1}} = \Omega_P\left(\log(v/d^2)\right) = \Omega_P\left(\log(n/d^3)\right)\,.
\end{align*}
which completes the proof of \cref{eq:lower-b}.

%%%%%%%%%%%% DO BIBLIOGRAPHY
%%%%%%%%%%%% %%%%%%%%%%%%%%%%%%%%%%%%%%%%%%%%%%%%%%%%%%
%\bibliography{bibfiles/koulikpaper}
\bibliography{koulikpaper}

\begin{thebibliography}{31}
% BibTex style file: imsart-nameyear.bst, 2017-11-03
% Default style options (sort=1,type=nameyear).
% Used options (sort=1,type=nameyear).

\bibitem[\protect\citeauthoryear{{\AA}str{\"o}m}{2012}]{aastrom2012introduction}
\begin{bbook}[author]
\bauthor{\bsnm{{\AA}str{\"o}m},~\bfnm{Karl~J}\binits{K.~J.}}
(\byear{2012}).
\btitle{Introduction to stochastic control theory}.
\bpublisher{Courier Corporation}.
\end{bbook}
\endbibitem

\bibitem[\protect\citeauthoryear{Abbasi-Yadkori, P{\'a}l and
  Szepesv{\'a}ri}{2011}]{abbasi2011online}
\begin{barticle}[author]
\bauthor{\bsnm{Abbasi-Yadkori},~\bfnm{Yasin}\binits{Y.}},
  \bauthor{\bsnm{P{\'a}l},~\bfnm{David}\binits{D.}} \AND
  \bauthor{\bsnm{Szepesv{\'a}ri},~\bfnm{Csaba}\binits{C.}}
(\byear{2011}).
\btitle{Online least squares estimation with self-normalized processes: An
  application to bandit problems}.
\bjournal{arXiv preprint arXiv:1102.2670}.
\end{barticle}
\endbibitem

\bibitem[\protect\citeauthoryear{Boucheron, Lugosi and
  Massart}{2013}]{BouLugMas13}
\begin{bbook}[author]
\bauthor{\bsnm{Boucheron},~\bfnm{S.}\binits{S.}},
  \bauthor{\bsnm{Lugosi},~\bfnm{G.}\binits{G.}} \AND
  \bauthor{\bsnm{Massart},~\bfnm{P.}\binits{P.}}
(\byear{2013}).
\btitle{Concentration inequalities: A nonasymptotic theory of independence}.
\bpublisher{Oxford University Press}, \baddress{Oxford, UK}.
\end{bbook}
\endbibitem

\bibitem[\protect\citeauthoryear{Box et~al.}{2015}]{box2015time}
\begin{bbook}[author]
\bauthor{\bsnm{Box},~\bfnm{George~EP}\binits{G.~E.}},
  \bauthor{\bsnm{Jenkins},~\bfnm{Gwilym~M}\binits{G.~M.}},
  \bauthor{\bsnm{Reinsel},~\bfnm{Gregory~C}\binits{G.~C.}} \AND
  \bauthor{\bsnm{Ljung},~\bfnm{Greta~M}\binits{G.~M.}}
(\byear{2015}).
\btitle{Time series analysis: forecasting and control}.
\bpublisher{John Wiley \& Sons}.
\end{bbook}
\endbibitem

\bibitem[\protect\citeauthoryear{Deshpande, Javanmard and
  Mehrabi}{2019}]{deshpande2019online}
\begin{barticle}[author]
\bauthor{\bsnm{Deshpande},~\bfnm{Yash}\binits{Y.}},
  \bauthor{\bsnm{Javanmard},~\bfnm{Adel}\binits{A.}} \AND
  \bauthor{\bsnm{Mehrabi},~\bfnm{Mohammad}\binits{M.}}
(\byear{2019}).
\btitle{Online debiasing for adaptively collected high-dimensional data}.
\bjournal{arXiv preprint arXiv:1911.01040}.
\end{barticle}
\endbibitem

\bibitem[\protect\citeauthoryear{Deshpande
  et~al.}{2018}]{deshpande2017accurate}
\begin{binproceedings}[author]
\bauthor{\bsnm{Deshpande},~\bfnm{Yash}\binits{Y.}},
  \bauthor{\bsnm{Mackey},~\bfnm{Lester}\binits{L.}},
  \bauthor{\bsnm{Syrgkanis},~\bfnm{Vasilis}\binits{V.}} \AND
  \bauthor{\bsnm{Taddy},~\bfnm{Matt}\binits{M.}}
(\byear{2018}).
\btitle{Accurate Inference for Adaptive Linear Models}.
In \bbooktitle{International Conference on Machine Learning}
\bvolume{35}
\bpages{1194--1203}.
\end{binproceedings}
\endbibitem

\bibitem[\protect\citeauthoryear{Dickey and
  Fuller}{1979}]{dickey1979distribution}
\begin{barticle}[author]
\bauthor{\bsnm{Dickey},~\bfnm{David~A}\binits{D.~A.}} \AND
  \bauthor{\bsnm{Fuller},~\bfnm{Wayne~A}\binits{W.~A.}}
(\byear{1979}).
\btitle{Distribution of the estimators for autoregressive time series with a
  unit root}.
\bjournal{Journal of the American statistical association}
\bvolume{74}
\bpages{427--431}.
\end{barticle}
\endbibitem

\bibitem[\protect\citeauthoryear{Dvoretzky
  et~al.}{1972}]{dvoretzky1972asymptotic}
\begin{binproceedings}[author]
\bauthor{\bsnm{Dvoretzky},~\bfnm{Aryeh}\binits{A.}} \betal{et~al.}
(\byear{1972}).
\btitle{Asymptotic normality for sums of dependent random variables}.
In \bbooktitle{Proceedings of the Sixth Berkeley Symposium on Mathematical
  Statistics and Probability, Volume 2: Probability Theory}.
\bpublisher{The Regents of the University of California}.
\end{binproceedings}
\endbibitem

\bibitem[\protect\citeauthoryear{Fontaine et~al.}{2019}]{fontaine2019online}
\begin{barticle}[author]
\bauthor{\bsnm{Fontaine},~\bfnm{Xavier}\binits{X.}},
  \bauthor{\bsnm{Perrault},~\bfnm{Pierre}\binits{P.}},
  \bauthor{\bsnm{Valko},~\bfnm{Michal}\binits{M.}} \AND
  \bauthor{\bsnm{Perchet},~\bfnm{Vianney}\binits{V.}}
(\byear{2019}).
\btitle{Online A-Optimal Design and Active Linear Regression}.
\bjournal{arXiv preprint arXiv:1906.08509}.
\end{barticle}
\endbibitem

\bibitem[\protect\citeauthoryear{Goodwin and Payne}{1977}]{goodwin1977dynamic}
\begin{bbook}[author]
\bauthor{\bsnm{Goodwin},~\bfnm{GC}\binits{G.}} \AND
  \bauthor{\bsnm{Payne},~\bfnm{RL}\binits{R.}}
(\byear{1977}).
\btitle{Dynamic system identification experiment design and data analysis.}
\bseries{Mathematics in Science and Engineering, Volume 136}.
\bpublisher{Elsevier}.
\end{bbook}
\endbibitem

\bibitem[\protect\citeauthoryear{Hadad et~al.}{2019}]{hadad2019confidence}
\begin{barticle}[author]
\bauthor{\bsnm{Hadad},~\bfnm{Vitor}\binits{V.}},
  \bauthor{\bsnm{Hirshberg},~\bfnm{David~A}\binits{D.~A.}},
  \bauthor{\bsnm{Zhan},~\bfnm{Ruohan}\binits{R.}},
  \bauthor{\bsnm{Wager},~\bfnm{Stefan}\binits{S.}} \AND
  \bauthor{\bsnm{Athey},~\bfnm{Susan}\binits{S.}}
(\byear{2019}).
\btitle{Confidence intervals for policy evaluation in adaptive experiments}.
\bjournal{arXiv preprint arXiv:1911.02768}.
\end{barticle}
\endbibitem

\bibitem[\protect\citeauthoryear{Howard et~al.}{}]{howard2018uniform}
\begin{barticle}[author]
\bauthor{\bsnm{Howard},~\bfnm{Steven~R}\binits{S.~R.}},
  \bauthor{\bsnm{Ramdas},~\bfnm{Aaditya}\binits{A.}},
  \bauthor{\bsnm{McAuliffe},~\bfnm{Jon}\binits{J.}} \AND
  \bauthor{\bsnm{Sekhon},~\bfnm{Jasjeet}\binits{J.}}
\btitle{Time-uniform, nonparametric, non-asymptotic confidence sequences}.
\bjournal{The Annals of Statistics}
\bpages{To appear}.
\end{barticle}
\endbibitem

\bibitem[\protect\citeauthoryear{Jamieson et~al.}{2014}]{jamieson2014lil}
\begin{binproceedings}[author]
\bauthor{\bsnm{Jamieson},~\bfnm{Kevin}\binits{K.}},
  \bauthor{\bsnm{Malloy},~\bfnm{Matthew}\binits{M.}},
  \bauthor{\bsnm{Nowak},~\bfnm{Robert}\binits{R.}} \AND
  \bauthor{\bsnm{Bubeck},~\bfnm{S{\'e}bastien}\binits{S.}}
(\byear{2014}).
\btitle{lil’ucb: An optimal exploration algorithm for multi-armed bandits}.
In \bbooktitle{Conference on Learning Theory}
\bvolume{27}
\bpages{423--439}.
\end{binproceedings}
\endbibitem

\bibitem[\protect\citeauthoryear{Kaufmann and
  Koolen}{2018}]{kaufmann2018mixture}
\begin{barticle}[author]
\bauthor{\bsnm{Kaufmann},~\bfnm{Emilie}\binits{E.}} \AND
  \bauthor{\bsnm{Koolen},~\bfnm{Wouter}\binits{W.}}
(\byear{2018}).
\btitle{Mixture martingales revisited with applications to sequential tests and
  confidence intervals}.
\bjournal{arXiv preprint arXiv:1811.11419}.
\end{barticle}
\endbibitem

\bibitem[\protect\citeauthoryear{Lai}{1994}]{lai1994asymptotic}
\begin{barticle}[author]
\bauthor{\bsnm{Lai},~\bfnm{Tze~Leung}\binits{T.~L.}}
(\byear{1994}).
\btitle{Asymptotic properties of nonlinear least squares estimates in
  stochastic regression models}.
\bjournal{The Annals of Statistics}
\bvolume{22}
\bpages{1917--1930}.
\end{barticle}
\endbibitem

\bibitem[\protect\citeauthoryear{Lai and Robbins}{1979}]{lai1979adaptive}
\begin{barticle}[author]
\bauthor{\bsnm{Lai},~\bfnm{T~L}\binits{T.~L.}} \AND
  \bauthor{\bsnm{Robbins},~\bfnm{Herbert}\binits{H.}}
(\byear{1979}).
\btitle{Adaptive design and stochastic approximation}.
\bjournal{The Annals of Statistics}
\bvolume{7}
\bpages{1196--1221}.
\end{barticle}
\endbibitem

\bibitem[\protect\citeauthoryear{Lai, Robbins and Wei}{1979}]{lai1979strong}
\begin{barticle}[author]
\bauthor{\bsnm{Lai},~\bfnm{T~L\_etc}\binits{T.~L.}},
  \bauthor{\bsnm{Robbins},~\bfnm{Herbert}\binits{H.}} \AND
  \bauthor{\bsnm{Wei},~\bfnm{C~Zi}\binits{C.~Z.}}
(\byear{1979}).
\btitle{Strong consistency of least squares estimates in multiple regression
  II}.
\bjournal{Journal of Multivariate Analysis}
\bvolume{9}
\bpages{343--361}.
\end{barticle}
\endbibitem

\bibitem[\protect\citeauthoryear{Lai and Wei}{1982}]{lai1982least}
\begin{barticle}[author]
\bauthor{\bsnm{Lai},~\bfnm{Tze~Leung}\binits{T.~L.}} \AND
  \bauthor{\bsnm{Wei},~\bfnm{Ching~Zong}\binits{C.~Z.}}
(\byear{1982}).
\btitle{Least squares estimates in stochastic regression models with
  applications to identification and control of dynamic systems}.
\bjournal{The Annals of Statistics}
\bvolume{10}
\bpages{154--166}.
\end{barticle}
\endbibitem

\bibitem[\protect\citeauthoryear{Lattimore}{2023}]{Lat23}
\begin{barticle}[author]
\bauthor{\bsnm{Lattimore},~\bfnm{T.}\binits{T.}}
(\byear{2023}).
\btitle{A Lower Bound for Linear and Kernel Regression with Adaptive
  Covariates}.
\bjournal{arXiv}.
\end{barticle}
\endbibitem

\bibitem[\protect\citeauthoryear{Lattimore and
  Szepesv{\'a}ri}{2020}]{lattimore2020bandit}
\begin{bbook}[author]
\bauthor{\bsnm{Lattimore},~\bfnm{Tor}\binits{T.}} \AND
  \bauthor{\bsnm{Szepesv{\'a}ri},~\bfnm{Csaba}\binits{C.}}
(\byear{2020}).
\btitle{Bandit algorithms}.
\bpublisher{Cambridge University Press}.
\end{bbook}
\endbibitem

\bibitem[\protect\citeauthoryear{Nie et~al.}{2018}]{nie2018adaptively}
\begin{binproceedings}[author]
\bauthor{\bsnm{Nie},~\bfnm{Xinkun}\binits{X.}},
  \bauthor{\bsnm{Tian},~\bfnm{Xiaoying}\binits{X.}},
  \bauthor{\bsnm{Taylor},~\bfnm{Jonathan}\binits{J.}} \AND
  \bauthor{\bsnm{Zou},~\bfnm{James}\binits{J.}}
(\byear{2018}).
\btitle{Why adaptively collected data have negative bias and how to correct for
  it}.
In \bbooktitle{International Conference on Artificial Intelligence and
  Statistics}
\bvolume{84}
\bpages{1261--1269}.
\end{binproceedings}
\endbibitem

\bibitem[\protect\citeauthoryear{Shin, Ramdas and
  Rinaldo}{2019a}]{shin2019sample}
\begin{binproceedings}[author]
\bauthor{\bsnm{Shin},~\bfnm{Jaehyeok}\binits{J.}},
  \bauthor{\bsnm{Ramdas},~\bfnm{Aaditya}\binits{A.}} \AND
  \bauthor{\bsnm{Rinaldo},~\bfnm{Alessandro}\binits{A.}}
(\byear{2019}a).
\btitle{Are sample means in multi-armed bandits positively or negatively
  biased?}
In \bbooktitle{Advances in Neural Information Processing Systems}
\bvolume{32}
\bpages{7102--7111}.
\end{binproceedings}
\endbibitem

\bibitem[\protect\citeauthoryear{Shin, Ramdas and
  Rinaldo}{2019b}]{shin2019bias}
\begin{barticle}[author]
\bauthor{\bsnm{Shin},~\bfnm{Jaehyeok}\binits{J.}},
  \bauthor{\bsnm{Ramdas},~\bfnm{Aaditya}\binits{A.}} \AND
  \bauthor{\bsnm{Rinaldo},~\bfnm{Alessandro}\binits{A.}}
(\byear{2019}b).
\btitle{On the bias, risk and consistency of sample means in multi-armed
  bandits}.
\bjournal{arXiv preprint arXiv:1902.00746}.
\end{barticle}
\endbibitem

\bibitem[\protect\citeauthoryear{Thompson}{1933}]{thompson1933likelihood}
\begin{barticle}[author]
\bauthor{\bsnm{Thompson},~\bfnm{William~R}\binits{W.~R.}}
(\byear{1933}).
\btitle{On the likelihood that one unknown probability exceeds another in view
  of the evidence of two samples}.
\bjournal{Biometrika}
\bvolume{25}
\bpages{285--294}.
\end{barticle}
\endbibitem

\bibitem[\protect\citeauthoryear{Vershynin}{2018}]{vershynin2018high}
\begin{bbook}[author]
\bauthor{\bsnm{Vershynin},~\bfnm{Roman}\binits{R.}}
(\byear{2018}).
\btitle{High-dimensional probability: An introduction with applications in data
  science}
\bvolume{47}.
\bpublisher{Cambridge university press}.
\end{bbook}
\endbibitem

\bibitem[\protect\citeauthoryear{Villar, Bowden and
  Wason}{2015}]{villar2015multi}
\begin{barticle}[author]
\bauthor{\bsnm{Villar},~\bfnm{Sof{\'\i}a~S}\binits{S.~S.}},
  \bauthor{\bsnm{Bowden},~\bfnm{Jack}\binits{J.}} \AND
  \bauthor{\bsnm{Wason},~\bfnm{James}\binits{J.}}
(\byear{2015}).
\btitle{Multi-armed bandit models for the optimal design of clinical trials:
  benefits and challenges}.
\bjournal{Statistical science: a review journal of the Institute of
  Mathematical Statistics}
\bvolume{30}
\bpages{199}.
\end{barticle}
\endbibitem

\bibitem[\protect\citeauthoryear{Wainwright}{2019}]{wainwright2019high}
\begin{bbook}[author]
\bauthor{\bsnm{Wainwright},~\bfnm{Martin~J}\binits{M.~J.}}
(\byear{2019}).
\btitle{High-dimensional statistics: A non-asymptotic viewpoint}.
\bpublisher{Cambridge University Press}.
\end{bbook}
\endbibitem

\bibitem[\protect\citeauthoryear{White}{1958}]{white1958limiting}
\begin{barticle}[author]
\bauthor{\bsnm{White},~\bfnm{John~S}\binits{J.~S.}}
(\byear{1958}).
\btitle{The limiting distribution of the serial correlation coefficient in the
  explosive case}.
\bjournal{The Annals of Mathematical Statistics}
\bpages{1188--1197}.
\end{barticle}
\endbibitem

\bibitem[\protect\citeauthoryear{Xu, Qin and Liu}{2013}]{xu2013estimation}
\begin{binproceedings}[author]
\bauthor{\bsnm{Xu},~\bfnm{Min}\binits{M.}},
  \bauthor{\bsnm{Qin},~\bfnm{Tao}\binits{T.}} \AND
  \bauthor{\bsnm{Liu},~\bfnm{Tie-Yan}\binits{T.-Y.}}
(\byear{2013}).
\btitle{Estimation bias in multi-armed bandit algorithms for search
  advertising}.
In \bbooktitle{Advances in Neural Information Processing Systems}
\bvolume{26}
\bpages{2400--2408}.
\end{binproceedings}
\endbibitem

\bibitem[\protect\citeauthoryear{Zhang, Janson and
  Murphy}{2020}]{zhang2020inference}
\begin{barticle}[author]
\bauthor{\bsnm{Zhang},~\bfnm{Kelly~W}\binits{K.~W.}},
  \bauthor{\bsnm{Janson},~\bfnm{Lucas}\binits{L.}} \AND
  \bauthor{\bsnm{Murphy},~\bfnm{Susan~A}\binits{S.~A.}}
(\byear{2020}).
\btitle{Inference for Batched Bandits}.
\bjournal{arXiv preprint arXiv:2002.03217}.
\end{barticle}
\endbibitem

\bibitem[\protect\citeauthoryear{Zhang, Janson and Murphy}{}]{zhangstatistical}
\begin{barticle}[author]
\bauthor{\bsnm{Zhang},~\bfnm{Kelly~W}\binits{K.~W.}},
  \bauthor{\bsnm{Janson},~\bfnm{Lucas}\binits{L.}} \AND
  \bauthor{\bsnm{Murphy},~\bfnm{Susan~A}\binits{S.~A.}}
\btitle{Statistical Inference with M-Estimators on Adaptively Collected Data}.
\end{barticle}
\endbibitem

\end{thebibliography}

%%%%%%%%%%%%%%%%%%%%%%%%%%%%%%%%%%%%%%%%%%%%%%%%%%%%%%%%%%%%%%%%%%%%%%
\end{document}